\documentclass[a4paper,twoside,11pt,english,intlimits]{article}

\usepackage[utf8]{inputenc}
\usepackage[T1]{fontenc}
\usepackage[english]{babel}

\usepackage{amsmath,amsthm,amssymb,stmaryrd}
\usepackage{mathtools}
\usepackage{mathrsfs}
\allowdisplaybreaks

\usepackage[ttscale=.875]{libertine}
\usepackage[libertine]{newtxmath}

\usepackage{setspace}
\usepackage{fancyhdr}
\usepackage{titlesec}
\usepackage{etoolbox}

\usepackage{graphicx}
\usepackage{xcolor}
\usepackage{tikz}
\usepackage[all]{xy}
\xyoption{rotate,color}
\usepackage{catex}

\usepackage{algorithm2e}
\usepackage[numbers,sort&compress]{natbib}
\usepackage{hyperref}

\fancypagestyle{plain}{%
  }

\newcommand{\periodafter}[1]{\ifstrempty{#1}{}{#1.}}
\titleformat{\section}[block]{\scshape\filcenter\Large\boldmath}{\thesection.}{.5em}{}
\titleformat{\subsection}[block]{\bfseries\filcenter\large\boldmath}{\thesubsection.}{.5em}{}
\titleformat{\subsubsection}[runin]{\bfseries\boldmath}{\thesubsubsection.}{.5em}{\periodafter}
\titlespacing{\subsubsection}{0pt}{\topsep}{.5em}

\newtheoremstyle{ntheorem}%
  {\topsep}{\topsep}{\itshape}{0pt}{\bfseries}{.}{.5em}%
  {\thmnumber{#2.\hspace{.5em}}\thmname{#1}\thmnote{ (#3)}}

\newtheoremstyle{ndefinition}%
  {\topsep}{\topsep}{\normalfont}{0pt}{\bfseries}{.}{.5em}%
  {\thmnumber{#2.\hspace{.5em}}\thmname{#1}\thmnote{ (#3)}}

\theoremstyle{ntheorem}
\newtheorem{theorem}[subsubsection]{Theorem}
\newtheorem{proposition}[subsubsection]{Proposition}
\newtheorem{lemma}[subsubsection]{Lemma}

\theoremstyle{ndefinition}
\newtheorem{example}[subsubsection]{Example}
\newtheorem{remark}[subsubsection]{Remark}

\makeatletter
\def\@equationname{equation}
\newenvironment{eqn}[1]{%
  \def\mymathenvironmenttouse{#1}%
  \ifx\mymathenvironmenttouse\@equationname%
    \refstepcounter{subsubsection}%
  \else
    \patchcmd{\@arrayparboxrestore}{equation}{subsubsection}{}{}%
    \patchcmd{\print@eqnum}{equation}{subsubsection}{}{}%
    \patchcmd{\incr@eqnum}{equation}{subsubsection}{}{}%
  \fi
  \csname\mymathenvironmenttouse\endcsname%
}{%
  \ifx\mymathenvironmenttouse\@equationname%
    \tag{\thesubsubsection}%
  \fi
  \csname end\mymathenvironmenttouse\endcsname%
}
\makeatother

\UseTips
\SelectTips{eu}{11}

\newcommand{\NF}[2]{\mathord{\downarrow}^{#1}_{#2}}
\newcommand{\NP}[1]{\widehat{#1}}

\newcommand{\Exte}[1]{\mathcal{E}(#1)}
\newcommand{\Int}{\mathsf{Int}}
\newcommand{\monex}[1]{m{#1}}

\newcommand{\strat}[1]{F_s^{#1}}

\newcommand{\fl}{\rightarrow}
\newcommand{\fll}{\longrightarrow}
\newcommand{\ofl}[1]{\overset{\displaystyle #1}{\fll}}

\renewcommand{\angle}[1]{\langle #1 \rangle}
\newcommand{\cl}[1]{\overline{#1}}

\makeatletter
\newlength{\xywd}
\def\xy@arrow#1#2#3#4#5#6#7{%
\mathrel{\mskip\@ne mu
  \sbox\z@{\scriptsize$\mkern#1mu#6\mkern#2mu$}%
  \xywd=\wd\z@
  \sbox\z@{\scriptsize$\mkern#1mu#7\mkern#2mu$}%
  \ifdim\wd\z@>\xywd\xywd=\wd\z@\fi\relax
  \ifdim#3\p@>\xywd \xywd=#3\p@ \fi
  \xy
    <\xywd,\z@> **\dir{-}
    ? *^!U!/\thr@@\p@/\txt\scriptsize{$\mkern#1mu#7\mkern#2mu$}
      *_!D!/\thr@@\p@/\txt\scriptsize{$\mkern#1mu#6\mkern#2mu$}
    ?<*\dir{#4}
    ?>*\dir{#5};
  \endxy
\mskip\@ne mu}%
}
\newcommand{\xtwoheadrightarrow}[2][]{\xy@arrow 26{11}{}{>>}{#2}{#1}}
\newcommand{\xtwoheadleftarrow}[2][]{\xy@arrow 62{11}{<<}{}{#2}{#1}}
\DeclareRobustCommand{\xto}{\@ifstar\xtwoheadrightarrow\xrightarrow}
\DeclareRobustCommand{\xot}{\@ifstar\xtwoheadleftarrow\xleftarrow}
\makeatother

\newcommand{\Nb}{\mathbb{N}}
\newcommand{\Zb}{\mathbb{Z}}

\newcommand{\Cr}{\mathcal{C}}
\newcommand{\Er}{\mathcal{E}}
\newcommand{\Or}{\mathcal{O}}
\newcommand{\Rr}{\mathcal{R}}
\newcommand{\Xr}{\mathscr{X}}

\newcommand{\bk}{{\mathbf{k}}}

\DeclareMathOperator{\lm}{lm}
\DeclareMathOperator{\lc}{lc}
\DeclareMathOperator{\lt}{lt}
\newcommand{\cf}[2]{\operatorname{cf}_{#1}(#2)}

\DeclareMathOperator{\supp}{Supp}
\DeclareMathOperator{\suppm}{MSupp}

\DeclareMathOperator{\Rem}{Rem}
\newcommand{\quo}[2]{\operatorname{quo}_{#1}(#2)}
\newcommand{\rem}[2]{\operatorname{rem}_{#1}(#2)}

\DeclareMathOperator{\End}{End}

\newcommand{\Cont}{\mathbf{Cont}}
\newcommand{\Act}{\mathbf{Act}}

\newcommand{\cat}[1]{#1^{\ast}}
\newcommand{\lcat}[1]{#1^{\ell}}
\newcommand{\pcat}[1]{#1^{\circ}}
\newcommand{\plcat}[1]{#1^{\circ\ell}}

\newcommand{\Ab}{\mathbf{Ab}}
\newcommand{\Ord}{\mathbf{Ord}}

\newcommand{\polygraph}[1]{\textsf{#1}}
\newcommand{\category}[1]{\mathcal{#1}}

\newcommand{\ACCat}{\category{AC}}
\newcommand{\ACPol}{\polygraph{AC}}

\newcommand{\qSCat}{q\category{S}}
\newcommand{\qSPol}{\polygraph{qS}}

\newcommand{\HCat}{\category{H}}
\newcommand{\HPol}{\polygraph{H}}

\DeclareMathOperator{\Nf}{Nf}
\newcommand{\abs}[1]{\lvert#1\rvert}

\newcommand{\lex}{\mathrm{lex}}
\newcommand{\Sch}{\mathrm{Sch}}
\newcommand{\q}{\mathrm{q}}

\DeclareMathOperator{\Sph}{Sph}

\newcommand{\Phivec}[4]{%
  \Phi_{#1}\begin{bmatrix}
    #2\\[1mm]
    #3\\[1mm]
    #4
  \end{bmatrix}}

\newcommand{\Psivec}[5]{%
  \Psi_{#1}\,\begin{bmatrix}
    #2\\[1mm]
    #3\\[1mm]
    #4\\[1mm]
    #5
  \end{bmatrix}}

\newcommand{\SSS}[1]{§\ref{#1}}
\newcommand{\down}[2]{{#1\!\!\downarrow_{\scriptscriptstyle #2}}}
\newcommand{\dlr}[2]{\down{\llbracket #1 \rrbracket}{#2}}
\newcommand{\minsub}{\min\nolimits_{\sqsubseteq}}
\newcommand{\CB}{\textsf{CB}}

\deftwocell[crossing]{tau : 2 -> 2}
\deftwocell[circle]{eta : 1 -> 1}
\deftwocell[braid1]{sigma1 : 2 -> 2}
\deftwocell[braid2]{sigma2 : 2 -> 2}
\deftwocell[left = i, right = j]{id_ij : 2 -> 2}
\deftwocell[mid = i]{i : 1 -> 0}
\deftwocell[mid = j]{j : 1 -> 0}
\deftwocell[mid = k]{k : 1 -> 0}
\deftwocell{id : 1 -> 1}

\deftwocell[rectangle]{rectangle:2 -> 2}
\deftwocell[text=\alpha]{alpha:2 -> 2}
\deftwocell[text=\beta]{beta:2 -> 2}
\deftwocell[text=\gamma]{gamma:2 -> 2}
\deftwocell[text=\alpha_1]{alpha1:2 -> 2}
\deftwocell[text=\alpha_2]{alpha2:2 -> 2}
\deftwocell[text=\alpha_3]{alpha3:2 -> 2}
\deftwocell[text=\alpha_k]{alphak:2 -> 2}
\deftwocell[text=\beta_1]{beta1:2 -> 2}
\deftwocell[text=\beta_2]{beta2:2 -> 2}
\deftwocell[text=\beta_3]{beta3:2 -> 2}
\deftwocell[text=f_1]{f1:2 -> 2}
\deftwocell[text=f_2]{f2:2 -> 2}
\deftwocell[text=f_2^i]{f2i:2 -> 2}
\deftwocell[text=f_3]{f3:2 -> 2}
\deftwocell[text=g_1]{g1:2 -> 2}
\deftwocell[text=g_2]{g2:2 -> 2}
\deftwocell[text=g_3]{g3:2 -> 2}
\deftwocell[text=f_i^1]{fi1:2 -> 2}
\deftwocell[text=f_i^{m_i}]{fimi:2 -> 2}
\deftwocell[text=f]{f:2 -> 2}
\deftwocell[text=g]{g:2 -> 2}
\deftwocell[text=h]{h:2 -> 2}
\deftwocell[text=\square]{square:2 -> 2}
\deftwocell[text=f']{f':2 -> 2}
\deftwocell[text=g']{g':2 -> 2}
\deftwocell[text=f_i]{fi:2 -> 2}
\deftwocell[text=f_k]{fk:2 -> 2}
\deftwocell[text=\square]{square1:1 -> 1}
\deftwocell[text=E]{E-1:1 -> 1}
\deftwocell[text=\alpha_\lambda]{alphalambda:1 -> 1}
\deftwocell[text=\square]{square3:3 -> 3}

\deftwocell[dots]{dot:2 -> 2}
\deftwocell[cap]{cap  : 0 -> 2}
\deftwocell[cup]{cup  : 2 -> 0}
\deftwocell[orange]{split : 2 -> 1}
\deftwocell[green]{merge : 1 -> 2}
\deftwocell[mid = a]{as : 0 -> 1}
\deftwocell[mid = a]{ax : 1 -> 0}
\deftwocell[mid = a_1]{a1s : 0 -> 1}
\deftwocell[mid = a_1]{a1x : 1 -> 0}
\deftwocell[mid = a_n]{ans : 0 -> 1}
\deftwocell[mid = a_n]{anx : 1 -> 0}
\deftwocell[mid = b]{bs : 0 -> 1}
\deftwocell[mid = b]{bx : 1 -> 0}
\deftwocell[mid = b']{b'x : 1 -> 0}
\deftwocell[mid = a']{a's : 0 -> 1}
\deftwocell[mid = a']{a'x : 1 -> 0}
\deftwocell[mid = b']{b's : 0 -> 1}
\deftwocell[mid = b']{b'x : 1 -> 0}
\deftwocell[mid = c]{cs : 0 -> 1}
\deftwocell[mid = c]{cx : 1 -> 0}
\deftwocell[mid = d]{ds : 0 -> 1}
\deftwocell[mid = d]{dx : 1 -> 0}
\deftwocell[mid = r]{rs : 0 -> 1}
\deftwocell[mid = r]{rx : 1 -> 0}
\deftwocell[mid = a]{a : 1 -> 1}
\deftwocell[mid = b]{b : 1 -> 1}
\deftwocell[mid = c]{c : 1 -> 1}
\deftwocell[mid = d]{d : 1 -> 1}
\deftwocell[mid = c-s]{cjs : 1 -> 1}
\deftwocell[mid = c-s]{cjsx : 1 -> 0}
\deftwocell[mid = d-s]{djs : 1 -> 1}
\deftwocell[mid = a+b]{ajbs : 0 -> 1}
\deftwocell[mid = a+b]{ajbx : 1 -> 0}
\deftwocell[mid = a+b+c]{ajbjcx : 1 -> 0}
\deftwocell[mid = b+c]{bjcs : 0 -> 1}
\deftwocell[mid = b+c]{bjcx : 1 -> 0}
\deftwocell[mid = d-c]{djcs : 0 -> 1}
\deftwocell[mid = d-c]{djc : 1 -> 1}
\deftwocell[mid = a-s]{ajs : 1 -> 1}
\deftwocell[mid = c-r]{cjr : 1 -> 1}
\deftwocell[mid = d-r]{djr : 1 -> 1}
\deftwocell[mid = 1]{one : 1 -> 1}
\deftwocell[mid = 2]{two : 1 -> 1}
\deftwocell[mid = 3]{three : 1 -> 1}
\deftwocell[mid = 4]{four : 1 -> 1}
\deftwocell[mid = 5]{five : 1 -> 1}

\deftwocell[mid = \raisebox{0pt}{\fontsize{5pt}{5pt}\selectfont $a\!+\!c$}]{ajiac : 1 -> 1}
\deftwocell[mid = \raisebox{0pt}
{\fontsize{5pt}{5pt}\selectfont $a\!+\!d$}]{ajiad : 1 -> 1}
\deftwocell[mid = \raisebox{0pt}
{\fontsize{5pt}{5pt}\selectfont $a\!+\!n$}]{ajian : 1 -> 1}
\deftwocell[mid = \raisebox{0pt}
{\fontsize{5pt}{5pt}\selectfont $a\!+\!m$}]{ajiam : 1 -> 1}
\deftwocell[mid = \raisebox{0pt}
{\fontsize{5pt}{5pt}\selectfont $y\!-\!t$}]{yjt : 1 -> 1}
\deftwocell[mid = \raisebox{0pt}
{\fontsize{5pt}{5pt}\selectfont $x\!-\!t$}]{xjt : 1 -> 1}
\deftwocell[mid = e]{es : 0 -> 1}

\deftwocell[mid = b+d]{bjdx : 1 -> 0}
\deftwocell[mid = \raisebox{0pt}
{\fontsize{5pt}{5pt}\selectfont $b\!+\!d$}]{bjd : 1 -> 1}
\deftwocell[mid = e]{es : 0 -> 1}
\deftwocell[mid = a-d]{ajd : 1 -> 1}
\deftwocell[mid = e]{ex : 1 -> 0}
\deftwocell[mid = e]{e : 1 -> 1}
\deftwocell[mid = a+e]{ajex : 1 -> 0}
\deftwocell[mid = a+e]{aje : 1 -> 1}
\deftwocell[mid = a-t]{ajt : 1 -> 1}
\deftwocell[mid = a+b+d]{ajbjdx : 1 -> 0}
\deftwocell[mid = m]{m : 1 -> 1}
\deftwocell[mid = m]{mx : 1 -> 0}
\deftwocell[mid = n]{n : 1 -> 1}
\deftwocell[mid = n]{ns : 0 -> 1}
\deftwocell[mid = n]{nx : 1 -> 0}
\deftwocell[mid = x]{x : 1 -> 1}
\deftwocell[mid = y]{y : 1 -> 1}
\deftwocell[right = \:\:\:\:\:\:\:\:\:\:\:\:\:\:\:\: b-c-m+t]{bjcjmjt : 1 -> 1}
\deftwocell[mid = u]{u : 1 -> 1}
\deftwocell[mid = v]{v : 1 -> 1}
\deftwocell[mid = w]{w : 1 -> 1}
\deftwocell[mid = g]{gs : 0 -> 1}
\deftwocell[mid = f]{fx : 1 -> 0}
\deftwocell[right = \:\:\:\:\:\:\:\:\:\:\:\:\:\:\:\:\:\:\:g-b-d+c+s]{gbdcs : 1 -> 1}
\deftwocell[right = \:\:\:\:\:\:\:\:\:\:\:\:\:\:\:\:\:\:\:\:\:\:\:\:\:g-b-d+c+m+l]{gbdcml : 1 -> 1}
\deftwocell[right = \:\:\:\:\:\:\:\:\:\:\:\:\:\:\:\:\:\:\:b-c+1-m+t]{bjcj1jmjt : 1 -> 1}
\deftwocell[mid = d-n]{djn : 1 -> 1}
\deftwocell[mid = c-n]{cjn : 1 -> 1}
\deftwocell[mid =  ]{emptys : 0 -> 1}
\deftwocell[mid =  ]{emptyx : 1 -> 0}
\deftwocell[mid = d-t]{djt : 1 -> 1}
\deftwocell[mid = \raisebox{0pt}
{\fontsize{5pt}{5pt}\selectfont $d\!-\!1$}]{dj1 : 1 -> 1}
\deftwocell[mid = \raisebox{0pt}
{\fontsize{5pt}{5pt}\selectfont $c\!-\!1$}]{cj1 : 1 -> 1}

\deftwocell[mid = a'+b']{a'jb's : 0 -> 1}
\deftwocell[mid = a'+b']{a'jb'x : 1 -> 0}
\deftwocell[mid = a'+b']{a'jb's : 0 -> 1}
\deftwocell[mid = a'+b']{a'jb'x : 1 -> 0}
\deftwocell[circle,orange]{epsilon : 1 -> 0}
\deftwocell[circle,green]{circ : 0 -> 1}
\deftwocell[mid = 0]{0s : 0 -> 1}
\deftwocell[mid = 0]{0x : 1 -> 0}
\deftwocell[mid = 0]{0z : 1 -> 1}
\deftwocell[mid = 1]{1x : 1 -> 0}
\deftwocell[mid = 2]{2x : 1 -> 0}
\deftwocell[mid = 1]{1s : 0 -> 1}
\deftwocell[mid = 2]{2s : 0 -> 1}
\deftwocell[mid = r_1]{r1x : 1 -> 0}
\deftwocell[mid = h^i]{hi : 1 -> 1}
\deftwocell[mid = k^i]{ki : 1 -> 1}
\deftwocell[mid = h_1]{h1 : 1 -> 1}
\deftwocell[mid = k_1]{k1 : 1 -> 1}
\deftwocell[mid = i]{ix : 1 -> 0}
\deftwocell[mid = a_1]{a1 : 1 -> 1}
\deftwocell[mid = a_m]{am : 1 -> 1}
\deftwocell[mid = a_1+\ldots+a_m]{a1m : 1 -> 1}

\deftwocell[mid = f]{onef : 1 -> 1}
\deftwocell[mid = g]{oneg : 1 -> 1}
\deftwocell[mid = f-t]{fjt : 1 -> 1}
\deftwocell[mid = g-t]{gjt : 1 -> 1}
\deftwocell[right = \:\:\:\:\:\:\:\:b-c+s]{bcs : 1 -> 1}
\deftwocell[mid = f-n]{fjn : 1 -> 1}
\deftwocell[mid = g-n]{gjn : 1 -> 1}
\deftwocell[right = \:\:\:\:\:\:\:\:\:b+d-c]{bdc : 1 -> 1}
\deftwocell[mid = d-m]{djm : 1 -> 1}
\deftwocell[mid = \raisebox{0pt}{\fontsize{5pt}{5pt}\selectfont $g\!\!-\!\!n\!\!-\!\!m$}]{gjnjm : 1 -> 1}
\deftwocell[right = \:\:\:\:\:\:\:\:\:\:\:\:\:\:\:b-c+m+l]{bcml : 1 -> 1}
\deftwocell[mid = u_i]{ui : 1 -> 1}
\deftwocell[mid = v_i]{vi : 1 -> 1}
\deftwocell[mid = u_{i+1}]{ui1 : 1 -> 1}
\deftwocell[mid = v_{i+1}]{vi1 : 1 -> 1}
\deftwocell[right = \raisebox{0pt}{\fontsize{6pt}{6pt}\selectfont $n$}]{nr : 1 -> 1}
\deftwocell[right = \raisebox{0pt}{\fontsize{6pt}{6pt}\selectfont $m$}]{mr : 1 -> 1}
\deftwocell[right = \raisebox{0pt}{\fontsize{6pt}{6pt}\selectfont $n\!+\!\!1$}]{nj1r : 1 -> 1}
\deftwocell[right = \raisebox{0pt}{\fontsize{6pt}{6pt}\selectfont $2$}]{2r : 1 -> 1}
\deftwocell[right = \raisebox{0pt}{\fontsize{6pt}{6pt}\selectfont $k$}]{kr : 1 -> 1}
\deftwocell[right = \raisebox{0pt}{\fontsize{6pt}{6pt}\selectfont $i$}]{ir : 1 -> 1}
\deftwocell[right = \raisebox{0pt}{\fontsize{6pt}{6pt}\selectfont $1$}]{1r : 1 -> 1}
\deftwocell[mid = \raisebox{0pt}{\fontsize{5pt}{5pt}\selectfont $c\!+\!m\!\!-\!\!t$}]{cjmjt : 1 -> 1}
\deftwocell[mid = \raisebox{0pt}{\fontsize{5pt}{5pt}\selectfont $b\!+\!d\!\!+\!\!y$}]{bjdjyx : 1 -> 0}
\deftwocell[mid = \raisebox{0pt}{\fontsize{5pt}{5pt}\selectfont $d\!+\!n\!\!-\!\!t$}]{djnjt : 1 -> 1}
\deftwocell[right = \raisebox{0pt}{\fontsize{6pt}{6pt}\selectfont $\leqslant \! 2$}]{leq2 : 1 -> 1}
\deftwocell[right = \raisebox{0pt}{\fontsize{6pt}{6pt}\selectfont $\geqslant \! 1$}]{geq1 : 1 -> 1}
\deftwocell[right = \raisebox{0pt}{\fontsize{6pt}{6pt}\selectfont $\leqslant \! n$}]{leqn : 1 -> 1}
\deftwocell[mid = \raisebox{0pt}{\fontsize{6pt}{6pt}\selectfont $a\!-\!c$}]{a-c : 1 -> 1}
\deftwocell[mid = \raisebox{0pt}{\fontsize{4pt}{4pt}\selectfont $a\!+\!c\!\!-\!\!d-\!\!t$}]{ajcjdjt : 1 -> 1}
\deftwocell[mid = \raisebox{0pt}{\fontsize{5pt}{5pt}\selectfont $b\!+\!c\!\!-\!\!s$}]{bjcjs : 1 -> 1}
\deftwocell[mid = \raisebox{0pt}{\fontsize{5pt}{5pt}\selectfont $d\!+\!n\!\!-\!\!m$}]{djnjm : 1 -> 1}
\deftwocell[mid = \raisebox{0pt}{\fontsize{5pt}{5pt}\selectfont $d\!-\!n\!\!-\!\!m$}]{djnjm : 1 -> 1}
\deftwocell[mid = \raisebox{0pt}{\fontsize{5pt}{5pt}\selectfont $d\!-\!m\!\!-\!\!l$}]{djmjl : 1 -> 1}
\deftwocell[mid = \raisebox{0pt}{\fontsize{5pt}{5pt}\selectfont $f\!-\!n\!\!-\!\!m$}]{fjnjm : 1 -> 1}
\deftwocell[mid = \raisebox{0pt}{\fontsize{5pt}{5pt}\selectfont $a\!+\!b\!\!-\!\!c$}]{ajbjc : 1 -> 1}
\deftwocell[mid = \raisebox{0pt}{\fontsize{5pt}{5pt}\selectfont $\boxed{m}$}]{boxm : 1 -> 1}
\deftwocell[mid = \raisebox{0pt}{\fontsize{5pt}{5pt}\selectfont $\boxed{n}$}]{boxn : 1 -> 1}
\deftwocell[mid = \raisebox{0pt}{\fontsize{5pt}{5pt}\selectfont $d\!-\!1\!\!-\!\!s$}]{dj1js : 1 -> 1}
\deftwocell[mid = \raisebox{0pt}{\fontsize{5pt}{5pt}\selectfont $c\!-\!1\!\!-\!\!s$}]{cj1js : 1 -> 1}

\deftwocell[mid = \raisebox{0pt}{\fontsize{5pt}{5pt}\selectfont $c\!-\!s$}]{xcjs : 1 -> 1}
\deftwocell[mid = \raisebox{0pt}{\fontsize{5pt}{5pt}\selectfont $d\!-\!s$}]{xdjs : 1 -> 1}

\deftwocell[orange,text=n]{nsplit : 2 -> 1}
\deftwocell[green,text=n]{nmerge : 1 -> 2}
\deftwocell[orange,text=n-1]{nj1split : 2 -> 1}
\deftwocell[green,text=n-1]{nj1merge : 1 -> 2}
\deftwocell[orange,text=n_1]{n1split : 2 -> 1}
\deftwocell[green,text=n_1]{n1merge : 1 -> 2}
\deftwocell[orange,text=n_2]{n2split : 2 -> 1}
\deftwocell[green,text=n_2]{n2merge : 1 -> 2}
\deftwocell[orange,text=n_3]{n3split : 2 -> 1}
\deftwocell[green,text=n_3]{n3merge : 1 -> 2}
\deftwocell[orange,text=n_k]{nksplit : 2 -> 1}
\deftwocell[green,text=n_k]{nkmerge : 1 -> 2}
\deftwocell[orange,text=1]{1split : 2 -> 1}
\deftwocell[green,text=1]{1merge : 1 -> 2}
\deftwocell[orange,text=0]{0split : 2 -> 1}
\deftwocell[green,text=0]{0merge : 1 -> 2}
\deftwocell[orange,text=m]{msplit : 4 -> 1}
\deftwocell[orange,text=k]{ksplit : 3 -> 1}
\deftwocell[orange,text=m-1]{mj1split : 4 -> 1}
\deftwocell[orange,text=k-1]{kj1split : 2 -> 1}
\deftwocell[green,text=n]{5nmerge : 1 -> 5}
\deftwocell[green,text=n-1]{5nj1merge : 1 -> 5}

\deftwocell[text=A]{A:2 -> 2}
\deftwocell[text=B]{B:2 -> 2}
\deftwocell[text=C]{C:2 -> 2}
\deftwocell[text=E]{E:2 -> 2}
\deftwocell[text=E']{E':2 -> 2}
\deftwocell[text=\Xr]{X:2 -> 2}
\deftwocell[text=E_{n_i}]{Ei:2 -> 2}
\deftwocell[text=E_{n_1}]{E1:2 -> 2}
\deftwocell[text=E_{n_2}]{E2:2 -> 2}
\deftwocell[text=E_{n_3}]{E3:2 -> 2}
\deftwocell[text=s\alpha]{salpha:2 -> 2}
\deftwocell[text=t\alpha]{talpha:2 -> 2}
\deftwocell[text=\varphi]{varphi:2 -> 2}
\deftwocell[text=\varphi_i]{varphii:2 -> 2}
\deftwocell[text=\varphi_{i+1}]{varphii1:2 -> 2}
\deftwocell[text=\psi]{psi:2 -> 2}
\deftwocell[text=h]{h2:4 -> 2}
\deftwocell[text=k]{k2:2 -> 4}
\deftwocell[text=\eta_1]{eta1:2 -> 2}
\deftwocell[text=\eta_i]{etai:2 -> 2}
\deftwocell[text=\eta_{i-2}]{etaij2:2 -> 2}
\deftwocell[text=\theta_{i-2}]{thetaij2:2 -> 2}
\deftwocell[text=\theta_{i-1}]{thetaij1:2 -> 2}
\deftwocell[text=\theta'_{i-1}]{thetaij1':2 -> 2}
\deftwocell[text=\theta'_{j-1}]{thetajj1':2 -> 2}

\deftwocell[orange]{split3 : 3 -> 1}
\deftwocell[green]{merge3 : 1 -> 3}
\deftwocell[orange]{split4 : 4 -> 1}
\deftwocell[green]{merge4 : 1 -> 4}
\deftwocell[orange]{split6 : 6 -> 1}
\deftwocell[green]{merge6 : 1 -> 6}

\deftwocell[text = h_i]{h-i : 2 -> 2}
\deftwocell[text = h_{i+1}]{h-i1 : 2 -> 2}

\deftwocell[braid1]{braidl : 2 -> 2}
\deftwocell[braid2]{braidr : 2 -> 2}

\deftwocell[text = f_n]{fn : 2 -> 2}
\deftwocell[text=k]{kk:2 -> 2}
\deftwocell[text = k_1]{kk1 : 2 -> 2}
\deftwocell[text = k_2]{kk2 : 2 -> 2}
\deftwocell[text = k_\ell]{kkn : 2 -> 2}
\deftwocell[text = k_i]{kki : 2 -> 2}
\deftwocell[text = k_{i+1}]{kkij1 : 2 -> 2}
\deftwocell[mid = j]{jj : 1 -> 1}

\newcommand{\aup}{
\begin{tikzpicture}[scale=0.55,baseline=(current bounding box.center)]
  \node at (0,-0.35) {$a$};

  \draw[line width=0.6mm]  (0,0) --  (0,-0.16);

  \draw[line width=0.3mm] (0,0) -- (-0.8,1.2);
  \draw[line width=0.3mm] (0,0) -- (-0.4,1.2);
  \draw[line width=0.3mm,dotted] (0,0) -- (0,1.2);
  \draw[line width=0.3mm] (0,0) -- (0.4,1.2);
  \draw[line width=0.3mm] (0,0) -- (0.8,1.2);

  \node[above] at (-1,1.1) {$a_1$};
  \node[above] at (-0.4,1.2) {};
  \node[above] at (0.4,1.2) {};
  \node[above] at (1,1.1) {$a_n$};

\end{tikzpicture}
}

\newcommand{\adown}{
\begin{tikzpicture}[scale=0.55,baseline=(current bounding box.center),rotate=180]
  \node at (0,-0.35) {$a$};

  \draw[line width=0.6mm]  (0,0) --  (0,-0.16);

  \draw[line width=0.3mm] (0,0) -- (-0.8,1.2);
  \draw[line width=0.3mm] (0,0) -- (-0.4,1.2);
  \draw[line width=0.3mm,dotted] (0,0) -- (0,1.2);
  \draw[line width=0.3mm] (0,0) -- (0.4,1.2);
  \draw[line width=0.3mm] (0,0) -- (0.8,1.2);

  \node[above] at (-0.9,1.9) {$a_n$};
  \node[above] at (-0.4,1.6) {};
  \node[above] at (0.4,1.6) {};
  \node[above] at (0.8,1.9) {$a_1$};
\end{tikzpicture}
}

\newcommand{\EXi}{
\begin{tikzpicture}[scale=0.7,baseline=-0.1em]
  \draw[-,line width=0.35mm] (0,0) -- (0.6,0.6);
  \draw[-,line width=0.35mm] (0,0) -- (-0.6,0.6);
  \draw[-,line width=0.35mm] (0,0) -- (-0.6,-0.6);
  \draw[-,line width=0.35mm] (0,0) -- (0.6,-0.6);

  \draw[-,line width=0.3mm] (0.6,0.6) -- (0.6,-0.6);

  \draw[line width=0.6mm] (-0.6,0.8) -- (-0.6,0.6);
  \draw[line width=0.6mm] (0.6,0.8) -- (0.6,0.6);
  \draw[line width=0.6mm] (-0.6,-0.8) -- (-0.6,-0.6);
  \draw[line width=0.6mm] (0.6,-0.8) -- (0.6,-0.6);

  \node[above left] at (-0.6,0.6) {\scriptsize 3};
  \node[above right] at (0.6,0.6) {\scriptsize 4};
  \node[below left] at (-0.6,-0.6) {\scriptsize 2};
  \node[below right] at (0.6,-0.6) {\scriptsize 5};

\node at (0.2,0.4) {$\scriptstyle 2$};
\node at (0.2,-0.4) {$\scriptstyle 3$};
\node at (0.7,0.05){$\scriptstyle 2$};
\end{tikzpicture}
}

\newcommand{\EXii}{
 \begin{tikzpicture}[scale=1.4,baseline=-0.1em]
\draw[-,line width=.6mm] (.212,.5) to (.212,.39);
\draw[-,line width=.75mm] (.595,.5) to (.595,.39);
\draw[-,line width=.15mm] (0.0005,-.396) to (.2,.4);
\draw[-,line width=.3mm] (0.01,-.4) to (.59,.4);
\draw[-,line width=.3mm] (.4,-.4) to (.607,.4);
\draw[-,line width=.45mm] (.79,-.4) to (.214,.4);
\draw[-,line width=.15mm] (.8035,-.398) to (.614,.4);
\draw[-,line width=.3mm] (.4006,-.5) to (.4006,-.395);
\draw[-,line width=.6mm] (.788,-.5) to (.788,-.395);
\draw[-,line width=.45mm] (0.011,-.5) to (0.011,-.395);
\node at (0.05,0.05) {$\scriptstyle 1$};
\node at (0.76,0.05) {$\scriptstyle 1$};
\node at (0.35,0.35) {$\scriptstyle 3$};
\node at (0.2,-0.26) {$\scriptstyle 2$};
\node at (0.5,-0.26) {$\scriptstyle 2$};
\end{tikzpicture}
}

\newcommand{\nmthree}{
\begin{tikzpicture}[scale=1.5]
  \draw[-,line width=0.35mm] (-1.2,-0.7) -- (1.8,0.7);
  \draw[-,line width=0.35mm] (-1.2,0.7) -- (1.8,-0.7);
  \draw[-,line width=0.35mm] (1.8,0.7) -- (1.8,-0.7);
  \draw[-,line width=0.35mm] (1.8,0.7) -- (0.3,-0.7);
  \draw[-,line width=0.35mm] (1.8,-0.7) -- (0.3,0.7);

  \draw[line width=0.7mm] (-1.2,0.85) -- (-1.2,0.7);
  \draw[line width=0.7mm] (0.3,0.85) -- (0.3,0.7);
  \draw[line width=0.7mm] (-1.2,-0.85) -- (-1.2,-0.7);
  \draw[line width=0.7mm] (0.3,-0.85) -- (0.3,-0.7);
  \draw[line width=0.7mm] (1.8,0.85) -- (1.8,0.7);
  \draw[line width=0.7mm] (1.8,-0.85) -- (1.8,-0.7);

  \node[above left] at (-1.2,0.7) {\scriptsize $b^1$};
  \node[above right] at (0.3,0.7) {\scriptsize $b^2$};
  \node[below left] at (-1.2,-0.7) {\scriptsize $a_1$};
  \node[below right] at (0.3,-0.7) {\scriptsize $a_2$};
  \node[below right] at (1.8,-0.7) {\scriptsize $a_3$};
  \node[above right] at (1.8,0.7) {\scriptsize $b^3$};

\node at (-0.5,0.23) {$\scriptstyle c^{11}$};
\node at (-0.5,-0.25) {$\scriptstyle c_{11}$};
\node at (0.5,0.23) {$\scriptstyle c_{12}$};
\node at (0.5,-0.25) {$\scriptstyle c^{12}$};
\node at (1.2,0.55) {$\scriptstyle c_{13}$};
\node at (1.2,-0.57) {$\scriptstyle c^{13}$};
\node at (0.7,0.55) {$\scriptstyle c^{21}$};
\node at (1.07,0.2) {$\scriptstyle c^{22}$};
\node at (1.5,-0.2) {$\scriptstyle c^{23}$};
\node at (0.65,-0.55) {$\scriptstyle c_{21}$};
\node at (1.05,-0.2) {$\scriptstyle c_{22}$};
\node at (1.47,0.2) {$\scriptstyle c_{23}$};
\node at (2.1,0) {$\scriptstyle c_3(c^3)$};
\end{tikzpicture}
}

\newcommand{\atob}{
\begin{tikzpicture}[scale=1.2]
  \draw[-,line width=0.35mm] (-1.2,-0.7) -- (1.8,0.7);
  \draw[-,line width=0.35mm] (1.8,0.7) -- (1.8,-0.7);
  \draw[-,line width=0.35mm] (1.8,0.7) -- (0.1,-0.7);
  \draw[-,line width=0.35mm] (1.8,0.7) -- (1.1,-0.7);
  \draw[-,line width=0.35mm] (-0.7,0.7) -- (0.788,-0.133);
  \draw[dashed,line width=0.35mm] (0.788,-0.133) -- (1.253,-0.394);
  \draw[-,line width=0.35mm] (1.253,-0.394) -- (1.8,-0.7);

  \node[above left] at (-0.6,0.7) {\scriptsize $b^i$};
  \node[below left] at (-1.15,-0.7) {\scriptsize $a_1$};
  \node[below right] at (-0.05,-0.7) {\scriptsize $a_2$};
  \node[below right] at (0.9,-0.7) {\scriptsize $a_{n-1}$};
  \node[below right] at (1.8,-0.7) {\scriptsize $a_n$};
  \node[below right] at (1.8,1) {\scriptsize $b^m$};
  \node[below right] at (0.4,-0.75) {$\cdots$};
  \node[below right] at (-0.2,0.7) {\scriptsize $c^{i1}$};
  \node[below right] at (0.55,0.27) {\scriptsize $c^{i2}$};
  \node[below right] at (1.35,-0.2) {\scriptsize $c^{in}$};
\end{tikzpicture}
}

\newcommand{\btoa}{
\begin{tikzpicture}[scale=1.2]
  \draw[-,line width=0.35mm] (-1.2,0.7) -- (1.8,-0.7);
  \draw[-,line width=0.35mm] (1.8,-0.7) -- (1.8,0.7);
  \draw[-,line width=0.35mm] (1.8,-0.7) -- (0.1,0.7);
  \draw[-,line width=0.35mm] (1.8,-0.7) -- (1.1,0.7);

  \draw[-,line width=0.35mm] (-0.7,-0.7) -- (0.788,0.133);
  \draw[dashed,line width=0.35mm] (0.788,0.133) -- (1.253,0.394);
  \draw[-,line width=0.35mm] (1.253,0.394) -- (1.8,0.7);

  \node[below left] at (-0.6,-0.7) {\scriptsize $a_j$};

  \node[above left]  at (-1.15,0.7) {\scriptsize $b^1$};
  \node[above right] at (-0.05,0.7) {\scriptsize $b^2$};
  \node[above right] at (0.9,0.7)  {\scriptsize $b^{n-1}$};
  \node[above right] at (1.8,0.7)  {\scriptsize $b^m$};

  \node[above right] at (1.8,-1) {\scriptsize $a_n$};

  \node[above right] at (0.4,0.65) {$\cdots$};

  \node[above right] at (-0.2,-0.65) {\scriptsize $c_{j1}$};
  \node[above right] at (0.53,-0.25)  {\scriptsize $c_{j2}$};
  \node[above right] at (1.28,0.18)  {\scriptsize $c_{jm}$};
\end{tikzpicture}
}

\newcommand{\Fthree}{
\begin{tikzpicture}[scale=0.7]
  \draw[-,line width=0.35mm] (0,0) -- (0.6,0.6);
  \draw[-,line width=0.35mm] (0,0) -- (-0.6,0.6);
  \draw[-,line width=0.35mm] (0,0) -- (-0.6,-0.6);
  \draw[-,line width=0.35mm] (0,0) -- (0.6,-0.6);

  \draw[line width=0.6mm] (-0.6,0.8) -- (-0.6,0.6);
  \draw[line width=0.5mm] (0.6,0.8) -- (0.6,0.6);
  \draw[line width=0.6mm] (-0.6,-0.8) -- (-0.6,-0.6);
  \draw[line width=0.5mm] (0.6,-0.8) -- (0.6,-0.6);

  \draw[-,line width=0.3mm] (-0.6,0.6) -- (-0.6,-0.6);
\end{tikzpicture}
}

\newcommand{\counterexample}{
\begin{tikzpicture}[scale=1.2]
  \draw[-,line width=0.35mm] (-1.2,-0.7) -- (1.8,0.7);
  \draw[-,line width=0.35mm] (-1.2,0.7) -- (1.8,-0.7);
  \draw[-,line width=0.35mm] (1.8,0.7) -- (1.8,-0.7);
  \draw[-,line width=0.35mm] (1.8,0.7) -- (0.3,-0.7);
  \draw[-,line width=0.35mm] (1.8,-0.7) -- (0.3,0.7);

  \draw[line width=0.7mm] (-1.2,0.85) -- (-1.2,0.7);
  \draw[line width=0.7mm] (0.3,0.85) -- (0.3,0.7);
  \draw[line width=0.7mm] (-1.2,-0.85) -- (-1.2,-0.7);
  \draw[line width=0.7mm] (0.3,-0.85) -- (0.3,-0.7);
  \draw[line width=0.7mm] (1.8,0.85) -- (1.8,0.7);
  \draw[line width=0.7mm] (1.8,-0.85) -- (1.8,-0.7);

  \node[above left] at (-1.2,0.7) {\scriptsize $b^1$};
  \node[above right] at (0.3,0.7) {\scriptsize $b^2$};
  \node[below left] at (-1.2,-0.7) {\scriptsize $a_1$};
  \node[below right] at (0.3,-0.7) {\scriptsize $a_2$};
  \node[below right] at (1.8,-0.7) {\scriptsize $a_3$};
  \node[above right] at (1.8,0.7) {\scriptsize $b^3$};

  \draw[-,line width=0.35mm] (-1.2,-0.7) -- (0.3,0.7);

  \draw[red, line width=0.6mm] (-1.2,-0.7) -- (-0.2, 0.23);
  \draw[red, line width=0.6mm] (-1.2,-0.7) -- (0.8, 0.23);
  \draw[red, line width=0.6mm] (-0.2,0.23) -- (0.8, -0.23);
\end{tikzpicture}
}

\newcommand{\sGammaone}{
\begin{tikzpicture}[scale=0.9,baseline=-0.1em]
  \draw[-,line width=0.35mm] (0,0) -- (0.6,0.6);
  \draw[-,line width=0.35mm] (0,0) -- (-0.6,0.6);
  \draw[-,line width=0.35mm] (0,0) -- (-0.6,-0.6);
  \draw[-,line width=0.35mm] (0,0) -- (0.6,-0.6);

  \draw[-,line width=0.3mm] (0.6,0.6) -- (0.6,-0.6);

  \draw[line width=0.6mm] (-0.6,0.8) -- (-0.6,0.6);
  \draw[line width=0.6mm] (0.6,0.8) -- (0.6,0.6);
  \draw[line width=0.6mm] (-0.6,-0.8) -- (-0.6,-0.6);
  \draw[line width=0.6mm] (0.6,-0.8) -- (0.6,-0.6);

  \node[above left] at (-0.6,0.6) {\scriptsize {\setlength{\fboxsep}{0.6pt}\fbox{c}}};
  \node[below left] at (-0.6,-0.6) {\scriptsize a};
  \node[below right] at (0.6,-0.6) {\scriptsize  b+d};

\node at (0.2,0.4) { };
\node at (0.2,-0.45) {\scriptsize {\setlength{\fboxsep}{0.6pt}\fbox{b}}};
\node at (0.7,0.05){};
\end{tikzpicture}
}

\newcommand{\tGammaone}{
\begin{tikzpicture}[scale=0.9,baseline=-0.1em]
  \draw[-,line width=0.35mm] (0,0) -- (0.6,0.6);
  \draw[-,line width=0.35mm] (0,0) -- (-0.6,0.6);
  \draw[-,line width=0.35mm] (0,0) -- (-0.6,-0.6);
  \draw[-,line width=0.35mm] (0,0) -- (0.6,-0.6);

  \draw[-,line width=0.3mm] (0.6,0.6) -- (0.6,-0.6);

  \draw[line width=0.6mm] (-0.6,0.8) -- (-0.6,0.6);
  \draw[line width=0.6mm] (0.6,0.8) -- (0.6,0.6);
  \draw[line width=0.6mm] (-0.6,-0.8) -- (-0.6,-0.6);
  \draw[line width=0.6mm] (0.6,-0.8) -- (0.6,-0.6);

  \node[above left] at (-0.6,0.6) {\scriptsize c};
  \node[above right] at (0.6,0.6) { };
  \node[below left] at (-0.6,-0.6) {\scriptsize a};
  \node[below right] at (0.6,-0.6) {\scriptsize b+d};

\node[rotate=45] at (0.25,0.45)
{$\scriptstyle \scalebox{0.63}{$a-s$}$};
\node at (0.7,0.05){ };
\node[rotate=-45] at (0.25,-0.43)
{$\scriptstyle \scalebox{0.63}{$c-s$}$};
\node at (0.7,0.05){ };
\end{tikzpicture}
}

\newcommand{\sGammak}{
\begin{tikzpicture}[scale=0.9,baseline=-0.1em]
  \draw[-,line width=0.35mm] (0,0) -- (0.6,0.6);
  \draw[-,line width=0.35mm] (0,0) -- (-0.6,0.6);
  \draw[-,line width=0.35mm] (0,0) -- (-0.6,-0.6);
  \draw[-,line width=0.35mm] (0,0) -- (0.6,-0.6);
  \draw[dashed,line width=0.35mm] (0.6,0.6) -- (1.2,1.2);
  \draw[dashed,line width=0.35mm] (0.6,-0.6) -- (1.2,-1.2);
  \draw[-,line width=0.35mm] (1.2,1.2) -- (1.8,1.8);
  \draw[-,line width=0.35mm] (1.2,-1.2) -- (1.8,-1.8);
  \draw[-,line width=0.35mm] (1.8,1.8) -- (2.4,2.4);
  \draw[-,line width=0.35mm] (1.8,-1.8) -- (2.4,-2.4);

  \draw[-,line width=0.35mm] (0.6,0.6) -- (0.6,-0.6);
  \draw[-,line width=0.35mm] (1.2,1.2) -- (1.2,-1.2);
  \draw[-,line width=0.35mm] (1.8,1.8) -- (1.8,-1.8);
  \draw[-,line width=0.35mm] (2.4,2.4) -- (2.4,-2.4);
  \draw[-,line width=0.35mm] (0.6,0.6) -- (0.6,1.2);
  \draw[-,line width=0.35mm] (0.6,-0.6) -- (0.6,-1.2);
  \draw[-,line width=0.35mm] (1.2,1.2) -- (1.2,1.8);
  \draw[-,line width=0.35mm] (1.2,-1.2) -- (1.2,-1.8);
  \draw[-,line width=0.35mm] (1.8,1.8) -- (1.8,2.4);
  \draw[-,line width=0.35mm] (1.8,-1.8) -- (1.8,-2.4);

  \node[above right] at (1.7,2.3) {\scriptsize {\setlength{\fboxsep}{0.6pt}\fbox{m}}};
  \node[below right] at (1.7,-2.3) {\scriptsize n};
\node at (1.5,-1.3) {\scriptsize {\setlength{\fboxsep}{0.6pt}\fbox{c}}};
\node at (2.1,-1.8) {\scriptsize {\setlength{\fboxsep}{0.6pt}\fbox{y}}};
\node at (1.5,1.35) {\scriptsize d};
\node at (2.1,1.9) {\scriptsize x};
\node at (0.95,0) {$\cdots$};
\end{tikzpicture}
}

\newcommand{\tGammak}{
\begin{tikzpicture}[scale=0.9,baseline=-0.1em]
  \draw[-,line width=0.35mm] (0,0) -- (0.6,0.6);
  \draw[-,line width=0.35mm] (0,0) -- (-0.6,0.6);
  \draw[-,line width=0.35mm] (0,0) -- (-0.6,-0.6);
  \draw[-,line width=0.35mm] (0,0) -- (0.6,-0.6);
  \draw[dashed,line width=0.35mm] (0.6,0.6) -- (1.2,1.2);
  \draw[dashed,line width=0.35mm] (0.6,-0.6) -- (1.2,-1.2);
  \draw[-,line width=0.35mm] (1.2,1.2) -- (1.8,1.8);
  \draw[-,line width=0.35mm] (1.2,-1.2) -- (1.8,-1.8);
  \draw[-,line width=0.35mm] (1.8,1.8) -- (2.7,2.7);
  \draw[-,line width=0.35mm] (1.8,-1.8) -- (2.7,-2.7);

  \draw[-,line width=0.35mm] (0.6,0.6) -- (0.6,-0.6);
  \draw[-,line width=0.35mm] (1.2,1.2) -- (1.2,-1.2);
  \draw[-,line width=0.35mm] (1.8,1.8) -- (1.8,-1.8);
  \draw[-,line width=0.35mm] (2.7,2.7) -- (2.7,-2.7);
  \draw[-,line width=0.35mm] (0.6,0.6) -- (0.6,1.2);
  \draw[-,line width=0.35mm] (0.6,-0.6) -- (0.6,-1.2);
  \draw[-,line width=0.35mm] (1.2,1.2) -- (1.2,1.8);
  \draw[-,line width=0.35mm] (1.2,-1.2) -- (1.2,-1.8);
  \draw[-,line width=0.35mm] (1.8,1.8) -- (1.8,2.4);
  \draw[-,line width=0.35mm] (1.8,-1.8) -- (1.8,-2.4);

  \node[above right] at (1.7,2.3) {\scriptsize m};
  \node[below right] at (1.7,-2.3) {\scriptsize n};
\node at (1.5,-1.35) {\scriptsize c};
\node[rotate=-45] at (2.3,-2.05)
{$\scriptstyle \scalebox{0.55}{$c+m-t$}$};
\node at (0.95,0) {$\cdots$};
\node at (1.5,1.35) {\scriptsize d};
\node[rotate=45] at (2.3,2.05)
{$\scriptstyle \scalebox{0.55}{$d+n-t$}$};
\node at (0.95,0) {$\cdots$};
\end{tikzpicture}
}

\newcommand{\sGammakprime}{
\begin{tikzpicture}[scale=0.9,baseline=-0.1em]
  \draw[-,line width=0.35mm] (0,0) -- (0.6,0.6);
  \draw[-,line width=0.35mm] (0,0) -- (-0.6,0.6);
  \draw[-,line width=0.35mm] (0,0) -- (-0.6,-0.6);
  \draw[-,line width=0.35mm] (0,0) -- (0.6,-0.6);
  \draw[-, line width=0.35mm] (0.6,-0.6) -- (1.8,0.6);
  \draw[-, line width=0.35mm] (0.6,0.6) -- (1.8,-0.6);
  \draw[-,line width=0.35mm] (2.6,0.6) -- (3.8,-0.6);
  \draw[-,line width=0.35mm] (2.6,-0.6) -- (3.8,0.6);

  \draw[-,line width=0.35mm] (3.8,0.6) -- (3.8,-0.6);

  \draw[line width=0.6mm] (-0.6,0.8) -- (-0.6,0.6);
  \draw[line width=0.6mm] (-0.6,-0.8) -- (-0.6,-0.6);

  \node[above left] at (-0.6,0.6) {\scriptsize {\setlength{\fboxsep}{0.6pt}\fbox{c}}};
  \node[above right] at (0.6,0.6) { };
  \node[below left] at (-0.6,-0.6) {\scriptsize a};

\node at (0.25,-0.45) {\scriptsize {\setlength{\fboxsep}{0.6pt}\fbox{b}}};
\node at (0.7,0.05){ };
\node at (2.2,0) {$\cdots$};
\end{tikzpicture}
}

\newcommand{\tGammakprime}{
\begin{tikzpicture}[scale=0.9,baseline=-0.1em]
  \draw[-,line width=0.35mm] (0,0) -- (0.6,0.6);
  \draw[-,line width=0.35mm] (0,0) -- (-0.6,0.6);
  \draw[-,line width=0.35mm] (0,0) -- (-0.6,-0.6);
  \draw[-,line width=0.35mm] (0,0) -- (0.6,-0.6);
  \draw[-, line width=0.35mm] (0.6,-0.6) -- (1.8,0.6);
  \draw[-, line width=0.35mm] (0.6,0.6) -- (1.8,-0.6);
  \draw[-,line width=0.35mm] (2.6,0.6) -- (3.8,-0.6);
  \draw[-,line width=0.35mm] (2.6,-0.6) -- (3.8,0.6);

  \draw[-,line width=0.35mm] (3.8,0.6) -- (3.8,-0.6);

  \draw[line width=0.6mm] (-0.6,0.8) -- (-0.6,0.6);
  \draw[line width=0.6mm] (-0.6,-0.8) -- (-0.6,-0.6);

  \node[above left] at (-0.6,0.6) {\scriptsize c};
  \node[above right] at (0.6,0.6) { };
  \node[below left] at (-0.6,-0.6) {\scriptsize a};

\node[rotate=45] at (0.25,0.45)
{$\scriptstyle \scalebox{0.63}{$a-s$}$};
\node at (0.7,0.05){ };
\node[rotate=-45] at (0.25,-0.43)
{$\scriptstyle \scalebox{0.63}{$c-s$}$};
\node at (0.7,0.05){ };
\node at (2.2,0) {$\cdots$};
\end{tikzpicture}
}

\newcommand{\sFthree}{
\begin{tikzpicture}[scale=0.9,baseline=-0.1em]
  \draw[-,line width=0.35mm] (0,0) -- (0.6,0.6);
  \draw[-,line width=0.35mm] (0,0) -- (-0.6,0.6);
  \draw[-,line width=0.35mm] (0,0) -- (-0.6,-0.6);
  \draw[-,line width=0.35mm] (0,0) -- (0.6,-0.6);

  \draw[-,line width=0.3mm] (-0.6,-0.6) -- (-0.6,0.6);

  \draw[line width=0.6mm] (-0.6,0.8) -- (-0.6,0.6);
  \draw[line width=0.6mm] (0.6,0.8) -- (0.6,0.6);
  \draw[line width=0.6mm] (-0.6,-0.8) -- (-0.6,-0.6);
  \draw[line width=0.6mm] (0.6,-0.8) -- (0.6,-0.6);

  \node[below left] at (-0.6,-0.6) {\scriptsize a};
  \node[below right] at (0.6,-0.6) {\scriptsize b};

\node at (-0.2,0.4) {$\scriptstyle c$};
\node at (-0.2,-0.4) {$\scriptstyle d$};
\end{tikzpicture}
}

\newcommand{\tFthree}{
\begin{tikzpicture}[scale=0.9,baseline=-0.1em]
  \draw[-,line width=0.35mm] (0,0) -- (0.6,0.6);
  \draw[-,line width=0.35mm] (0,0) -- (-0.6,0.6);
  \draw[-,line width=0.35mm] (0,0) -- (-0.6,-0.6);
  \draw[-,line width=0.35mm] (0,0) -- (0.6,-0.6);

  \draw[-,line width=0.3mm] (0.6,0.6) -- (0.6,-0.6);

  \draw[line width=0.6mm] (-0.6,0.8) -- (-0.6,0.6);
  \draw[line width=0.6mm] (0.6,0.8) -- (0.6,0.6);
  \draw[line width=0.6mm] (-0.6,-0.8) -- (-0.6,-0.6);
  \draw[line width=0.6mm] (0.6,-0.8) -- (0.6,-0.6);

  \node[above right] at (0.6,0.6) { };
  \node[below left] at (-0.6,-0.6) {\scriptsize a};
  \node[below right] at (0.6,-0.6) {\scriptsize b};

\node[rotate=45] at (0.25,0.45)
{$\scriptstyle \scalebox{0.63}{$c-s$}$};
\node at (0.7,0.05){ };
\node[rotate=-45] at (0.25,-0.43)
{$\scriptstyle \scalebox{0.63}{$d-s$}$};
\node at (0.7,0.05){ };
\end{tikzpicture}
}

\newcommand{\sFtwo}{
\begin{tikzpicture}[scale=0.55,baseline=(current bounding box.center),rotate=180]
  \draw[line width=0.6mm]  (0,0) --  (0,-0.16);

  \draw[line width=0.3mm] (0,0) -- (-0.8,1.2);
  \draw[line width=0.3mm] (0,0) -- (-0.4,1.2);
  \draw[line width=0.3mm,dotted] (0,0) -- (0,1.2);
  \draw[line width=0.3mm] (0,0) -- (0.4,1.2);
  \draw[line width=0.3mm] (0,0) -- (0.8,1.2);

  \node[above] at (-0.9,2) {$a_m$};
  \node[above] at (0.8,2) {$a_1$};
\end{tikzpicture}
}

\newcommand{\tFtwo}{
\begin{tikzpicture}[scale=0.55,baseline=(current bounding box.center)]
  \draw[line width=0.6mm]  (0,0) --  (0,-0.16);

  \draw[line width=0.3mm] (0,0) -- (-0.8,1.2);
  \draw[line width=0.3mm] (0,0) -- (-0.4,1.2);
  \draw[line width=0.3mm,dotted] (0,0) -- (0,1.2);
  \draw[line width=0.3mm] (0,0) -- (0.4,1.2);
  \draw[line width=0.3mm] (0,0) -- (0.8,1.2);
\end{tikzpicture}
}

\newcommand{\ttFtwo}{
\begin{tikzpicture}[scale=0.55,baseline=-3em]
  \draw[line width=0.8mm]  (0,0) --  (0,-3);
  \node[below left] at (1.65,-3) {\scriptsize $a_1+\cdots + a_m$};
\end{tikzpicture}
}

\newcommand{\twofivetothreefour}{
\begin{tikzpicture}[scale=1]
  \draw[-,line width=0.35mm] (0,0) -- (0.6,0.6);
  \draw[-,line width=0.35mm] (0,0) -- (-0.6,0.6);
  \draw[-,line width=0.35mm] (0,0) -- (-0.6,-0.6);
  \draw[-,line width=0.35mm] (0,0) -- (0.6,-0.6);

  \draw[-,line width=0.3mm] (0.6,0.6) -- (0.6,-0.6);

  \draw[line width=0.6mm] (-0.6,0.8) -- (-0.6,0.6);
  \draw[line width=0.6mm] (0.6,0.8) -- (0.6,0.6);
  \draw[line width=0.6mm] (-0.6,-0.8) -- (-0.6,-0.6);
  \draw[line width=0.6mm] (0.6,-0.8) -- (0.6,-0.6);

  \node[above left] at (-0.6,0.6) {\scriptsize 3};
  \node[above right] at (0.6,0.6) {\scriptsize 4};
  \node[below left] at (-0.6,-0.6) {\scriptsize 2};
  \node[below right] at (0.6,-0.6) {\scriptsize 5};

\node at (-0.2,0.5) {$\scriptstyle c^{11}$};
\node at (-0.2,-0.5) {$\scriptstyle c_{11}$};
\node at (0.2,0.5) {$\scriptstyle c_{12}$};
\node at (0.2,-0.5) {$\scriptstyle c^{12}$};
\node at (0.8,0.05){$\scriptstyle c_2$};
\end{tikzpicture}
}

\newcommand{\twofivetothreefouri}{
\begin{tikzpicture}[scale=0.7]
  \draw[-,line width=0.35mm] (0,0) -- (0.6,0.6);
  \draw[-,line width=0.35mm] (0,0) -- (-0.6,0.6);
  \draw[-,line width=0.35mm] (0,0) -- (-0.6,-0.6);
  \draw[-,line width=0.35mm] (0,0) -- (0.6,-0.6);

  \draw[-,line width=0.3mm] (0.6,0.6) -- (0.6,-0.6);

  \draw[line width=0.6mm] (-0.6,0.8) -- (-0.6,0.6);
  \draw[line width=0.6mm] (0.6,0.8) -- (0.6,0.6);
  \draw[line width=0.6mm] (-0.6,-0.8) -- (-0.6,-0.6);
  \draw[line width=0.6mm] (0.6,-0.8) -- (0.6,-0.6);

  \node[above left] at (-0.6,0.6) {\scriptsize 3};
  \node[above right] at (0.6,0.6) {\scriptsize 4};
  \node[below left] at (-0.6,-0.6) {\scriptsize 2};
  \node[below right] at (0.6,-0.6) {\scriptsize 5};

\node at (0.2,0.4) {$\scriptstyle 4$};
\node at (0.2,-0.4) {$\scriptstyle 5$};
\node at (0.7,0.05){$\scriptstyle 0$};
\end{tikzpicture}
}

\newcommand{\twofivetothreefourii}{
\begin{tikzpicture}[scale=0.7]
  \draw[-,line width=0.35mm] (0,0) -- (0.6,0.6);
  \draw[-,line width=0.35mm] (0,0) -- (-0.6,0.6);
  \draw[-,line width=0.35mm] (0,0) -- (-0.6,-0.6);
  \draw[-,line width=0.35mm] (0,0) -- (0.6,-0.6);

  \draw[-,line width=0.3mm] (0.6,0.6) -- (0.6,-0.6);

  \draw[line width=0.6mm] (-0.6,0.8) -- (-0.6,0.6);
  \draw[line width=0.6mm] (0.6,0.8) -- (0.6,0.6);
  \draw[line width=0.6mm] (-0.6,-0.8) -- (-0.6,-0.6);
  \draw[line width=0.6mm] (0.6,-0.8) -- (0.6,-0.6);

  \node[above left] at (-0.6,0.6) {\scriptsize 3};
  \node[above right] at (0.6,0.6) {\scriptsize 4};
  \node[below left] at (-0.6,-0.6) {\scriptsize 2};
  \node[below right] at (0.6,-0.6) {\scriptsize 5};

\node at (0.2,0.4) {$\scriptstyle 3$};
\node at (0.2,-0.4) {$\scriptstyle 4$};
\node at (0.7,0.05){$\scriptstyle 1$};
\end{tikzpicture}
}

\newcommand{\twofivetothreefouriii}{
\begin{tikzpicture}[scale=0.7]
  \draw[-,line width=0.35mm] (0,0) -- (0.6,0.6);
  \draw[-,line width=0.35mm] (0,0) -- (-0.6,0.6);
  \draw[-,line width=0.35mm] (0,0) -- (-0.6,-0.6);
  \draw[-,line width=0.35mm] (0,0) -- (0.6,-0.6);

  \draw[-,line width=0.3mm] (0.6,0.6) -- (0.6,-0.6);

  \draw[line width=0.6mm] (-0.6,0.8) -- (-0.6,0.6);
  \draw[line width=0.6mm] (0.6,0.8) -- (0.6,0.6);
  \draw[line width=0.6mm] (-0.6,-0.8) -- (-0.6,-0.6);
  \draw[line width=0.6mm] (0.6,-0.8) -- (0.6,-0.6);

  \node[above left] at (-0.6,0.6) {\scriptsize 3};
  \node[above right] at (0.6,0.6) {\scriptsize 4};
  \node[below left] at (-0.6,-0.6) {\scriptsize 2};
  \node[below right] at (0.6,-0.6) {\scriptsize 5};

\node at (0.2,0.4) {$\scriptstyle 2$};
\node at (0.2,-0.4) {$\scriptstyle 3$};
\node at (0.7,0.05){$\scriptstyle 2$};
\end{tikzpicture}
}

\newcommand{\twofivetothreefouriv}{
\begin{tikzpicture}[scale=0.7]
  \draw[-,line width=0.35mm] (0,0) -- (0.6,0.6);
  \draw[-,line width=0.35mm] (0,0) -- (-0.6,0.6);
  \draw[-,line width=0.35mm] (0,0) -- (-0.6,-0.6);
  \draw[-,line width=0.35mm] (0,0) -- (0.6,-0.6);

  \draw[-,line width=0.3mm] (0.6,0.6) -- (0.6,-0.6);

  \draw[line width=0.6mm] (-0.6,0.8) -- (-0.6,0.6);
  \draw[line width=0.6mm] (0.6,0.8) -- (0.6,0.6);
  \draw[line width=0.6mm] (-0.6,-0.8) -- (-0.6,-0.6);
  \draw[line width=0.6mm] (0.6,-0.8) -- (0.6,-0.6);

  \node[above left] at (-0.6,0.6) {\scriptsize 3};
  \node[above right] at (0.6,0.6) {\scriptsize 4};
  \node[below left] at (-0.6,-0.6) {\scriptsize 2};
  \node[below right] at (0.6,-0.6) {\scriptsize 5};

\node at (0.2,0.4) {$\scriptstyle 1$};
\node at (0.2,-0.4) {$\scriptstyle 2$};
\node at (0.7,0.05){$\scriptstyle 3$};
\end{tikzpicture}
}

\newcommand{\twofivetothreefourv}{
\begin{tikzpicture}[scale=0.7]
  \draw[-,line width=0.35mm] (0,0) -- (0.6,0.6);
  \draw[-,line width=0.35mm] (0,0) -- (-0.6,0.6);
  \draw[-,line width=0.35mm] (0,0) -- (-0.6,-0.6);
  \draw[-,line width=0.35mm] (0,0) -- (0.6,-0.6);

  \draw[-,line width=0.3mm] (0.6,0.6) -- (0.6,-0.6);

  \draw[line width=0.6mm] (-0.6,0.8) -- (-0.6,0.6);
  \draw[line width=0.6mm] (0.6,0.8) -- (0.6,0.6);
  \draw[line width=0.6mm] (-0.6,-0.8) -- (-0.6,-0.6);
  \draw[line width=0.6mm] (0.6,-0.8) -- (0.6,-0.6);

  \node[above left] at (-0.6,0.6) {\scriptsize 3};
  \node[above right] at (0.6,0.6) {\scriptsize 4};
  \node[below left] at (-0.6,-0.6) {\scriptsize 2};
  \node[below right] at (0.6,-0.6) {\scriptsize 5};

\node at (0.2,0.4) {$\scriptstyle 0$};
\node at (0.2,-0.4) {$\scriptstyle 1$};
\node at (0.7,0.05){$\scriptstyle 4$};
\end{tikzpicture}
}

\newcommand{\NFBi}[1][]{
\begin{tikzpicture}[scale=0.7,#1]
  \draw[-,line width=0.35mm] (0,0) -- (0.6,0.6);
  \draw[-,line width=0.35mm] (0,0) -- (-0.6,0.6);
  \draw[-,line width=0.35mm] (0,0) -- (-0.6,-0.6);
  \draw[-,line width=0.35mm] (0,0) -- (0.6,-0.6);

  \draw[-,line width=0.3mm] (0.6,0.6) -- (0.6,-0.6);

  \draw[line width=0.6mm] (-0.6,0.8) -- (-0.6,0.6);
  \draw[line width=0.6mm] (0.6,0.8) -- (0.6,0.6);
  \draw[line width=0.6mm] (-0.6,-0.8) -- (-0.6,-0.6);
  \draw[line width=0.6mm] (0.6,-0.8) -- (0.6,-0.6);

  \node[above left] at (-0.6,0.6) {\scriptsize 3};
  \node[above right] at (0.6,0.6) {\scriptsize 4};
  \node[below left] at (-0.6,-0.6) {\scriptsize 2};
  \node[below right] at (0.6,-0.6) {\scriptsize 5};

\node at (0.2,0.4) {$\scriptstyle 1$};
\node at (0.2,-0.4) {$\scriptstyle 2$};
\node at (0.7,0.05){$\scriptstyle 3$};
\end{tikzpicture}
}

\newcommand{\NFBii}[1][]{
\begin{tikzpicture}[scale=0.7,#1]
  \draw[-,line width=0.35mm] (0,0) -- (0.6,0.6);
  \draw[-,line width=0.35mm] (0,0) -- (-0.6,0.6);
  \draw[-,line width=0.35mm] (0,0) -- (-0.6,-0.6);
  \draw[-,line width=0.35mm] (0,0) -- (0.6,-0.6);

  \draw[line width=0.6mm] (-0.6,0.8) -- (-0.6,0.6);
  \draw[line width=0.6mm] (0.6,0.8) -- (0.6,0.6);
  \draw[line width=0.6mm] (-0.6,-0.8) -- (-0.6,-0.6);
  \draw[line width=0.6mm] (0.6,-0.8) -- (0.6,-0.6);

  \node[above left] at (-0.6,0.6) {\scriptsize 3};
  \node[above right] at (0.6,0.6) {\scriptsize 4};
  \node[below left] at (-0.6,-0.6) {\scriptsize 2};
  \node[below right] at (0.6,-0.6) {\scriptsize 5};
\end{tikzpicture}
}

\newcommand{\NFBiii}[1][]{
\begin{tikzpicture}[scale=0.7,#1]
  \draw[-,line width=0.35mm] (0,0) -- (-0.6,0.6);
  \draw[-,line width=0.35mm] (0,0) -- (0.6,-0.6);

  \draw[-,line width=0.3mm] (-0.6,0.6) -- (-0.6,-0.6);

  \draw[-,line width=0.3mm] (0.6,0.6) -- (0.6,-0.6);

  \draw[line width=0.6mm] (-0.6,0.8) -- (-0.6,0.6);
  \draw[line width=0.6mm] (0.6,0.8) -- (0.6,0.6);
  \draw[line width=0.6mm] (-0.6,-0.8) -- (-0.6,-0.6);
  \draw[line width=0.6mm] (0.6,-0.8) -- (0.6,-0.6);

  \node[above left] at (-0.6,0.6) {\scriptsize 3};
  \node[above right] at (0.6,0.6) {\scriptsize 4};
  \node[below left] at (-0.6,-0.6) {\scriptsize 2};
  \node[below right] at (0.6,-0.6) {\scriptsize 5};

\node at (0,0.15) {$\scriptstyle 1$};
\node at (-0.7,0.05) {$\scriptstyle 2$};
\node at (0.7,0.05){$\scriptstyle 4$};
\end{tikzpicture}
}

\newcommand{\SBi}[1][]{
\begin{tikzpicture}[scale=0.7,#1]
  \draw[-,line width=0.35mm] (0,0) -- (0.6,0.6);
  \draw[-,line width=0.35mm] (0,0) -- (-0.6,0.6);
  \draw[-,line width=0.35mm] (0,0) -- (-0.6,-0.6);
  \draw[-,line width=0.35mm] (0,0) -- (0.6,-0.6);

  \draw[-,line width=0.3mm] (0.6,0.6) -- (0.6,-0.6);

  \draw[line width=0.6mm] (-0.6,0.8) -- (-0.6,0.6);
  \draw[line width=0.6mm] (0.6,0.8) -- (0.6,0.6);
  \draw[line width=0.6mm] (-0.6,-0.8) -- (-0.6,-0.6);
  \draw[line width=0.6mm] (0.6,-0.8) -- (0.6,-0.6);

  \node[above left] at (-0.6,0.6) {\scriptsize 3};
  \node[above right] at (0.6,0.6) {\scriptsize 4};
  \node[below left] at (-0.6,-0.6) {\scriptsize 2};
  \node[below right] at (0.6,-0.6) {\scriptsize 5};

\node at (0.2,0.4) {$\scriptstyle 2$};
\node at (0.2,-0.4) {$\scriptstyle 3$};
\node at (0.7,0.05){$\scriptstyle 2$};
\end{tikzpicture}
}

\newcommand{\SBii}[1][]{
\begin{tikzpicture}[scale=0.7,#1]
  \draw[-,line width=0.35mm] (0,0) -- (0.6,0.6);
  \draw[-,line width=0.35mm] (0,0) -- (-0.6,0.6);
  \draw[-,line width=0.35mm] (0,0) -- (-0.6,-0.6);
  \draw[-,line width=0.35mm] (0,0) -- (0.6,-0.6);

  \draw[line width=0.6mm] (-0.6,0.8) -- (-0.6,0.6);
  \draw[line width=0.6mm] (0.6,0.8) -- (0.6,0.6);
  \draw[line width=0.6mm] (-0.6,-0.8) -- (-0.6,-0.6);
  \draw[line width=0.6mm] (0.6,-0.8) -- (0.6,-0.6);

  \draw[-,line width=0.3mm] (-0.6,0.6) -- (-0.6,-0.6);

  \draw[-,line width=0.3mm] (0.6,0.6) -- (0.6,-0.6);

  \node[above left] at (-0.6,0.6) {\scriptsize 3};
  \node[above right] at (0.6,0.6) {\scriptsize 4};
  \node[below left] at (-0.6,-0.6) {\scriptsize 2};
  \node[below right] at (0.6,-0.6) {\scriptsize 5};

\node at (0.2,0.4) {$\scriptstyle 1$};
\node at (0.2,-0.4) {$\scriptstyle 2$};
\node at (0.7,0.05){$\scriptstyle 3$};
\node at (-0.7,0.05){$\scriptstyle 1$};
\end{tikzpicture}
}

\newcommand{\SBiii}[1][]{
\begin{tikzpicture}[scale=0.7,#1]
  \draw[-,line width=0.35mm] (0,0) -- (-0.6,0.6);
  \draw[-,line width=0.35mm] (0,0) -- (0.6,-0.6);

  \draw[-,line width=0.3mm] (-0.6,0.6) -- (-0.6,-0.6);

  \draw[-,line width=0.3mm] (0.6,0.6) -- (0.6,-0.6);

  \draw[line width=0.6mm] (-0.6,0.8) -- (-0.6,0.6);
  \draw[line width=0.6mm] (0.6,0.8) -- (0.6,0.6);
  \draw[line width=0.6mm] (-0.6,-0.8) -- (-0.6,-0.6);
  \draw[line width=0.6mm] (0.6,-0.8) -- (0.6,-0.6);

  \node[above left] at (-0.6,0.6) {\scriptsize 3};
  \node[above right] at (0.6,0.6) {\scriptsize 4};
  \node[below left] at (-0.6,-0.6) {\scriptsize 2};
  \node[below right] at (0.6,-0.6) {\scriptsize 5};

\node at (0,0.15) {$\scriptstyle 1$};
\node at (-0.7,0.05) {$\scriptstyle 2$};
\node at (0.7,0.05){$\scriptstyle 4$};
\end{tikzpicture}
}

\newcommand{\CoBi}[1][]{
\begin{tikzpicture}[scale=0.7,#1]
  \draw[-,line width=0.35mm] (0,0) -- (0.6,0.6);
  \draw[-,line width=0.35mm] (0,0) -- (-0.6,0.6);
  \draw[-,line width=0.35mm] (0,0) -- (-0.6,-0.6);
  \draw[-,line width=0.35mm] (0,0) -- (0.6,-0.6);

  \draw[line width=0.6mm] (-0.6,0.8) -- (-0.6,0.6);
  \draw[line width=0.6mm] (0.6,0.8) -- (0.6,0.6);
  \draw[line width=0.6mm] (-0.6,-0.8) -- (-0.6,-0.6);
  \draw[line width=0.6mm] (0.6,-0.8) -- (0.6,-0.6);

  \node[above left] at (-0.6,0.6) {\scriptsize 3};
  \node[above right] at (0.6,0.6) {\scriptsize 4};
  \node[below left] at (-0.6,-0.6) {\scriptsize 2};
  \node[below right] at (0.6,-0.6) {\scriptsize 5};
\end{tikzpicture}
}

\newcommand{\CoBii}[1][]{
\begin{tikzpicture}[scale=0.7,#1]
  \draw[-,line width=0.35mm] (0,0) -- (0.6,0.6);
  \draw[-,line width=0.35mm] (0,0) -- (-0.6,0.6);
  \draw[-,line width=0.35mm] (0,0) -- (-0.6,-0.6);
  \draw[-,line width=0.35mm] (0,0) -- (0.6,-0.6);

  \draw[-,line width=0.3mm] (0.6,0.6) -- (0.6,-0.6);

  \draw[line width=0.6mm] (-0.6,0.8) -- (-0.6,0.6);
  \draw[line width=0.6mm] (0.6,0.8) -- (0.6,0.6);
  \draw[line width=0.6mm] (-0.6,-0.8) -- (-0.6,-0.6);
  \draw[line width=0.6mm] (0.6,-0.8) -- (0.6,-0.6);

  \node[above left] at (-0.6,0.6) {\scriptsize 3};
  \node[above right] at (0.6,0.6) {\scriptsize 4};
  \node[below left] at (-0.6,-0.6) {\scriptsize 2};
  \node[below right] at (0.6,-0.6) {\scriptsize 5};

\node at (0.2,0.4) {$\scriptstyle 3$};
\node at (0.2,-0.4) {$\scriptstyle 4$};
\node at (0.7,0.05){$\scriptstyle 1$};
\end{tikzpicture}
}

\newcommand{\CoBiii}[1][]{
\begin{tikzpicture}[scale=0.7,#1]
  \draw[-,line width=0.35mm] (0,0) -- (0.6,0.6);
  \draw[-,line width=0.35mm] (0,0) -- (-0.6,0.6);
  \draw[-,line width=0.35mm] (0,0) -- (-0.6,-0.6);
  \draw[-,line width=0.35mm] (0,0) -- (0.6,-0.6);

  \draw[-,line width=0.3mm] (0.6,0.6) -- (0.6,-0.6);

  \draw[line width=0.6mm] (-0.6,0.8) -- (-0.6,0.6);
  \draw[line width=0.6mm] (0.6,0.8) -- (0.6,0.6);
  \draw[line width=0.6mm] (-0.6,-0.8) -- (-0.6,-0.6);
  \draw[line width=0.6mm] (0.6,-0.8) -- (0.6,-0.6);

  \node[above left] at (-0.6,0.6) {\scriptsize 3};
  \node[above right] at (0.6,0.6) {\scriptsize 4};
  \node[below left] at (-0.6,-0.6) {\scriptsize 2};
  \node[below right] at (0.6,-0.6) {\scriptsize 5};

\node at (0.2,0.4) {$\scriptstyle 2$};
\node at (0.2,-0.4) {$\scriptstyle 3$};
\node at (0.7,0.05){$\scriptstyle 2$};
\end{tikzpicture}
}

\newcount\hh
\newcount\mm
\mm=\time
\hh=\time
\divide\hh by 60
\divide\mm by 60
\multiply\mm by 60
\mm=-\mm
\advance\mm by \time
\def\hhmm{\number\hh:\ifnum\mm<10{}0\fi\number\mm}

\definecolor{vert}{rgb}{0,0.45,0}
\definecolor{rouge}{rgb}{0.89,0.04,0.36}
\definecolor{oorange}{RGB}{148,76,5}
\definecolor{MyGray}{gray}{0.6}
\definecolor{MyRed}{RGB}{212,42,42}

\newcommand{\auteur}[3]{%
\noindent
\begin{minipage}[t]{.45\textwidth}
\begin{flushright}
\textsc{#1} \\
{\footnotesize\textsf{#2}}
\end{flushright}
\end{minipage}
\qquad
\begin{minipage}[t]{.45\textwidth}
#3
\end{minipage}%
}

\begin{document}
\thispagestyle{empty}

\vspace*{-2cm}

\begin{center}
\begin{doublespace}
\begin{Large}
{\bfseries Diagrammatic bases from stratified normalization}
\end{Large}

\vskip-0.4pt

\begin{large}
{\scshape Stéphane Gaussent -- Zuan Liu -- Philippe Malbos}
\end{large}
\end{doublespace}

\vskip+20pt

\begin{small}\begin{minipage}{14cm}
\noindent\textbf{Abstract --}
In this article, we introduce a new rewriting method for computing hom-bases of diagrammatic categories. Our approach is based on stratified presentations, obtained by partitioning an equational presentation into strata according to the rewriting behavior of its relations. We establish modularity principles for termination and confluence in stratified rewriting systems with coefficients in a ring. We then introduce a completion procedure that transforms a stratified presentation into a normalized one, in which mixed interactions between successive strata are absorbed into the normalized rewriting rules, thereby inducing modularity of confluence. This leads to a stratified hom-basis theorem for diagrammatic categories. We illustrate the method by applying it to the $q$-Schur category, for which we construct a new hom-basis.

\medskip


\smallskip\noindent\textbf{M.S.C. 2020 --} 18N30, 68Q42, 17B10, 17B37.

\end{minipage}
\end{small}

\begin{small}\begin{minipage}{12cm}
\renewcommand{\contentsname}{}
\setcounter{tocdepth}{0}
\tableofcontents
\end{minipage}
\end{small}
\end{center}

\section{Introduction}

The determination of explicit linear bases for morphism spaces in diagrammatic categories is a central problem in representation theory and higher algebra. Such bases provide concrete descriptions of morphisms, make computations effective, and often reveal structural properties of the categories under consideration. This problem appears prominently in many categories arising from categorification. For instance, the $q$-Schur category, as presented by Brundan, admits distinguished bases for its morphism spaces, closely related to classical bases of $q$-Schur algebras and to bases in $U_q(\mathfrak{gl}_n)$-web categories~\cite{Brundan2025}.
Diagrammatic integral bases for the Hom-spaces of affine and cyclotomic web and $q$-Schur categories have been constructed in~\cite{SongWang25}.
Another fundamental example is provided by Soergel bimodules: Libedinsky's light leaves basis, and subsequently the double leaves basis, give explicit combinatorial bases for morphism spaces between Bott--Samelson bimodules. Similar questions arise in the diagrammatic categorification of quantum groups, where understanding spaces of $2$-morphisms is an essential part of the theory.

A natural way to construct such bases from presentations by generators and relations is to orient the defining relations into rewriting rules and study the resulting rewriting system. 
When the system is terminating and confluent, normal forms provide canonical representatives and, under suitable monicity assumptions in the linear setting, yield bases of the corresponding hom-spaces. 
This principle underlies Gröbner--Shirshov basis method, originating in the work of Buchberger~\cite{Buchberger65} and Shirshov~\cite{Shirshov}, and has subsequently been extended to many algebraic structures, including operads~\cite{DotsenkoKhoroshkin10,DotsenkoKhoroshkin13,MalbosRen23}.

In the categorical and higher-dimensional setting, rewriting has developed into a diagrammatic theory adapted to presentations of monoidal categories and $2$-categories. 
Elias~\cite{Elias22} developed a version of Bergman's diamond lemma~\cite{Bergman78} for monoidal categories presented by generators and relations, with applications to Hecke-type algebras and categories. In a higher-dimensional linear setting, Dupont~\cite{Dupont22} developed rewriting modulo isotopy relations for pivotal linear $2$-categories and used it to compute bases of spaces of $2$-cells. 
More recently, Schelstraete developed rewriting modulo for linear $2$-categories, and applied it to a basis theorem for a super-$2$-category related to odd Khovanov homology~\cite{Schelstraete26} and to categories of affine Brauer type~\cite{BarbierSchelstraete26}. Diagrammatic rewriting has now developed into an active line of research, with several complementary techniques and formalisms; see, for instance,~\cite{Polybook2025}.

Rewriting in linear monoidal categories and, more generally, in linear $2$-categories can be formalized using higher-dimensional rewriting structures such as linear polygraphs \cite{GuiraudHoffbeckMalbos19,Alleaume18,Dupont22,Schelstraete26}. In the approach adopted here, rewriting is formulated using $3$-prepolygraphs, whose underlying $2$-dimensional structure is a precategory rather than a strict $2$-category. Rewriting acts on morphisms, viewed as $2$-cells, while termination and confluence control the existence and uniqueness of normal forms.

These rewriting methods have been used to construct hom-bases in several diagrammatic settings. In particular, rewriting without relying on global monomial orders was developed for associative algebras in~\cite{GuiraudHoffbeckMalbos19} and subsequently adapted to diagrammatic categories, yielding constructive hom-bases for the affine oriented Brauer category~\cite{Alleaume18,Dupont22}, the Khovanov--Lauda--Rouquier category~\cite{Dupont21}, the loop Hecke algebras~\cite{JanssensLacabanneSchelstraeteVaz25}, and the graded $2$-category of $\mathfrak{gl}_2$-foams~\cite{Schelstraete26}.

Despite these developments, constructing convergent diagrammatic presentations remains difficult for two main reasons. The first concerns the application of the critical branching lemma, which characterizes confluence of the whole rewriting system in terms of local confluence on critical branchings. The first difficulty is combinatorial: in the setting of rewriting in $2$-categories, the analysis of critical branchings gives rise to intricate families of confluence obstructions~\cite{GuiraudMalbos09,GuiraudMalbos12mscs,Mimram14}. The second difficulty concerns termination, which requires the construction of a well-founded measure compatible with the orientation of all rewriting rules.

The purpose of this article is to address these two difficulties through a stratified approach to diagrammatic rewriting. Instead of treating all rewriting rules simultaneously, we partition them into strata according to their rewriting behavior. On the termination side, we introduce \emph{stratified termination functions}, which allow termination to be proved stratum by stratum by combining different termination methods, in particular derivation-based and monomial-order-based arguments. On the confluence side, we introduce a normalization procedure in which the interactions between a new stratum and the preceding ones are absorbed into normalized rewriting rules. This makes confluence modular with respect to the stratification and reduces its verification to confluence within the individual strata.

A further motivation for the present work comes from the coefficient rings occurring naturally in diagrammatic algebra. Many diagrammatic algebras are defined over a commutative ring, and it is therefore desirable to perform rewriting directly over the ground ring, without passing to a field extension in which the relevant coefficients become invertible. Rewriting with coefficients in a ring has already been developed in Gröbner basis theory and reduction rings, notably in the works of Buchberger~\cite{Buchberger84}, Möller~\cite{Moller88}, and Stifter~\cite{Stifter87,Stifter91}.

The absence of inverses requires a refinement of the usual reduction relation, since the applicability of a rewriting rule depends on divisibility properties of coefficients. In the higher-dimensional setting considered here, this arithmetic becomes part of the diagrammatic rewriting process itself: coefficient reduction affects both admissible rewriting steps and the structure of local and critical branchings. We therefore develop rewriting directly over the ground ring for linear monoidal categories and $2$-categories presented by $3$-prepolygraphs, together with the corresponding notions of reduction, critical branching, and completion.

As an application, we study Brundan’s presentation of the $q$-Schur category from~\cite[Thm.~2]{Brundan2025}. Applying the normalization procedure to this presentation, we construct a new hom-basis for the $q$-Schur category, called the \emph{normal form basis}. In Section~\ref{S:PictureCalculus}, we compare this basis with other known hom-bases for the $q$-Schur category, including the standard, canonical, and codeterminant bases.

\bigskip

The paper is organized as follows.
In Section~\ref{S:DiagrammaticRewriting}, we introduce a notion of diagrammatic rewriting over a commutative ring, formulated for linear monoidal categories and $2$-categories using linear $3$-prepolygraphs. We then recall the notions of termination, confluence, and normal forms, and explain under the relevant monicity assumptions how convergent presentations give rise to hom-bases. We also develop the rewriting-theoretic machinery used throughout the article.

Section~\ref{S:StratifiedPrepolygraphsTerminationStratification} introduces stratified $3$-prepolygraphs modulo and stratified termination functions. Theorem~\ref{T:MainTheoremOnTermination} states that any $3$-prepolygraph modulo admitting a stratified termination function is terminating modulo. This approach provides a modular proof of termination, carried out stratum by stratum and allowing different strata to be controlled by different termination methods, in particular derivation-based and monomial-order-based arguments. As an illustration, we revisit the presentation of Khovanov--Lauda--Rouquier algebras and obtain a streamlined proof of termination compared with that of~\cite{Dupont21}.

Section~\ref{S:HomBasesStratifiedNormalization} introduces a normalized completion procedure for stratified presentations and shows that, under suitable compatibility assumptions, confluence becomes a modular property that can be analysed stratum by stratum. Combined with the monicity assumptions required for linear normal forms, this leads to our main hom-basis theorem:
\begin{quote}
\noindent {\bf Theorem~\ref{T:MainTheorem}.}
\emph{
Let $\Cr$ be a linear monoidal category presented by a stratified $3$-prepolygraph $F^\ast X$.
If there exists an $\strat{1}X$-compatible monomial order $\preccurlyeq$ such that every stratum $\strat{i}\NP{X}$ of the associated $\preccurlyeq$-normalization $F^\ast\NP{X}$ is left-monomial and confluent, then $\Nf_m(\NP{X})$ forms a hom-basis of $\Cr$.
}
\end{quote}
Subsection~\ref{SS:HomBasisComputationProcedure} develops a completion procedure for computing hom-bases from stratified presentations whenever the procedure terminates.

In Section~\ref{S:NormFormHomBasisqSchurCategories}, we apply Theorem~\ref{T:MainTheorem} to Brundan’s presentation of the $q$-Schur category. We construct a stratified presentation, prove its termination, perform its normalization, and derive a new hom-basis, called the \emph{normal form basis}.

Finally, Section~\ref{S:PictureCalculus} gives an explicit diagrammatic description of the normal form basis. We develop a picture calculus for the normal forms associated with the stratified presentation and use it to compare the resulting basis with the standard, canonical, and codeterminant bases of the $q$-Schur category.

\section{Hom-bases by diagrammatic rewriting}
\label{S:DiagrammaticRewriting}

This section recalls the foundations of rewriting for diagrammatic categories formulated in the framework of linear $2$-categories. Higher-dimensional rewriting originates from a series of works on equational presentations of higher categories~\cite{Street76,Power91,Street87,Burroni93}. More recently, these methods have been developed through several approaches, depending on whether rewriting is performed modulo the interchange relations~\cite{GuiraudMalbos09,Alleaume18,Dupont22}, or whether these relations are made explicit as rewriting rules within the system itself~\cite{ForestMimram2022,Schelstraete26}.

\subsection{(Pre)polygraphs and diagrammatic rewriting}

We shall work within the framework of precategories~\cite{ForestMimram2022}, which generalize the structure of sesquicategories~\cite{Street96} to arbitrary dimensions.
Although we focus here on rewriting in two-dimensional precategories, most of the constructions developed below extend to arbitrary higher dimensions.

\subsubsection{Linear categories}
Let $\Cr$ be a $2$-category. 
For each $0\leq k\leq 2$, we use $\Cr_k$ to denote both the set of $k$-cells of $\Cr$ and the
corresponding $k$-category.
For $f \in \Cr_k$ and $i<k$, we write $s_i(f)$ and 
$t_i(f)$ for the $i$-source and $i$-target of $f$.
We further set $sf\coloneq s_{k-1}(f)$ and $tf\coloneq t_{k-1}(f)$ for $k>0$.

For $k>0$, a \emph{$k$-sphere} of $\Cr$ is a pair $(f,g)$ of parallel $k$-cells, that is, such that $sf=sg$ and $tf=tg$. 
For $f\in\Cr_k$, we denote by $f^{\partial}:=(sf,tf)$ the \emph{boundary $(k-1)$-sphere} of $f$.
The set of $k$-spheres of $\Cr$ is denoted by $\Sph_k(\Cr)$.
A \emph{(cellular) extension} of $\Cr_k$ is a set $X$ equipped with a map $\rho : X \fl \Sph_k(\Cr)$.
We denote by $\Exte{\Cr_k}$ the set of extensions of $\Cr_k$.
If $X$ and $Y$ are two such extensions, we denote their disjoint union by $XY$.

By regarding the elements of an extension $X$ of $\Cr_k$ as generating $(k+1)$-cells with sources and targets in $\Cr$, and by forming all their formal composites, one obtains the \emph{free $(k+1)$-category generated by $X$ over~$\Cr$}, denoted by $\Cr[X]$. The \emph{size} of a $(k+1)$-cell $f$ in $\Cr[X]$ is the number of occurrences of generating $(k+1)$-cells from $X$ in $f$, denoted by $\abs{f}$.
The quotient of $\Cr_k$ by $X$, denoted by $\Cr_k/X$, is the $k$-category obtained from $\Cr$ by identifying $s(\alpha)$ with $t(\alpha)$ for every $\alpha\in X$.

Recall that a \emph{linear $2$-category} over a commutative ring $\bk$ with unity is a $2$-category $\Cr$ such that, for every pair of $1$-cells $p,q \in \Cr_1$, the hom-set $\Cr_2(p,q)$ is endowed with a structure of a $\bk$-module, and such that, for all $p,q,r \in \Cr_1$, the composition functor 
\begin{eqn}{equation}
\label{E:starUnComposition}
\circ_1: \Cr_2(p,q)\otimes_{\bk} \Cr_2(q,r)\to \Cr_2(p,r)
\end{eqn}
is $\bk$-linear.
A \emph{linear $2$-precategory} is defined in the same way, except that the interchange law, or equivalently the functoriality of the composition~\SSS{E:starUnComposition}, is not required to hold.

A \emph{linear monoidal category} is a linear $2$-category with a single $0$-cell. Since this article is concerned with presentations of linear monoidal categories, all the $2$-categorical structures considered below are assumed to have a single $0$-cell. 
Nevertheless, all the constructions and results extend straightforwardly to the multi-object setting.

The $2$-categories and $2$-precategories considered here, as well as their linear counterparts, are generated by \emph{$2$-polygraphs}. 
A \emph{$2$-polygraph} is a datum $X_{\leq 2}=(X_0,X_1,X_2)$ consisting of a directed graph $(X_0,X_1)$, whose elements are called generating $0$-cells and generating $1$-cells, respectively, together with a cellular extension $X_2$ of the free category $X_1^\ast$ generated by this graph. In accordance with the convention above, we assume throughout that $X_0$ is a singleton.

We denote by $\cat{X_2}$, $\lcat{X_2}$, $\pcat{X_2}$, and $\plcat{X_2}$ the $2$-category, linear $2$-category, $2$-precategory, and linear $2$-precategory freely generated by $X$, respectively. These structures have the same underlying $1$-category, namely $X_1^\ast$.

A $2$-cell of $\pcat{X_2}$ is called a \emph{$2$-monomial}.
A $2$-cell of $\plcat{X_2}$ of the form $a h$, where $h$ is a $2$-monomial and $a\in\bk\setminus\{0\}$, is called a \emph{$2$-term}.
Every $2$-monomial $h$ admits a unique decomposition $h = h_1 \circ_1 \cdots \circ_1 h_{\abs{h}}$, where each $h_i$ is a $0$-composition $h_i = u_i \varphi_i v_i$, with $\varphi_i \in X_2$ and $u_i, v_i \in X_1^\ast$.

In the sequel, we will use a diagrammatic representation of $2$-cells. 
The $2$-monomial $h_i$ and the composition $h_i\circ_1 h_{i+1}$ will be respectively denoted by the following diagram
\[
h_i\:=\: \twocell{(id*0dot*0id)*1(ui*0 varphii *0 vi)*1(id*0dot*0id)},
\quad\text{and}\quad
h_i\circ_1h_{i+1} \:=\: \twocell{dot*1h-i1*1dot*1h-i*1dot}\,.
\]

The $2$-monomials of $\pcat{X_2}(p,q)$ form a basis of the $\bk$-modules $\plcat{X_2}(p,q)$ of $2$-cells, for all $1$-cells $p$ and $q$.
Equivalently, every $2$-cell $f$ of $\plcat{X_2}$ admits a unique decomposition $f=\sum_ia_i h_i$, where the $h_i$ are pairwise distinct $2$-monomials and $a_i\in \bk\setminus\{0\}$. 
The coefficient $a_i$ of the monomial $h_i$ in the decomposition of $f$ is denoted by $\cf{h_i}{f}$.
The set of $2$-terms $\{a_ih_i\}_i$ forms the \emph{support} of~$f$, denoted by $\supp(f)$, while the set of $2$-monomials $\{h_i\}_i$ forms its \emph{monomial support}, denoted by $\suppm(f)$.

The $2$-category $\cat{X_2}$ (resp. linear $2$-category $\lcat{X_2}$) is obtained as the quotient of the $2$-precategory~$\pcat{X_2}$ (resp. linear $2$-precategory $\plcat{X_2}$) by the \emph{interchange extension} $\Int(X_2)$.
This extension consists of $3$-generators
\begin{eqn}{equation}
\label{E:InterchangeRules}
\alpha_{\varphi,u,\psi} : (\varphi us_1(\psi))\circ_1(t_1(\varphi)u\psi)
\:\fl\:
(s_1(\varphi)u\psi)\circ_1 (\varphi ut_1(\psi)),
\end{eqn} 
for all $\varphi,\psi\in X_2$ and $u\in\cat{X_1}$, diagrammatically represented as
\[
\alpha_{\varphi,u,\psi} : 
\twocell{(dot*0 id *0 psi)*1(dot *0 u *0 dot)*1(varphi*0 id *0 dot)}
\fl
\twocell{(varphi*0 id *0 dot)*1(dot *0 u *0 dot)*1 (dot *0 id *0 psi)}\,.
\]

A \emph{(one-hole) context}, or simply a \emph{context}, of the $2$-precategory $\plcat{X_2}$ is a pair $(\square,C)$, denoted by $C[\square]$, or simply by $C$, where $\square$
is a $1$-sphere of $\cat{X_1}$ and $C$ is a $2$-term $ag$, with $a\in\bk\setminus\{0\}$ and
$g\in\pcat{(X_2\sqcup \{\square\})}$ containing exactly one occurrence of the generating $2$-cell~$\square$.
Every context can be written uniquely in the form
\[
C[\square] = a\bigl(h\circ_1 u\,\square\,v\circ_1 k\bigr),
\]
where $a\in \bk$, $h,k\in\pcat{X_2}$ and $u,v\in\cat{X_1}$.
The \emph{size} of $C$ is defined by $\abs{C}=\abs{h}+\abs{k}$.
A context $C[\square]$ is \emph{compatible} with
$f\in\pcat{X_2}$ if $\square=f^\partial$.
A context of the form $C=au\square v$ is called a \emph{whisker context}.

Similarly, a \emph{two-hole context} of $\plcat{X_2}$ is a triple
$(\square_1,\square_2,D)$, denoted by $D[\square_1,\square_2]$, where $\square_1$ and $\square_2$ are
$1$-spheres of $\cat{X_1}$ and $D$ is a $2$-term $ag$,
with $a\in\bk\setminus\{0\}$ and
$g\in\pcat{(X_2\sqcup\{\square_1,\square_2\})}$ containing
exactly one occurrence of each of the $2$-generators $\square_1$ and $\square_2$.

The category of contexts is denoted by $\Cont(\plcat{X_2})$. Its objects are the $2$-terms of $\plcat{X_2}$, and a morphism from $f$ to $g$ is a compatible context $C$ such that $C[f]=g$. 
Composition is given by the nesting of contexts: for contexts $C$ and $D$, we set $C_1C_2[\square]\coloneq C_1[C_2[\square]]$.
When $C[f]=g$, we say that \emph{$f$ divides $g$} and write $f\!\mid\!g$.
In a same way we define the category of contexts on $2$-monomials, denoted by $\Cont(\pcat{X_2})$

\subsubsection{Three-dimensional prepolygraphs}
\label{SSS:LinearPolygraphs}
An equational presentation of a linear monoidal category can be studied from a rewriting perspective by orienting its defining relations. This leads to the notion of a \emph{linear $3$-prepolygraph}, defined as a pair $X=(X_{\leq 2},X_3)$, where $X_{\leq 2}$ is a $2$-polygraph and $X_3\in\Exte{\plcat{X_2}}$.

A linear $3$-prepolygraph is  \emph{modulo} if it is defined as $(X_{\leq 2},X_3\sqcup E)$ where $E\in\Exte{\plcat{X_2}}$ is symmetric, it is then denoted by $(X,E)$.
For $2$-cells $f,g$ in $\plcat{X_2}$, we denote $f\sim_E g$ if there is a $3$-cell in the free $3$-category $\plcat{E}$ with source $f$ and target $g$. 
We denote ${}_E X_E$ the extension of $\plcat{X_2}$ consisting of spheres $(f,g)$ such that 
$f \sim_E f' \xto[X]{} g' \sim_E g$.
For instance, when $E$ is the symmetric closure of the extension $\Int(X_2)$, we recover the structure of linear $3$-polygraph~\cite[Sec. 18.3]{Polybook2025}.

We denote by $\plcat{X_3}$ the free linear $3$-precategory generated by $X$, and by $\cl{X} := \plcat{X_2}/(X_3\sqcup\Int(X_2))$ the linear $2$-category presented by~$X$.
Since all the $2$-categorical structures considered here have a single $0$-cell, $\cl{X}$ can equivalently be regarded as a linear monoidal category.
Two linear $3$-prepolygraphs are \emph{Tietze equivalent} if they present isomorphic linear $2$-categories.

A linear $3$-prepolygraph is \emph{left-term} if the source of each of its $3$-generators is a $2$-term. Up to Tietze equivalence, any linear $3$-prepolygraph may be assumed to be left-term.
Unless otherwise stated, all $3$-prepolygraphs considered in what follows are assumed to be linear and left-term. 

The elements of the extension $X_3$ are called \emph{$X$-rewriting rules}, denoted by $\alpha : s\alpha \fl t\alpha$, and diagrammatically written as
\[
\alpha : 
a\; \twocell{dot*1h*1dot}
\fl 
\sum_i a_i \;\twocell{dot*1 h-i *1 dot} \;,
\]
when $s\alpha=a h$ and $t\alpha=\sum_i a_i h_i$.
We assume throughout that $h\notin\suppm(t\alpha)$.
Whenever $a$ is invertible, we normalize the rule by replacing $ah\fl t\alpha$ with $h\fl a^{-1}t\alpha$.
Thus, the coefficient $a$ is either $1$ or a nonzero noninvertible element of $\bk$.
We call $a$ the \emph{left coefficient} and $h$ the \emph{left monomial} of $\alpha$, and write $\lc(\alpha)=a$ and $\lm(\alpha)=h$.
Finally, we set $\partial\alpha:=s\alpha-t\alpha$.

\subsubsection{Rings with remainder system}
When we rewrite over a ring, we have to take into account the interaction of the context of rewriting with that divisibility properties of the ring in order to guaranty a decreasing reduction order and to do not introduce any cycle of reductions just induced by this interaction and not due to the rule system itself~\cite{Buchberger84}.

We define a \emph{normalized remainder system} (NRS) on a commutative integral domain $\bk$ as a family
$\Rr = \big(\Rem(a)\big)_{a\in\bk\setminus\{0\}}$, where $\Rem(a)\subseteq\bk$, such that, for every $a\in\bk\setminus\{0\}$, one has $0\in\Rem(a)$ and the canonical map
\[
\Rem(a)\longrightarrow\bk/(a),
\]
sending $r$ to $r+(a)$ is bijective.
Equivalently, $\Rem(a)$ is a fixed set of representatives in $\bk$ for the residue classes modulo the ideal $(a)$.
For every $b\in\bk$, there are therefore unique elements $\operatorname{quo}_{a}(b)$ in $\bk$ and $\rem{a}{b}$ in $\Rem(a)$
satisfying
\[
b=a\quo{a}{b} + \rem{a}{b}.
\]

These maps satisfy $\rem{a}{b}=0$ if and only if $b\in(a)$ and $\operatorname{quo}_{a}(b)=0$ if and only if $b\in\Rem(a)$. We have also $\rem{a}{b+ca}=\rem{a}{b}$ for every $c\in\bk$.
In particular, for every unit $u\in\bk$, we have 
$\Rem(u)=\{0\}$,
$\rem{u}{b}=0$,
and $\operatorname{quo}_{u}(b)=u^{-1}b$.

A NRS is \emph{decreasing} if there exists a well-order \(\preccurlyeq_{\bk}\) on $\bk$, with $0$ as its least element, such that
\[
\rem{a}{b}\ne b
\quad\text{implies}\quad
\rem{a}{b}\prec_{\bk}b,
\]
for all $b\in\bk$ and $a\in\bk\setminus\{0\}$.
Such a system can be constructed by fixing a well-order \(\preccurlyeq_{\bk}\) on $\bk$ with least element \(0\) and setting
\begin{eqn}{equation}
\label{E:DecreasingNRS}
\rem{a}{b}
=
\min_{\preccurlyeq_{\bk}}\bigl(b+(a)\bigr),
\qquad
a\quo{a}{b} = b-\rem{a}{b},
\quad\text{and}\quad
\Rem(a)
=
\{\rem{a}{b}\mid b\in\bk\}.
\end{eqn}
The resulting system is decreasing because $b\in b+(a)$. 
Consequently, there is no infinite decreasing sequence $(b_i)_i$, with
$b_{i+1}=\rem{a_i}{b_i}$ for some $a_i\in\bk$.

The existence of a decreasing NRS is weaker than Euclideanity. A Euclidean structure controls remainders relative to the divisor, whereas a decreasing NRS only provides well-founded normalization inside each residue class.

\subsubsection{Examples}
\label{Examples:NRS}
When $\bk$ is a field, every nonzero element is a unit. Hence the unique normalized remainder system is given by $\Rem(a)=\{0\}$, $\rem{a}{b}=0$,
and $\operatorname{quo}_{a}(b)=a^{-1}b$. 
It is decreasing for any well-order on $\bk$ having $0$ as its least element.

When $\bk$ is an Euclidean domain equipped with a Euclidean function
$\varphi:\bk\setminus\{0\}\fl\Nb$.
Choose a well-order $\preccurlyeq_{\bk}$ with least element $0$ such that, for nonzero $a,b$,
$\varphi(a)<\varphi(b)$ implies that $a\prec_{\bk}b$.
Such an order can be constructed by ordering the nonzero elements first by their Euclidean value and then by a fixed well-order within each level.
Euclidean division ensures that every coset modulo $(a)$ contains either $0$ or an element of Euclidean value strictly smaller than $\varphi(a)$. Therefore, we set $\rem{a}{b}=0$ or 
$\varphi(\rem{a}{b})\bigr)<\varphi(a)$.

For example, when $\bk=\mathbb{Z}$, for every $a\in\mathbb{Z}\setminus\{0\}$, we set
$\Rem(a) = \{0,1,\ldots,|a|-1\}$ and for every $b\in\bk$, we define the remainder $\rem{a}{b}$ and quotient $\quo{a}{b}$ by Euclidian division as follows
\[
 \rem{a}{b} = b-|a| \left\lfloor\frac{b}{|a|}\right\rfloor,
\qquad\text{and}\qquad
a\quo{a}{b} = b-\rem{a}{b}.
\]
This NRS is decreasing for the well-order
$0\prec \ldots \prec n \prec \ldots\prec -1\prec \ldots \prec -n \prec \ldots$.

When $\bk=\Zb[q,q^{-1}]$, for $f=\sum_i a_iq^i$, we set $H(f)=\sum_i(1+|i|)|a_i|$.
We order $\bk$ first by $H$, breaking ties lexicographically
on $(a_0,a_1,a_{-1},a_2,a_{-2},\ldots)$, with the
coefficient order $0<1<-1<2<-2<\cdots$.
Since $H$ has finite sublevel sets, this defines a
well-order $\preccurlyeq_\bk$ with least element $0$.
We then define a decreasing NRS by using definitions~\eqref{E:DecreasingNRS}.

\subsubsection{Rewriting steps}
\label{SSS:RewritingStepsOverRing}
From now on, we assume that $\bk$ is a commutative integral domain equipped with a NRS $\Rr$ with associated quotient and remainder maps $\operatorname{quo}$ and $\Rem$.

Let $X$ be a $3$-prepolygraph. 
An \emph{$X$-rewriting step} is a $3$-cell of $\plcat{X_3}$ of size $1$,  of the form
\begin{eqn}{equation}
\label{E:RewritingStepsOverRing}
\sigma = aC[\alpha]+1_f,
\end{eqn}
where $a\in\bk\setminus\{0\}$, $\alpha\in X_3$, $f\in\plcat{X_2}$, and $C\in\Cont(\pcat{X_2})$, subject to the following condition.
Writing
\[
f=bC[\lm(\alpha)]+f'
\quad\text{and}\quad
C[\lm(\alpha)]\notin\suppm(f'),
\]
we require $\operatorname{quo}_{\operatorname{lc}(\alpha)}(b)=0$.
The $3$-cell $\sigma$ is diagrammatically written
\[
\sigma
\;:\;
a
\cdot
\twocell{dot*1 k2 *1(id*0dot*0 id)*1(u*0 salpha *0 v)*1(id*0dot*0 id)*1 h2*1 dot}
+
\twocell{dot*1f*1dot}
\;\fl\;
a
\cdot
\twocell{dot*1 k2 *1(id*0dot*0 id)*1(u*0 talpha *0 v)*1(id*0dot*0 id)*1 h2*1 dot}
+
\twocell{dot*1f*1dot}\;.
\]
Finally, when $f=0$, we call $\sigma$ \emph{left-term}.

\begin{remark}
\label{R:ClassicalLeftMonomialSteps}
When $X$ is left-monomial, the condition in the definition of an $X$-rewriting step $\sigma$ \eqref{E:RewritingStepsOverRing} reduces to $C[\lm(\alpha)]\notin\suppm(f)$.
Indeed, put $h=C[\lm(\alpha)]$ and let $b=\cf{h}{f}$.
Since $X$ is left-monomial, we have $\lc(\alpha)=1$.
It follows from $\Rem(1)=\{0\}$ that $\quo{1}{b}=b$ and $\rem{1}{b}=0$.
By definition, the $3$-cell $\sigma$ is an $X$-rewriting step if and only if $\quo{1}{b}=b=0$, or equivalently, if and only if $ C[\lm(\alpha)]\notin\suppm(f)$.

In particular, over a field, this recovers the standard notion of a rewriting step for a left-monomial linear polygraph, as formulated in the associative-algebra setting in~\cite[Sec.~3.1]{GuiraudHoffbeckMalbos19} and in the higher-dimensional setting in~\cite{Alleaume18,Dupont21,Schelstraete26}.
\end{remark}

\subsubsection{Example}
Let $\bk=\mathbb{Z}$ be an NRS as in \SSS{Examples:NRS}.
Let consider the $3$-prepolygraph $X$ with two
$2$-generators $\twocell{\eta}$ and $\twocell{\tau}$ and the
$3$-generator
$\alpha: 7\,\twocell{\eta}\fl 4\,\twocell{id}$.
Then the $3$-cell
\[
\sigma := C[\alpha]+8h
\: : \:
15h
\fl
4\,\twocell{id}+8h\,,
\]
where
$
C=
\raisebox{-8.75pt}{
\begin{tikzpicture}
\begin{scope}[
    x=10pt,
    y=10pt,
    join=round,
    cap=round,
    thick,
    black,
    solid,
    -
]
\draw
    (0.00,2.00)--(0.00,1.75)
    (1.00,2.00)--(1.00,1.50);
\draw
    (0.00,0.00)--(0.00,-0.25)
    (1.00,0.00)--(1.00,-0.25);
\draw[rounded corners=1pt,fill=white]
    (-0.35,0.75) rectangle (0.35,1.75);
\node at (0.00,1.25) {$\scriptstyle\square$};
\draw[fill=lightgray]
    (1.00,1.50)--(1.00,1.00)--cycle;
\draw
    (0.00,0.75)--(0.00,0.50)
    (1.00,1.00)--(1.00,0.50);
\draw
    (0.00,0.50)--(1.00,0.00)
    (1.00,0.50)--(0.00,0.00);
\end{scope}
\end{tikzpicture}
},
$
and $h=\twocell{(eta*0id)*1tau}$ is not an $X$-rewriting step, since 
 $\rem{7}{8}=1$ and $\quo{7}{8}=1$. However, it decomposes
as a composite of two $3$-cells $\sigma = \tau\circ_2 \rho^{-1}$, where
\[
\tau := 2C[\alpha] +h 
\: : \:
15h
=
2\cdot7h +h
\fl
2\cdot4\,\twocell{id}\,+h
=
8\,\twocell{id}+h\,,
\]
and
\[
\rho := C[\alpha] + h + 4\,\twocell{id}
\: : \:
4\,\twocell{id}\,+8h
=
4\,\twocell{id}\,+7h+h
\fl
4\,\twocell{id}\,+4\,\twocell{id}\,+h
=
8\,\twocell{id}\,+h\,.
\]
Indeed, both $\tau$ and $\rho$ are $X$-rewriting steps, since the
coefficient of $h$ in their additive parts is
$1$, with $\quo{7}{1}=0$ and $\rem{7}{1}=1$.

This observation leads to the following lemma, analogous to \cite[Lem.~3.1.3]{GuiraudHoffbeckMalbos19}, which was established in the setting of rewriting in associative algebras over a field.
\begin{lemma}
\label{L:DecopositionOfSize1Cell}
Let $X$ be a $3$-prepolygraph.
Every $3$-cells $\sigma$ of size $1$ in $\plcat{X_3}$ admits a decomposition $\sigma=\tau\circ_2 \rho^{-1}$, where each of $\tau$ and $\rho$ is either an identity or an $X$-rewriting step.
\end{lemma}
\begin{proof}
We write $\sigma=aC[\alpha]+1_g$,
where $a\in\bk\setminus\{0\}$, $\alpha\in X_3$, $g\in\plcat{X_2}$, and $C\in\Cont(\pcat{X_2})$.
We set $\lambda=\lc(\alpha)$ and $h=C[\lm(\alpha)]$, and we write $g=bh+g'$ with $h\notin\suppm(g')$.
Using the NRS, we consider $q=\operatorname{quo}_{\lambda}(b)$ and $r=\rem{\lambda}{b}$.
Thus $b=q\lambda+r$, with $r\in\Rem(\lambda)$ and
$\operatorname{quo}_{\lambda}(r)=0$.

Let us define the following $3$-cells in $\plcat{X_3}$  
\[
\tau=(a+q)C[\alpha]+1_{rh+g'}
\quad\text{and}\quad
\rho=qC[\alpha]+1_{aC[t\alpha]+rh+g'},
\]
that form the following diagram
\[
\xymatrix @R=1.8em @C=0.9em {
((a+q)\lambda+r)h+g'
\ar[dr]_{\tau}
\ar[rr]^{\sigma}
&
{}
&
aC[t\alpha]+(q\lambda+r)h+g'
\ar[dl]^{\rho}
\\
&
(a+q)C[t\alpha]+rh+g'.
&
}
\]

Since $\lm(\alpha)\notin\suppm(t\alpha)$ and monomial
contexts in the free precategory $\pcat{X_2}$ are injective on monomials,
we have $h\notin\suppm(C[t\alpha])$.

Consequently, the coefficient of $h$ in each of
$rh+g'$ and $aC[t\alpha]+rh+g'$ is $r$.
As $\quo{\lambda}{r}=0$, by definition
of an $X$-rewriting step, if $a+q\neq0$, then $\tau$ is an $X$-rewriting step, and if $q\neq0$, then $\rho$ is an $X$-rewriting step. Otherwise, $\sigma$ and $\tau$ are identities.
Finally, the linear composition identity gives
\[
\sigma\circ_2\rho
=\sigma+\rho-1_{t\sigma}
=(a+q)C[\alpha]+1_{rh+g'}
=\tau.
\]
Hence $\sigma=\tau\circ_2\rho^{-1}$.
\end{proof}

\subsubsection{Termination and normalization}

A \emph{$X$-rewriting path of length $n$} is a $3$-cell of $\plcat{X_3}$ defined as a $\circ_2$-composite $\sigma=\sigma_1\circ_2\cdots\circ_2\sigma_n$,
where each $\sigma_i$ is an $X$-rewriting step.

If $\sigma$ is an $X$-rewriting step  (resp. $X$-rewriting path) from $f$ to $g$, we write 
$f\xto[\sigma]{} g$ (resp. $f\xto*[\sigma]{} g$).
When there exists an $X$-rewriting step  (resp. $X$-rewriting path) with source $f$ and target $g$, we write $f \xto[X]{} g$ (resp. $f \xto*[X]{} g$).

A $2$-cell $f$ of $\plcat{X_2}$ is \emph{$X$-reducible} is there is a rule $\alpha$ in $X_3$ and a context $C$ such that $h=C[\lm(\alpha)]\in\suppm(f)$, and writing
\[
f \: =\: bh+f'
\quad\text{and}\quad
h\notin\suppm(f'),
\]
we have $q=\operatorname{quo}_{\lc(\alpha)}(b)\neq0$.
The corresponding $X$-rewriting step is thus:
\[
\sigma
=
qC[\alpha]+1_{f-q\lc(\alpha)h}
:
f\fl f-qC[\partial\alpha].
\]

A $3$-prepolygraph $X$ is \emph{terminating}, if the induced rewrite relation $\xto[X]{}$ is contained in a well-founded ordering.
A $3$-prepolygraph modulo $(X,E)$ is \emph{$E$-terminating} if the relation $\sim_E \xto[X]{} \sim_E$ is contained in a well-founded ordering.

A $2$-cell $f$ of $\plcat{X_2}$ is in \emph{$X$-normal form} if it cannot be reduced by any $X$-rewriting step.
We denote by $\Nf(X)$ the set of $X$-normal forms, and by $\Nf_m(X)$ its subset of monomials.
A $2$-cell $g$ is a \emph{normal form of a $2$-cell $f$} if $f \xto*[X]{} g$ and $g$ is in $X$-normal form, and we write $g = f\NF{}{X}$.
The polygraph $X$ is \emph{normalizing} if every $2$-cell of $\plcat{X_2}$ has a normal form.
In that case, a \emph{$X$-normalization strategy} is a map $\down{(-)}{X}$ sending every $2$-cell $f$ of $\plcat{X_2}$ to an $X$-normal form $\down{f}{X}$.

A context $C = a(h \circ_1 u \square v \circ_1 k)$ in $\plcat{X_2}$ is \emph{$X$-normal}, denoted by $C\in\Nf(X)$, if $ah,ak\in\Nf(X)$.

Let $Y$ be a normalizing extension of $\plcat{X_2}$, together with a $Y$-normalization strategy $\down{(-)}{Y}$, and let $\alpha$ be a rule in $X_3$.
A context $C$ compatible with $s\alpha$ is said to be \emph{$Y$-active on $\alpha$} if either $s\alpha\notin\Nf(Y)$ and $\abs{C}=0$, or $s\alpha\in\Nf(Y)$, $C\in\Nf(Y)$, and $C[s\alpha]\notin\Nf(Y)$.
We denote by $\Act(\alpha,Y)$ the set of $Y$-active contexts on $\alpha$.
For instance, let $Y\coloneq \Int(X_2)$.
If $s\alpha\in\Nf(Y)$, we have
\[
\Act(\alpha,Y)
=
\left(\,
\raisebox{-17.50pt}{\begin{tikzpicture} \begin{scope} [ x = 10pt, y = 10pt, join = round, cap = round, thick, black, solid, -] \draw (0.00,3.75)--(0.00,3.00) (1.00,3.75)--(1.00,3.00) (2.00,3.75)--(2.00,3.25) (3.00,3.75)--(3.00,3.00) (4.00,3.75)--(4.00,3.00) (5.00,3.75)--(5.00,3.25) (6.00,3.75)--(6.00,3.50) (7.00,3.75)--(7.00,3.50) ; \draw (0.00,0.00)--(0.00,-0.25) (1.00,0.00)--(1.00,-0.25) (2.00,0.25)--(2.00,-0.25) (3.00,0.50)--(3.00,-0.25) (4.00,0.50)--(4.00,-0.25) (5.00,0.25)--(5.00,-0.25) (6.00,0.50)--(6.00,-0.25) (7.00,0.50)--(7.00,-0.25) ; \draw [dash pattern = on 0.25pt off 2pt] (0.25,3.00)--(0.75,3.00) ; \draw [fill = lightgray] (2.00,3.25)--(2.00,2.75)--cycle ; \draw [dash pattern = on 0.25pt off 2pt] (3.25,3.00)--(3.75,3.00) ; \draw [fill = lightgray] (5.00,3.25)--(5.00,2.75)--cycle ; \draw [rounded corners = 1pt, fill = lightgray] (5.75,2.50) rectangle (7.25,3.50) ; \node at (6.50,3.00) {$\scriptstyle f$} ; \draw (0.00,3.00)--(0.00,1.75) (1.00,3.00)--(1.00,1.75) (2.00,2.75)--(2.00,2.00) (3.00,3.00)--(3.00,2.25) (4.00,3.00)--(4.00,2.25) (5.00,2.75)--(5.00,2.00) (6.00,2.50)--(6.00,1.75) (7.00,2.50)--(7.00,1.75) ; \draw [dash pattern = on 0.25pt off 2pt] (0.25,1.75)--(0.75,1.75) ; \draw [fill = lightgray] (2.00,2.00)--(2.00,1.50)--cycle ; \draw [rounded corners = 1pt, fill = white] (2.75,1.25) rectangle (4.25,2.25) ; \node at (3.50,1.75) {$\scriptstyle \square$} ; \draw [fill = lightgray] (5.00,2.00)--(5.00,1.50)--cycle ; \draw [dash pattern = on 0.25pt off 2pt] (6.25,1.75)--(6.75,1.75) ; \draw (0.00,1.75)--(0.00,1.00) (1.00,1.75)--(1.00,1.00) (2.00,1.50)--(2.00,0.75) (3.00,1.25)--(3.00,0.50) (4.00,1.25)--(4.00,0.50) (5.00,1.50)--(5.00,0.75) (6.00,1.75)--(6.00,0.50) (7.00,1.75)--(7.00,0.50) ; \draw [rounded corners = 1pt, fill = lightgray] (-0.25,0.00) rectangle (1.25,1.00) ; \node at (0.50,0.50) {$\scriptstyle g$} ; \draw [fill = lightgray] (2.00,0.75)--(2.00,0.25)--cycle ; \draw [dash pattern = on 0.25pt off 2pt] (3.25,0.50)--(3.75,0.50) ; \draw [fill = lightgray] (5.00,0.75)--(5.00,0.25)--cycle ; \draw [dash pattern = on 0.25pt off 2pt] (6.25,0.50)--(6.75,0.50) ; \end{scope} \end{tikzpicture}}
\quad \big| \quad f,g\in\Nf(Y)\;
\right\}.
\]
When $Y$ is left-monomial, every $Y$-active context on $\alpha$ has coefficient $1$.

We define the \emph{context relation} $\sqsubseteq$ on the $3$-cells of $\plcat{X_3}$ by setting, for all $\beta_1,\beta_2\in \plcat{X_3}$, $\beta_1\sqsubseteq\beta_2$ if there exists a  context $C$ in $\plcat{X_2}$ such that $C[\beta_1]=\beta_2$.
For an extension $Y$ of $\plcat{X_2}$, we denote by $\minsub Y$ the set of minimal $3$-generators in $Y$ with respect to the relation $\sqsubseteq$.

\subsection{Confluence and hom-bases by confluence}

We now introduce the rewriting notions used throughout this work, together with the rewriting machinery allowing one to construct hom-bases from convergent presentations of categories.

\subsubsection{Confluence}
A \emph{branching} of a $3$-prepolygraph $X$ is a pair $(\sigma,\tau)$ of $X$-rewriting paths with the same source, that is, $s\sigma=s\tau$.
It is called \emph{local} if both $\sigma$ and $\tau$ are $X$-rewriting steps.
The $3$-prepolygraph $X$ is \emph{confluent} (resp. \emph{local confluent}) if for any branching (resp. local branching) $(\sigma,\tau)$, there exists a pair of $X$-rewriting paths $(\rho_1,\rho_2)$, called a \emph{confluence}, as in the following diagram
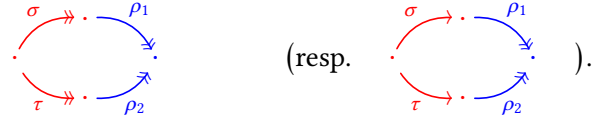
\[
\vcenter{\xymatrix @R=0.5em @C=1.6em {
& {\color{red}\cdot}
  \ar@{->>}@/^/@[blue][dr]^-{\color{blue}\rho_1}
\\
{\color{red}\cdot}
  \ar@{->>}@/^/@[red][ur]^-{\color{red}\sigma}
  \ar@{->>}@/_/@[red][dr]_-{\color{red}\tau}
&& {\color{blue}\cdot}
\\
& {\color{red}\cdot}
  \ar@{->>}@/_/@[blue][ur]_-{\color{blue}\rho_2}
}}
\qquad\qquad \bigl(\text{resp.}\quad
\vcenter{\xymatrix @R=0.5em @C=1.6em {
& {\color{red}\cdot}
  \ar@{->>}@/^/@[blue][dr]^-{\color{blue}\rho_1}
\\
{\color{red}\cdot}
  \ar@{->}@/^/@[red][ur]^-{\color{red}\sigma}
  \ar@{->}@/_/@[red][dr]_-{\color{red}\tau}
&& {\color{blue}\cdot}
\\
& {\color{red}\cdot}
  \ar@{->>}@/_/@[blue][ur]_-{\color{blue}\rho_2}
}}
\quad\bigr).
\]
We call $X$ \emph{convergent} if it is both terminating and confluent.
In this case, every $2$-cell $f$ admits a unique normal form, denoted by $\down{f}{X}$. 
Thus, the normal-form map $f\fl \down{f}{X}$ is uniquely determined. 
However, the $X$-rewriting paths from $f$ to $\down{f}{X}$ need not be unique. 
A normalization rewriting strategy amounts to choosing such a path for each $2$-cell $f$.

\subsubsection{Classification of local branchings}
\label{SSS:CLB}

Let $X$ be a left-term $3$-prepolygraph, and let $(\sigma,\tau)$ be a local branching of $X$ with common source $f$. We write
\[
\sigma=q_1C_1[\alpha_1]+1_{f-q_1\lambda_1 h_1},
\qquad
\tau=q_2C_2[\alpha_2]+1_{f-q_2\lambda_2 h_2},
\]
where, for $i=1,2$, we have $\alpha_i\in X_3$,
$C_i\in\Cont(\pcat{X_2})$,
$\lambda_i=\lc(\alpha_i)$,
$h_i=C_i[\lm(\alpha_i)]\in\suppm(f)$,
$b_i=\cf{h_i}{f}$,
$q_i=\quo{\lambda_i}{b_i}\neq0$, and
$r_i=\rem{\lambda_i}{b_i}$, so that
$b_i=q_i\lambda_i+r_i$.
We distinguish the following four families of local branchings.

\begin{enumerate}
\item \emph{Aspherical branchings}, when the two rewriting steps coincide, that is, $\sigma=\tau$.

\item \emph{Additive branchings}, when the two rules $\alpha_1$ and $\alpha_2$ are applied to
distinct terms in $\supp(f)$, that is, $h_1\neq h_2$.
The common source is $f=b_1h_1+b_2h_2+g$, where $h_1,h_2\notin\suppm(g)$.
More explicitly,
\[
\sigma=q_1C_1[\alpha_1]+1_{r_1h_1+b_2h_2+g},
\qquad
\tau=q_2C_2[\alpha_2]+1_{b_1h_1+r_2h_2+g}.
\]

\item \emph{Peiffer branchings}, when the two rules $\alpha_1$ and $\alpha_2$ are applied to disjoint occurrences within the same monomial $h=h_1=h_2$, and the product of their leading coefficients divides the coefficient $b=b_1=b_2$ of $h$ in $f$, that is, $\lambda_1\lambda_2\mid b$.
Thus, there exist a monomial two-hole context $D$ and a nonzero scalar $\mu$ such that
$h=D[\lm(\alpha_1),\lm(\alpha_2)]$ and $b=\mu\lambda_1\lambda_2$.
The common source is $f=\mu\lambda_1\lambda_2h+g$, where $h\notin\suppm(g)$.
More explicitly,
\[
\sigma=\mu\lambda_2D[\alpha_1,\lm(\alpha_2)]+1_g,
\qquad
\tau=\mu\lambda_1D[\lm(\alpha_1),\alpha_2]+1_g.
\]

\item \emph{Overlapping branchings}, when the two rewriting steps are distinct and apply the rules
$\alpha_1$ and $\alpha_2$ to the same monomial
$h=h_1=h_2$.
Moreover, either the occurrences of $\lm(\alpha_1)$ and
$\lm(\alpha_2)$ in $h$ overlap, or they are disjoint
but the product of their leading coefficients does not
divide the coefficient $b=b_1=b_2$ of $h$ in $f$,
that is, $\lambda_1\lambda_2\nmid b$.
The common source is $f=bh+g$, where $h\notin\suppm(g)$.
More explicitly,
\[
\sigma=q_1C_1[\alpha_1]+1_{r_1h+g},
\qquad
\tau=q_2C_2[\alpha_2]+1_{r_2h+g}.
\]
\end{enumerate}

The \emph{critical branchings} of $X$ are the overlapping branchings
$(\sigma,\tau)$ of $X$ whose common source $s\sigma=s\tau$ is a
$2$-term and which are minimal with respect to the context relation
$\sqsubseteq$ on $\plcat{X_3}$.
For instance, the rules
$
\alpha:2\,\twocell{eta}\fl\twocell{id}
$
and
$
\beta:3\,\twocell{eta}\fl\twocell{id}\,
$
give rise to the following two critical branchings:
{
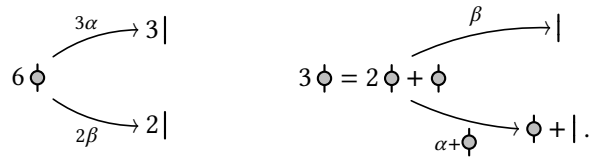
\[
\vcenter{\xymatrix @R=0.1em @C=2.9em {
& 3\,\twocell{id}
\\
6\,\twocell{eta}
  \ar@{->}@/^/[ur]^-{3\alpha}
  \ar@{->}@/_/[dr]_-{2\beta}
\\
& 2\,\twocell{id}
}}
\qquad\qquad
\vcenter{\xymatrix @R=0.2em @C=2.2em {
& \twocell{id}
\\
3\,\twocell{eta}
=2\,\twocell{eta}+\twocell{eta}
  \ar@{->}@/^/[ur]^-{\beta}
  \ar@{->}@/_/[dr]_-{\alpha+\twocell{eta}}
\\
& \twocell{eta}+\twocell{id}\;.
}}
\]
}
The first branching arises from the common multiple $6$ of the leading coefficients, whereas the second arises from the nonzero remainder in $3=1\cdot2+1$; in particular, $\alpha+\twocell{eta}$ is an $X$-rewriting step.

A key feature of the precategorical setting is that overlaps between $2$-monomials in $\pcat{X_2}$ can be decomposed into overlaps between strings.
Indeed, every $2$-monomial $h\in\pcat{X_2}$ admits a unique decomposition $h = h_1 \circ_1 \cdots \circ_1 h_{\abs{h}}$, where $h_i=C_i[\varphi_i]$, with $\varphi_i \in X_2$ and $C_i$ is a whisker context. An occurrence of a $2$-monomial $k$ inside $u$, given by a monomial $2$-context $C$ such that $u=C[v]$, corresponds, up to whiskering by $1$-cells, to a contiguous factor in the decomposition of $h$.
Hence, if two occurrences of leading monomials $\lm(\alpha_1)$ and $\lm(\alpha_2)$ overlap in a $2$-monomial $h$, the corresponding factors in the canonical decomposition of $h$ overlap. The study of overlapping branchings is therefore reduced to the study of compatible overlaps of finite strings. In particular, for a fixed pair of rules, their possible overlaps can be determined combinatorially from the string decompositions of their leading monomials.
This reduction to string overlaps is one of the main advantages of the precategorical formulation of $3$-dimensional rewriting introduced in~\cite{ForestMimram2022}.

We now state the critical branching lemma.

\begin{lemma}
\label{L:CBL}
Assume that $X$ is a terminating $3$-prepolygraph.
If every critical branching of $X$ is confluent, then $X$ is locally
confluent. Moreover, $X$ is confluent.
\end{lemma}
\begin{proof}
Following \cite[3.2.2 and Thm.~4.2.1]{GuiraudHoffbeckMalbos19}, we proceed by noetherian induction on the sources of local branchings, with respect to the well-founded rewrite order $\prec_X$ of $X$.
Let $f$ be a reducible $2$-cell, and assume that $X$ is locally confluent at every $2$-cell $f'\prec_X f$.
By Newman's lemma, it follows that $X$ is confluent at every $f'\prec_X f$.

Let $(\sigma,\tau)$ be a local branching with source $f$.
If $(\sigma,\tau)$ is aspherical, then it is trivially confluent.
We consider the remaining three cases.
For each case, we use the notation of the classification in \SSS{SSS:CLB}, and construct a confluence diagram of the following form
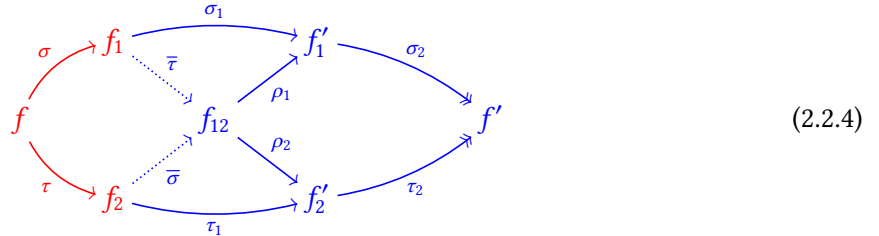
\begin{eqn}{equation}
\label{E:LocalBranchingConfluence}
\vcenter{\xymatrix @R=1.15em @C=2em {
& \textcolor{red}{f_1}
  \ar@[blue]@/^/[rr]^-{\textcolor{blue}{\sigma_1}}
  \ar@[blue]@{.>}[dr]^-{\textcolor{blue}{\overline{\tau}}}
&& \textcolor{blue}{f_1'}
  \ar@[blue]@{->>}@/^/[drr]^-{\textcolor{blue}{\sigma_2}}
\\
\textcolor{red}{f}
  \ar@[red]@/^1.5ex/[ur]^-{\textcolor{red}{\sigma}}
  \ar@[red]@/_1.5ex/[dr]_-{\textcolor{red}{\tau}}
&& \textcolor{blue}{f_{12}}
  \ar@[blue][ur]_-{\textcolor{blue}{\rho_1}}
  \ar@[blue][dr]^-{\textcolor{blue}{\rho_2}}
&&& \textcolor{blue}{f'}
\\
& \textcolor{red}{f_2}
  \ar@[blue]@{.>}[ur]_-{\textcolor{blue}{\overline{\sigma}}}
  \ar@[blue]@/_/[rr]_-{\textcolor{blue}{\tau_1}}
&& \textcolor{blue}{f_2'}
  \ar@[blue]@{->>}@/_/[urr]_-{\textcolor{blue}{\tau_2}}
}}
\end{eqn}

\noindent {\bf Case 1.} 
Suppose that $(\sigma,\tau)$ is additive.
By definition of $\sigma$ and $\tau$, we have $f=b_1h_1+b_2h_2+g$.
We set 
\[
f_1=t(\sigma)=q_1\overline{h}_1+r_1h_1+b_2h_2+g,
\qquad
f_2=t(\tau)=b_1h_1+q_2\overline{h}_2+r_2h_2+g,
\qquad
f_{12}=q_1\overline{h}_1+r_1h_1+q_2\overline{h}_2+r_2h_2+g,
\]
with $\overline{h}_i=C_i[t(\alpha_i)]$.
Let consider the $3$-cells
\[
\overline{\sigma}
=q_1C_1[\alpha_1]
+1_{r_1h_1+q_2\overline{h}_2+r_2h_2+g} :f_2\fl f_{12},
\qquad\text{and}\qquad
\overline{\tau}
=q_2C_2[\alpha_2]
+1_{q_1\overline{h}_1+r_1h_1+r_2h_2+g}:f_1\fl f_{12}.
\]
They may be not $X$-rewriting steps, since we may have $h_1\in\suppm(\overline{h}_2)$ or
$h_2\in\suppm(\overline{h}_1)$ and the coefficients of $h_1$ and $h_2$ in the additive parts of
$\overline{\sigma}$ and $\overline{\tau}$ need not
belong to $\Rem(\lambda_1)$ and $\Rem(\lambda_2)$, respectively.
However, they have size $1$.
Lemma~\ref{L:DecopositionOfSize1Cell} therefore yields $2$-cells
$f_1'$ and $f_2'$, together with $X$-rewriting paths $\sigma_1, \rho_1, \tau_1$ and $\rho_2$ as listed in \eqref{E:LocalBranchingConfluence}.
Hence $X$ is confluent at $f_{12}\prec_X f$ by hypothesis, so the branching
$(\rho_1,\rho_2)$ admits a confluence $(\sigma_2,\tau_2)$.

\noindent {\bf Case 2.} 
Suppose next that $(\sigma,\tau)$ is Peiffer.
Set $\overline{h}_i=t(\alpha_i)$ for $i=1,2$.
The common source is $f=\mu\lambda_1\lambda_2h+g$,
where $h=D[\lm(\alpha_1),\lm(\alpha_2)]$. Put
\[
f_1=\mu\lambda_2D[\overline{h}_1,\lm(\alpha_2)]+g,
\qquad
f_2=\mu\lambda_1D[\lm(\alpha_1),\overline{h}_2]+g,
\qquad
f_{12}=\mu D[\overline{h}_1,\overline{h}_2]+g.
\]
Then $f_1=t(\sigma)$ and $f_2=t(\tau)$.
There are $3$-cells
\[
\overline{\sigma}
 =\mu D[\alpha_1,\overline{h}_2]+1_g
 :f_2\fl f_{12},
\qquad
\overline{\tau}
 =\mu D[\overline{h}_1,\alpha_2]+1_g
 :f_1\fl f_{12},
\]
which need not be $X$-rewriting paths, since we may have
$\suppm(D[\overline{h}_1,\lm(\alpha_2)])\cap\suppm(g)\neq\varnothing$
or
$\suppm(D[\lm(\alpha_1),\overline{h}_2])\cap\suppm(g)\neq\varnothing$,
and the corresponding coefficients in $g$ need not belong to
$\Rem(\lambda_2)$ and $\Rem(\lambda_1)$, respectively.
We decompose them into $3$-cells of size $1$ and apply
Lemma~\ref{L:DecopositionOfSize1Cell} successively, as in
\cite[Lem.~4.1.2]{GuiraudHoffbeckMalbos19}.
This yields $2$-cells $f_1'$ and $f_2'$, together with
$X$-rewriting paths $\sigma_1, \rho_1, \tau_1$ and $\rho_2$ as listed in \SSS{E:LocalBranchingConfluence}.
Hence $X$ is confluent at $f_{12}\prec_X f$ by hypothesis, so the branching
$(\rho_1,\rho_2)$ admits a confluence $(\sigma_2,\tau_2)$.

\noindent {\bf Case 3.} 
Finally, suppose that $(\sigma,\tau)$ is overlapping.
The common source is $f=bh+g$, where
$h=C_1[\lm(\alpha_1)]=C_2[\lm(\alpha_2)]$
and $b=q_i\lambda_i+r_i$ for $i=1,2$. Set
\[
\overline{h}_i=C_i[t(\alpha_i)],
\qquad
f_i=q_i\overline{h}_i+r_ih+g,
\qquad i=1,2.
\]
Then $f_1=t(\sigma)$ and $f_2=t(\tau)$.
By minimality, there exist a critical branching
$(\sigma_0,\tau_0)$ and a context $C$ such that
$(q_1C_1[\alpha_1]+1_{r_1h},q_2C_2[\alpha_2]+1_{r_2h})=C[\sigma_0,\tau_0]$.
The critical branching is confluent by hypothesis, it admits a confluence $(\sigma_0',\tau_0')$ with common target $e_0$.
Set $f_{12}=C[e_0]+1_g$.
Applying $C$ gives the $3$-cells
\[
\overline{\sigma}
 =C[\tau_0']+1_{g}:f_2\fl f_{12},
\qquad
\overline{\tau}
 =C[\sigma_0']+1_{g}:f_1\fl f_{12},
\]
which may not be $X$-rewriting paths, since 
$\suppm(g)$ may intersect $\suppm(C[\lm(\sigma_0')])$ or \linebreak $\suppm(C[\lm(\tau_0')])$, and the corresponding coefficients in $g$ need not belong to
$\Rem(\lambda_2)$ and $\Rem(\lambda_1)$, respectively.
We decompose them into $3$-cells of size $1$ and apply
Lemma~\ref{L:DecopositionOfSize1Cell} successively.
This yields $2$-cells $f_1'$ and $f_2'$, together with $X$-rewriting paths $\sigma_1, \rho_1, \tau_1$ and $\rho_2$ as listed in \SSS{E:LocalBranchingConfluence}.
Hence $X$ is confluent at $f_{12}\prec_X f$ by hypothesis, so the branching
$(\rho_1,\rho_2)$ admits a confluence $(\sigma_2,\tau_2)$.

\medskip

Therefore, $X$ is locally confluent. Since $X$ is terminating, it is confluent by Newman's Lemma.
\end{proof}

We denote by $\CB(X)$ the set of critical branchings of $X$. For two $3$-prepolygraphs $X$ and $Y$ sharing the same underlying $2$-polygraph, we denote by $\CB(X,Y)$ the set of critical branchings involving both an $X$-rewriting step and a $Y$-rewriting step.

In \cite[Sec.~5.1]{GuiraudMalbos09}, critical branchings of $3$-polygraphs are classified into three families: regular, inclusion, and indexed branchings. The classification for $3$-prepolygraphs is similar, except that indexed branchings do not appear, since interchange relations are not available in this setting. This significantly simplifies the structure of critical branchings. On the other hand, interchange relations become explicit rewriting rules, generating new families of critical branchings that must be analyzed. The stratification approach introduced in Subsection~\ref{SS:StratifiedPrepolygraphsTermination} is precisely designed to handle this additional complexity.

\subsubsection{Hom-bases by confluence over rings}
Following~\cite{Alleaume18,Dupont21,Dupont22,Schelstraete26}, in the case
of rewriting over a field, if $X$ is convergent, then $\Nf_m(X)$ forms
a \emph{hom-basis} of the category $\Cr$ presented by $X$, that is, a family of linear bases
$B(u,v)$ of the vector spaces $\Cr_2(u,v)$, for all
$u,v\in\cat{X_1}$. Indeed, convergence ensures that every monomial has
a unique normal-form representative, and the result follows from the
equality
$
\Nf(X)=\bk\Nf_m(X),
$
that is, the module of normal forms need not be the free $\bk$-module
spanned by the monomial normal forms.

When $\bk$ is a ring, however, this equality need not hold, and hence
the hom-basis theorem may fail for rewriting over a ring.
For instance, let $\bk=\mathbb Z$, and consider the $3$-prepolygraph $X$
with one $1$-cell $i$, one $2$-generator
$
\twocell{\eta}:i\fl i,
$
and one $3$-generator
$
2\,\twocell{\eta}\ \fl\ 0.
$
Then, in $\plcat{X_2}(i,i)$, every monomial 
$\twocell{\eta}^{\,n} := \twocell{\eta} \circ_1 \underset{\text{$n$-times}}{\ldots} \circ_1 \twocell{\eta}$ is $X$-irreducible. Hence
$\Nf_m(X)$ is the set of $2$-monomial $\twocell{\eta}^{\,n}$, for $n\geq 1$.
On the other hand, a linear combination is irreducible precisely when all
its nonzero coefficients are odd. 
Thus, the set $\Nf(X)$ of $2$-cells in $X$-normal forms is made of $2$-cells 
$\sum_{n\geq 1} \lambda_n\twocell{\eta}^{\,n}$ where $\lambda_n$ is odd.
In particular, $\Nf(X)$ is not closed under addition. Consequently, we have
$\Nf(X)\neq \mathbb \bk\Nf_m(X)$.
This failure stems from the fact that the underlying $\bk$-module of
$\overline{X}$ is not free.
However, when $X$ is left-monomial, we have the following result.

\begin{theorem}
\label{T:HomBasisFromConvergence}
Let $\Cr$ be a linear monoidal category presented by a left-monomial and convergent $3$-prepolygraph $X$. Then $\Nf_m(X)$ forms a hom-basis of $\Cr$.
\end{theorem}
\begin{proof}
First we prove that under the hypothesis that we have $\Nf(X)=\bk\Nf_m(X)$, or equivalently for every $2$-cell $f$, we have $f\in\Nf(X)$ if and only if $\suppm(f)\subseteq\Nf_m(X)$.

Suppose first that $f\in \Nf(X)$, and assume that some $u\in\suppm(f)$ is reducible.
Then $u=C[\lm(\alpha)]$ for some $\alpha\in X_3$ and some monomial context $C$. 
Let $a=\cf{u}{f}\neq 0$. 
Since $X$ is left-monomial, we have $\lc(\alpha)=1$ and $\quo{1}{a}=a\neq 0$. 
Hence, the $3$-cell $aC[\alpha]+1_{f-au}$ is an $X$-rewriting step with source $f$, so $f$ is reducible.
This contradicts $f\in\Nf(X)$. Therefore every monomial in $\suppm(f)$ is irreducible, and thus $\suppm(f)\subseteq\Nf_m(X)$.

Conversely, suppose that $f$ is reducible. Then there exists an
$X$-rewriting step with source $f$. Such a step rewrites a monomial
$C[\lm(\alpha)]$ occurring in $\suppm(f)$, for some $\alpha\in X_3$ and some monomial
context $C$. Hence $\suppm(f)$ contains a reducible monomial.

We have thus proved that a $2$-cell is in normal form if and only if
all the monomials in its support are in normal form. Equivalently,
$\Nf(X)=\bk\Nf_m(X)$.

\medskip

We show that $\Nf_m(X)$ is a hom-basis of $\Cr$ by considering $u,v\in\cat{X_1}$ and $\pi_{u,v}$ the canonical projection of $\plcat{X_2}(u,v)$ on $\Cr_2(u,v)$. 
Since $X$ is convergent, every $2$-cell
admits a unique normal form. It follows that $\pi_{u,v}$ restricts to
a bijection
\[
\Nf(X)\cap\plcat{X_2}(u,v)
\ \longrightarrow\
\Cr_2(u,v).
\]
As its domain is the free $\bk$-module generated by the monomial normal forms from $u$ to $v$.
Hence the images of these monomial
normal forms form a $\bk$-basis of $\Cr_2(u,v)$. Since this holds for
all $u,v\in\cat{X_1}$, the family $\Nf_m(X)$ is a hom-basis
of $\Cr$.
\end{proof}

In order to apply the above result when $X$ is terminating but not confluent, one may apply a completion procedure, in the spirit of those introduced for algebras in~\cite{Buchberger65} and for term rewriting systems in~\cite{KnuthBendix70}.
In many situations, however, this completion procedure does not terminate, due to the large number of rules and the interactions between them. To reduce the number of interactions that need to be considered and to obtain a more tractable completion procedure, we introduce in the next sections a modular stratification of the set of rules. The aim is to decompose the rewriting system into strata in such a way that confluence can be established stratum by stratum and then lifted to the whole rewriting system.

\subsection{Termination by orderings and derivations}

Termination techniques have been extensively developed in the setting of term rewriting systems~\cite{Dershowitz87,Terese03}. 
These techniques often rely on interpretation methods, which assign to each term a value in a well-founded ordered structure so that every rewriting step strictly decreases this value. In the higher-dimensional setting, however, classical interpretation methods do not directly provide suitable reduction orders for $3$-polygraphs~\cite{Guiraud06jpaa,Guiraud06cras}. This has led to the development of derivation-based methods, which adapt the principle of interpretations to the structure of $3$-polygraphs. In the linear setting, termination can also be established by means of monomial orders.

In this subsection, we recall these two approaches, which will be combined through stratification in the next section.

\subsubsection{Monomial orders}
\label{SSS:MnomialOrders}
A \emph{total preorder} on a set $Z$ is a transitive relation $\preccurlyeq_Z$ on $Z$ such that either $x\preccurlyeq_Z y$ or $y\preccurlyeq_Z x$ for all $x,y\in Z$.
For distinct $x$ and $y$, if $x \preccurlyeq_Z y$, we write $x \prec_Z y$.
Such an order is a \emph{total order} if it is moreover antisymmetric, that is, if $x \preccurlyeq_Z y$ and $y \preccurlyeq_Z x$ imply $x=y$.

A \emph{(hom-set) relation} $\prec$ on a free $2$-precategory $\pcat{X_2}$ is a family of relations $\prec_{u,v}$ indexed by pairs $(u,v)$ of $1$-cells of $\cat{X_1}$, where each relation $\prec_{u,v}$ is defined on the hom-set $\pcat{X_2}(u,v)$.
The notions of \emph{total preorder} and \emph{total order} on $\pcat{X_2}$ are then defined pointwise on each hom-set.
A relation $\prec$ is \emph{monotone} if, for any parallel $2$-cells $f,g : u \fl v$ and compatible context $C$ of $\pcat{X_2}$, we have
\[
f \prec_{u,v} g
\quad\text{implies}\quad
C[f]\prec_{C[u],C[v]} C[g].
\]
It is \emph{well-founded} if each relation $\prec_{u,v}$ is well-founded, that is, admits no infinite descending chain.
A total order $\preccurlyeq$ on a precategory $\pcat{X_2}$ is a \emph{monomial order} if it is monotone and well-founded.

For every nonzero $2$-cell $f$ of $\plcat{X_2}$, the \emph{$\preccurlyeq$-leading monomial} of $f$, denoted by $\lm_{\preccurlyeq}(f)$, is the $\preccurlyeq$-maximum element of $\supp(f)$. The \emph{$\preccurlyeq$-leading coefficient} of $f$, denoted by $\lc_{\preccurlyeq}(f)$, is the coefficient of $\lm_{\preccurlyeq}(f)$ in $f$. 
Finally, the \emph{$\preccurlyeq$-leading term} of $f$ is defined by 
$\lt_{\preccurlyeq}(f) =\lc_{\preccurlyeq}(f)\,\lm_{\preccurlyeq}(f)$.

Let $X$ be a $3$-prepolygraph. 
A monomial order $\preccurlyeq$ on $\pcat{X_2}$ is \emph{$X$-compatible} if $v \prec \lm(\alpha)$, for all $\alpha\in X_3$ and $v\in\suppm(t\alpha)$.
For every $3$-cell $\beta\in\plcat{X_3}$ such that
$\partial\beta\neq 0$, set
\[
\lambda_\beta
\: := \:
\begin{cases}
\lc_{\preccurlyeq}(\partial\beta)^{-1},
&\text{if $\lc_{\preccurlyeq}(\partial\beta)$ is invertible},\\
1,
&\text{otherwise}.
\end{cases}
\]
The \emph{$\preccurlyeq$-oriented left-term $3$-cell associated with
$\beta$}, denoted by $\beta^{\preccurlyeq}$, is defined by setting
\[
s(\beta^{\preccurlyeq})
\: := \:
\lambda_\beta\lt_{\preccurlyeq}(\partial\beta),
\qquad
t(\beta^{\preccurlyeq})
\: := \:
\begin{cases}
\lambda_\beta\lt_{\preccurlyeq}(\partial\beta)
+\lambda_\beta\partial\beta,
&
\text{if }
\lm_{\preccurlyeq}(\partial\beta)
\in\suppm(t\beta)\setminus\suppm(s\beta),
\\[0.5em]
\lambda_\beta\lt_{\preccurlyeq}(\partial\beta)
-\lambda_\beta\partial\beta,
&
\text{otherwise}.
\end{cases}
\]
The existence of an $X$-compatible monomial order implies that $X$ is termination, see {\cite[Sec.~3.2]{GuiraudHoffbeckMalbos19}}.

\begin{remark}
Compared with $3$-prepolygraphs, proving  termination of a $3$-polygraph by constructing a monomial order is often difficult \cite{GuiraudMalbos09,Alleaume18,DupontPhD2020}.
For instance, consider a $3$-polygraph  $X$ with a single $2$-generator $\twocell{cap}$ and a single $3$-generator
\[
\twocell{id*0 id*0 cap} \: \fl \: \twocell{ cap*0 id*0 id}\,.
\]
If there exists an $X$-compatible monomial order $\preccurlyeq$, then necessarily $\twocell{ cap*0 id*0 id}\prec\twocell{id*0 id*0 cap}$\;. Composing both sides on top with the $2$-cell $\twocell{cap}$ yields
\[
\twocell{ (cap)*1(cap*0 id*0 id)}
\:\prec\:
\twocell{(cap)*1(id*0 id*0 cap)}\,.
\]
This contradicts the interchange relations.
\end{remark}

We now construct a monomial order for $3$-prepolygraphs, which will be used in Section~\ref{S:NormFormHomBasisqSchurCategories}.
\subsubsection{Lexicographic orders}
\label{SSS:LexicographicOrders}
Given a family $(\preccurlyeq_{o_i})_{1 \leq i \leq N}$ of total preorders on a free $2$-precategory~$\pcat{X_2}$, with $N \geq 2$, we define a new total preorder on $\pcat{X_2}$ by induction on $i$, denoted by $\preccurlyeq_{o_1,\dots,o_N}$, as follows:
\[
u \preccurlyeq_{o_1,\dots,o_N} v
\quad\text{if}\quad
u \prec_{o_1} v
\;\text{or}\;
\big(u =_{o_1} v \;\text{and}\; u \preccurlyeq_{o_2,\dots,o_N} v\big),
\quad
\text{ for $u,v \in \pcat{X_2}$.}
\]
If the last order $\preccurlyeq_{o_N}$ is a total order, then the order $\preccurlyeq_{o_1,\dots,o_N}$ is itself a total order \cite[Lem.~1.15]{LIU25}.

Similarly, the orders $(\preccurlyeq_{o_i})_{1 \leq i \leq N}$ also induce a total preorder
on the Cartesian product $(\pcat{X_2})^{N}$, denoted by $\preccurlyeq_{[o_1,\ldots,o_N]}$, defined as follows:
for $(u_1,\ldots,u_N)$ and $(v_1,\ldots,v_N)$ in $(\pcat{X_2})^{N}$,
\[
(u_1,\ldots,u_N) \preccurlyeq_{[o_1,\ldots,o_N]} (v_1,\ldots,v_N)
\quad\text{if}\quad
u_1 \prec_{o_1} v_1
\;\text{or}\;
\big(u_1 =_{o_1} v_1 \;\text{and}\; (u_2,\ldots,u_N) \preccurlyeq_{[o_2,\ldots,o_N]} (v_2,\ldots,v_N)\big).
\]
Moreover, if each total preorder $\preccurlyeq_{o_i}$ is well-founded, then the order $\preccurlyeq_{[o_1,\ldots,o_N]}$ is also well-founded
\cite[Lem.~1.15]{LIU25}.

Suppose that there exists a well-founded total order $\preccurlyeq_o$ on the set $X_1\sqcup X_2$.
We define a total order~$\preccurlyeq_{\mathrm{lex}}$ on the free $2$-precategory $\pcat{X_2}$ as follows.
\begin{enumerate}
\item Let $f$ and $g$ be two parallel $2$-cells of $\pcat{X_2}$ of size~$1$, written as $ f = x_n  \cdots  x_1 \circ_0 \varphi \circ_0 y_1 \cdots  y_m$ and $g = x'_{n'}  \cdots  x'_{1}  \circ_0  \varphi' \circ_0  y'_1  \cdots  y'_{m'}$,  where $\varphi, \varphi' \in X_2$ and $x_i, y_i, x'_i, y'_i \in X_1$.
We define $f\preccurlyeq_{\lex} g$ if either $m+n<m'+n'$, or $m+n=m'+n'$ and
 \[
(x_n,\dots,x_1,\varphi,y_1,\dots,y_m)
\preccurlyeq_{[o,\ldots,o]}   (x'_{n'},\dots,x'_{1},\varphi',y'_1,\dots,y'_{m'}).
  \]
For convenience, we still write $\preccurlyeq_o \coloneq \preccurlyeq_{[o,\ldots,o]}$ for the order on parallel $2$-cells of size~$1$.
\item In general, we consider $\abs{f}=n$ and $\abs{g}=m$, and write their unique $\circ_1$-decompositions as
$f = C_1[\varphi_1] \circ_1 \ldots \circ_1 C_n[\varphi_n]$ and 
$g = D_1[\psi_1] \circ_1 \ldots \circ_1 D_m[\psi_m]$,
where $\varphi_i,\psi_i\in X_2$.
We define $f \preccurlyeq_{\lex} g$ if either $n<m$, or $n=m$ and
\[
(C_1[\varphi_1], \ldots, C_n[\varphi_n])
 \preccurlyeq_{[o,\ldots,o]}
 (D_1[\psi_1], \ldots, D_m[\psi_m]).
\]
\end{enumerate}
We prove that the relation $\preccurlyeq_{\lex}$ is a monomial order on $\pcat{X_2}$, called the \emph{lexicographic order}, using the same arguments as for the lexicographic order on free words
\cite[Lem.~1.15]{LIU25}.

\subsubsection{Derivations} 
\label{SSS:TerminationDerivation}
Given a $2$-category $\Cr$, a \emph{$\Cr$-module $\Or$}  is a functor from the category of contexts $\Cont(\Cr)$ to the category $\Ab$ of abelian groups.
Denote by $\Ord$ the $2$-category with a single $0$-cell, whose $1$-cells are partially ordered sets and whose $2$-cells are monotone maps. The $0$-composition is given by the cartesian product while the $1$-composition is given by the composition of monotone maps.
Given a $2$-functor $Z: \Cr \fl \Ord$, and an ordered abelian group $(G,\leq_G)$, the following assignments define a $\Cr$-module, denoted by $\Or_{Z,G}$:
\begin{enumerate}
\item Every $2$-cell $f$ is sent to the ordered abelian group of morphisms $\Or_{Z,G}(f) \:=\: \Ord\big(\: Z(sf), \: G \big)$, consisting of monotone maps from $Z(sf)$ to $G$, provided with the natural order define by $a \leq b$, if $a(y)\leq_G b(y)$, for all $a,b\in\Or_{Z,G}(f)$ and $y\in Z(sf)$.

\item If $C=w\square w'$ is a context of $\Cr$, with  a $1$-sphere $\square$ and $w,w'\in\Cr_1$, then $\Or_{Z,G}(C)$ sends a monotone map $a:Z(s\square)\fl G$ to the map
\[
Z(w)\times Z(s\square) \times Z(w') \:\fl\: G
\]
sending $(y,z,y')$ to $a(z)$.
\item If $C = g\circ_1 \square \circ_1 h$  is a context of $\Cr$, with  a $1$-sphere $\square$ and $g,h\in\Cr_2$, then $\Or_{Z,G}(C)$ sends a monotone map $a:Z(s\square)\fl G$ to the map $Z(sg) \:\fl\: G$
sending $y$ to $a(Z(g)(y))$.
\end{enumerate}

Suppose now that the $2$-category $\Cr$ is presented by a $3$-polygraph $X$.  
Then the $\Cr$-module $\Or_{Z,G}$ is uniquely determined by the partially ordered sets $Z(x)$, for $x\in X_1$, together with the monotone maps $Z(\varphi): Z(s\varphi) \fl Z(t\varphi)$, for $\varphi\in X_2$. 

A \emph{derivation} of $\Cr$ with values in a $\Cr$-module $\Or$ is a map  
$d:\Cr\to\Or$ sending  each $2$-cell $f$ to an element $d(f)\in\Or(f)$ such that, for every $i\in\{0,1\}$ and every $i$-composable pair
$(f,g)$ of $2$-cells in $\Cr$,
\[
d\left(f \circ_i g\right) = f \circ_i d(g) + d(f) \circ_i g.
\]
Following {\cite[Thm.~4.2.1]{GuiraudMalbos09}}, a $3$-polygraph $X$ is terminating if there exist a $2$-functor $Z : \cat{X_2} \fl \Ord$, an ordered abelian group $G$, and a derivation $d : \cat{X_2} \fl \Or_{Z,G}$ satisfying the following two conditions:
\begin{enumerate}
\item For any $1$-cell $c$ in $X_1$, the set $Z(c)$ is non-empty and, $Z(s\alpha) \geq Z(t\alpha)$ for every $\alpha \in X_3$.
\item For any $2$-cell $f$ in $\cat{X_2}$, we have $d(f) > 0$ and,
$d(s\alpha) \geq d(t\alpha)$ for every $\alpha \in X_3$. 
\end{enumerate}

This termination criterion was extended to linear $3$-polygraphs in {\cite[§1.3.3]{Dupont22}} and applies directly to $3$-prepolygraphs. 
However, it is not sufficiently flexible for certain diagrammatic algebras involving large rewriting systems, such as the KLR algebras considered in Subsection~\ref{SS:KLRalgebras} and the $q$-Schur categories studied in Section~\ref{S:NormFormHomBasisqSchurCategories}. In such cases, it is difficult to construct functors $Z$ and derivations $d$ satisfying the above conditions on the entire set of rewriting rules, which motivates the idea of stratifying the rewriting system.

\section{Stratified prepolygraphs and termination by stratification}
\label{S:StratifiedPrepolygraphsTerminationStratification}

This section introduces stratifications of 3-prepolygraphs modulo together with stratified termination functions, generalizing both derivation-based and monomial-order-based methods. We also illustrate this approach on Khovanov-Lauda-Rouquier algebras, obtaining a more efficient proof of termination than the one given in~\cite{Dupont21}.

\subsection{Stratified prepolygraphs and termination}
\label{SS:StratifiedPrepolygraphsTermination}

\subsubsection{Stratified prepolygraphs modulo}
\label{SSS:StratifiedPolygraphModulo}
A \emph{$k$-stratified $3$-prepolygraph modulo} $F^\ast(X,E)$ is a $3$-prepolygraph modulo $(X,E)$ together with a sequence $(F^iX)_{1\leq i \leq k}$ in $\Exte{\plcat{X_2}}$ such that
\[
F^0 X=E \subseteq F^1 X \subseteq \; \ldots \; \subseteq F^k X = X_3E,
\]
which we call a \emph{$k$-stratification} of the $3$-prepolygraph modulo $(X,E)$.
The sequence $(\strat{i} X)_{1\leq i \leq k}$, where $\strat{i} X \coloneq F^i X \setminus F^0 X$, is called the \emph{reduced stratification} of the $3$-prepolygraph~$(X,E)$. 
For $i\geq 1$, the extension $\strat{i}X \coloneq F^i X \setminus F^{i - 1} X$ is called the \emph{stratum of depth $i$} of the extension $X_3$.
When $E$ is empty, we call it a \emph{$k$-stratified $3$-prepolygraph}
and write $F^\ast X$.

\subsubsection{Weighted functions}
\label{SSS:StrafiedWeightedMaps}

A \emph{weighted function} on $\pcat{X_2}$ is a hom-set–wise function
\[
\delta = \big(\delta_{u,v} : \pcat{X_2}(u,v) \to (Z_{u,v},\preccurlyeq_{u,v})\big)_{u,v\in \cat{X_1}}\raisebox{-0.8ex}{,}
\]
where each $(Z_{u,v},\preccurlyeq_{u,v})$ is a well-founded totally preordered set. We require that $\delta$ be monotone with respect to contexts, that is for all $2$-cells $f,g$ in $\pcat{X_2}(u,v)$ and every compatible context $C$, one has
\[
\delta(f) \preccurlyeq_{u,v} \delta(g)
\quad\text{implies}\quad
\delta(C[f]) \preccurlyeq_{u,v} \delta(C[g]).
\]
Given an extension $Y$ of $\plcat{X_2}$, we say that $\delta$ is \emph{$Y$-compatible} (resp. \emph{strictly $Y$-compatible}; an \emph{$Y$-identity}) if, for every $\alpha\in Y$, one has 
$\delta(v) \preccurlyeq \delta(\lm(\alpha))$ (resp. $\delta(v) \prec \delta(\lm(\alpha))$; $ \delta(v) =\delta(\lm(\alpha))$), for every $v\in\suppm(t\alpha)$. 
Note that the preorders~$\preccurlyeq_{u,v}$ need not be antisymmetric.

\subsubsection{Termination function}
\label{SSS:DefinitionWeightedMap}
A \emph{termination function} of a $k$-stratified $3$-prepolygraph modulo $F^\ast(X,E)$ is a family of weighted functions $(\delta_i)_{1\leq i \leq k}$, satisfying the following conditions, 
\begin{enumerate}
\item $\delta_i$ is strictly $\strat{i}X$-compatible, for every $1\leq i \leq k$, 
\item $\delta_i$ is $\strat{i-1}X$-compatible, for every $2\leq i \leq k$,
\item $\delta_i$ is $E$-identity, for every $1\leq i \leq k$.
\end{enumerate}

Condition {\bf i)} ensures termination within each stratum, while Condition {\bf ii)} provides the compatibility between successive strata required for the modularity of termination established in Theorem~\ref{T:MainTheoremOnTermination}.
The proof relies on a reduction from linear rewriting to monomial rewriting, which we formulate in terms of the \emph{monomial $3$-prepolygraph} $\monex{X}=(X_{\leq 2},\monex{X_3})$ associated with a left-monomial $3$-prepolygraph $X$, defined by
\[
\monex{X_3} \coloneq \{(\lm(\alpha),v) \mid \alpha \in X_3,\; v \in \suppm(t\alpha)\}.
\]

\begin{lemma}
\label{L:PolAndLinearPol}
A $3$-prepolygraph modulo $(X,E)$ is $E$-terminating if and only if its associated monomial $3$-prepolygraph modulo $(\monex{X},\monex{E})$ is $\monex{E}$-terminating.
\end{lemma}
\begin{proof}
This result follows from properties of multiset relations,
as introduced in \cite{DershowitzManna79}, that we adapt to our modulo setting. 
We define a relation $\succcurlyeq_\vdash$  on $\pcat{X_2}$ by setting
\[
s\alpha \succ_\vdash t\alpha \quad \text{for every } \alpha \in \monex{X_3},
\qquad
s\beta =_\vdash t\beta \quad \text{for every } \beta \in \monex{E}.
\]
Moreover, for any compatible context $C$, we require $C[f] \succcurlyeq_\vdash C[g]$  if $f \succcurlyeq_\vdash g$.
This induces a binary relation $\vdash$ on $\pcat{X_2}/\monex{E}$ defined by
\[
u \vdash v
\quad \text{if there exist } u', v' \in \pcat{X_2}
\text{ such that }
u =_\vdash u' \succ_\vdash v' =_\vdash v .
\]
We then define analogously the relation on $\plcat{X_2}/E$ as the restriction to finite subsets of $\pcat{X_2}/\monex{E}$, of the multiset relation generated by $\vdash$.
We refer to {\cite[Sec.~2.5]{BaaderNipkow98}} for the definition and properties of the multiset relation generated by a binary relation~$\vdash$.  
In particular, the relation $\vdash$ is well-founded on the $2$-cells of
$\plcat{X_2}/\monex{E}$ if and only if it is well-founded on the monomials of
$\pcat{X_2}/E$, whose proof is based on K\"onig's lemma.
\end{proof}

\begin{theorem}
\label{T:MainTheoremOnTermination}
A $k$-stratified $3$-prepolygraph modulo $F^\ast(X,E)$ having a termination function is \linebreak  $E$-terminating.
\end{theorem}
\begin{proof}
We have $X_3=\strat{1}X\strat{2}\ldots X\strat{k}X$.
We first consider the case where $X_3$ is a monomial extension. 
Suppose, for a contradiction, that there exists an infinite ${}_E X_E$-rewriting path
\[
f_{1} \sim_E g_{1} 
\ofl{\sigma_{1}} 
f_{2} \sim_E g_{2}
\ofl{\sigma_{2}}
f_{3} \sim_E g_{3}
\ofl{\sigma_{3}} \cdots
\ofl{\sigma_{n-1}}
f_{n} \sim_E g_{n}
\ofl{\sigma_{n}} \cdots,
\]
where each $\sigma_i$ is a $X$-rewriting step.
Since $\delta$ is a termination function, we have $\delta_k(f_i)=\delta_k(g_i)$, for every $1\leq i$.
Moreover, according to whether $\sigma_i$ is a $\strat{k}X$-rewriting step or a $\strat{k-1}X$-rewriting step, we have respectively 
\[
\delta_k(g_i)\succ\delta_k(f_{i+1}) 
\qquad \text{or} \qquad
\delta_k(g_i)\succcurlyeq \delta_k(f_{i+1}).
\]
By well-foundedness, the rewriting path can therefore contain only finitely many $\strat{k}X$-rewriting steps. Hence, there exists an index $j_k$ such that $\sigma_i$ is a $\strat{k-1}X$-rewriting step for every $i\geq j_k$.

Repeating the same argument for each function $\delta_r$ for $r \leq k-1$, we eventually obtain an index $j_2$ such that all $\sigma_i$ is an $\strat{1}X$-rewriting step for every $i \geq j_2$.
This is sufficient to prove that the stratum $\strat{1}X$ is terminating.
By definition of the termination function $\delta$, the map $\delta_1$ is strictly $\strat{1}X$-compatible. Hence, $\strat{1}X$ is terminating, that is, the $3$-prepolygraph modulo $(\strat{1}X,E)$ is $E$-terminating.

In the general case, if $X_3$ is an extension of $\plcat{X_2}$, the stratification $F^\ast(X,E)$ and the termination function $\delta$ induce a stratification of $G^\ast(\monex{X},\monex{E})$ and an associated  termination function $\delta'$ by setting 
\[
G_s^\ast\monex{X}=\strat{1}\monex{X}\strat{2}\monex{X}\ldots \strat{k}\monex{X}\quad \text{ and }\quad \delta'(f) = \delta(f),
\]
 for every $f\in\pcat{X_2}$, where $\strat{i}\monex{X}\coloneq\monex{\strat{i}X}$.
Following the previous monomial case, we deduce that the $3$-prepolygraph modulo $G^\ast(\monex{X},\monex{E})$ is $\monex{E}$-terminating.  
By Lemma~\ref{L:PolAndLinearPol}, we therefore conclude that $F^\ast(X,E)$  is 
$E$-terminating.
\end{proof}

\subsection{Stratified termination using monomial orders and derivations}
\label{SS:TerminationByStratification}

In practice, termination functions can often be constructed from familiar termination measures. We first explain how a global monomial order gives rise to a termination function. We then formulate a stratified version of the classical derivation criterion for termination.

\begin{proposition}
\label{P:WeightedMapsvsMonomialOrders}
Let $F^\ast(X,E)$ be a $k$-stratified $3$-prepolygraph modulo.
If $\delta$ is a termination function of $F^\ast(X,E)$, such that $\delta_1$ takes values in a well-founded totally ordered set, then the relation $\preccurlyeq_\delta$ defined by
\[
f \preccurlyeq_\delta g 
\quad\text{if}\quad
(\delta_k(f),\ldots,\delta_1(f))\preccurlyeq_{[k,\ldots,1]} (\delta_k(g),\ldots,\delta_1(g)),
\]
for all $f,g\in \pcat{X_2}$, is an $X$-compatible monomial order.
Conversely, for any $X$-compatible monomial order~$\preccurlyeq$ on $\pcat{X_2}$, the identity map on $\pcat{X_2}$ defines a termination function of $F^\ast X$.
\end{proposition}
\begin{proof}
Let $\delta$ be a termination function of $F^\ast(X,E)$.
Since for each $2\leq i \leq k$ the relation $\preccurlyeq_i$ is a well-founded total preorder
and $\preccurlyeq_1$ is a well-founded total order,
the lexicographic order $\preccurlyeq_{[k,\ldots,1]}$ is a well-founded total order
following \SSS{SSS:LexicographicOrders}, and hence so is $\preccurlyeq_\delta$.
Moreover, the $X$-compatibility of $\delta$ implies that $\preccurlyeq_\delta$
is an $X$-compatible monomial order.

Conversely, assume that $\preccurlyeq$ is an $X$-compatible monomial order on $\pcat{X_2}$.
Then $\preccurlyeq$ is a well-founded total order on $\pcat{X_2}$, and by $X$-compatibility 
$\mathrm{id}_{\pcat{X_2}}(t\alpha)\prec \mathrm{id}_{\pcat{X_2}}(s\alpha)$ 
holds, for every $\alpha\in X_3$. 
Therefore, $\mathrm{id}_{\pcat{X_2}}$ is a termination function.
\end{proof}

\subsubsection{Bounded below derivations}
Let $(X,E)$ be a $3$-prepolygraph modulo and $(G,\leq)$ be an ordered abelian group.
A $2$-functor $Z : \pcat{X_2} \to \Ord$ is \emph{$X_3E$-monotone} if, for every $1$-cell $c \in X_1$, the set $Z(c)$ is non-empty and,
for all $\alpha \in \monex{X_3}$ and $\beta \in E$, we have
\[
Z(s\alpha) \geq Z(t\alpha)
\quad\text{and}\quad
Z(s\beta) = Z(t\beta).
\]
A derivation $d : \pcat{X_2} \to \Or_{Z,G}$ is \emph{bounded below}, if there exists an element $m \in G$ such that, for all $\alpha \in X_2$ and $y\in Z(s\alpha)$, we have $d(\alpha)(y) \geq m$.

\begin{proposition}
\label{P:DerivationvsWeightedmap}
Let $(X,E)$ be a $3$-prepolygraph modulo.
Then any $X_3E$-monotone $2$-functor $Z : \pcat{X_2} \fl \Ord$ and bounded below derivation $d : \pcat{X_2} \fl \Or_{Z,G}$ induces a weighted function $\delta^{Z,d}$ defined, for all $u,v \in \cat{X_1}$, by
\[
\delta_{u,v}^{Z,d} \colon \pcat{X_2}(u,v) \fl (\Ord(Z(u),G),\leq),
\]
sending $f$ to $d(f)$, where $\leq$ is the natural order on $\Ord(Z(u),G)$.
\end{proposition}
\begin{proof}
For every $f\in \pcat{X_2}(u,v)$, we have $d(f)\in \Or_Z(f)=\Ord(Z(u),G)$, that is $\delta^{Z,d}_{u,v}$ is well-defined.
Assume that there exist two \ $2$-cells $f$ and $g$ in $\pcat{X_2}$ such that
$\delta^{Z,d}(f)< \delta^{Z,d}(g)$.
For all compatible context $C = h \circ_1 w \,\square\, w' \circ_1 h'$, by the definition of derivations, we have
\[
d(C[f])=d(h)\circ_1 wf w'\circ_1 h'+h\circ_1 w d(f) w'\circ_1 h'+h\circ_1 wf w'\circ_1 d(h')
\]
 and similarly for $d(C[g])$.
By the hypothesis that $d(f) < d(g)$, we have
\[
h\circ_1 w d(f) w'\circ_1 h' < h\circ_1 w d(g) w'\circ_1 h'
\]
Since the functor $Z$ is $X_3E$-monotone, it is decreasing on generators and thus 
\[
(d(h)\circ_1 w f w'\circ_1 h') \leq (d(h)\circ_1 w g w'\circ_1 h') 
\quad
\text{and}
\quad
(h\circ_1 w f w'\circ_1 d(h')) \leq (h\circ_1 w g w'\circ_1 d(h')) .
\]
Therefore, we deduce that $d(C[f])<d(C[g])$, that is $\delta^{Z,d}(C[f])< \delta^{Z,d}(C[g])$.

Since the derivation $d$ is bounded below, there exists $m \in G$ such that,
for every $\alpha \in X_2$ and every $y \in Z(s\alpha)$, one has $d(\alpha)(y) \geq m$.
For every $f \in \pcat{X_2}$, write its unique $\circ_1$-decomposition
$
f = C_1[\alpha_1] \circ_1 \ldots \circ_1 C_n[\alpha_n],
$
 where  $\alpha_i \in X_2$.
Then we have
\[
d(f)(a)
=
\Bigl(
C_1[d(\alpha_1)] \cdot \ldots \cdot C_n[\alpha_n]
+
C_1[\alpha_1] \cdot \ldots \cdot C_n[d(\alpha_n)]
\Bigr)(a)
\;\geq\; nm,
\]
for every $a \in Z(sf)$.
If the order $\leq$ is not well-founded, then there exists an infinite sequence of parallel $2$-cells $(f_k)_{k}$ with the same source $s$, which yields a strictly decreasing sequence of integers $\bigl(d(f_k)(a_k)\bigr)_{k}$, 
where $a_k \in Z(s)$.
This contradicts the hypothesis that the derivation $d$ is bounded below.
Hence, $\delta^{Z,d}$ is a weighted function.
\end{proof}

\begin{remark}
Stratified termination functions generalize both monomial-order-based and derivation-based termination criteria for a $k$-stratified $3$-prepolygraph $F^\ast(X,E)$. 
Indeed, by Proposition~\ref{P:WeightedMapsvsMonomialOrders}, every $X$-compatible monomial order on $\pcat{X_2}$ induces a termination function on $F^\ast(X,E)$. Similarly, by Proposition~\ref{P:DerivationvsWeightedmap}, every $X_3$-monotone $2$-functor equipped with a bounded below derivation induces such a termination function.
In both cases, Theorem~\ref{T:MainTheoremOnTermination} implies that $X$ is terminating, thereby recovering the monomial-order-based and derivation-based termination criteria.

Consequently, stratified termination functions extend these classical termination criteria by making it possible to establish termination stratum by stratum in situations where suitable global monomial orders or derivations are difficult to construct due to the large number of rewriting rules.
\end{remark}

\subsection{Examples: termination of a presentation of KLR algebras}
\label{SS:KLRalgebras}

In this subsection, we illustrate how termination of a $3$-prepolygraph modulo can be established by means of a stratification of its set of rules. As an application, we consider the challenging presentation of KLR algebras~\cite{Khovanov_2009,Rouquier08}, whose rewriting system involves a large number of rules. Throughout this subsection, we set $\bk=\mathbb{Z}$.

\subsubsection{Polygraphic presentation of KLR algebras}
Let $\Gamma$ be a finite tree (a connected graph with no loops, nor multiple edges) with set of vertices denoted by $I$. 
Let $\cdot$ be a bilinear form on $\mathbb Z[I]$ defined by $i\cdot i = 2$ for any $i\in I$ and for $i \ne j$ in $I$
\[
i\cdot j =
\left\{
\begin{array}{ll}
-1 & \hbox{ if there is an edge between } i \hbox{ and } j \hbox{ in } \Gamma,\\
0 & \hbox { otherwise.}
\end{array}
\right.
\]
Let $(X,E)$ be the $3$-prepolygraph modulo with a single $0$-cell, whose $1$-generators are the elements of $I$ and whose $2$-generators are 
\[
\twocell{tau *1 (i *0 j)} : i \circ_0 j \longrightarrow j \circ_0 i,
\quad \text{and} \quad
\twocell{eta *1 i} : i \longrightarrow i,
\]
for any $i,j \in I$.
We set $E\coloneq \Int(X_2)$, and the $3$-generators in $X_3$ are 
for all $i,j,k\in I$,
\[
\alpha_{i,j}^L:\twocell{((eta *0 1)*1tau)*1 (i*0 j)} \fl \twocell{(tau *1 (1 *0 eta))*1 (i*0 j)}\ , 
\;\text{when $i\neq j$,}\qquad
\alpha_{i,j}^R:\twocell{((1 *0 eta)*1tau)*1 (i*0 j)} \fl \twocell{(tau *1 (eta *0 1))*1 (i*0 j)}\ , 
\;\text{when $i\neq j$},
\]
\[
\alpha_{i}^{L}:\twocell{((eta *0 1)*1tau)*1 (i*0 i)} \fl \twocell{(tau *1 (1 *0 eta))*1 (i*0 i)}+\twocell{(id*0 id)*1 (i*0 i)} \ ,
\quad
\alpha_{i}^{R}:\twocell{((1 *0 eta)*1tau)*1 (i*0 i)} \fl \twocell{(tau *1 (eta *0 1))*1 (i*0 i)}-\twocell{(id*0 id)*1 (i*0 i)} \ ,
\quad
\beta_i:\twocell{tau*1 tau*1 (i*0i)}\fl 0\ ,
\]
\[
\beta_{j,k}:\twocell{(tau*1tau)*1(j*0 k)} \fl \twocell{(id*0 id)*1(j*0 k)}\ ,
\;\text{when $j\cdot k=0$,}\qquad
\beta_{i,j}:\twocell{(tau*1tau)*1(i*0 j)} \fl \twocell{(eta*0 1)*1(i*0 j)}+ \twocell{(1*0 eta)*1(i*0 j)}\ ,
\;\text{when $i\cdot j=-1$,}
\]
\[
\gamma_{i,j,k}:\twocell{(tau *0 1) *1 (1 *0 tau) *1 (tau *0 1)*1(i*0 j*0 k)}
\fl \twocell{(1 *0 tau) *1 (tau *0 1) *1 (1 *0 tau)*1(i*0 j*0 k)}\ ,
\;\text{when $i\neq k$ or $i\cdot j\neq -1$,}\qquad
\gamma_{i,j,i}:\twocell{(tau *0 1) *1 (1 *0 tau) *1 (tau *0 1)*1(i*0 j*0 i)}
\fl \twocell{(1 *0 tau) *1 (tau *0 1) *1 (1 *0 tau)*1(i*0 j*0 i)} + \twocell{(id*0 id*0 id)*1(i*0j*0i)}\ ,
\;\text{when $i\cdot j=-1$.}
\]
The $3$-prepolygraph modulo $(X,E)$ is presenting the KLR monoidal category, seen as a $2$-category with only one $0$-cell, associated with $\Gamma$, see \cite[Thm.2.3.2]{Dupont21}. 
This monoidal category contains the KLR algebras as hom-spaces.

\subsubsection{Stratified termination of $(X,E)$}
Thanks to the Lemma~\ref{L:PolAndLinearPol}, we prove $E$-termination of the $3$-prepolygraph modulo $(X,E)$ by proving that its associated monomial $3$-prepolygraph modulo is $E$-terminating.
For each $i,j\in I$, we set
\[
\alpha_{i}^{L1}:
\twocell{((eta *0 1)*1tau)*1 (i*0 i)}
\fl
\twocell{(tau *1 (1 *0 eta))*1 (i*0 i)},
\qquad
\alpha_{i}^{L2}:\twocell{((eta *0 1)*1tau)*1 (i*0 i)} \fl \twocell{(id*0 id)*1 (i*0 i)}\,,
\qquad
\alpha_{i}^{R1}:
\twocell{((1 *0 eta)*1tau)*1 (i*0 i)}
\fl
\twocell{(tau *1 (eta *0 1))*1 (i*0 i)},
\qquad
\alpha_{i}^{R2}:\twocell{((1 *0 eta)*1tau)*1 (i*0 i)} \fl \twocell{(id*0 id)*1 (i*0 i)}\,,
\]
\[
\beta^1_{i,j}:\twocell{(tau*1tau)*1(i*0 j)} \fl \twocell{(eta*0 1)*1(i*0 j)}\,,
\qquad
\beta^2_{i,j}:\twocell{(tau*1tau)*1(i*0 j)} \fl \twocell{(1*0 eta)*1(i*0 j)}\,,
\qquad
\gamma^1_{i,j,i}:\twocell{(tau *0 1) *1 (1 *0 tau) *1 (tau *0 1)*1(i*0 j*0 i)}
\fl \twocell{(1 *0 tau) *1 (tau *0 1) *1 (1 *0 tau)*1(i*0 j*0 i)} \,,
\qquad
\gamma^2_{i,j,i}:\twocell{(tau *0 1) *1 (1 *0 tau) *1 (tau *0 1)*1(i*0 j*0 i)}
\fl \twocell{(id*0 id*0 id)*1(i*0j*0i)}\,,
\]
and define a $4$-stratification of $\monex{X}$ by setting
\[
\strat{4} \monex{X}
= \{\alpha_i^{L2},\ \alpha_i^{R2},\ \beta_i,\ \beta_{j,k},\ \beta_{i,j}^{1}, \ \beta_{i,j}^{2},\ \gamma_{i,j,i}^{2}\},
\qquad
\strat{3} \monex{X}
= \{\alpha_{i,j}^{L},\ \alpha_{i}^{L1}\},
\]
\[
\strat{2} \monex{X}
= \{\gamma_{i,j,k},\ \gamma_{i,j,i}^{1}\},
\qquad
\strat{1} \monex{X}
= \{\alpha_{i,j}^{R},\ \alpha_{i}^{R1}\}.
\]
This defines a $4$-stratified $3$-prepolygraph modulo
$F^\ast(\monex{X},E)$, on which we now define a termination function
$\delta=(\delta_k)_{1\leq k\leq 4}$.

We define $\delta_4(f)$ to be the size of $f$ for every $f\in \pcat{X_2}$.
Since every source in $\strat{4} \monex{X}$ has strictly greater size than its target, whereas for every rule in $\strat{3}\monex{X}$ the source and the target have the same size, the weighted function $\delta_4$ is strictly $\strat{4} \monex{X}$-compatible and $\strat{3}\monex{X}$-compatible.

For every $k\leq 3$, following Proposition~\ref{P:DerivationvsWeightedmap}, we construct a weighted map $\delta_k$ induced by a $F^k \monex{X}$-monotone $2$-functor
$Z_k : \pcat{X_2} \to \Ord$
and a bounded below derivation $d_k:\pcat{X_2} \to \Or_{Z_k,\Zb}$.

We first set $Z_k(i)=\mathbb{N}$ for every simple root $i\in I$, so that  $Z_k(i \circ_0 j)=\mathbb{N}\times\mathbb{N}$ for all $i,j\in I$.
We then define the $Z_k$ and $d_k$ by setting (omitting the labels on the strands).
\begin{enumerate}
\item  We define the functor $Z_3$ and the derivation $d_3$ by setting the following values on $X_2$
\[
Z_3(\twocell{\tau})(n,m) = (m+1,n),
\qquad
Z_3(\twocell{\eta})(n) = n,
\qquad
d_3(\twocell{\tau})(n,m) = m,
\qquad
d_3(\twocell{\eta})(n) = n.
\]
for all $n, m \in \Nb$.
Since $m,n \geq 0$, the derivation $d_3$ is bounded below.
Moreover, the functor $Z_3$ is $F^3 \monex{X}$-monotone, since the following
relations hold:
\[
Z_3(\twocell{((eta *0 1)*1tau)})(n,m)
=
Z_3(\twocell{(tau *1 (1 *0 eta))})(n,m)
= (m+1,n),
\qquad
Z_3(\twocell{((1 *0 eta)*1tau)})(n,m)
=
Z_3(\twocell{(tau *1 (eta *0 1))})(n,m)
= (m+1,n),
\]
\[
Z_3(\twocell{(tau *0 1) *1 (1 *0 tau) *1 (tau *0 1)})(n,m,l)
=
Z_3(\twocell{(1 *0 tau) *1 (tau *0 1) *1 (1 *0 tau)})(n,m,l)
= (l+2,m+1,n).
\]
We set $\delta_3$ to be the weighted function $\delta^{Z_3,d_3}$ induced by $Z_3$ and $d_3$. 
It is strictly $\strat{3} \monex{X}$-compatible and $\strat{2}\monex{X}$-compatible, since the following relations hold:
\[
d_3(\twocell{((eta *0 1)*1tau)})(n,m)
= 2m+1
> 2m
= d_3(\twocell{(tau *1 (1 *0 eta))})(n,m),
\qquad
d_3(\twocell{((1 *0 eta)*1tau)})(n,m)
= m+n
= d_3(\twocell{(tau *1 (eta *0 1))})(n,m),
\]
\[
d_3(\twocell{(tau *0 1) *1 (1 *0 tau) *1 (tau *0 1)})(n,m,l)
= m+2l+1
=
d_3(\twocell{(1 *0 tau) *1 (tau *0 1) *1 (1 *0 tau)})(n,m,l).
\]

\item  
We define the functor $Z_2$ and the derivation $d_2$ by setting the following values on $X_2$
\[
Z_2(\twocell{\tau})(n,m)=(m+1,n),
\qquad
Z_2(\twocell{\eta})(n)=n,
\qquad
d_2(\twocell{\tau})(n,m)=n,
\qquad
d_2(\twocell{\eta})(n)=n,
\]
for all $n, m \in \Nb$.
Since $n \geq 0$, the derivation $d_2$ is bounded below.
Moreover, the functor $Z_2$ is $F^2 \monex{X}$-monotone, since the following
relations hold:
\[
Z_2(\twocell{(tau *0 1) *1 (1 *0 tau) *1 (tau *0 1)})(n,m,l)
=
Z_2(\twocell{(1 *0 tau) *1 (tau *0 1) *1 (1 *0 tau)})(n,m,l)=(l+2,m+1,n),
\quad
Z_2(\twocell{((1 *0 eta)*1tau)})(n,m)= Z_2(\twocell{(tau *1 (eta *0 1))})(n,m)=(m+1,n),
\]
We set $\delta_2$ to be the weighted function $\delta^{Z_2,d_2}$ induced by $Z_2$ and $d_2$. 
It is strictly $\strat{2} \monex{X}$-compatible and $\strat{1}\monex{X}$-compatible, since the following relations hold:
\[
d_2(\twocell{(tau *0 1) *1 (1 *0 tau) *1 (tau *0 1)})(n,m,l)=2n+m+1
>2n+m=
d_2(\twocell{(1 *0 tau) *1 (tau *0 1) *1 (1 *0 tau)})(n,m,l),
\qquad
d_2(\twocell{((1 *0 eta)*1tau)})(n,m)=2n= d_2(\twocell{(tau *1 (eta *0 1))})(n,m).
\]

\item  
We define the functor $Z_1$ and the derivation $d_1$ by setting the following values on $X_2$
\[
Z_1(\twocell{\tau})(n,m)=(m,n+1),
\qquad
Z_1(\twocell{\eta})(n)=n,
\qquad
d_1(\twocell{\tau})(n,m)=n,
\qquad
d_1(\twocell{\eta})(n)=n,
\]
for all $n, m \in \Nb$.
Since $n \geq 0$, the derivation $d_1$ is bounded below.
Moreover, the functor $Z_1$ is $F^1 \monex{X}$-monotone, since the following
relation hold:
\[
Z_1(\twocell{((1 *0 eta)*1tau)})(n,m)=Z_1(\twocell{(tau *1 (eta *0 1))})(n,m)=(m,n+1).
\]
We set $\delta_1$ to be the weighted function $\delta_{Z_1,d_1}$ induced by $Z_1$ and $d_1$. 
It is strictly $\strat{1} \monex{X}$-compatible, since the following relation hold:
\[
d_1(\twocell{((1 *0 eta)*1tau)})(n,m)=2n+1>2n=d_1(\twocell{(tau *1 (eta *0 1))})(n,m).
\]
\end{enumerate}
Note that each $\delta_i$ is $F^0 \monex{X}$-identity, since
\[
d_i\bigl(\twocell{(dot*0 id *0 psi)*1(dot *0 u *0 dot)*1(varphi*0 id *0 dot)}\bigr)
=
d_i\bigl(\twocell{(varphi*0 id *0 dot)*1(dot *0 u *0 dot)*1 (dot *0 id *0 psi)}\bigr)
\]
for every $1\leq i\leq 4$, all $1$-cell $u\in \cat{X_1}$, and $2$-generators $\varphi,\psi\in X_2$.
Therefore, $\delta$ is a termination function of the stratified $3$-prepolygraph modulo $F^\ast(\monex{X},E)$.
By Theorem~\ref{T:MainTheoremOnTermination}, the monomial $3$-prepolygraph modulo $(\monex{X},E)$ is $E$-terminating. It then follows from Lemma~\ref{L:PolAndLinearPol} that the original $3$-prepolygraph modulo $(X,E)$ is also $E$-terminating.

\subsubsection{Remarks}
This example illustrates how stratified termination simplifies the construction of termination arguments by allowing monomial orders and derivations to be adapted to subsets of rewriting rules grouped into strata according to their termination behaviour. Rather than defining a single global measure on the whole rewriting system, one may choose, for each stratum, the termination method that is best suited to the corresponding family of rules. In particular, monomial orders and derivations can be combined within the same stratified proof, with different strata controlled by different types of termination measures. In this sense, the argument given here extends the derivation-based approach used for KLR algebras in~\cite{Dupont21} and provides a flexible method for treating rewriting systems with large and structurally heterogeneous sets of rules.

We have illustrated the stratification method for proving termination in the case of a one-object presentation of the KLR $2$-category. In future work, it would be natural to extend this approach to the Kac--Moody $2$-category $\mathcal{KLR}$, defined in~\cite{Khovanov_2009,Rouquier08}, which is its multi-object counterpart. The $3$-cells of this presentation encode structural relations, including isotopy and bubble relations, whose termination raises difficulties inherent in the $2$-dimensional structure. The main challenge will therefore be to construct a stratification adapted to the category as a whole.

\section{Modular confluence and hom-bases from stratified normalization}
\label{S:HomBasesStratifiedNormalization}

This section develops a modular approach to confluence for stratified rewriting systems and applies it to the construction of hom-bases. We first introduce confluence modularity and derive a general hom-basis theorem for terminating stratified presentations satisfying this property. We then construct a stratified normalization procedure that enforces modularity by normalizing each new stratum with respect to the preceding ones, thereby absorbing mixed rewriting interactions into the normalized rules. We prove that the resulting presentation is terminating, presents the same linear category, and satisfies confluence modularity. In this sense, stratified normalization refines classical completion procedures~\cite{KnuthBendix70,Marche96} by organizing completion along the successive strata of the presentation. The hom-basis theorem then applies directly to the normalized presentation. We conclude the section by illustrating the method on the associative and coassociative category and on the Hecke category.

\subsection{Modular confluence and stratified normalization}
\label{SS:ModularConfluenceNormalization}

\subsubsection{Modularity of confluence}
Let $F^\ast X$ be a $k$-stratified $3$-prepolygraph. 
Confluence of the individual strata $\strat{i}X$ does not imply confluence of their union, since rewriting steps belonging to different strata may interact and produce additional mixed branchings.
We say that $F^\ast X$ satisfies the \emph{confluence modularity property} if, for every $1\leq j\leq k$, the confluence of each stratum $\strat{i}X$, for $1\leq i\leq j$, implies the confluence of the extension $F^jX$.

This property separates two aspects of a confluence proof. On the one hand, confluence may be checked independently within each stratum; on the other hand, the compatibility between different strata is encoded globally by the modularity property. As a consequence, for a terminating stratified presentation satisfying confluence modularity, the confluence problem for the whole rewriting system reduces to the confluence analysis of its individual strata.
By Theorem~\ref{T:HomBasisFromConvergence}, we have thus the following hom-bases theorem from modular confluence.
\begin{theorem}
\label{T:HomBasesFromConfluence}
Let $\Cr$ be a linear monoidal category presented by a  left-monomial terminating stratified $3$-prepolygraph $F^\ast X$, satisfing the confluence modularity property, and whose each stratum $\strat{i}X$ is confluent. Then the set $\Nf_m(X)$ forms a hom-basis of $\Cr$.
\end{theorem}

We now introduce a normalization procedure on stratified prepolygraphs in order to reach this modularity property holds without changing the presented monoidal category. Indeed, by normalizing each new stratum with respect to the preceding ones, we mix interactions between successive strata are incorporated into the normalized rewriting rules.

\subsubsection{Stratified normalization}
\label{SSS:Normalization}
Let $F^\ast X$ be a $k$-stratified $3$-prepolygraph, and $\preccurlyeq$ be a monomial order on $\pcat{X_2}$.
A \emph{$\preccurlyeq$-normalization} of $F^\ast X$ is a $k$-stratified $3$-prepolygraph, denoted by $F^\ast \NP{X}$, whose extensions $\strat{i}\NP{X}$ in $\Exte{\plcat{X_2}}$ are defined inductively.

We set $F^1\NP{X}\coloneq F^1X$.
Assume that $F^j\NP{X}$ has been defined for all $1\leq j\leq i$, and fix an $F^i\NP{X}$-normalization strategy $\down{(-)}{i}$.
We then construct $F^{i+1}\NP{X}$ as follows.

For every $\alpha \in \strat{i+1}X$ and context $C$ of $\pcat{X_2}$, there is an induced $3$-cell $C[\alpha]$ in $\plcat{X_3}$. 
When $\down{C[\partial \alpha]}{i} \neq 0$ this yields a $3$-cell
\[
\down{C[\alpha]}{i} : 
\down{C[s\alpha]}{i}
\to
\down{C[t\alpha]}{i}
\]
and, as defined in \SSS{SSS:MnomialOrders}, a $\preccurlyeq$-oriented left-term associated $3$-generator $\down{C[\alpha]}{i}^\preccurlyeq$ defined on $F^i\NP{X}$-normal forms.
We denote by $\dlr{\alpha}{i}$ the extension of $\plcat{X_2}$ consisting of such $3$-generators deduced from an $F^i\NP{X}$-active context on $\alpha$, that is,
\[
\dlr{\alpha}{i} 
\:\coloneq\:
\{\down{C[\alpha]}{i}^\preccurlyeq \: \mid \: C \in \Act(\alpha,F^i\NP{X})\}.
\]
The $3$-generators in $\dlr{\alpha}{i}$ are called the \emph{$F^i\NP{X}$-normal reductions of $\alpha$}.
Finally, we set 
\[
\strat{i+1}\NP{X} \: \coloneq \: \bigsqcup_{\alpha \in \strat{i+1}X} \minsub\dlr{\alpha}{i},
\]
where $\minsub$ is context relation,
so that $F^{i+1}\NP{X}=F^i\NP{X} \sqcup \strat{i+1}\NP{X}$.

\subsubsection{Remarks}
A $\preccurlyeq$-normalization of a stratified prepolygraph $F^\ast X$
need not be unique, since it depends on the normalization strategy
$\down{(-)}{i}$ chosen at each stratum $1\leq i\leq k$.
If $F^i\NP{X}$ is confluent, however, the strategy
$\down{(-)}{i}$ yields the unique $F^i\NP{X}$-normal form of each
$2$-cell. In this case, the extension $F^{i+1}\NP{X}$ obtained from
the above construction is uniquely determined.

In addition, the construction of the extension $F^{i+1}\NP{X}$  works as a completion of the critical branchings in $\CB(F^i\NP{X},\strat{i+1}X)$.
Indeed, for each non-confluent critical branching $(\sigma,\tau)\in \CB(F^i\NP{X},\strat{i+1}X)$, the construction yields a rule making this branching confluent.
More precisely, the $\strat{i+1}X$-rewriting step~$\tau$ can be written as $\tau=C[\alpha]$  for some $\alpha \in \strat{i+1}X$ yielding the following diagram: 
\[
\xymatrix 
@C=3em
@R=0.8em 
{& 
{\color{blue} f} 
  \ar@{->>}@[blue][r]_{\hspace{-10pt}{\color{blue}F^i\NP{X}}}  
& 
{\color{blue}\down{C[s\alpha]}{i}}
  \ar@{||..||}@[blue][dd] ^-{{\color{blue}{\down{C[\alpha]\,}{i}}^\preccurlyeq}}
\\
{\color{red}C[s\alpha]}
\ar@[red]@/^2ex/[ur] ^-{{\color{red}\sigma}} _-{{\color{red}F^i\NP{X}}} 
\ar@[red]@/_2ex/[dr] _-{{\color{red}\tau}} ^-{{\color{red}\strat{i+1}X}} &
\\
& 
{\color{blue}C[t\alpha]} 
\ar@{->>}@[blue][r]_{\hspace{-10pt}{\color{blue}F^i\NP{X}}}  & {\color{blue}\down{C[t\alpha]}{i}}
}
\]
Note that $\down{f}{i}=\down{C[s\alpha]}{i}$ since $F^i\NP{X}$ is confluent.
In this case, the normalization procedure adds a $F^i\NP{X}$-normal reduction 
$\down{C[\alpha]}{i}^\preccurlyeq$ to $\strat{i+1}\NP{X}$, connecting $\down{C[s\alpha]}{i}$  and $\down{C[t\alpha]}{i}$, thereby making this critical branching confluent.
In particular, all critical branching in $\CB(F^i\NP{X},\strat{i+1}X)$ is confluent if and only if $\strat{i+1}\NP{X}=\strat{i+1}X$.

\begin{lemma} 
\label{L:ConfluentOfS} 
For every $\strat{i+1}X$-rewriting step $\sigma = b C[\alpha] + 1_f$, with $i\geq 0$, if $C[s\alpha]$ is $F^{i}\NP{X}$-reducible and $\strat{i+1} \NP{X}$ is confluent, there exist two $\strat{i+1} \NP{X}$-rewriting paths $\tau$ and $\rho$ (possibly identities) as in the following diagram 
\begin{eqn}{equation} 
\label{E:MainLemma} 
\raisebox{1cm}{ 
\xymatrix @C=3.5em@R=0.8em { 
{\color{red}s\sigma} 
   \ar@{->>}@[red][r]_{\hspace{-10pt}{\color{red}F^{i} \NP{X}}} 
   \ar@[red][dd] ^{{\color{red}\sigma}}_{{\color{red}\strat{i+1}X}} 
&  
{\color{blue}\down{(s\sigma)}{i}} 
  \ar@[blue] @{->>} @/^2ex/ [dr]^{{\color{blue}\tau}}_{{\color{blue}\strat{i+1} \NP{X}}} & 
\\ 
&  & {\color{blue} e} 
\\ 
{\color{blue}t\sigma} 
	\ar@{->>}@[blue][r]_{\hspace{-10pt}{\color{blue}F^{i} \NP{X}}} 
& {\color{blue}\down{(t\sigma)}{i}}\ar@{->>}@/_2ex/@[blue][ur]^{{\color{blue}\rho}}_{{\color{blue}\strat{i+1} \NP{X}}} & 
}} 
\end{eqn} 
Moreover, the same statement holds for every $\strat{i+1} \NP{X}$-rewriting step $\sigma$. 
\end{lemma} 
\begin{proof} 
The proof relies on the $F^{i}\NP{X}$-normal reduction defined in \SSS{SSS:Normalization} to connect $\down{(s\sigma)}{i}$ and $\down{(t\sigma)}{i}$. 
 
Let~$\sigma = b C[\alpha] + 1_f$ be a $\strat{i+1}X$-rewriting step, where $\alpha\in \strat{i+1}X$ and $f\in\plcat{X_2}$. 
We set the context $C \coloneq \lambda(g \circ_1 u \square v \circ_1 h)$. 
If $\abs{C}=0$, then $C$ is $F^{i}\NP{X}$-active on $\alpha$, since $s\alpha$ is reducible by hypothesis.
Otherwise, if $\lambda g\notin\Nf(F^{i}\NP{X})$ or $\lambda  h\notin\Nf(F^{i}\NP{X})$,  
we have the $F^{i}\NP{X}$-normal reduction 
$\down{C[\alpha]}{i}^\preccurlyeq=\down{C'[\alpha]}{i}^\preccurlyeq$, where $C'=\down{\lambda g}{i} \circ_1 u \square v \circ_1 \down{h}{i}$ or $C'=\down{g}{i} \circ_1 u \square v \circ_1 \down{\lambda h}{i}$.
Hence, we only consider the case of $F^{i}\NP{X}$-active contexts $C$ on $\alpha$. 
 Let $h=\lm_{\preccurlyeq}(\down{C[\partial \alpha]}{i})$ and $a=\lc_{\preccurlyeq}(\down{(C[\partial \alpha])}{i})$. 
We assume that $a$ is either $1$ or noninvertible; the case where 
$a$ is invertible is treated similarly. 
We distinguish three cases.

\medskip

\noindent \textbf{Case 1.} If $h\in \suppm(\down{C[s\alpha]}{i})\setminus\suppm(\down{C[t\alpha]}{i})$, then there exists a $F^{i}\NP{X}$-normal reduction 
\[ 
\down{C[\alpha]}{i}^\preccurlyeq: ah \fl \down{C[t\alpha]}{i} - \down{C[s\alpha]}{i}+a h 
\] 
in the extension $\dlr{\alpha}{i}$. 
If $\down{C[\alpha]}{i}^\preccurlyeq \in \minsub\dlr{\alpha}{i}$, we have $\down{C[\alpha]}{i}^\preccurlyeq\in \strat{i+1} \NP{X}$. 
Then there exists a $3$-cell $\sigma'$ of size $1$ in $\plcat{(\strat{i+1}\NP{X})}$ 
\[ 
\sigma'\coloneq b\down{C[\alpha]}{i}^\preccurlyeq+b\down{C[s\alpha]}{i}-b a h+\down{f}{i} 
\] 
with source $\down{(s\sigma)}{i}$ and target $\down{(t\sigma)}{i}$. 
By Lemma~\ref{L:DecopositionOfSize1Cell}, there exist identities or $\strat{i+1} \NP{X}$-rewriting paths 
$\tau$ and $\rho$ as in \eqref{E:MainLemma}. 
If $\down{C[\alpha]}{i}^\preccurlyeq \notin \minsub\dlr{\alpha}{i}$, then by definition of the $\sqsubseteq$-minimal set $\minsub\dlr{\alpha}{i}$, there exist a context $D$ and a rule $\beta \in \minsub\dlr{\alpha}{i}$ such that 
$ 
\down{C[\alpha]}{i}^\preccurlyeq = D[\beta]. 
$ 
It follows that $\sigma'$ can be written as 
\[ 
\sigma' = b D[\beta] + b\down{C[s\alpha]}{i} - b a h + \down{f}{i}, 
\] 
where $\beta \in \strat{i+1}\NP{X}$, so that $\sigma'$ is again a $3$-cell in $\plcat{(\strat{i+1}\NP{X})}$. 
Therefore, in following two cases, we may assume without loss of generality that 
$ \down{C[\alpha]}{i}^\preccurlyeq \in \minsub\dlr{\alpha}{i}$.
 
\medskip
 
\noindent \textbf{Case 2.} If $h\in \suppm(\down{C[t\alpha]}{i})\setminus\suppm(\down{C[s\alpha]}{i})$, then there exists a $F^{i}\NP{X}$-normal reduction 
\[ 
\down{C[\alpha]}{i}^\preccurlyeq: ah \fl \down{C[s\alpha]}{i}-\down{C[t\alpha]}{i}+a h 
\] 
in $\minsub\dlr{\alpha}{i}$. 
Then there exists a $3$-cell $\sigma'$ of size $1$ in $\plcat{(\strat{i+1}\NP{X})}$ 
\[ 
\sigma'\coloneq b\down{C[\alpha]}{i}^\preccurlyeq+b\down{C[t\alpha]}{i}-b a h+\down{f}{i} 
\] 
with source $\down{(t\sigma)}{i}$ and target $\down{(s\sigma)}{i}$. 
By Lemma~\ref{L:DecopositionOfSize1Cell}, there exist identities or $\strat{i+1} \NP{X}$-rewriting paths 
$\tau$ and $\rho$ as in \eqref{E:MainLemma}. 
 
 \medskip
 
\noindent \textbf{Case 3.}
If $h$ occurs in both
$\suppm(\down{C[s\alpha]}{i})$ and $\suppm(\down{C[t\alpha]}{i})$, we write
\[
\down{C[s\alpha]}{i}=a_1h+s_0,
\qquad\text{and}\qquad
\down{C[t\alpha]}{i}=a_2h+t_0,
\]
where $h\notin\suppm(s_0)\sqcup\suppm(t_0)$.
Note that $a=a_1-a_2$. There exists an
$F^{i}\NP{X}$-normal reduction
\[
\down{C[\alpha]}{i}^{\preccurlyeq}
   :ah\fl t_0-s_0
\]
in $\minsub\dlr{\alpha}{i}$.
For $j=1,2$, set
$
q_j\coloneq\quo{a}{a_j},
$
and
$
r_j\coloneq\rem{a}{a_j}.
$
Since $a_1-a_2=a$, we have $r_1,r_2\in\Rem(a)$, that is $r_1=r_2$. Following $a_1=q_1a+r_1$ and $a_2=q_2a+r_2$, we have
\[
a=(q_1-q_2)a+(r_1-r_2)=(q_1-q_2)a.
\]
Since $\bk$ is an integral domain and $a\neq0$,
we obtain $q_1-q_2=1$.
We define two $3$-cells of size $1$ in
$\plcat{(\strat{i+1}\NP{X})}$ by
\[
\sigma_1\coloneq bq_1\beta+br_1h+bs_0+\down{f}{i}
:\down{(s\sigma)}{i}\fl e_1
\qquad\text{and}\qquad
\sigma_2\coloneq bq_2\beta+br_2h+bt_0+\down{f}{i}
:\down{(t\sigma)}{i}\fl e_2,
\]
where
$
e_1=bq_1(t_0-s_0)+br_1h+bs_0+\down{f}{i},
$
and
$
e_2=bq_2(t_0-s_0)+br_2h+bt_0+\down{f}{i}.
$
Using $r_1=r_2$ and $q_1-q_2=1$, we have
$
e_1-e_2
=b(q_1-q_2)(t_0-s_0)+b(s_0-t_0)=0.
$
Thus, we obtain the diagram
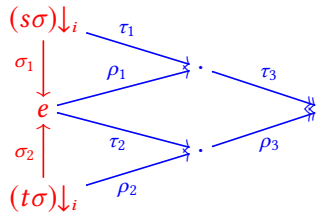
\[
\xymatrix @R=0.5em @C=3.5em {
{\color{red}\down{(s\sigma)}{i}}
  \ar@[red][dd]_{\color{red}\sigma_1}
  \ar@[blue][dr]^{\color{blue}\tau_1}
\\
& {\color{blue}\cdot}
  \ar@{->>}@[blue][dr]^{\color{blue}\tau_3}
\\
{\color{red}e}
  \ar@[blue][ur]^{\color{blue}\rho_1}
  \ar@[blue][dr]_{\color{blue}\tau_2}
&& {\color{blue}\cdot}
\\
& {\color{blue}\cdot}
  \ar@{->>}@[blue][ur]_{\color{blue}\rho_3}
\\
{\color{red}\down{(t\sigma)}{i}}
  \ar@[red][uu]^{\color{red}\sigma_2}
  \ar@[blue][ur]_{\color{blue}\rho_2}
}
\]
where $e\coloneq e_1=e_2$ and the identities or $\strat{i+1}\NP{X}$-rewriting paths 
$\tau_1,\rho_1,\tau_2$ and $\rho_2$ are obtained from 
Lemma~\ref{L:DecopositionOfSize1Cell}. Moreover, the paths 
$\tau_3$ and $\rho_3$ exist since the branching 
$(\tau_2,\rho_1)$ is confluent by hypothesis. Hence, this yields the identities or $\strat{i+1}\NP{X}$-rewriting paths 
$\tau=\tau_1\circ_1\tau_3$ and 
$\rho=\rho_1\circ_1\rho_3$, as required in~\eqref{E:MainLemma}. 

\medskip
 
Finally, if $\sigma$ be a $\strat{i+1} \NP{X}$-rewriting step, we write 
$\sigma = b D[\down{C[\alpha]}{i}^\preccurlyeq] + f$, 
where $\down{C[\alpha]}{i}^\preccurlyeq \in \strat{i+1}\NP{X}$ and $\alpha \in \strat{i+1}X$. 
We have $\down{(s\sigma)}{i}=b\down{(DC[s\alpha])}{i}+\down{f}{i}$ and $\down{(t\sigma)}{i}=b\down{(DC[t\alpha])}{i}+\down{f}{i}$. 
The conclusion then follows exactly as in the case of an $\strat{i+1}X$-rewriting step $\sigma =b (D C[\alpha] )+ f$. 
\end{proof} 

\begin{lemma} 
\label{L:Terminatingnormalization} 
If the monomial order $\preccurlyeq$ is $\strat{1}X$-compatible, then any $\preccurlyeq$-normalization $F^\ast \NP{X}$ is terminating and Tietze equivalent to $F^\ast X$. 
\end{lemma} 
\begin{proof} 
The stratified normalization procedure to define normal reductions in \SSS{SSS:Normalization} ensures that the monomial order $\preccurlyeq$ is $\strat{i+1}\NP{X}$-compatible for every $1\leq i\leq k-1$. Together with the assumption that $\preccurlyeq$ is $\strat{1}X$-compatible, it follows that $\preccurlyeq$ is $F^k\NP{X}$-compatible. 
 Hence, the extension $F^\ast \NP{X}$ is terminating. It remains to prove that $\plcat{\NP{X_2}}/\angle{\NP{X_3}}=\plcat{X_2}/\angle{X_3}$. 
 
For each $1\leq i\leq k-1$ and every rule $f \fl h$ in $\strat{i+1}X$, by Lemma~\ref{L:ConfluentOfS} there exists two $\strat{i+1} \NP{X}$-rewriting paths $\tau$ and $\rho$ such that 
\[ 
   f \xto*[F^{i}\NP{X}]{} \down{f}{i} \xto*[\strat{i+1} \NP{X}]{\tau} e \xot*[\strat{i+1} \NP{X}]{\rho}  \down{h}{i} \xot*[F^{i}\NP{X}]{}  h. 
\] 
Hence we obtain 
$ 
\angle{\strat{i+1}X} \subseteq \angle{F^{i+1}\NP{X}}, 
$ 
that is, 
$ 
\plcat{X_2}/\angle{F^{i+1}\NP{X}} \subseteq \plcat{X_2}/\angle{\strat{i+1}X}. 
$ 
When $i=0$, we have $ F^1 \NP{X}=\strat{1}X $ by the definition in \SSS{SSS:Normalization}. 
By induction on $i$, we have $\plcat{\NP{X_2}}/\angle{\NP{X_3}}\subseteq\plcat{X_2}/\angle{X_3}$. 
 
For each $1\leq i\leq k-1$ and every rule $f \fl h$ in $\strat{i+1} \NP{X}$, by the construction of normalization procedure there exist  
$\alpha \in \strat{i+1}X$, a $3$-cell $g$ in $\plcat{(F^{i} \NP{X})}$, and a scalar $\mu$ 
such that 
\[ 
\mu f + g = \down{C[s\alpha]}{i} 
\quad\text{and}\quad 
\mu h + g = \down{C[t\alpha]}{i}, 
\] 
that is  
\[  
\mu f+g =\down{C[s\alpha]}{i} \xot*[F^{i} \NP{X}]{} C[s\alpha] \xto[\strat{i+1}X]{C[\alpha]} C[t\alpha] \xto*[F^{i} \NP{X}]{} \down{C[t\alpha]}{i}= \mu h+g.  
\] 
By indection, we have 
\[ 
\angle{\strat{i+1} \NP{X}} 
\subseteq 
\angle{\strat{i+1}X \sqcup F^{i} \NP{X}} 
= 
\angle{\strat{i+1}X \sqcup \strat{i} \NP{X} \sqcup F^{i-1} \NP{X}} 
\subseteq 
\angle{\strat{i+1}X\sqcup \strat{i} X \sqcup F^{i-1} \NP{X}} 
\subseteq \cdots . 
\] 
Since $F^1 \NP{X} = F^1 X$, it follows that 
$ 
\angle{\strat{i+1} \NP{X}} \subseteq \angle{F^{i+1} X}, 
$ 
that is, 
$ 
\plcat{X_2}/\angle{F^{i+1} X} \subseteq \plcat{X_2}/\angle{\strat{i+1}\NP{X}}. 
$ 
By induction on $i$, we have $\plcat{X_2}/\angle{X_3}\subseteq\plcat{\NP{X_2}}/\angle{\NP{X_3}}$. 
\end{proof}

The main structural property of stratified normalization is that it makes confluence modular with respect to the stratification as stated by the following lemma.
\begin{lemma} 
\label{L:StratumLemma} 
Any $\preccurlyeq$-normalization $F^\ast \NP{X}$ satisfies the confluence modularity property. 
\end{lemma} 
\begin{proof} 
For $j=1$, the result follows from the hypothesis, since $F^1\NP{X}=\strat{1}\NP{X}$. 
Assume now that $1\leq j < k-1$ and that $F^j\NP{X}$ is confluent. We prove that $F^{j+1}\NP{X}$ is confluent. 
For every critical branching $(\sigma,\tau)\in\CB(F^{j+1} \NP{X})$, 
we distinguish two cases: 
\begin{enumerate} 
\item  If $(\sigma,\tau)\in\CB(\strat{j+1} \NP{X})$, or $(\sigma,\tau)\in\CB(F^j \NP{X})$, then it is confluent by hypothesis. 
 
\item If $(\sigma,\tau)\in\CB(\strat{j+1} \NP{X},F^j \NP{X})$ with common source $f$, 
we have the following diagram: 
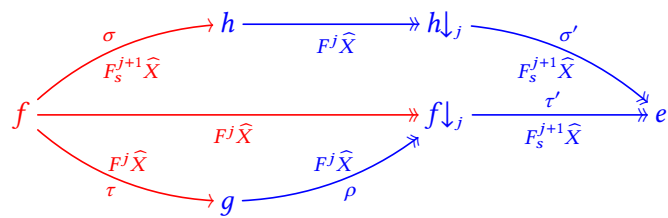
\[ 
\xymatrix @R=1.6em @C=6em { 
& {\color{blue}h} 
  \ar@{->>}@[blue][r]_ 
    {\color{blue}F^j\NP{X}} 
& {\color{blue}\down{h}{j}} 
  \ar@/^0.7pc/@{->>}@[blue][dr] 
    ^{\color{blue}\sigma'}_ 
    {\color{blue}\strat{j+1}\NP{X}} 
& 
\\ 
{\color{red}f} 
  \ar@/^0.7pc/@{->}@[red][ur] 
    ^{\color{red}\sigma}_ 
    {\color{red}\strat{j+1}\NP{X}} 
  \ar@/_0.7pc/@{->}@[red][dr] 
    _{\color{red}\tau}^ 
    {\color{red}F^j\NP{X}} 
  \ar@{->>}@[red][rr]_ 
    {\color{red}F^j\NP{X}} 
&& {\color{blue}\down{f}{j}} 
  \ar@{->>}@[blue][r] 
    ^{\color{blue}\tau'}_ 
    {\color{blue}\strat{j+1}\NP{X}} 
& {\color{blue}e} 
\\ 
& {\color{blue}g} 
  \ar@/_0.7pc/@{->>}@[blue][ur] 
    _{\color{blue}\rho}^ 
    {\color{blue}F^j\NP{X}} 
&& 
} 
\] 
where $\sigma'$ and $\tau'$ are obtained from Lemma~\ref{L:ConfluentOfS}, since $\sigma$ is an $\strat{j+1} \NP{X}$-rewriting step, and $\rho$ is obtained from the equality $\down{f}{j}=\down{g}{j}$, since $F^j \NP{X}$ is confluent. 
\end{enumerate} 
By Lemma~\ref{L:Terminatingnormalization}, the extension $\NP{X}$ is terminating, and so is $F^{j+1} \NP{X}$. 
Hence $F^{j+1} \NP{X}$ is locally confluent by Lemma~\ref{L:CBL}. 
\end{proof}

\begin{theorem}
\label{T:MainTheorem}
Let $\Cr$ be a linear monoidal category presented by a stratified $3$-prepolygraph $F^\ast X$.
If there exists an $\strat{1}X$-compatible monomial order $\preccurlyeq$ such that every stratum $\strat{i}\NP{X}$ of the associated $\preccurlyeq$-normalization $F^\ast\NP{X}$ is left-monomial and confluent, then $\Nf_m(\NP{X})$ forms a hom-basis of $\Cr$.
\end{theorem}
\begin{proof}
Let $F^\ast X$ be a left-monomial $k$-stratified  $3$-prepolygraph for $k\geq 1$ satisfying conditions of the statement.
By Lemma~\ref{L:Terminatingnormalization}, the $3$-prepolygraph $F^\ast \NP{X}$ is terminating and presents the category $\Cr$. 
By Lemma~\ref{L:StratumLemma}, it satisfies the confluence modularity property.
Since each stratum $\strat{i}\NP{X}$ is confluent and left-monomial, Theorem~\ref{T:HomBasesFromConfluence} applies and yields that $\Nf_m(\NP{X})$ forms a hom-basis of $\Cr$.
\end{proof}

\subsubsection{Generic branchings}
Let $F^\ast X$ be a stratified $3$-prepolygraph.
For $2\leq i\leq k$, let $u,v\in\pcat{X_1}$ and $C$ be a context.
For every $1\leq j<i$, the context $C$ is called \emph{$F^j\NP{X}$-stable with $(u,v)$} if, for every $\strat{i}\NP{X}$-rewriting step $\rho$ whose source and target belong to $\plcat{X_2}(u,v)$, there exists an $\strat{i}\NP{X}$-rewriting step
\[
\rho^{C,j} :
\down{C[s\rho]}{j}
\fl
\down{C[t\rho]}{j}.
\]
If moreover
$
\rho \not\sqsubseteq \rho^{C,j},
$
with respect to the context relation $\sqsubseteq$, then $C$ is called \emph{strictly $F^j\NP{X}$-stable with~$\rho$}.

Let $(\sigma,\tau)$ be an $\strat{i}\NP{X}$-branching with source $f$. Any $F^j\NP{X}$-stable context $C$ with $f^\partial$ induces a branching
$
(\sigma^{C,j},\tau^{C,j})
$
with source $\down{C[f]}{j}$. This yields a well-founded relation $\vdash$ on the set of $\strat{i}\NP{X}$-branchings, defined by
\[
(\sigma,\tau)
\vdash
(\sigma^{C,j},\tau^{C,j}).
\]
A branching that is minimal with respect to $\vdash$ is called a \emph{generic $\strat{i}\NP{X}$-branching}.

Suppose that the branching $(\sigma_1,\tau_1)$ is confluent and admits a confluence diagram of the form
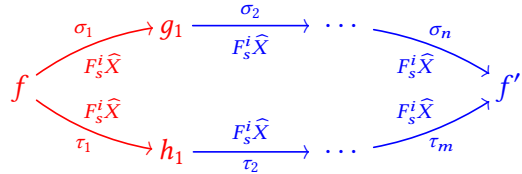
\[
\xymatrix @R=0.75em @C=1.7em {
&& {\color{red}g_1}
\ar@{->}@[blue][rr] ^-{{\color{blue}\sigma_2}}_-{{\color{blue}\strat{i}\NP{X}}}
&& \;{\color{blue} \cdots}\;
\ar@/^0.4pc/@{->}@[blue] [drr] ^-{{\color{blue}\sigma_n}}_{\hspace{-5pt}{\color{blue}\strat{i}\NP{X}}}
\\
{\color{red} f}
\ar@/^0.4pc/@{->}@[red][urr]^{{\color{red} \sigma_1}}_-{{\color{red}\strat{i}\NP{X}}}
\ar@/_0.4pc/@{->}@[red][drr]
  ^{{\color{red}\strat{i}\NP{X}}}_-{{\color{red}\tau_1}}
&&&
&&& {\color{blue}f'}
\\
&& {\color{red}h_1}
\ar@{->}@[blue][rr]
  ^-{{\color{blue}\strat{i}\NP{X}}}_-{{\color{blue}\tau_2}}
&& \;{\color{blue} \cdots}\;
\ar@/_0.4pc/@{->}@[blue][urr]
  ^{\hspace{-5pt}{\color{blue}\strat{i}\NP{X}}}_-{{\color{blue}\tau_m}}
}
\]
Then the induced branching $(\sigma_1^{C,j},\tau_1^{C,j})$ is also confluent, as witnessed by the following diagram:
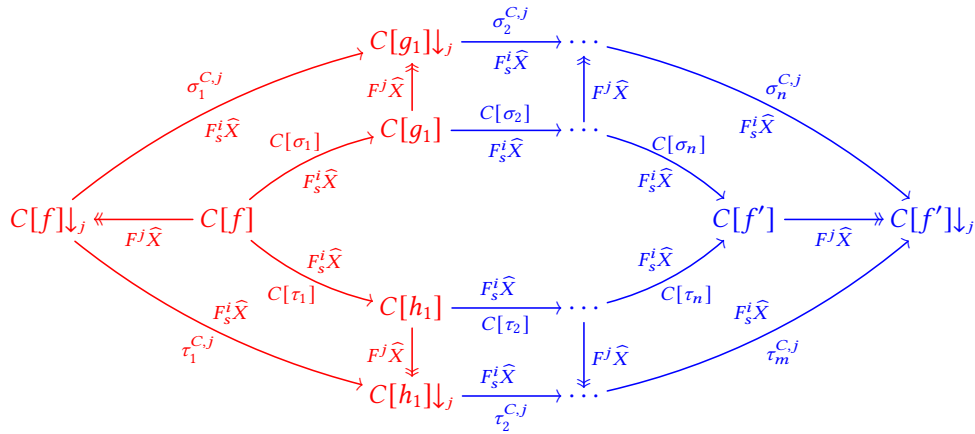
\[
\scalebox{0.95}{$
\xymatrix @R=1.7em @C=1.5em {
&&&&{\color{red}\down{C[g_1]}{j}}
\ar@[blue] [rr] ^-{\color{blue}\sigma_2^{C,j}} _-{{\color{blue}\strat{i}\NP{X}}}
&& {\color{blue}\cdots}
\ar@/^0.9pc/@[blue] [ddrrrr] ^{\color{blue}\sigma_n^{C,j}} _{\hspace{-1pt}{\color{blue}\strat{i}\NP{X}}}
\\
&&&& {\color{red}C[g_1]}
\ar@[red]@{->>}[u]^-{{\color{red}F^j\NP{X}}}
\ar@[blue] [rr] ^-{\color{blue}C[\sigma_2]} _-{{\color{blue}\strat{i}\NP{X}}}
&& {\color{blue}\cdots}
\ar@[blue]@{->>}[u]_-{{\color{blue}F^j\NP{X}}}
\ar@/^/@[blue] [drr] ^{\hspace{-6pt}\color{blue}C[\sigma_n]} _-{{\color{blue}\strat{i}\NP{X}}}
\\
{\color{red}\down{C[f]}{j}} 
\ar@/^0.9pc/@[red][uurrrr]^{\color{red}\sigma_1^{C,j}} _-{{\color{red}\strat{i}\NP{X}}}
\ar@/_0.9pc/@[red] [ddrrrr] ^{{\color{red}\strat{i}\NP{X}}} _-{\color{red}\tau_1^{C,j}}
&&{\color{red}C[f]} 
\ar@[red]@{->>}[ll]^-{{\color{red}F^j\NP{X}}}
\ar@/^0.6pc/@[red][urr]^{\color{red}C[\sigma_1]} _-{{\color{red}\strat{i}\NP{X}}}
\ar@/_0.6pc/@[red] [drr] ^{{\color{red}\strat{i}\NP{X}}} _-{\color{red}C[\tau_1]}
&&& 
&&& {\color{blue}C[f']}
\ar@[blue]@{->>}[rr]_-{{\color{blue}F^j\NP{X}}}
&&{\color{blue}\down{C[f']}{j}}
\\
&&&&{\color{red}C[h_1]}
\ar@[red]@{->>}[d]_-{{\color{red}F^j\NP{X}}}
\ar@[blue] [rr] ^{{\color{blue}\strat{i}\NP{X}}} _-{\color{blue}C[\tau_2]}
&& {\color{blue}\cdots}
\ar@[blue]@{->>}[d]^-{{\color{blue}F^j\NP{X}}}
\ar@/_/@[blue] [urr] ^-{{\color{blue}\strat{i}\NP{X}}} _-{\color{blue}C[\tau_n]}
\\
&&&& {\color{red}\down{C[h_1]}{j}}
\ar@[blue] [rr] ^{{\color{blue}\strat{i}\NP{X}}} _-{\color{blue}\tau_2^{C,j}}
&& {\color{blue}\cdots}
\ar@/_0.9pc/@[blue] [uurrrr] ^{\hspace{-10pt}{\color{blue}\strat{i}\NP{X}}} _{\color{blue}\tau_m^{C,j}}
}
$}
\]
We therefore obtain the following criterion.

\begin{lemma}
All $\strat{i}\NP{X}$-critical branchings are confluent if and only if all generic $\strat{i}\NP{X}$-critical branchings are confluent.
\end{lemma}

\subsection{Hom-basis computation procedure}
\label{SS:HomBasisComputationProcedure}

This subsection presents a procedure for computing a hom-basis of a linear monoidal category from a stratified presentation and the normalization-completion procedure.
We also introduce a stratified method for active contexts ensuring the termination of these procedures.

\subsubsection{Normalization-completion procedure}
\label{SSS:NormalizationCompletionProcedure}

We summarize the constructions of the normalization of stratified prepolygraphs in \SSS{SSS:Normalization} into the following \emph{stratified normalization procedure}~$\mathtt{SN}(F^\ast X,\preccurlyeq~)$.

\begin{algorithm}[H]
\DontPrintSemicolon
\SetAlgoVlined
\SetKwFor{ForEach}{for each}{do}{}
\SetKwFor{While}{while}{do}{}

\KwIn{A $k$-stratified $3$-prepolygraph $F^\ast X$,
an $\strat{1}X$-compatible monomial order $\preccurlyeq$.}

\KwOut{The normalization $F^\ast\NP{X}$.}

$F^1\NP{X}\leftarrow F^1X$, $i\leftarrow 1$\;

\While{$i+1\leq k$}{

  $T\leftarrow\varnothing$\;

  \ForEach{$\alpha\in\strat{i+1}X$}{

    \ForEach{
      $C\in\Act(\alpha,F^iX)$
    }{

      $d\leftarrow
      \Nf_{F^i\NP{X}}
      \bigl(C[s\alpha]-C[t\alpha]\bigr)$\;

      $w\leftarrow\lm_{\preccurlyeq}(d)$,\quad
      $\lambda\leftarrow\lc_{\preccurlyeq}(d)$\;

     \lIf{$\lambda$ is invertible}{$\mu\leftarrow\lambda^{-1}$}
\lElse{$\mu\leftarrow 1$}

      \eIf{
        $w\in
        \suppm\bigl(\Nf_{F^i\NP{X}}(C[t\alpha])\bigr)
        \setminus
        \suppm\bigl(\Nf_{F^i\NP{X}}(C[s\alpha])\bigr)$
      }{
        Add the rule
        $\mu\lambda w\to\mu\bigl(d+\lambda w\bigr)$
        to $T$\;
      }{
        Add the rule
        $\mu\lambda w\to\mu\bigl(-d+\lambda w\bigr)$
        to $T$\;
      }

    }

  }

 $\strat{i+1} \NP{X} \leftarrow T \setminus
  \{\,\beta \in T \mid \exists\,\gamma \in T \text{ and a context $C$ such that $\abs{C}>0$ and } C[\gamma]=\beta\,\}$\;
  $i \leftarrow i + 1$\;

}

$\strat{\ast}\NP{X}
\leftarrow
\strat{1}\NP{X}\cdots\strat{k}\NP{X}$\;

\Return{$\strat{\ast}\NP{X}$}\;

\end{algorithm}
If the set of active contexts is finite, the procedure terminates and yields the normalization. In general, however, this set is infinite. We restrict the set of normal reductions by filtering active contexts according to their size, as explained in~\ref{SSS:FiltrationNormalReductions}.

\subsubsection{Filtrations of normal reductions}
\label{SSS:FiltrationNormalReductions}
Let $F^\ast X$ be a $k$-stratified $3$-prepolygraph and $\preccurlyeq$ an $\strat{1}{X}$-compatible monomial order.
For $i\leq k$ and $\alpha\in \strat{i+1}X$, we define a filtration of the extension $\dlr{\alpha}{i}$ by the size of active contexts:
\[
L^\ast\dlr{\alpha}{i}\;:\;
L^0\dlr{\alpha}{i}\;\subseteq\; L^1\dlr{\alpha}{i}\;\subseteq\; \cdots \;\subseteq\;
L^j\dlr{\alpha}{i}\;\subseteq\; \cdots,
\]
where, for every $j\ge 0$,
\[
L^j\dlr{\alpha}{i}
\;\coloneq\;
\bigl\{\down{C[\alpha]}{i}\ \bigm|\ C\in \Act(\alpha,L^i\NP{X})\text{ and } \abs{C}\leq j\bigr\}.
\]
We define $L^j\minsub\dlr{\alpha}{i}$ as the set of minimal elements of $L^j\dlr{\alpha}{i}$ with respect to the context relation~$\sqsubseteq$. 
This defines a filtration $L^\ast\minsub\dlr{\alpha}{i}$ of the extension $\minsub\dlr{\alpha}{i}$.
We say that $\minsub\dlr{\alpha}{i}$ is \emph{bounded} if there exists $n\ge 0$ such that
\[
\minsub\dlr{\alpha}{i}=L^n\minsub\dlr{\alpha}{i}.
\]

\begin{proposition}
\label{P:FiltrationOfContexts}
Let $\alpha\in\strat{i+1}X$, and assume that $F^i\NP{X}$ is
left-monomial and that
$
\abs{s\beta}=\abs{t\beta}
$
for every $\beta\in F^i\NP{X}$.
If there exists a natural number $n$ such that
\[
n \ge \max_{\beta\in F^i\NP{X}} \{\abs{s\beta} - \abs{\alpha}\}
\quad \text{and} \quad
L^n\minsub\dlr{\alpha}{i} = L^{n+1}\minsub\dlr{\alpha}{i}\,,
\]
then the extension $\minsub\dlr{\alpha}{i}$ is bounded.
\end{proposition}
\begin{proof}
For every context $C\in \Act(\alpha,F^i \NP{X})$ of size $(n+2)$, $C[\alpha]$ falls two cases:
\[
f\circ_1C'[\alpha] \quad \text{ and } \quad C'[\alpha]\circ_1g,
\]
where $\abs{f}=\abs{g}=1$ and $\abs{C'}=n+1$.
We consider the first case. The other case is treated similarly.

 Since $\abs{s\beta}=\abs{t\beta}$ for every $\beta\in F^i\NP{X}$, the source and target of $\down{C'[\alpha]}{i}^\preccurlyeq$ have the same size as those of $C'[\alpha]$. By the assumption
$
L^n\minsub\dlr{\alpha}{i}=L^{n+1}\minsub\dlr{\alpha}{i},
$
it follows that $f\circ_1\down{C'[\alpha]}{i}^\preccurlyeq$ is of the following two forms:
\[
f\circ_1 h_1\circ_1\down{D[\alpha]}{i}^\preccurlyeq, \quad \text{ and } \quad
f\circ_1\down{D[\alpha]}{i}^\preccurlyeq\circ_1 h_2,
\]
where $\abs{h_1}=\abs{h_2}=1$ and $\abs{D}=n$.
Note that the $2$-cells $f\circ_1 h_1$ and $f\circ_1 s\down{D[\alpha]}{i}^\preccurlyeq$ may be reducible. To compute the normal reduction $\down{C[\alpha]}{i}$, one needs to further reduce their sources and targets to normal forms. Since the monomial order $\preccurlyeq$ is $X$-compatible, this process will terminate. Consequently, $\down{C[\alpha]}{i}^\preccurlyeq$ is of one of the following three forms:
\[
f'\circ_1g'\circ_1\down{D'[\alpha]}{i}^\preccurlyeq, \qquad
\down{D'[\alpha]}{i}^\preccurlyeq\circ_1f'\circ_1g', \quad \text{ and } \quad
f'\circ_1\down{D'[\alpha]}{i}^\preccurlyeq\circ_1 g',
\]
where $\abs{f'}=\abs{g'}=1$, $\abs{D'}=n$, and $\down{D'[\alpha]}{i}^\preccurlyeq\in L^n\minsub\dlr{\alpha}{i}$.
 Hence, by taking minimal rules with respect to $\sqsubseteq$, we obtain
$
L^{n+2}\minsub\dlr{\alpha}{i}=L^n\minsub\dlr{\alpha}{i}.
$
By induction on the size of active contexts, one obtains
\[
L^j\minsub\dlr{\alpha}{i}=L^n\minsub\dlr{\alpha}{i}, \quad \text{for all } j\ge n.
\]
Therefore, we have
$
\minsub\dlr{\alpha}{i}=L^n\minsub\dlr{\alpha}{i}.
$
\end{proof}

\subsubsection{Hom-basis procedure}
\label{SSS:HomBasisProcedure}

We denote by $\mathtt{KB}(Y,\preccurlyeq)$ the $3$-prepolygraph obtained by applying the Knuth-Bendix completion procedure~\cite{KnuthBendix70} to a $3$-prepolygraph $Y$ with respect to $\preccurlyeq$.
For every equational presentation $\Er$ of a linear monoidal category $\Cr$, we compute its hom-basis using the following \emph{hom-basis procedure} $\mathtt{HB}(\Er,\preccurlyeq)$:

\begin{algorithm}[H]
\DontPrintSemicolon
\SetAlgoVlined
\SetKwFor{ForEach}{for each}{do}{}
\SetKwFor{While}{while}{do}{}
\KwIn{
A presentation $\Er$ of a linear monoidal category $\Cr$, a monomial order $\preccurlyeq$ on~$\Cr$
}
\KwOut{A hom-basis of $\Cr$.}

$X_{\leq 2} \leftarrow \Er_{\leq 2}$;\quad
$X_3 \leftarrow \varnothing$;\;
\ForEach{$(u,v)$ in $\Er_3$}{
  $\beta \leftarrow \{u\fl v\}$ ;\\
$X_3 \leftarrow X_3 \sqcup \{\beta^\preccurlyeq\}$;
}

Fix a stratification $\strat{\ast}X$ of length $k$ of $X$\,

$\strat{\ast} \NP{X} \leftarrow \mathtt{SN}(\strat{\ast} X,\preccurlyeq)$\;

\While{$\strat{\ast} \NP{X}\text{ is not confluent}$}{
$I\leftarrow\{j\in\{1,\ldots,k\}\mid \strat{j}\NP{X} \text{ is not confluent}\}$\\
\ForEach{$j\in I$}{
$\strat{j}\NP{X} \leftarrow \mathtt{KB}(\strat{j}\NP{X},\preccurlyeq)$\;
}
$\strat{\ast}\NP{X} \leftarrow \mathtt{SN}(\strat{\ast}\NP{X},\preccurlyeq)$\;
}
\Return{$\Nf(\NP{X})$}\;
\end{algorithm}
If the stratum $\strat{j} \NP{X}$ is confluent for every $1 \leq j \leq k$, then the procedure stops.
As a consequence of Theorem~\ref{T:MainTheorem}, we obtain the following result.
\begin{proposition}
If the procedure $\mathtt{HB}(\Er,\preccurlyeq)$ terminates and its output
$\NP{X}$ is left-monomial, then it computes a hom-basis of the
category $\Cr$.
\end{proposition}

\subsection{Examples of the hom-basis procedure}
In this subsection, we illustrate the hom-basis computation procedure on two examples. The first computes a hom-basis of the associative-coassociative category, which provides the starting point for the $q$-Schur category in the next section. The second applies the procedure to the Hecke category, yielding bases for the Iwahori--Hecke algebras.

\begin{example}
\label{E:AssociativeCoassociativeCategory}
The \emph{associative-coassociative category}, denoted by $\ACCat$, is the linear monoidal category generated by $1$-cells $(r)$, for $r \in \mathbb{N}^+$, together with $2$-cells, called \emph{merges} and \emph{splits}, given by
\[
\twocell{ajbs *1 merge*1(ax *0 bx)}: (a)\circ_0 (b) \fl (a+b)\ , \qquad 
\twocell{(as *0 bs)*1 split *1 ajbx}: (a+b)\fl (a)\circ_0 (b)\ ,
\]
satisfying the following  \emph{associativity} and \emph{coassociativity relations}
\begin{eqn}{equation}
\label{E:Associativityrelations}
\twocell{merge*1((merge*1(ax *0 bx)) *0 cx)} = \twocell{merge*1 (ax *0 (merge*1(bx *0 cx)))}
\quad\text{and}\quad  
\twocell{( as *0 ((bs*0 cs)*1 split))*1 split} = \twocell{(((as*0 bs)*1 split)*0 cs)*1 split},
\end{eqn}
for all $a,b,c>0$.
Let $F^\ast \ACPol$ be the following $2$-stratified $3$-prepolygraph presenting $\ACCat$:
\begin{enumerate}
\item $\ACPol_1$ has the $1$-generators $(r)$, for each  $r\in \mathbb{N}^+$, and $\ACPol_2$ has the following 2-generators, for all $a,b>0$,
\[
\twocell{ajbs *1 merge*1(ax *0 bx)}: (a)\circ_0(b) \fl (a+b)\ , \qquad 
\twocell{(as *0 bs)*1 split *1 ajbx}: (a+b)\fl (a)\circ_0(b)\ ,
\]
\item  $\strat{1} \ACPol=\Int(\ACPol_2)$ as defined in \eqref{E:InterchangeRules} and $\strat{2} \ACPol$ consists of the $3$-generators
\[
\alpha_{a,b,c}: 
\twocell{merge*1((merge*1(ax *0 bx)) *0 cx)}
\fl 
\twocell{merge*1 (ax *0 (merge*1(bx *0 cx)))}, 
\qquad  
\beta_{a,b,c}:
\twocell{( as *0 ((bs*0 cs)*1 split))*1 split}
\fl 
\twocell{(((as*0 bs)*1 split)*0 cs)*1 split},
\]
for all $a,b,c> 0$.
\end{enumerate}
These rules will be written simply as $\alpha$ and $\beta$, omitting subscripts whenever no confusion arises.
 We now compute its normalization with respect to the monomial order $\preccurlyeq_\Sch$ defined in \SSS{SSS:PartialOrderQSchurPolygraph}.

We set $\strat{1}\NP{\ACPol} \coloneq \strat{1}\ACPol$. 
Since $s\alpha\in \Nf(\strat{1}\NP{\ACPol})$, we have $F^0\minsub\dlr{\alpha}{1}=\{\alpha\}$.
For every $\strat{1}\NP{\ACPol}$-active context $C$ of size $1$, we have
$C[s\alpha] \in \Nf(\strat{1}\NP{\ACPol})$, except in the following cases:
\[
C_1[\alpha]:\twocell{(id*0merge)*1(id*0merge*0id)*1(merge*0id*0id*0id)}\fl\twocell{(id*0merge)*1(id*0id*0merge)*1(merge*0id*0id*0id)}
\qquad\leadsto\qquad
\down{C_1[\alpha]}{1}^\preccurlyeq:
\twocell{(merge*0id)*1(id*0id*0merge)*1(id*0id*0merge*0id)}\fl\twocell{(merge*0id)*1(id*0id*0merge)*1(id*0id*0id*0merge)}\,,
\]

\[
C_2[\alpha]:
\twocell{(id*0id*0merge)*1(id*0id*0merge*0id)*1(split*0id*0id*0id)}\fl\twocell{(id*0id*0merge)*1(id*0id*0id*0merge)*1(split*0id*0id*0id)}
\qquad\leadsto\qquad
\down{C_2[\alpha]}{1}^\preccurlyeq:
\twocell{(split*0id)*1(id*0merge)*1(id*0merge*0id)}\fl\twocell{(split*0id)*1(id*0merge)*1(id*0id*0merge)}\,,
\]

\[
C_3[\alpha]:
\twocell{(id*0merge)*1(merge*0id*0id)*1(merge*0id*0id*0id)}\fl\twocell{(id*0merge)*1(merge*0id*0id)*1(id*0merge*0id*0id)}
\qquad\leadsto\qquad
\down{C_3[\alpha]}{1}^\preccurlyeq:
\twocell{(merge*0id)*1(merge*0id*0id)*1(id*0id*0id*0merge)}\fl\twocell{(merge*0id)*1(id*0merge*0id)*1(id*0id*0id*0merge)}\,,
\]

\[
C_4[\alpha]:
\twocell{(id*0split)*1(merge*0id)*1(merge*0id*0id)}\fl\twocell{(id*0split)*1(merge*0id)*1(id*0merge*0id)}
\qquad\leadsto\qquad
\down{C_4[\alpha]}{1}^\preccurlyeq:
\twocell{(merge*0id*0id)*1(merge*0id*0id*0id)*1(id*0id*0id*0split)}\fl\twocell{(merge*0id*0id)*1(id*0merge*0id*0id)*1(id*0id*0id*0split)}\,.
\]
Hence,
$L^1\minsub\dlr{\alpha}{1} = \{\alpha\}
= L^0\minsub\dlr{\alpha}{1}$.
Since $\abs{\alpha}=\abs{\beta}=2$,
we obtain
$\minsub\dlr{\alpha}{1} = \{\alpha\}$ by Proposition~\ref{P:FiltrationOfContexts}.
Similarly,
$\minsub\dlr{\beta}{1} = \{\beta\}$.
Therefore, we have $\strat{2}\NP{\ACPol} = \strat{2}\ACPol$.
It is straightforward to verify that $\strat{2}\NP{\ACPol}$ and $\strat{1}\NP{\ACPol}$ are both confluent.
By Procedure~\ref{SSS:HomBasisProcedure}, we obtain a hom-basis $\Nf_m(\NP{\ACPol})$ of the category $\ACCat$, in which each $2$-cell admits a unique decomposition
$f_1\circ_1 f_2\circ_1\ldots \circ_1 f_n$, where each $f_i$ is a $2$-cell of $\pcat{\ACPol_2}$ of size~$1$, and each composite $f_i\circ_1 f_{i+1}$ is  one of the following forms:
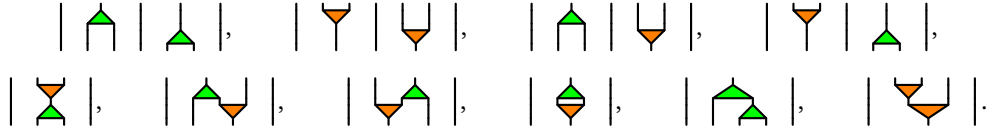
\[
\twocell{(id*0 merge*0 id*0id *0id)*1(id*0id*0id*0 id*0merge*0id)}\,,
\qquad
\twocell{(id*0 split*0 id*0 id*0id *0id)*1(id*0id*0 id*0split*0id)}\,,
\qquad
\twocell{(id*0 merge*0 id*0 id*0id *0id)*1(id*0id*0 id*0 id*0split*0id)}\,,
\qquad
\twocell{(id*0 split*0 id*0id *0id)*1(id*0 id*0 id*0merge*0id)}\,,
\]
\[
\twocell{(id*0split*0id)*1 (id*0merge*0id)}\,,
\qquad
\twocell{(id*0merge*0 id*0id)*1(id*0id*0 split*0id)}\,,
\qquad
\twocell{(id*0id*0 merge*0id)*1( id*0split*0 id*0id)}\,,
\qquad
\twocell{(id*0merge*0id)*1( id*0split*0id)}\,,
\qquad
\twocell{(id*0merge*0id)*1(id*0id*0merge*0id)}\,,
\qquad
\twocell{(id*0split*0id*0id)*1(id*0split*0id)}\,.
\]

Since the $q$-Schur category is obtained from the category~$\ACCat$ by imposing additional relations, in the Section~\ref{S:NormFormHomBasisqSchurCategories}, we will extend the stratified prepolygraph $F^\ast \ACPol$ by adding further strata corresponding to the $q$-Schur relations, and apply Procedure~\ref{SSS:HomBasisProcedure} to compute a hom-basis of the $q$-Schur category.
\end{example}

\begin{example}
\label{E:HeckeAlgebras}
We set
$\bk=\mathbb{Z}[q,q^{-1}]$ and $z=q-q^{-1}$.
For $n\geq 1$, the type~$A_{n-1}$ \emph{Iwahori--Hecke algebra} $H_n(q)$ is the $\bk$-algebra generated by $T_1,\ldots,T_{n-1}$, subject to the relations
\[
T_i^2=1+zT_i,
\qquad
T_iT_{i+1}T_i=T_{i+1}T_iT_{i+1},
\qquad
T_iT_j=T_jT_i
\quad\text{if } |i-j|>1.
\]
Equivalently, the quadratic relation can be written as
$T_i-T_i^{-1}=z$.
The \emph{Hecke category}, denoted by $\HCat$, is the linear monoidal categories, generated by one $1$-generator~$1$ and two $2$-generators
\[
\twocell{braidr}\,,\ \twocell{braidl}:
1\circ_0 1
\longrightarrow
1 \circ_0 1,
\]
corresponding to $T_i$ and $T_i^{-1}$ respectively, and subject to the \emph{skein}, \emph{inverse}, and \emph{braid relations}
\[
\twocell{braidl}
=
\twocell{braidr}-z\,\twocell{id*0id}\,,
\qquad
\twocell{braidl*1braidr}
=
\twocell{id*0id}\,,
\qquad
\twocell{braidr*1braidl}
=
\twocell{id*0id}\,,
\qquad
\twocell{(braidr*0id)*1(id*0braidr)*1(braidr*0id)}
=
\twocell{(id*0braidr)*1(braidr*0id)*1(id*0braidr)}\,.
\]
For every $n\geq 1$, the algebra
$\End_{\HCat}(1^{\circ_0 n})$
is isomorphic to $H_n(q)$, where $T_i$ corresponds to the crossing of the $i$-th and $(i+1)$-st strands.
The distant braid relations follow from the exchange law.

Let $F^\ast\HPol$ be the following $4$-stratified
$3$-prepolygraph presenting the $2$-category~$\HCat$:

\begin{enumerate}

\item $\HPol_{\leq 2}$ has one $0$-generator, one $1$-generator,
and the two crossing $2$-generators $\twocell{braidr}$ and
$\twocell{braidl}$ above.

\item Its strata are defined by
$\strat{1}{\HPol}=\Int(\HPol_2)$,
$\strat{2}{\HPol}=\{\alpha\}$,
$\strat{3}{\HPol}=\{\beta_1,\beta_2\}$,
and $\strat{4}{\HPol}=\{\gamma\}$, where
\[
\alpha:\twocell{braidl}\fl
\twocell{braidr}-z\,\twocell{id*0id}\,,\quad
\beta_1:\twocell{braidl*1braidr}\fl\twocell{id*0id}\,,
\quad
\beta_2:\twocell{braidr*1braidl}\fl\twocell{id*0id}
\quad\text{and}\quad
\gamma:\twocell{(braidr*0id)*1(id*0braidr)*1(braidr*0id)}
\fl
\twocell{(id*0braidr)*1(braidr*0id)*1(id*0braidr)}\,.
\]
\end{enumerate}
We now compute its normalization with respect to a monomial order.
Let $\preccurlyeq$ be a well-founded total order on
$\HPol_1\sqcup\HPol_2$ satisfying
$
1
\prec
\twocell{braidr}
\prec
\twocell{braidl}\,.
$
By \SSS{SSS:LexicographicOrders}, this order $\preccurlyeq$ induces a lexicographic
monomial order $\preccurlyeq_{\HPol}$ on $\pcat{\HPol_2}$.
It is straightforward to verify that
$\preccurlyeq_{\HPol}$ is $\HPol$-compatible.

The first two strata are already normalized, so that
$\strat{1}{\NP{\HPol}}=\strat{1}{\HPol}$ and $\strat{2}{\NP{\HPol}}=\strat{2}{\HPol}$.
For the normalization of the strata $\strat{3}{\HPol}$, we consider the $\strat{2}{\NP{\HPol}}$-active contexts of size~$0$ on $\beta_1$ and $\beta_2$. 
We have the following normalization
\[
\raisebox{0.5cm}{
\xymatrix @C=3em @R=1.5em {
{\twocell{braidl*1braidr}}
   \ar[r]^{\beta_1}
   \ar@{->>}[d]_{\strat{2}{\NP{\HPol}}}
&
{\twocell{id*0id}}
   \ar@{=}[d]
\\
{\twocell{braidr*1braidr}-z\,\twocell{braidr}}
&
{\twocell{id*0id}}
}}
\qquad\text{and}\qquad
\raisebox{0.5cm}{
\xymatrix @C=3em @R=1.5em {
{\twocell{braidr*1braidl}}
   \ar[r]^{\beta_2}
   \ar@{->>}[d]_{\strat{2}{\NP{\HPol}}}
&
{\twocell{id*0id}}
   \ar@{=}[d]
\\
{\twocell{braidr*1braidr}-z\,\twocell{braidr}}
&
{\twocell{id*0id}}
}}\,\raisebox{-7ex}{.}
\]
Thus, both inverse relations $\beta_1$ and $\beta_2$ yield the same
$\strat{2}{\NP{\HPol}}$-normal reduction
$
\twocell{braidr*1braidr}
\fl
\twocell{id*0id}+z\,\twocell{braidr}\,.
$
The $\strat{2}{\NP{\HPol}}$-active contexts of size~$1$ on $\beta_1$ and $\beta_2$ are of the forms
\[
\raisebox{-8.75pt}{\begin{tikzpicture} \begin{scope} [ x = 10pt, y = 10pt, join = round, cap = round, thick, black, solid, -] \draw (0.00,2.00)--(0.00,1.75) (1.00,2.00)--(1.00,1.75) (2.00,2.00)--(2.00,1.75) (3.00,2.00)--(3.00,1.75) (4.00,2.00)--(4.00,1.75) ; \draw (0.00,0.00)--(0.00,-0.25) (1.00,0.00)--(1.00,-0.25) (2.00,0.25)--(2.00,-0.25) (3.00,0.25)--(3.00,-0.25) (4.00,0.25)--(4.00,-0.25) ; \draw [fill = white] (0.00,1.75)--(0.00,1.25)--cycle ; \draw [fill = white] (1.00,1.75)--(1.00,1.25)--cycle ; \draw [fill = white] (2.00,1.75)--(2.00,1.25)--cycle ; \draw [ line width = 2.5pt, white ] (3.20,1.65)--(3.80,1.35) ; \draw (3.00,1.75)--(4.00,1.25) ; \draw [ line width = 2.5pt, white ] (3.80,1.65)--(3.20,1.35) ; \draw (4.00,1.75)--(3.00,1.25) ; \draw (0.00,1.25)--(0.00,1.00) (1.00,1.25)--(1.00,1.00) (2.00,1.25)--(2.00,0.75) (3.00,1.25)--(3.00,0.75) (4.00,1.25)--(4.00,0.75) ; \draw [rounded corners = 1pt, fill = white] (-0.25,0.00) rectangle (1.25,1.00) ; \node at (0.50,0.50) {$\scriptstyle \square$} ; \draw [fill = white] (2.00,0.75)--(2.00,0.25)--cycle ; \draw [fill = white] (3.00,0.75)--(3.00,0.25)--cycle ; \draw [fill = white] (4.00,0.75)--(4.00,0.25)--cycle ; \end{scope} \end{tikzpicture}}
\qquad\text{and}\qquad
\raisebox{-8.75pt}{\begin{tikzpicture} \begin{scope} [ x = 10pt, y = 10pt, join = round, cap = round, thick, black, solid, -] \draw (0.00,2.00)--(0.00,1.50) (1.00,2.00)--(1.00,1.50) (2.00,2.00)--(2.00,1.50) (3.00,2.00)--(3.00,1.75) (4.00,2.00)--(4.00,1.75) ; \draw (0.00,0.00)--(0.00,-0.25) (1.00,0.00)--(1.00,-0.25) (2.00,0.00)--(2.00,-0.25) (3.00,0.00)--(3.00,-0.25) (4.00,0.00)--(4.00,-0.25) ; \draw [fill = white] (0.00,1.50)--(0.00,1.00)--cycle ; \draw [fill = white] (1.00,1.50)--(1.00,1.00)--cycle ; \draw [fill = white] (2.00,1.50)--(2.00,1.00)--cycle ; \draw [rounded corners = 1pt, fill = white] (2.75,0.75) rectangle (4.25,1.75) ; \node at (3.50,1.25) {$\scriptstyle \square$} ; \draw (0.00,1.00)--(0.00,0.50) (1.00,1.00)--(1.00,0.50) (2.00,1.00)--(2.00,0.50) (3.00,0.75)--(3.00,0.50) (4.00,0.75)--(4.00,0.50) ; \draw [ line width = 2.5pt, white ] (0.20,0.40)--(0.80,0.10) ; \draw (0.00,0.50)--(1.00,0.00) ; \draw [ line width = 2.5pt, white ] (0.80,0.40)--(0.20,0.10) ; \draw (1.00,0.50)--(0.00,0.00) ; \draw [fill = white] (2.00,0.50)--(2.00,0.00)--cycle ; \draw [fill = white] (3.00,0.50)--(3.00,0.00)--cycle ; \draw [fill = white] (4.00,0.50)--(4.00,0.00)--cycle ; \end{scope} \end{tikzpicture}}\,,
\]
which produce no further $\strat{2}{\NP{\HPol}}$-normalized reductions. Hence, by
Proposition~\ref{P:FiltrationOfContexts}, we obtain
\[
\strat{3}{\NP{\HPol}}
=
\left\{\;
\beta':
\twocell{braidr*1braidr}
\fl
\twocell{id*0id}+z\,\twocell{braidr}
\;\right\}.
\]
The $\strat{3}{\NP{\HPol}}$-active contexts of size~$1$ on $\beta'$ are of the forms
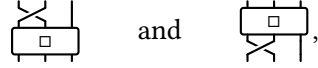
\[
\raisebox{-8.75pt}{\begin{tikzpicture} \begin{scope} [ x = 10pt, y = 10pt, join = round, cap = round, thick, black, solid, -] \draw (0.00,2.00)--(0.00,1.75) (1.00,2.00)--(1.00,1.75) (2.00,2.00)--(2.00,1.75) ; \draw (0.00,0.00)--(0.00,-0.25) (1.00,0.00)--(1.00,-0.25) (2.00,0.00)--(2.00,-0.25) ; \draw [ line width = 2.5pt, white ] (0.20,1.65)--(0.80,1.35) ; \draw (0.00,1.75)--(1.00,1.25) ; \draw [ line width = 2.5pt, white ] (0.80,1.65)--(0.20,1.35) ; \draw (1.00,1.75)--(0.00,1.25) ; \draw [fill = white] (2.00,1.75)--(2.00,1.25)--cycle ; \draw (0.00,1.25)--(0.00,1.00) (1.00,1.25)--(1.00,1.00) (2.00,1.25)--(2.00,1.00) ; \draw [rounded corners = 1pt, fill = white] (-0.25,0.00) rectangle (2.25,1.00) ; \node at (1.00,0.50) {$\scriptstyle \square$} ; \end{scope} \end{tikzpicture}}
\qquad\text{and}\qquad
\raisebox{-8.75pt}{\begin{tikzpicture} \begin{scope} [ x = 10pt, y = 10pt, join = round, cap = round, thick, black, solid, -] \draw (0.00,2.00)--(0.00,1.75) (1.00,2.00)--(1.00,1.75) (2.00,2.00)--(2.00,1.75) ; \draw (0.00,0.00)--(0.00,-0.25) (1.00,0.00)--(1.00,-0.25) (2.00,0.00)--(2.00,-0.25) ; \draw [rounded corners = 1pt, fill = white] (-0.25,0.75) rectangle (2.25,1.75) ; \node at (1.00,1.25) {$\scriptstyle \square$} ; \draw (0.00,0.75)--(0.00,0.50) (1.00,0.75)--(1.00,0.50) (2.00,0.75)--(2.00,0.50) ; \draw [ line width = 2.5pt, white ] (0.20,0.40)--(0.80,0.10) ; \draw (0.00,0.50)--(1.00,0.00) ; \draw [ line width = 2.5pt, white ] (0.80,0.40)--(0.20,0.10) ; \draw (1.00,0.50)--(0.00,0.00) ; \draw [fill = white] (2.00,0.50)--(2.00,0.00)--cycle ; \end{scope} \end{tikzpicture}}\,,
\]
which produce no additional $\strat{3}{\NP{\HPol}}$-normalized reductions. Hence, we obtain
\[
\strat{4}{\NP{\HPol}}:
\qquad
\twocell{
(braidr*0id*0id)
*1(id*0id*0dot)
*1(id*0id*0kk)
*1(id*0id*0dot)
*1(id*0braidr*0id)
*1(braidr*0id*0id)
}
\fl
\twocell{
(id*0id*0dot)
*1(id*0id*0kk)
*1(id*0id*0dot)
*1(id*0braidr*0id)
*1(braidr*0id*0id)
*1(id*0braidr*0id)
}\raisebox{-0.4cm}{\; ,}
\qquad
\text{where \;} \twocell{dot*1kk*1dot}\coloneq 
\raisebox{-10.00pt}{\begin{tikzpicture}
\begin{scope}[x=10pt,y=10pt,join=round,cap=round,thick,black,solid,-]
\draw (3.00,3.00)--(3.00,2.75) (4.00,3.00)--(4.00,2.75);
\draw[line width=2.5pt,white] (3.20,2.65)--(3.80,2.35);
\draw (3.00,2.75)--(4.00,2.25);
\draw[line width=2.5pt,white] (3.80,2.65)--(3.20,2.35);
\draw (4.00,2.75)--(3.00,2.25);
\draw (3.00,2.25)--(3.00,2.00) (4.00,2.25)--(4.00,2.00);
\draw[densely dashed,line width=1.2pt] (1.15,0.85)--(2.85,1.90);
\node[rotate=31.70,above] at (2.00,1.38) {$\scriptstyle k$};
\draw (0.00,0.75)--(0.00,0.50) (1.00,0.75)--(1.00,0.50);
\draw[line width=2.5pt,white] (0.20,0.40)--(0.80,0.10);
\draw (0.00,0.50)--(1.00,0.00);
\draw[line width=2.5pt,white] (0.80,0.40)--(0.20,0.10);
\draw (1.00,0.50)--(0.00,0.00);
\draw (0.00,0.00)--(0.00,-0.25) (1.00,0.00)--(1.00,-0.25);
\end{scope}
\end{tikzpicture}}
\text{\, and } k\geq 0.
\]
It is straightforward to verify that the extensions
$\strat{i}{\NP{\HPol}}$, for $1\leq i\leq 3$, are confluent, while
the confluence of $\strat{4}{\NP{\HPol}}$ follows from
\cite[5.4.4]{GuiraudMalbos09}. 
Therefore, by
Procedure~\ref{SSS:HomBasisProcedure}, the set
$\Nf_m(\NP{\HPol})$ forms a hom-basis of the Hecke category $\HCat$.
Each $2$-cell in $\Nf_m(\NP{\HPol})$ admits a unique decomposition
\[
\twocell{(id*0dot*0id)*1(id*0kk1*0id)*1(id*0dot*0id)}
\;\circ_1\;
\twocell{(id*0dot*0id)*1(id*0kk2*0id)*1(id*0dot*0id)}
\;\circ_1\;\cdots\;\circ_1\;
\twocell{(id*0dot*0id)*1(id*0kkn*0id)*1(id*0dot*0id)}\;,
\qquad k_i\geq 1.
\]
such that, for each $1\leq i<\ell$, the part connecting
$\twocell{(id*0dot*0id)*1(id*0kki*0id)*1(id*0dot*0id)}$ and
$\twocell{(id*0dot*0id)*1(id*0kkij1*0id)*1(id*0dot*0id)}$
has one of the forms
\[
\twocell{(id*0id*0braidr*0id)*1(id*0braidr*0id*0id)}
\qquad
\twocell{
(id*0braidr*0jj*0id*0id*0id)
*1(id*0id*0id*0id*0braidr*0id)
}\;,
\]
or, when $k_i=1$, the form
\[
\twocell{(id*0braidr*0id*0id)*1(id*0id*0braidr*0id)}\,.
\]
Here, $\twocell{jj}$ denotes an identity $2$-cell,
with $j\geq k_i-2$.
In particular, for every $n\geq 1$, the set
\[
\Nf_m(\NP{\HPol})
\cap
\End_{\HCat}(1^{\circ_0 n})
\]
forms a $\mathbb{Z}[q,q^{-1}]$-basis of the Iwahori--Hecke algebra
$H_n(q)$.

When $z=0$, the quadratic relation
$
T_i^2=1+zT_i
$
becomes $T_i^2=1$, and hence $H_n(q)$ specialises to the group algebra
$\mathbb{Z}[S_n]$ of the symmetric group. Under this specialization,
the hom-basis
$
\Nf_m(\NP{\HPol})\cap\End_{\HCat}(1^{\circ_0 n})
$
of $H_n(q)$ obtained above specialises to a basis indexed by the
permutations $w\in S_n$. For each reduced expression
$
w=s_{i_1}\cdots s_{i_r},
$
the corresponding basis element $T_w$ is represented by the reduced
braid diagram $T_{i_1}\cdots T_{i_r}$.
\end{example}

\section{The normal form basis of the $q$-Schur category}
\label{S:NormFormHomBasisqSchurCategories}
In this section, we construct a new hom-basis of the $q$-Schur category, the \emph{normal form basis}, by applying the stratified rewriting approach to Brundan’s presentation~\cite[Thm.~2]{Brundan2025}.
In Subsection~\ref{SS:PresentationQSchurCategories}, we first derive from this presentation a stratified $3$-prepolygraph, whose termination is established in Subsection~\ref{SS:TerminationQSchurPrepolygraph} using a monomial-order-based criterion.
In Subsection~\ref{SS:NormalizationQSchurPresentation}, we apply the stratified normalization procedure to this presentation and deduce the new hom-basis. 
Throughout this section, we set~$\bk=\mathbb{Z}[q,q^{-1}]$.

\subsection{Presentations of $q$-Schur categories}
\label{SS:PresentationQSchurCategories}

We begin by recalling the presentation of the $q$-Schur category from~\cite[Thm.~2]{Brundan2025}, and reformulate it as a $5$-stratified $3$-prepolygraph.
For an integer $n> 0$, we will denote the \emph{quantum integer} by
\[
[n]_q := \frac{q^n - q^{-n}}{q - q^{-1}}\raisebox{-0.8ex}{.}
\]
In particular, $[0]_q=0$ and $[-n]_q=-[n]_q$.
The corresponding \emph{quantum binomial coefficient} is given by
\[
{n \brack s}_q
\coloneq
\frac{[n]_q [n-1]_q \cdots [n-s+1]_q}{[s]_q [s-1]_q \cdots [1]_q}\raisebox{-1.2ex}{,}
\]
with the convention that
${n \brack 0}_q = 1$
and ${n \brack s}_q = 0$ whenever $s < 0$.

\subsubsection{$q$-Schur category}
The \emph{$q$-Schur category}, denoted by $\qSCat$, is the linear $2$-category generated by 
$\ACCat_i$, for $0\leq i\leq 2$, subject to the relations \eqref{E:Associativityrelations}, together with the following \emph{square-switch relation}:
\[
\twocell{(merge *0 id) *1 (id *0 c *0 id)  *1 (id *0 split)*1(id*0 merge)*1 (id *0 d *0 id) *1(split *0 id)  *1(ax *0 bx)}
=
\sum_{s=\max (0, c-b)}^{\min (c, d)}{a-b+c-d \brack s}_q
\twocell{(id *0 merge)*1 (id *0 djs *0 id)  *1 (split*0 id)*1 (merge*0 id) *1 (id *0 cjs *0 id)  *1 (id *0 split) *1(ax *0 bx)}
\]
for $a, b, c, d \geq 0$ with $d \leq a$ and $c \leq b+d$, where
$\twocell{merge*1(ax*0 0x)}\coloneq\twocell{id*1 ax}$ 
and
$\twocell{(bs*0 0s)*1 split}\coloneq\twocell{id*1 bx}$.

We present the $q$-Schur category $\qSCat$ by the $3$-prepolygraph $\qSPol$, defined by $\qSPol_i=\ACPol_i$ for $0\leq i\leq 2$, and
$
\qSPol_3 \coloneq \ACPol_3  \qSPol_3',
$
where $\qSPol_3'$ consists of the following $3$-generators:
\[
\gamma_{a,b,c,d}:
\twocell{(merge *0 id) *1 (id *0 c *0 id)  *1 (id *0 split)*1(id*0 merge)*1 (id *0 d *0 id) *1(split *0 id)  *1(ax *0 bx)}
\fl
\sum_{s=\max (0, c-b)}^{\min (c, d)}{a-b+c-d \brack s}_q
\twocell{(id *0 merge)*1 (id *0 djs *0 id)  *1 (split*0 id)*1 (merge*0 id) *1 (id *0 cjs *0 id)  *1 (id *0 split) *1(ax *0 bx)}
\]
for all $a, b, c, d \geq 0$ with $d \leq a$ and $c \leq b+d$.

\subsubsection{Stratification of $\qSPol_3$}
We stratify the extension $\qSPol_3'$ according to the values of the parameters $a$, $b$, $c$, $d$ of the rule $\gamma_{a,b,c,d}$. Several special choices of parameters yield degenerate forms of
$\gamma_{a,b,c,d}$:
\begin{enumerate}
\item If $c=0$ or $d=0$, the rule reduces to the identity:
\[
\gamma_{a,b,0,d}:\twocell{(id*0 merge)*1(id*0 d*0 id)*1(split*0 id)*1(ax*0 bx)}\fl\twocell{(id*0 merge)*1(id*0 d*0 id)*1(split*0 id)*1(ax*0 bx)}
\quad\text{ or }\quad
\gamma_{a,b,c,0}:\twocell{( merge*0 id)*1(id*0 c *0 id)*1(id*0 split)*1(ax*0 bx)}\fl\twocell{( merge*0 id)*1(id*0 c *0 id)*1(id*0 split)*1(ax*0 bx)}\,.
\]

\item When $b=0$, the rule takes the form
\begin{eqn}{equation}
\label{E:UpperTriangular}
\gamma_{a,0,c,d}:\twocell{(merge*0 id)*1 (id*0 c *0 djc)*1 (id*0 split)*1 split*1 ax}
\fl
{a+c-d \brack c}_q\, \twocell{(id*0 djcs)*1split*1 ax}.
\end{eqn}
If moreover $c=d$, the rule further simplifies to
\begin{eqn}{equation}
\label{E:Triangular1}
\gamma_{a,0,c,d}:\twocell{merge*1(id*0 d)*1 split*1 ax}
\fl
{a \brack d}_q\, \twocell{id*1id*1ax}.
\end{eqn}

\item When $b+d-c=0$, we obtain
\begin{eqn}{equation}
\label{E:DownTriangular}
\gamma_{a,b,c,d}:\twocell{merge*1(id*0merge)*1(id*0 d*0 id)*1(split*0 id)*1(ax*0 bx)}\fl 
{a \brack d}_q\twocell{merge*1(ax*0 bx)}.
\end{eqn}
If moreover $b=0$, this further reduces to
\[
\gamma_{a,0,c,d}:\twocell{merge*1(id*0 d)*1 split*1 ax}
\fl
{a \brack d}_q\, \twocell{id*1 id*1 ax}.
\]

\item When $a-d=0$, we have
\[
s\gamma_{a,b,c,d}=
\twocell{(cs*0 id)*1split*1 merge*1(ax*0 bx)}
\quad\text{and}\quad
t\gamma_{a,b,c,d}=
\sum_{s=\max (0, c-b)}^{\min (c, a)}
\left[\begin{array}{c}
c-b \\[1mm] s
\end{array}\right]_q 
\twocell{(id *0 merge)*1 (id *0 ajs *0 id)  *1 (split*0 id)*1 (merge*0 id) *1 (id *0 cjs *0 id)  *1 (id *0 split) *1(ax *0 bx)}\,.
\]
If $c \geq b$, then
$
{c-b \brack s}_q = 0
$ for all $s > c-b$.
Since $c-b \leq s \leq \min(c,a)$, it follows that $t\gamma_{a,b,c,d}$ has only one non-zero term with $s = c-b$.
In this case, $\gamma_{a,b,c,d}$ is the identity.
If $c<b$, we rewrite $\gamma_{a,b,c,d}$ in the oriented form
\begin{eqn}{equation}
\label{E:Triangular2}
\gamma_{a,b,c,d}:
\twocell{(id *0 merge)*1 (id *0 a *0 id)  *1 (split*0 id)*1 (merge*0 id) *1 (id *0 c *0 id)  *1 (id *0 split) *1(ax *0 bx)}
\fl
\twocell{(cs*0 id)*1split*1 merge*1(ax*0 bx)}
-
\sum_{s=1}^{\min (c, a)}
{c-b \brack s}_q
\twocell{(id *0 merge)*1 (id *0 ajs *0 id)  *1 (split*0 id)*1 (merge*0 id) *1 (id *0 cjs *0 id)  *1 (id *0 split) *1(ax *0 bx)}\,.
\end{eqn}
The orientation is induced by the monomial order defined in~\SSS{SSS:PartialOrderQSchurPolygraph}, which ensures termination of $\qSPol$.
\end{enumerate}
Therefore, we construct a $5$-stratification $F^\ast \qSPol$ of the prepolygraph $\qSPol$ as follows.
We set $\strat{1} \qSPol = \strat{1} \ACPol$ and $\strat{2} \qSPol = \strat{2} \ACPol$ from Example~\ref{E:AssociativeCoassociativeCategory}.
The extension $\strat{3} \qSPol$ consists of the following rules:
\[ 
\gamma_{a,b,c}^\triangle: \twocell{(as*0 bs)*1(merge*0 id)*1 (id*0 c *0 id)*1 (id*0 split)*1 split}\fl {a \brack c}_q\,\twocell{(as*0 bs)*1split}, \qquad \gamma_{a,b,c}^\triangledown: \twocell{merge*1(id*0merge)*1(id*0 c*0 id)*1(split*0 id)*1(ax*0 bx)}\fl {a \brack c}_q\,\twocell{merge*1(ax*0 bx)}, \qquad \gamma_{a,c}^\diamond: \twocell{merge*1(id*0 c)*1 split*1 ax}\fl {a \brack c}_q\,\twocell{id*1id*1ax},
\]
for all $a,b,c > 0$, which correspond respectively to the degenerate rules in \eqref{E:UpperTriangular}, \eqref{E:DownTriangular}, and
\eqref{E:Triangular1}.
The extension $\strat{4} \qSPol$ consists of the following non-degenerate rules:
\[
\gamma_{a,b,c,d}:\twocell{(merge *0 id) *1 (id *0 c *0 id)  *1 (id *0 split)*1(id*0 merge)*1 (id *0 d *0 id) *1(split *0 id)  *1(ax *0 bx)}
\fl
\sum_{s=\max (0, c-b)}^{\min (c, d)}{a-b+c-d \brack s}_q
\twocell{(id *0 merge)*1 (id *0 djs *0 id)  *1 (split*0 id)*1 (merge*0 id) *1 (id *0 cjs *0 id)  *1 (id *0 split) *1(ax *0 bx)}\;,
\]
for all $a,b,c,d > 0$, with $a>d$ and $b+d>c$.
Finally, the extension $\strat{5} \qSPol$ consists of the following degenerate rules:
\[
\eta_{a,b,c}:\twocell{(id *0 merge)*1 (id *0 a *0 id)  *1 (split*0 id)*1 (merge*0 id) *1 (id *0 c *0 id)  *1 (id *0 split) *1(ax *0 bx)}
\fl
\twocell{(cs*0 id)*1split*1 merge*1(ax*0 bx)}
-\sum_{s=1}^{\min (c, a)}
{c-b \brack s}_q
\twocell{(id *0 merge)*1 (id *0 ajs *0 id)  *1 (split*0 id)*1 (merge*0 id) *1 (id *0 cjs *0 id)  *1 (id *0 split) *1(ax *0 bx)}\;,
\]
for all $a,b,c > 0$, with $b>c$, which correspond to the degenerate rules in \eqref{E:Triangular2}.

Hence, the $5$-stratified $3$-prepolygraph $F^\ast \qSPol$ presents the $q$-Schur category $\qSCat$. Its rules will be written simply by omitting subscripts whenever no confusion arises.

\subsection{Termination by stratification}
\label{SS:TerminationQSchurPrepolygraph}

\subsubsection{A total preorder on $\pcat{\qSPol_2}$}
\label{SSS:PartialOrderQSchurPolygraph}
We define a total preorder $\succcurlyeq_{\q}$ on $\pcat{\qSPol_2}$ as follows. 
\begin{enumerate}
\item\label{I:1OrderOfqS} 
For any two parallel $2$-cells $f,g$ in $\pcat{\qSPol_2}$ of size~$1$, written as
\[
f = a_n \cdots a_1 \circ_0 \varphi \circ_0 b_1 \cdots b_m
\quad \text{and} \quad
g = a'_{n'} \cdots a'_1 \circ_0\varphi'\circ_0 b'_1 \cdots b'_{m'},
\]
where 
$a_i,a'_i,b_i,b'_i \in \qSPol_1$ and $\varphi,\varphi' \in \qSPol_2$.
Since $sf=sg$, we have
\[
\sum_{i=1}^n a_i
+
\sum_{i=1}^m b_i
+
\deg(\varphi)
=
\sum_{i=1}^{n'} a'_i
+
\sum_{i=1}^{m'} b'_i
+
\deg(\varphi'),
\]
and denote this common number by $r$, where $\deg(\varphi)\coloneq c$ if $\varphi=\twocell{cs*1merge}$ or $\varphi=\twocell{split*1cx}$\,.
We set $j\coloneq\min(n+m+1,n'+m'+1)$,
and define $f \succ_{\q} g$ if
\[
(a_n,\dots,a_1,\varphi,b_1,\dots,b_m)_j
\succ_{[\q,\ldots,\q]}
(a'_{n'},\dots,a'_1,\varphi',b'_1,\dots,b'_{m'})_j,
\]
where $(-)_j$ denotes the truncation to the first $j$ components from the left and
the order $\succcurlyeq_{\q}$ on $\qSPol_1\sqcup\qSPol_2$ is defined by
\begin{eqn}{equation}
\label{E:WellFounded1Element}
\cdots\succ_{\q}
\twocell{2s*1merge}\succ_{\q}
\twocell{1s*1merge}\succ_{\q}
\twocell{split*1 1x}\succ_{\q}
\twocell{id*1 1x}\succ_{\q}
\twocell{split*1 2x}\succ_{\q}
\twocell{id*1 2x}\succ_{\q}
\cdots\succ_{\q}
\twocell{split*1 rx}\succ_{\q}
\twocell{id*1 rx}.
\end{eqn}

\item\label{I:nOrderOfqS} More generally, for any two parallel $2$-cells $f, g \in \pcat{\qSPol_2}$, write their unique $\circ_1$-decompositions as
\[
f = C_1[\varphi_1] \circ_1 C_2[\varphi_2] \circ_1 \cdots \circ_1 C_n[\varphi_n]
\quad\text{and}\quad
g = D_1[\psi_1] \circ_1 D_2[\psi_2] \circ_1 \cdots \circ_1 D_m[\psi_m].
\]
We define $f \succ_{\q} g$ if $n>m$, or $n=m$ and 
\begin{eqn}{equation}
\label{E:AnyLength}
(C_1[\varphi_1], \ldots, C_n[\varphi_n]) \succ_{[\q,\ldots,\q]} (D_1[\psi_1],\ldots, D_m[\psi_m]).
\end{eqn}
\end{enumerate}

\begin{lemma}
\label{L:qSchurMonomialOrder}
The total preorder $\preccurlyeq_{\q}$ is monotone and well-founded.
\end{lemma}
\begin{proof}
As in the case of the lexicographic order on free words, it is straightforward that the order $\preccurlyeq_{\q}$ on $\pcat{\qSPol_2}$ is monotone.
By the definition of $\preccurlyeq_{\q}$ on parallel $2$-cells of arbitrary size given in \eqref{E:AnyLength}, it remains to prove that $\preccurlyeq_{\q}$ is well-founded on parallel $2$-cells of size $1$, following \cite[Lem.~1.15]{LIU25}.

From \SSS{SSS:PartialOrderQSchurPolygraph}-\textbf{i)},
the set of $2$-cells of size~$1$ is decomposed into subsets $(Z_i)_{i \geq 1}$,
where each $Z_i$ consists of $2$-cells of the form
\[
f = a_n \cdots a_1 \circ_1\varphi\circ_1 b_1 \cdots b_m
\quad \text{with} \quad
i = \sum_{i=1}^n a_i
+
\sum_{i=1}^m b_i
+
s\varphi.
\]
For each $Z_i$, \eqref{E:WellFounded1Element} shows that the order  $\preccurlyeq_{\q}$ is well-founded on $\qSPol_1\sqcup\qSPol_2$ with minimal element $\twocell{id*1 ix}$.  
It follows that $\preccurlyeq_{\q}$ is well-founded on $Z_i$.
Therefore, the order $\preccurlyeq_{\q}$ is well-founded on the set of parallel $2$-cells of size~$1$.
\end{proof}

We define a well-founded total order $\preccurlyeq$ on the set of generating $2$-cells of $\qSPol$ by setting
\[
\twocell{merge*1(ax*0bx)}
\succ
\twocell{merge*1(cx*0dx)}
\succ
\twocell{(as*0bs)*1split}
\succ
\twocell{(cs*0ds)*1split}
\succ
\twocell{id*1 mx}
\succ
\twocell{id*1 nx}\,,
\]
if $(a,b)>(c,d)$ with respect to the lexicographic order and $m>n$.
By \SSS{SSS:LexicographicOrders}, this order induces a lexicographic order $\preccurlyeq_{\lex}$ on $\pcat{\qSPol_2}$.
By Lemma~\ref{L:qSchurMonomialOrder}, the order $\preccurlyeq_{\q,\lex}$ is a monomial order on $\pcat{\qSPol_2}$, and we denote it by $\preccurlyeq_\Sch \coloneq \preccurlyeq_{\q,\lex}$.

\begin{lemma}
\label{L:CompatibleMonomial Order}
The order $\preccurlyeq_{\mathrm{Sch}}$ is a $\qSPol$-compatible monomial order on $\pcat{\qSPol_2}$.
Hence, the prepolygraph $\qSPol$ is terminating.
\end{lemma}
\begin{proof}
We compare the sources and targets of all rules in $\monex{\qSPol_3}$ with respect to $\preccurlyeq_{\q}$.
\begin{enumerate}
\item For the extension $\strat{1} \qSPol$, since $ \twocell{ajbs*1merge *1 (ax *0 bx)}\succ_{\q}\twocell{id*1 ax}$, it follows that
\[
\twocell{(dot*0 merge) *1 (merge*0 dot*0 id)*1(ax*0bx*0emptyx*0emptyx*0emptyx)}
\succ_{\q}\twocell{( merge*0 dot) *1 ( id *0 dot*0 merge)*1(ax*0bx*0emptyx*0emptyx*0emptyx)}
\quad\text{and}\quad
\twocell{( dot*0 split)*1( merge*0 dot)*1(ax*0bx*0emptyx*0emptyx)}
\succ_{\q}
\twocell{(merge*0 dot*0 id)*1(id*0 dot *0 split)*1(ax*0bx*0emptyx*0emptyx)}.
\]
Similarly, since  $\twocell{split*1 ax}\succ_{\q}\twocell{id*1 ax} $, we have 
\[
\twocell{(id*0 dot*0 split) *1 (split*0 dot)*1(ax*0emptyx*0emptyx)}
\succ_{\q}
\twocell{(split*0 dot*0 id) *1 (dot*0 split)*1(ax*0emptyx*0emptyx)}
\quad\text{and}\quad
\twocell{(id*0 dot*0 merge)*1(split*0 dot *0 id)*1(ax*0emptyx*0emptyx*0emptyx)}
\succ_{\q}
\twocell{(split*0 dot)*1( dot *0 merge)*1(ax*0emptyx*0emptyx*0emptyx)}.
\]
\item For the extension $\strat{2} \qSPol$, since $ \twocell{merge *1 (ax *0 bx)}\succ_{\q}\twocell{id*1 ax}$, $\twocell{(as *0 bjcs)*1 split*1(ajbjcx)}=\twocell{(ajbs *0 cs)*1 split*1(ajbjcx)}$ and $\twocell{id*1 ax}\succ_{\q} \twocell{(as *0 bs)*1 split*1(ajbx)}$,
we obtain 
\[
\twocell{merge*1((merge*1(ax *0 bx)) *0 cx)}\succ_{\q}\twocell{merge*1 (ax *0 (merge*1(bx *0 cx)))}
\quad\text{and}\quad
\twocell{( as *0 ((bs*0 cs)*1 split))*1 split} \succ_{\q} \twocell{(((as*0 bs)*1 split)*0 cs)*1 split}.
\]
\item For the extension $\strat{3} \qSPol$, comparing the sizes of $2$-cells yields
\[
\twocell{(as*0 bs)*1(merge*0 id)*1 (id*0 c *0 id)*1 (id*0 split)*1 split}
\succ_{\q} \twocell{(as*0 bs)*1split}\,,
\quad
\twocell{merge*1(id*0merge)*1(id*0 c*0 id)*1(split*0 id)*1(ax*0 bx)}
\succ_{\q}
\twocell{merge*1(ax*0 bx)}\,,
\quad\text{and}\quad
\twocell{merge*1(id*0 c)*1 split*1 ax}
\succ_{\q} \twocell{id*1ax}\,.
\]

\item For the extension $\strat{4} \qSPol$, since $\twocell{split*1 ax}\succ_{\q} \twocell{id*1 ax}$ for every $s\geq0$, we have 
\[
\twocell{(merge *0 id) *1 (id *0 c *0 id)  *1 (id *0 split)*1(id*0 merge)*1 (id *0 d *0 id) *1(split *0 id)  *1(ax *0 bx)}
\succ_{\q}
\twocell{(id *0 merge)*1 (id *0 djs *0 id)  *1 (split*0 id)*1 (merge*0 id) *1 (id *0 cjs *0 id)  *1 (id *0 split) *1(ax *0 bx)}\;.
\]
\item  For the extension $\strat{5} \qSPol$, since 
$\twocell{merge*1(ax*0 cx)}\succ_{\q}\twocell{merge*1(ax*0 cjsx)} $
for every $s>0$, and by comparing the sizes, we obtain
\[
\twocell{(id *0 merge)*1 (id *0 a *0 id)  *1 (split*0 id)*1 (merge*0 id) *1 (id *0 c *0 id)  *1 (id *0 split) *1(ax *0 bx)}
\succ_{\q}
\twocell{(id *0 merge)*1 (id *0 ajs *0 id)  *1 (split*0 id)*1 (merge*0 id) *1 (id *0 cjs *0 id)  *1 (id *0 split) *1(ax *0 bx)}\,
\quad\text{and}\quad
\twocell{(id *0 merge)*1 (id *0 a *0 id)  *1 (split*0 id)*1 (merge*0 id) *1 (id *0 c *0 id)  *1 (id *0 split) *1(ax *0 bx)}
\succ_{\q}
\twocell{(cs*0 id)*1split*1 merge*1(ax*0 bx)}\,.
\]
\end{enumerate}
Therefore, the order $\preccurlyeq_\Sch = \preccurlyeq_{\q,\lex}$ is also compatible.
\end{proof}

\subsection{Stratified normalization and normal form basis}
\label{SS:NormalizationQSchurPresentation}

In this subsection, we denote by $\twocell{nr*0id}$ the $n$ vertical strands, and by $\twocell{leqn*0id}$ any $i$ vertical strands $\twocell{ir*0id}$ with $1\leq i \leq n$.
The notation $\twocell{merge*1(1r*0id)}$ (resp.\;$\twocell{(1r*0id)*1split}$\;) is defined to be the identity. For $n\geq 2$, we set
\[
\twocell{merge*1(2r*0id)}\coloneq \twocell{merge}
\qquad
(\text{resp. }\; \twocell{(2r*0id)*1split}\coloneq \twocell{split}\;),
\]
and define inductively
\[
\twocell{merge*1(nj1r*0id)}
\coloneq
\twocell{merge*1(id*0merge)*1(id*0nr*0id)}
\qquad
(\text{resp. }\;
\twocell{(nj1r*0id)*1split}
\coloneq
\twocell{(nr*0id*0id)*1(split*0id)*1split}\;).
\]
We also introduce the following notation
\[
\twocell{ f *0 g *0 h }
\;\coloneq\;
\twocell{( f*0 dot *0 dot)*1(dot *0 g *0 dot)*1( dot *0 dot *0 h)}\,.
\]
A $2$-cell $f$ in $\pcat{\qSPol_2}$ is \emph{minimal} if $\down{f}{}$  cannot be written as
\; $\twocell{ f1 *0 dot*0 fi }$\,, for every $i \geq 2$. 

\setlength{\parskip}{1em}
We now compute the $\preccurlyeq_{\mathrm{Sch}}$-normalization $F^\ast \NP{\qSPol}$ of the $5$-stratified prepolygraph $F^\ast \qSPol$.
We set $\strat{1}\NP{\qSPol} \coloneq \strat{1}\qSPol$.
Thanks to Example~\ref{E:AssociativeCoassociativeCategory}, we have $\strat{2}\NP{\qSPol} \coloneq \strat{2}\qSPol$ and $\Nf(\strat{2}\NP{\qSPol})=\Nf_m(\NP{\ACPol})$.

\subsubsection{$F^2\NP{\qSPol}$-normal reductions}
\label{SSS:StratumS2}
For each $\gamma^\diamond\in \strat{3}\qSPol$, there exist two contexts
\[
C_1[\gamma^\diamond]:\twocell{merge*1(merge*0id)*1(split*0id)}\fl\twocell{merge}
\quad\text{and}\quad
C_2[\gamma^\diamond]:\twocell{(id*0merge)*1(id*0split)*1split}\fl\twocell{split}
\]
such that 
$
\gamma^\triangle=\down{C_1[\gamma^\diamond]}{2}^\preccurlyeq
$ and
$
\gamma^\triangledown=\down{C_2[\gamma^\diamond]}{2}^\preccurlyeq.
$
Consequently, we obtain
\[
\dlr{\gamma^\triangle}{2} \subseteq \dlr{\gamma^\diamond}{2}
\quad\text{and}\quad
\dlr{\gamma^\triangledown}{2} \subseteq \dlr{\gamma^\diamond}{2}.
\]
It suffices to compute $\dlr{\gamma^\diamond}{2}$ in order to obtain $\strat{3}\NP{\qSPol}$.

Since $s\gamma^\diamond\in \Nf(F^2\NP{\qSPol})$, we have the stratum $L^0\minsub\dlr{\gamma^\diamond}{2}=\{\gamma^\diamond\}$.
A straightforward verification shows that only $C_1$ and $C_2$ yield new rules to $\minsub\dlr{\gamma^\diamond}{2}$ among contexts of size $1$.  Hence
\[
L^1\minsub\dlr{\gamma^\diamond}{2}=\{\gamma^\diamond,\gamma^\triangle,\gamma^\triangledown\}.
\]
Moreover, among contexts of size $2$, none adds new rules to $\minsub\dlr{\gamma^\diamond}{2}$ except the following one:
\[
C_3[\gamma^\diamond_{a,c}]:
\twocell{(id*0merge)*1(id*0merge*0id)*1(id*0a-c*0c*0id)*1(id*0split*0id)*1(split*0id)}  \fl 
{a \brack c}_q\,\twocell{(id*0merge)*1(id*0a*0id)*1(split*0id)}
\qquad\leadsto\qquad
\down{C_3[\gamma^\diamond_{a,c}]}{2}^\preccurlyeq:
\twocell{(id*0merge3)*1(id*0id*0id*0id)  *1(id*0a-c*0c*0id)*1(id*0id*0id*0id)*1(split3*0id)}
\fl
{a \brack c}_q\,
\twocell{(id*0merge)*1(id*0id*0id) *1(id*0a*0id)*1(id*0id*0id)*1(split*0id)}\;.
\]
Therefore, we obtain
\[
L^2\minsub\dlr{\gamma^\diamond}{2}=\{\gamma^\diamond,\gamma^\triangle,\gamma^\triangledown,\down{C_3[\gamma^\diamond_{a,c}]}{2}\}.
\]
Finally, one checks that
$L^2\minsub\dlr{\gamma^\diamond}{2}=L^3\minsub\dlr{\gamma^\diamond}{2}$.
By Proposition~\ref{P:FiltrationOfContexts}, it follows that
\[
\strat{3}\NP{\qSPol}=\{\gamma^\diamond,\gamma^\triangle,\gamma^\triangledown,\down{C_3[\gamma^\diamond_{a,c}]}{2}^\preccurlyeq\}.
\]
By induction, there are two families of $3$-cells in $\plcat{(\strat{3}\NP{\qSPol})}$:
\[
\twocell{(merge4*0 id*0 id)*1(id*0id*0dot*0id*0id)*1(dot*0 a1*0am *0 dot)*1(id*0id*0dot*0id*0id)*1(id*0 id*0 split4)}
\fl
\prod_{i=2}^{m} 
\left[\begin{array}{c}
a_1+\cdots + a_i \\
a_i
\end{array}\right]_q\,
\twocell{(merge3*0 id*0 id)*1(dot*0 b *0 dot)*1(id*0 id*0 split3)}
\quad\text{or}\quad
\twocell{(id*0 id*0 merge4)*1(id*0id*0dot*0id*0id)*1(dot*0 a1*0am *0 dot)*1(id*0id*0dot*0id*0id)*1( split4*0 id*0 id)}
\fl
\prod_{i=2}^{m} 
\left[\begin{array}{c}
a_1+\cdots + a_i \\
a_i
\end{array}\right]_q\,
\twocell{( id*0 id*0 merge3)*1(dot*0 b *0 dot)*1( split3 *0 id*0 id)}\;,
\]
where $b=a_1+\ldots + a_m$,
This shows that any $\twocell{merge*1(mr*0id)}$ and $\twocell{(nr*0id)*1split}$ share at most one common strand. Hence $\strat{3}\NP{\qSPol}$ is confluent, since every source of critical branchings in $\strat{3}\NP{\qSPol}$ admits a unique normal form \cite[Thm.~4.3.3]{Alleaume18}.
In fact, the minimal $2$-cell in $\Nf(F^3\NP{\qSPol})$ can be described as follows:
\begin{eqn}{equation}
\label{E:S2-NormalForm}
\twocell{(id*0id*0merge*0merge*0 merge)*1
(id*0nr*0mr*0nr*0mr*0nr*0nr*0id)*1
(id*0split*0split*0split*0id)*1
(id*0merge*0merge *0merge*0id)*1
(nr*0mr*0mr*0mr*0mr*0mr*0id*0id)*1
(split*0split*0split*0id*0id)*1 (merge*0merge*0merge*0id*0id)*1
(nr*0mr*0mr*0mr*0mr*0mr*0id*0id)*1 (id*0split*0split*0split*0id)*1
(id*0merge*0merge*0merge*0id)*1(id*0nr*0mr*0nr*0mr*0nr*0id*0id)*1
(id*0id*0split*0split*0id*0id)}\;\raisebox{-8.5ex}{,}
\end{eqn}
where $2 \le m \le 3$ and $n \ge 2$, with each occurrence of $m$ and $n$ taking independent values.

\subsubsection{$F^3\NP{\qSPol}$-normal reductions}
\label{SSS:R3StratifiedCompletionStep}
For simplicity, we write
\[
\sum_s\coloneq\sum_{s=\max(0,c-b)}^{\min(c,d)} 
\quad \text{and}\quad
\Psivec{s}{a}{b}{c}{d}\coloneq
{
a-b+c-d 
\brack
s
}_q.
\]
For $\gamma \in \strat{3} \qSPol$, among contexts of size $1$, namely $C_1$ and $C_2$, and of size $2$, namely $C_3$,
\[
C_1[s\gamma]:\twocell{merge*1(merge*0id)*1(id*0split)*1(id*0merge)*1(split*0id)}
\qquad
C_2[s\gamma]:\twocell{(merge*0id)*1(id*0split)*1(id*0merge)*1(split*0id)*1split}
\qquad
C_3[s\gamma]:\twocell{(merge*0id)*1(id*0merge*0id)*1(id*0id*0split)*1(id*0id*0merge)*1(id*0split*0id)*1(split*0id)}\;,
\]
we have $\down{C_1[\partial\gamma]}{3}=0$ and $\down{C_2[\partial\gamma]}{3}=0$, and the rule $\down{C_3[\gamma]}{3}^\preccurlyeq$ is trivial (see Appendix~\ref{SS:ProofsOfThreeContexts}).
Therefore, none of these contexts yields a new rule in $\minsub\dlr{\gamma}{3}$.

Note that the extension $\minsub\dlr{\gamma}{3}$ is not bounded.
It suffices to apply active contexts on $\gamma$ along the two directions indicated by the arrows in the following diagram:
\[
\gamma:
\begin{tikzpicture}[baseline={(base)},x=10pt,y=10pt,join=round,cap=round,thick,black,solid]
  \coordinate (base) at (1.0,1.5);
  \useasboundingbox (-0.35,-0.55) rectangle (2.35,3.35);
  \draw[->,>=stealth,blue!70!black,very thick] (2.00,3.25)--(2.00,2.90);
  \draw[->,>=stealth,blue!70!black,very thick] (0.50,-0.35)--(0.50,0.00);
  \draw (0.50,3.00)--(0.50,2.75);
  \draw (2.00,0.00)--(2.00,-0.25);

  \draw[fill=green] (0.50,2.75)--(1.00,2.25)--(0.00,2.25)--cycle;
  \draw[fill=lightgray] (2.00,2.75)--(2.00,2.25)--cycle;

  \draw (0.00,2.25)--(0.00,2.00)
        (1.00,2.25)--(1.00,2.00)
        (2.00,2.25)--(2.00,2.00);

  \draw[fill=lightgray] (0.00,2.00)--(0.00,1.50)--cycle;
  \draw[fill=orange]   (1.00,2.00)--(2.00,2.00)--(1.50,1.50)--cycle;

  \draw (0.00,1.50)--(0.00,1.25)
        (1.50,1.50)--(1.50,1.25);

  \draw[fill=lightgray] (0.00,1.25)--(0.00,0.75)--cycle;
  \draw[fill=green]     (1.50,1.25)--(2.00,0.75)--(1.00,0.75)--cycle;

  \draw (0.00,0.75)--(0.00,0.50)
        (1.00,0.75)--(1.00,0.50)
        (2.00,0.75)--(2.00,0.50);

  \draw[fill=orange]   (0.00,0.50)--(1.00,0.50)--(0.50,0.00)--cycle;
  \draw[fill=lightgray](2.00,0.50)--(2.00,0.00)--cycle;
\end{tikzpicture}
\;\longrightarrow\;
\sum
\begin{tikzpicture}[baseline={(base)},x=10pt,y=10pt,join=round,cap=round,thick,black,solid]
  \coordinate (base) at (1.0,1.5);
  \useasboundingbox (-0.35,-0.55) rectangle (2.35,3.35);
  \draw[->,>=stealth,blue!70!black,very thick] (1.50,3.25)--(1.50,2.90);
  \draw[->,>=stealth,blue!70!black,very thick] (0.00,-0.35)--(0.00,0.00);
  \draw (0.00,3.00)--(0.00,2.75);
  \draw (1.50,0.00)--(1.50,-0.25);

  \draw[fill=lightgray] (0.00,2.75)--(0.00,2.25)--cycle;
  \draw[fill=green]     (1.50,2.75)--(2.00,2.25)--(1.00,2.25)--cycle;

  \draw (0.00,2.25)--(0.00,2.00)
        (1.00,2.25)--(1.00,2.00)
        (2.00,2.25)--(2.00,2.00);

  \draw[fill=orange]   (0.00,2.00)--(1.00,2.00)--(0.50,1.50)--cycle;
  \draw[fill=lightgray](2.00,2.00)--(2.00,1.50)--cycle;

  \draw (0.50,1.50)--(0.50,1.25)
        (2.00,1.50)--(2.00,1.25);

  \draw[fill=green]     (0.50,1.25)--(1.00,0.75)--(0.00,0.75)--cycle;
  \draw[fill=lightgray] (2.00,1.25)--(2.00,0.75)--cycle;

  \draw (0.00,0.75)--(0.00,0.50)
        (1.00,0.75)--(1.00,0.50)
        (2.00,0.75)--(2.00,0.50);

  \draw[fill=lightgray] (0.00,0.50)--(0.00,0.00)--cycle;
  \draw[fill=orange]    (1.00,0.50)--(2.00,0.50)--(1.50,0.00)--cycle;
\end{tikzpicture}
\]
and to compute the corresponding $F^3\NP{\qSPol}$-normal reductions.
Consequently,
 $\strat{4}\NP{\qSPol}=\bigsqcup\minsub\dlr{\gamma_{a,b,c,d}}{3}$ consists of 
\[
\begin{tikzpicture}[baseline={(base)}]
\begin{scope}[x=10pt,y=10pt,join=round,cap=round,thick,black,solid]

\coordinate (base) at (2.00,4.50);
\useasboundingbox (-0.60,-0.20) rectangle (4.60,11.10);

\draw (0.00,8.75)--(0.00,8.50) (2.00,8.75)--(2.00,8.25) (4.00,8.75)--(4.00,8.25) ;
\node at (0.00,8.00) {$\scriptstyle e$} ;
\draw [fill = green] (2.00,8.25)--(3.00,7.75)--(2.00,7.75)--(1.00,7.75)--cycle ;
\draw [fill = lightgray] (4.00,8.25)--(4.00,7.75)--cycle ;
\draw (0.00,7.50)--(0.00,7.25) (1.00,7.75)--(1.00,7.25) (2.00,7.75)--(2.00,7.25) (3.00,7.75)--(3.00,7.00) (4.00,7.75)--(4.00,7.00) ;
\draw (0.00,7.25)--(0.00,6.75) ;
\node at (0.50,7.00) {$\scriptstyle \raisebox{0pt}{\fontsize{6pt}{6pt}\selectfont$\leq\!2$}$} ;
\draw [fill = lightgray] (1.00,7.25)--(1.00,6.75)--cycle ;
\draw (2.00,7.25)--(2.00,6.75) ;
\node at (2.50,7.00) {$\scriptstyle \raisebox{0pt}{\fontsize{6pt}{6pt}\selectfont$\leq\!2$}$} ;
\draw [dash pattern = on 0.25pt off 2pt] (3.25,7.00)--(3.75,7.00) ;
\draw (0.00,6.75)--(0.00,6.25) (1.00,6.75)--(1.00,6.25) (2.00,6.75)--(2.00,6.50) (3.00,7.00)--(3.00,6.50) (4.00,7.00)--(4.00,6.50) ;
\draw [fill = orange] (0.00,6.25)--(1.00,6.25)--(0.50,5.75)--cycle ;
\node at (2.00,6.00) {$\scriptstyle c$} ;
\draw [rounded corners = 1pt, fill = lightgray] (2.75,5.50) rectangle (4.25,6.50) ;
\node at (3.50,6.00) {$\scriptstyle f$} ;
\draw (0.50,5.75)--(0.50,5.25) (2.00,5.50)--(2.00,5.25) (3.00,5.50)--(3.00,5.00) (4.00,5.50)--(4.00,5.00) ;
\draw [fill = lightgray] (0.50,5.25)--(0.50,4.75)--cycle ;
\draw [fill = lightgray] (2.00,5.25)--(2.00,4.75)--cycle ;
\draw [dash pattern = on 0.25pt off 2pt] (3.25,5.00)--(3.75,5.00) ;
\draw (0.50,4.75)--(0.50,4.50) (2.00,4.75)--(2.00,4.50) (3.00,5.00)--(3.00,4.50) (4.00,5.00)--(4.00,4.50) ;
\draw [fill = lightgray] (0.50,4.50)--(0.50,4.00)--cycle ;
\draw [fill = orange] (2.00,4.50)--(3.00,4.50)--(2.50,4.00)--cycle ;
\draw [fill = lightgray] (4.00,4.50)--(4.00,4.00)--cycle ;
\draw (0.50,4.00)--(0.50,3.75) (2.50,4.00)--(2.50,3.75) (4.00,4.00)--(4.00,3.75) ;
\draw [fill = lightgray] (0.50,3.75)--(0.50,3.25)--cycle ;
\draw [fill = green] (2.50,3.75)--(3.00,3.25)--(2.00,3.25)--cycle ;
\draw [fill = lightgray] (4.00,3.75)--(4.00,3.25)--cycle ;
\draw (0.50,3.25)--(0.50,2.75) (2.00,3.25)--(2.00,3.00) (3.00,3.25)--(3.00,2.75) (4.00,3.25)--(4.00,2.75) ;
\draw [fill = lightgray] (0.50,2.75)--(0.50,2.25)--cycle ;
\node at (2.00,2.50) {$\scriptstyle d$} ;
\draw [fill = lightgray] (3.00,2.75)--(3.00,2.25)--cycle ;
\draw [fill = lightgray] (4.00,2.75)--(4.00,2.25)--cycle ;
\draw (0.50,2.25)--(0.50,1.75) (2.00,2.00)--(2.00,1.75) (3.00,2.25)--(3.00,1.75) (4.00,2.25)--(4.00,1.75) ;
\draw [fill = orange] (0.50,1.75)--(2.00,1.75)--(1.25,1.25)--cycle ;
\draw [fill = lightgray] (3.00,1.75)--(3.00,1.25)--cycle ;
\draw [fill = lightgray] (4.00,1.75)--(4.00,1.25)--cycle ;
\draw (1.25,1.25)--(1.25,1.00) (3.00,1.25)--(3.00,1.00) (4.00,1.25)--(4.00,1.00) ;
\node at (1.25,0.50) {$\scriptstyle a+e$} ;
\node at (3.00,0.50) {$\scriptstyle b$} ;

\draw[blue,very thick,dashed] (-0.35,9.2) rectangle (4.5,3.85);

\end{scope}
\end{tikzpicture}
\quad
\longrightarrow
\quad
\sum_s\Psivec{s}{a}{b}{c}{d}\;
\begin{tikzpicture}[baseline={(base)}]
\begin{scope}[x=10pt,y=10pt,join=round,cap=round,thick,black,solid]

\coordinate (base) at (2.00,4.50);
\useasboundingbox (-0.60,-0.20) rectangle (4.60,11.10);

\draw (0.00,10.75)--(0.00,10.50) (1.75,10.75)--(1.75,10.50) (4.00,10.75)--(4.00,10.50) ;
\draw [fill = lightgray] (0.00,10.50)--(0.00,10.00)--cycle ;
\draw [fill = green] (1.75,10.50)--(2.50,10.00)--(1.00,10.00)--cycle ;
\draw [fill = lightgray] (4.00,10.50)--(4.00,10.00)--cycle ;
\draw (0.00,10.00)--(0.00,9.75) (1.00,10.00)--(1.00,9.75) (2.50,10.00)--(2.50,9.50) (4.00,10.00)--(4.00,9.50) ;
\draw [fill = lightgray] (0.00,9.75)--(0.00,9.25)--cycle ;
\draw (1.00,9.75)--(1.00,9.25) ;
\node at (1.50,9.50) {$\scriptstyle \raisebox{0pt}{\fontsize{6pt}{6pt}\selectfont\;\;$\leq\!2$}$} ;
\draw [dash pattern = on 0.25pt off 2pt] (2.88,9.50)--(3.63,9.50) ;
\draw (0.00,9.25)--(0.00,8.75) (1.00,9.25)--(1.00,8.75) (2.50,9.50)--(2.50,9.00) (4.00,9.50)--(4.00,9.00) ;
\draw [fill = lightgray] (0.00,8.75)--(0.00,8.25)--cycle ;
\draw [fill = lightgray] (1.00,8.75)--(1.00,8.25)--cycle ;
\draw [rounded corners = 1pt, fill = lightgray] (2.25,8.00) rectangle (4.25,9.00) ;
\node at (3.25,8.50) {$\scriptstyle f$} ;
\draw (0.00,8.25)--(0.00,7.75) (1.00,8.25)--(1.00,7.75) (2.50,8.00)--(2.50,7.50) (4.00,8.00)--(4.00,7.50) ;
\draw [fill = lightgray] (0.00,7.75)--(0.00,7.25)--cycle ;
\draw [fill = lightgray] (1.00,7.75)--(1.00,7.25)--cycle ;
\draw [dash pattern = on 0.25pt off 2pt] (2.88,7.50)--(3.63,7.50) ;
\draw (0.00,7.25)--(0.00,7.00) (1.00,7.25)--(1.00,7.00) (2.50,7.50)--(2.50,7.00) (4.00,7.50)--(4.00,7.00) ;
\draw [fill = lightgray] (0.00,7.00)--(0.00,6.50)--cycle ;
\draw [fill = lightgray] (1.00,7.00)--(1.00,6.50)--cycle ;
\draw [fill = green] (2.50,7.00)--(3.00,6.50)--(2.00,6.50)--cycle ;
\draw [fill = lightgray] (4.00,7.00)--(4.00,6.50)--cycle ;
\draw (0.00,6.50)--(0.00,6.00) (1.00,6.50)--(1.00,6.00) (2.00,6.50)--(2.00,6.25) (3.00,6.50)--(3.00,6.00) (4.00,6.50)--(4.00,6.00) ;
\draw [fill = lightgray] (0.00,6.00)--(0.00,5.50)--cycle ;
\draw [fill = lightgray] (1.00,6.00)--(1.00,5.50)--cycle ;
\node at (2.00,5.75) {$\scriptstyle d-s$} ;
\draw [fill = lightgray] (3.00,6.00)--(3.00,5.50)--cycle ;
\draw [fill = lightgray] (4.00,6.00)--(4.00,5.50)--cycle ;
\draw (0.00,5.50)--(0.00,5.00) (1.00,5.50)--(1.00,5.00) (2.00,5.25)--(2.00,5.00) (3.00,5.50)--(3.00,5.00) (4.00,5.50)--(4.00,5.00) ;
\draw [fill = lightgray] (0.00,5.00)--(0.00,4.50)--cycle ;
\draw [fill = orange] (1.00,5.00)--(2.00,5.00)--(1.50,4.50)--cycle ;
\draw [fill = lightgray] (3.00,5.00)--(3.00,4.50)--cycle ;
\draw [fill = lightgray] (4.00,5.00)--(4.00,4.50)--cycle ;
\draw (0.00,4.50)--(0.00,4.25) (1.50,4.50)--(1.50,4.00) (3.00,4.50)--(3.00,4.00) (4.00,4.50)--(4.00,4.00) ;
\node at (0.00,3.75) {$\scriptstyle e$} ;
\draw [fill = green] (1.50,4.00)--(2.00,3.50)--(1.00,3.50)--cycle ;
\draw [fill = lightgray] (3.00,4.00)--(3.00,3.50)--cycle ;
\draw [fill = lightgray] (4.00,4.00)--(4.00,3.50)--cycle ;
\draw (0.00,3.25)--(0.00,2.75) (1.00,3.50)--(1.00,2.75) (2.00,3.50)--(2.00,3.00) (3.00,3.50)--(3.00,2.75) (4.00,3.50)--(4.00,2.75) ;
\draw (0.00,2.75)--(0.00,2.25) ;
\node at (0.50,2.50) {$\scriptstyle \raisebox{0pt}{\fontsize{6pt}{6pt}\selectfont$\leq\!2$}$} ;
\draw [fill = lightgray] (1.00,2.75)--(1.00,2.25)--cycle ;
\node at (2.00,2.50) {$\scriptstyle c-s$} ;
\draw [fill = lightgray] (3.00,2.75)--(3.00,2.25)--cycle ;
\draw [fill = lightgray] (4.00,2.75)--(4.00,2.25)--cycle ;
\draw (0.00,2.25)--(0.00,1.75) (1.00,2.25)--(1.00,1.75) (2.00,2.00)--(2.00,1.75) (3.00,2.25)--(3.00,1.75) (4.00,2.25)--(4.00,1.75) ;
\draw [fill = orange] (0.00,1.75)--(1.00,1.75)--(0.50,1.25)--cycle ;
\draw [fill = orange] (2.00,1.75)--(3.00,1.75)--(2.50,1.25)--cycle ;
\draw [fill = lightgray] (4.00,1.75)--(4.00,1.25)--cycle ;
\draw (0.50,1.25)--(0.50,1.00) (2.50,1.25)--(2.50,1.00) (4.00,1.25)--(4.00,1.00) ;
\node at (0.50,0.50) {$\scriptstyle a+e$} ;
\node at (2.50,0.50) {$\scriptstyle b$} ;

\end{scope}
\end{tikzpicture}\;\raisebox{-8ex}{,}
\]
where $f\in \pcat{\qSPol_2}$ such that the $2$-cells of the form of the box of the source of the above $3$-cell belong to $\Nf(F^3\NP{\qSPol})$.
We also write the above rule as $\gamma_{a,b,c,d}$ when no confusion arises.

\begin{lemma}
\label{L:ConfluentOfS3}
The extension $\strat{4}\NP{\qSPol}$ is confluent.
\end{lemma}
\begin{proof}
We list below the three families of generic critical branchings:
\[
\left(\twocell{(emptys*0 gs)*1(merge3*0id)*1(id*0c*0split) *1(id*0id*0merge)*1(id*0split*0id)*1 (id*0merge*0id)*1(id*0 d*0id*0id)*1(split*0id*0id)*1(ax*0bx*0fx)} \,
\gamma
\begin{bmatrix}
b+d\\f\\
b+d-c+f-g\\ b+d-c
\end{bmatrix},
\gamma
\begin{bmatrix}
a\\
b\\
c\\
d
\end{bmatrix}
\right)\,\raisebox{-8ex}{,}
\quad
\left(\twocell{(merge*0id*0id)*1(id*0oneg*0id*0id)*1(id*0split*0id) *1(id*0merge*0id)*1(id*0id*0c*0id)*1(id*0onef*0split)*1(id*0id*0merge)*1 (id*0id*0d*0id)*1(split3*0id)*1(ax*0bx)}\,,
\gamma
\begin{bmatrix}
f+d \\b\\
c\\d
\end{bmatrix},
\gamma
\begin{bmatrix}
a-d\\c\\
g\\f
\end{bmatrix}\right)\,\raisebox{-8.5ex}{,}
\quad
\left(\twocell{(merge3*0id)*1(id*0c*0 f)*1(id*0split*0id) *1(id*0merge*0id)*1(id*0d*0id*0id)*1(split*0id*0id)*1(ax*0bx*0emptyx)}\,,
\gamma
\begin{bmatrix}
a\\b\\
c\\d
\end{bmatrix},\gamma'\right)\,\raisebox{-6.5ex}{,}
\]
where $\gamma':\twocell{f}\fl\sum_i \lambda_i\twocell{fi}$ is a $F^4 \NP{\qSPol}$-rewriting step.
Here we prove confluence for the last family, while that of the other two is proved in Appendix~\ref{ASSS:CriticalBranchingsOfS3}.
\[
\xymatrix @R=1em {
\twocell{(merge3*0id)*1(id*0c*0 f)*1(id*0split*0id) *1(id*0merge*0id)*1(id*0d*0id*0id)*1(split*0id*0id)*1(ax*0bx*0emptyx)}
  \ar[r]^{\hspace{-0.4cm}}
  \ar[d]_{\hspace{0.5cm}}
&
\sum_i \lambda_i\,
\twocell{(merge3*0id)*1(id*0c*0 fi)*1(id*0split*0id) *1(id*0merge*0id)*1(id*0d*0id*0id)*1(split*0id*0id)*1(ax*0bx*0emptyx)}
  \ar[d]_{\hspace{0.5cm}}
\\
\sum_s\Psivec{s}{a}{b}{c}{d}
\twocell{(merge*0id)*1(id*0f)*1(id *0 merge*0id)*1 (id *0 djs *0 id*0id) *1 (split*0 id*0id)*1 (merge*0 id*0id) *1 (id *0 cjs *0 id*0id) *1 (id *0 split*0id) *1(ax *0 bx*0emptyx)}
  \ar[r]^{\hspace{-0.4cm}}
&
\sum_{s,i} \lambda_i\Psivec{s}{a}{b}{c}{d}
\twocell{(merge*0id)*1(id*0fi)*1(id *0 merge*0id)*1 (id *0 djs *0 id*0id) *1 (split*0 id*0id)*1 (merge*0 id*0id) *1 (id *0 cjs *0 id*0id) *1 (id *0 split*0id) *1(ax *0 bx*0emptyx)}\,\raisebox{-6.5ex}{.}
}
\]
\end{proof}

By Theorem~\ref{T:MainTheorem}, the extension $F^4\NP{\qSPol}$ is confluent.
The shapes of the minimal $2$-cells in $\Nf(\strat{4}\NP{\qSPol})$ are described in pictorial language in \SSS{SSS:NewPictureAlgorithm}.
We now compute the final extension $\strat{5}\NP{\qSPol}$.
Without loss of generality, we restrict to active contexts $C$ such that $C[s\eta]$ is not in $\Nf(\strat{4}\NP{\qSPol})$ in order to compute the corresponding $F^4\NP{\qSPol}$-normal reductions.
There are three cases.

\subsubsection{Active context of type I}
For $\eta_{a,b,c}\in \strat{5}\qSPol$, we consider the $F^4\NP{\qSPol}$-active context $D_1=\twocell{(id*0merge)*1(square*0d)*1(id*0split)}$. Its application to $\eta_{a,b,c}$ gives
\[
D_1[\eta_{a,b,c}]:
\twocell{(id *0 merge)*1(id *0 merge*0 id)*1 (id *0 a *0 id*0 id)  *1 (split*0 id*0 id)*1 (merge*0 id*0 d) *1 (id *0 c *0 id*0 id)  *1 (id *0 split*0 id)*1 (id *0 split) *1(ax *0 bjdx)}
\fl
\twocell{(id *0 merge)*1 (c *0 id *0 id)  *1 (split*0 id)*1 (merge*0 id) *1 (id *0 b *0 id)  *1 (id *0 split) *1(ax *0 bjdx)}
-\sum_{s=1}^{\min (c, a)}{c-b\brack s}_q
\twocell{(id *0 merge)*1(id *0 merge*0 id)*1 (id *0 ajs *0 id*0 id)  *1 (split*0 id*0 id)*1 (merge*0 id*0 d) *1 (id *0 cjs *0 id*0 id)  *1 (id *0 split*0 id)*1 (id *0 split) *1(ax *0 bjdx)}\,\raisebox{-6.5ex}{.}
\]
Under the order $\preccurlyeq_{\mathrm{Sch}}$, its $F^4\NP{\qSPol}$-normal reduction is
\[
\down{D_1[\eta_{a,b,c}]}{4}^\preccurlyeq:
\twocell{(id *0 merge)*1 (c *0 id *0 id)  *1 (split*0 id)*1 (merge*0 id) *1 (id *0 b *0 id)  *1 (id *0 split) *1(ax *0 bjdx)}
\fl
\sum_{s=1}^{\min (c, a)}{c-b\brack s}_q
{b-c+s+d\brack d}_q
\twocell{(id *0 merge)*1 (c *0 ajs *0 id)  *1 (split*0 id)*1 (merge*0 id) *1 (id *0 cjs *0 id)  *1 (id *0 split) *1(ax *0 bjdx)}
+
{
b-c+d \brack
d
}_q 
\twocell{(id *0 merge)*1 (c *0 a *0 id)  *1 (split*0 id)*1 (merge*0 id) *1 (id *0 c *0 id)  *1 (id *0 split) *1(ax *0 bjdx)}\,,
\]
since $\twocell{merge*1(ax*0 bx)} \succ_{\q} \twocell{merge*1(ax*0 cjsx)}$ and $\twocell{merge*1(ax*0 bx)} \succ_{\q} \twocell{merge*1(ax*0 cx)}$ for $b > c$.
The target of $\down{D_1[\eta_{a,b,c}]}{4}^\preccurlyeq$ is not in $\Nf(\strat{5}\NP{\qSPol})$, since we can further apply $\eta_{a,b+d,c}$ to the second term on the right-hand side above, which yields a new rule in $\strat{5}\NP{\qSPol}$, denoted by $\zeta_{a,b,c,d}$:
\[
\twocell{(id *0 merge)*1 (c *0 id *0 id)  *1 (split*0 id)*1 (merge*0 id) *1 (id *0 b *0 id)  *1 (id *0 split) *1(ax *0 bjdx)}
\fl
\left[\begin{array}{c}
b-c+d \\
d
\end{array}\right]_q 
\twocell{(cs*0 id)*1split*1 merge*1(ax*0 bjdx)}
+\sum_{s=1}^{\min (c, a)}
\left(
\left[\begin{array}{c}
c-b \\
s
\end{array}\right]_q 
\left[\begin{array}{c}
b-c+s+d \\
d
\end{array}\right]_q 
-
\left[\begin{array}{c}
b-c+d \\
d
\end{array}\right]_q 
\left[\begin{array}{c}
c-b-d \\
s
\end{array}\right]_q 
\right)
\twocell{(id *0 merge)*1 (c *0 ajs *0 id)  *1 (split*0 id)*1 (merge*0 id) *1 (id *0 cjs *0 id)  *1 (id *0 split) *1(ax *0 bjdx)}\,.
\]
In fact, when $b=c$, we have $\zeta_{a,b,c,d}=\eta_{a,b+d,c}$, which shows that the rule $\eta$ can be replaced by $\zeta$ in $\strat{5}\NP{\qSPol}$.
For simplicity, for $b \geq c$, we write
\[
\zeta_{a,b,c,d}:
\twocell{(id *0 merge)*1 (c *0 id *0 id)  *1 (split*0 id)*1 (merge*0 id) *1 (id *0 b *0 id)  *1 (id *0 split) *1(ax *0 bjdx)}
\fl
\sum_s
\Phivec{s}{c}{b}{d}
\twocell{(id *0 merge)*1 (c *0 ajs *0 id)  *1 (split*0 id)*1 (merge*0 id) *1 (id *0 cjs *0 id)  *1 (id *0 split) *1(ax *0 bjdx)}\,,
\]
where
\[
\sum_s 
\coloneq
\sum_{\substack{s=1 \\ s=c-b-d}}^{\min(c,a)}
\quad\text{and}\quad
\Phivec{s}{c}{b}{d}
\coloneq
\left(
\left[\begin{array}{c} c-b \\ s \end{array}\right]_q
\left[\begin{array}{c} b-c+s+d \\ d \end{array}\right]_q
-
\left[\begin{array}{c} b-c+d \\ d \end{array}\right]_q
\left[\begin{array}{c} c-b-d \\ s \end{array}\right]_q
\right).
\]

\subsubsection{Active context of type II}
\label{SSS:AdmissibleContextD2}
We consider the $F^4\NP{\qSPol}$-active context $D_2$ of the form
$	
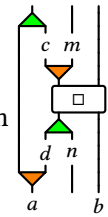
\raisebox{-36.25pt}{\begin{tikzpicture} \begin{scope} [ x = 10pt, y = 10pt, join = round, cap = round, thick, black, solid, -] \draw (0.50,8.00)--(0.50,7.75) (2.00,8.00)--(2.00,7.75) (3.00,8.00)--(3.00,7.75) ; \draw  ; \draw [fill = green] (0.50,7.75)--(1.00,7.25)--(0.00,7.25)--cycle ; \draw [fill = lightgray] (2.00,7.75)--(2.00,7.25)--cycle ; \draw [fill = lightgray] (3.00,7.75)--(3.00,7.25)--cycle ; \draw (0.00,7.25)--(0.00,6.75) (1.00,7.25)--(1.00,7.00) (2.00,7.25)--(2.00,7.00) (3.00,7.25)--(3.00,6.75) ; \draw [fill = lightgray] (0.00,6.75)--(0.00,6.25)--cycle ; \node at (1.00,6.50) {$\scriptstyle c$} ; \node at (2.00,6.50) {$\scriptstyle m$} ; \draw [fill = lightgray] (3.00,6.75)--(3.00,6.25)--cycle ; \draw (0.00,6.25)--(0.00,5.75) (1.00,6.00)--(1.00,5.75) (2.00,6.00)--(2.00,5.75) (3.00,6.25)--(3.00,5.75) ; \draw [fill = lightgray] (0.00,5.75)--(0.00,5.25)--cycle ; \draw [fill = orange] (1.00,5.75)--(2.00,5.75)--(1.50,5.25)--cycle ; \draw [fill = lightgray] (3.00,5.75)--(3.00,5.25)--cycle ; \draw (0.00,5.25)--(0.00,4.75) (1.50,5.25)--(1.50,5.00) (3.00,5.25)--(3.00,5.00) ; \draw [fill = lightgray] (0.00,4.75)--(0.00,4.25)--cycle ; \draw [rounded corners = 1pt, fill = white] (1.25,4.00) rectangle (3.25,5.00) ; \node at (2.25,4.50) {$\scriptstyle \square$} ; \draw (0.00,4.25)--(0.00,3.75) (1.50,4.00)--(1.50,3.75) (3.00,4.00)--(3.00,3.75) ; \draw [fill = lightgray] (0.00,3.75)--(0.00,3.25)--cycle ; \draw [fill = green] (1.50,3.75)--(2.00,3.25)--(1.00,3.25)--cycle ; \draw [fill = lightgray] (3.00,3.75)--(3.00,3.25)--cycle ; \draw (0.00,3.25)--(0.00,2.75) (1.00,3.25)--(1.00,3.00) (2.00,3.25)--(2.00,3.00) (3.00,3.25)--(3.00,2.75) ; \draw [fill = lightgray] (0.00,2.75)--(0.00,2.25)--cycle ; \node at (1.00,2.50) {$\scriptstyle d$} ; \node at (2.00,2.50) {$\scriptstyle n$} ; \draw [fill = lightgray] (3.00,2.75)--(3.00,2.25)--cycle ; \draw (0.00,2.25)--(0.00,1.75) (1.00,2.00)--(1.00,1.75) (2.00,2.00)--(2.00,1.75) (3.00,2.25)--(3.00,1.75) ; \draw [fill = orange] (0.00,1.75)--(1.00,1.75)--(0.50,1.25)--cycle ; \draw [fill = lightgray] (2.00,1.75)--(2.00,1.25)--cycle ; \draw [fill = lightgray] (3.00,1.75)--(3.00,1.25)--cycle ; \draw (0.50,1.25)--(0.50,1.00) (2.00,1.25)--(2.00,1.00) (3.00,1.25)--(3.00,1.00) ; \node at (0.50,0.50) {$\scriptstyle a$} ; \node at (2.00,0.50) {$\scriptstyle $} ; \node at (3.00,0.50) {$\scriptstyle b$} ; \end{scope} \end{tikzpicture}}\,.
$
We then compute 
$
\down{sD_2[\zeta_{d+n,y,c+m,b-y}]}{4}^\preccurlyeq
$
and
$
\down{tD_2[\zeta_{d+n,y,c+m,b-y}]}{4}^\preccurlyeq,
$
which are respectively given by
\[
\sum_s
\Psivec{s}{a}{n+y}{c}{d}\twocell{(id*0id*0merge)*1(id*0m*0x*0id)*1(id*0split*0id) *1(id*0merge*0id)*1(id*0djs*0id*0id)*1(split*0id*0id) *1(merge*0id*0id)*1(a*0cjs*0id*0id)*1(id*0split*0id) *1(id*0merge*0id)*1(id*0n*0y*0id)*1(id*0id*0split)*1(id*0id*0bx)}
\qquad\text{and}\qquad
\sum_{s,t}
\Phivec{t}{c+m}{y}{b-y}\Psivec{s}{a}{n+c+m-t}{c}{d} 
\twocell{(id*0m*0merge)*1(id*0id*0djnjt*0id)*1(id*0split*0id) *1(id*0merge*0id)*1(id*0djs*0id*0id)*1(split*0id*0id) *1(merge*0id*0id)*1(a*0cjs*0id*0id)*1(id*0split*0id) *1(id*0merge*0id)*1(id*0id*0cjmjt*0id)*1(id*0id*0split)*1(id*0n*0bx)}\,\raisebox{-8.5ex}{.}
\]
where $c+m\leq y$ and $x=y+d+n-c-m$.
Under the order $\preccurlyeq_{\mathrm{Sch}}$, the corresponding $F^4\NP{\qSPol}$-normal reduction is given by
\[
\down{D_2[\zeta]}{4}:
\twocell{(id*0id*0merge)*1(id*0m*0x*0id)*1(id*0split*0id) *1(id*0merge*0id)*1(id*0d*0id*0id)*1(split*0id*0id) *1(merge*0id*0id)*1(a*0c*0id*0id)*1(id*0split*0id) *1(id*0merge*0id)*1(id*0n*0y*0id)*1(id*0id*0split)*1(id*0id*0bx)}
\fl
-\sum_{s=1}
\Psivec{s}{a}{n+y}{c}{d}\twocell{(id*0id*0merge)*1(id*0m*0x*0id)*1(id*0split*0id) *1(id*0merge*0id)*1(id*0djs*0id*0id)*1(split*0id*0id) *1(merge*0id*0id)*1(a*0cjs*0id*0id)*1(id*0split*0id) *1(id*0merge*0id)*1(id*0n*0y*0id)*1(id*0id*0split)*1(id*0id*0bx)}
+
\sum_{s,t}
\Phivec{t}{c+m}{y}{b-y}\Psivec{s}{a}{n+c+m-t}{c}{d} 
\twocell{(id*0m*0merge)*1(id*0id*0djnjt*0id)*1(id*0split*0id) *1(id*0merge*0id)*1(id*0djs*0id*0id)*1(split*0id*0id) *1(merge*0id*0id)*1(a*0cjs*0id*0id)*1(id*0split*0id) *1(id*0merge*0id)*1(id*0id*0cjmjt*0id)*1(id*0id*0split)*1(id*0n*0bx)}\,\raisebox{-8.5ex}{.}
\]
The target of $\down{D_2[\zeta]}{4}^\preccurlyeq$ is not in $\Nf(\strat{5}\NP{\qSPol})$, since we can further apply $\down{D_2[\zeta]}{4}^\preccurlyeq$ to the first term on the right-hand side above, which yields a new rule in $\strat{5}\NP{\qSPol}$, denoted by $\theta_1$: 
\[
\twocell{(id*0id*0merge)*1(id*0m*0x*0id)*1(id*0split*0id) *1(id*0merge*0id)*1(id*0d*0id*0id)*1(split*0id*0id) *1(merge*0id*0id)*1(a*0c*0id*0id)*1(id*0split*0id) *1(id*0merge*0id)*1(id*0n*0y*0id)*1(id*0id*0split)*1(id*0id*0bx)}
\fl
\sum_t
\Phivec{t}{c+m}{y}{b-y}
\twocell{(id*0m*0merge)*1(id*0id*0djnjt*0id)*1(id*0split*0id) *1(id*0merge*0id)*1(id*0d*0id*0id)*1(split*0id*0id) *1(merge*0id*0id)*1(a*0c*0id*0id)*1(id*0split*0id) *1(id*0merge*0id)*1(id*0id*0cjmjt*0id)*1(id*0id*0split)*1(id*0n*0bx)}\,\raisebox{-8.5ex}{,}
\]
where $a>d$ and $c+m\leq y$.

In fact, by repeating the above construction, we can apply the context $D_2$ to the rule $\theta_{k-1}$ inductively, thereby obtaining further rules in $\strat{5}\NP{\qSPol}$, which we denote by $\theta_{k}$:
\[
\begin{tikzpicture}[baseline=15ex] \begin{scope} [ x = 10pt, y = 10pt, join = round, cap = round, thick, black, solid, -] \draw (0.00,14.50)--(0.00,14.25) (1.00,14.50)--(1.00,14.25) (2.00,14.50)--(2.00,14.25) (3.50,14.50)--(3.50,14.25) ; \draw (0.00,0.00)--(0.00,-0.50) (1.00,0.00)--(1.00,-0.50) (2.00,0.00)--(2.00,-0.50) ; \draw [fill = lightgray] (0.00,14.25)--(0.00,13.75)--cycle ; \draw [fill = lightgray] (1.00,14.25)--(1.00,13.75)--cycle ; \draw [fill = lightgray] (2.00,14.25)--(2.00,13.75)--cycle ; \draw [fill = green] (3.50,14.25)--(4.00,13.75)--(3.00,13.75)--cycle ; \draw (0.00,13.75)--(0.00,13.25) (1.00,13.75)--(1.00,13.25) (2.00,13.75)--(2.00,13.50) (3.00,13.75)--(3.00,13.50) (4.00,13.75)--(4.00,13.25) ; \draw [fill = lightgray] (0.00,13.25)--(0.00,12.75)--cycle ; \draw [fill = lightgray] (1.00,13.25)--(1.00,12.75)--cycle ; \node at (2.00,13.00) {$\scriptstyle m$} ; \node at (3.00,13.00) {$\scriptstyle x$} ; \draw [fill = lightgray] (4.00,13.25)--(4.00,12.75)--cycle ; \draw (0.00,12.75)--(0.00,12.25) (1.00,12.75)--(1.00,12.25) (2.00,12.50)--(2.00,12.25) (3.00,12.50)--(3.00,12.25) (4.00,12.75)--(4.00,12.25) ; \draw [fill = lightgray] (0.00,12.25)--(0.00,11.75)--cycle ; \draw [fill = lightgray] (1.00,12.25)--(1.00,11.75)--cycle ; \draw [fill = orange] (2.00,12.25)--(3.00,12.25)--(2.50,11.75)--cycle ; \draw [fill = lightgray] (4.00,12.25)--(4.00,11.75)--cycle ; \draw (0.00,11.75)--(0.00,11.50) (1.00,11.75)--(1.00,11.50) (2.50,11.75)--(2.50,11.50) (4.00,11.75)--(4.00,11.50) ; \draw [fill = lightgray] (0.00,11.50)--(0.00,11.00)--cycle ; \draw [fill = lightgray] (1.00,11.50)--(1.00,11.00)--cycle ; \draw [fill = green] (2.50,11.50)--(3.00,11.00)--(2.00,11.00)--cycle ; \draw [fill = lightgray] (4.00,11.50)--(4.00,11.00)--cycle ; \draw (0.00,11.00)--(0.00,10.50) (1.00,11.00)--(1.00,10.50) (2.00,11.00)--(2.00,10.75) (3.00,11.00)--(3.00,10.50) (4.00,11.00)--(4.00,10.50) ; \draw [fill = lightgray] (0.00,10.50)--(0.00,10.00)--cycle ; \draw [fill = lightgray] (1.00,10.50)--(1.00,10.00)--cycle ; \node at (2.00,10.25) {$\scriptstyle d$} ; \draw [fill = lightgray] (3.00,10.50)--(3.00,10.00)--cycle ; \draw [fill = lightgray] (4.00,10.50)--(4.00,10.00)--cycle ; \draw (0.00,10.00)--(0.00,9.25) (1.00,10.00)--(1.00,9.50) (2.00,9.75)--(2.00,9.50) (3.00,10.00)--(3.00,9.25) (4.00,10.00)--(4.00,9.25) ; \draw [fill = lightgray] (0.00,9.25)--(0.00,8.75)--cycle ; \draw [rounded corners = 1pt, fill = lightgray] (0.75,8.50) rectangle (2.25,9.50) ; \node at (1.50,9.00) {$\scriptstyle f$} ; \draw [fill = lightgray] (3.00,9.25)--(3.00,8.75)--cycle ; \draw [fill = lightgray] (4.00,9.25)--(4.00,8.75)--cycle ; \draw (0.00,8.75)--(0.00,8.25) (1.00,8.50)--(1.00,8.25) (2.00,8.50)--(2.00,8.25) (3.00,8.75)--(3.00,8.25) (4.00,8.75)--(4.00,8.25) ; \draw [fill = orange] (0.00,8.25)--(1.00,8.25)--(0.50,7.75)--cycle ; \draw [fill = lightgray] (2.00,8.25)--(2.00,7.75)--cycle ; \draw [fill = lightgray] (3.00,8.25)--(3.00,7.75)--cycle ; \draw [fill = lightgray] (4.00,8.25)--(4.00,7.75)--cycle ; \draw (0.50,7.75)--(0.50,7.50) (2.00,7.75)--(2.00,7.50) (3.00,7.75)--(3.00,7.50) (4.00,7.75)--(4.00,7.50) ; \draw [fill = green] (0.50,7.50)--(1.00,7.00)--(0.00,7.00)--cycle ; \draw [fill = lightgray] (2.00,7.50)--(2.00,7.00)--cycle ; \draw [fill = lightgray] (3.00,7.50)--(3.00,7.00)--cycle ; \draw [fill = lightgray] (4.00,7.50)--(4.00,7.00)--cycle ; \draw (0.00,7.00)--(0.00,6.50) (1.00,7.00)--(1.00,6.75) (2.00,7.00)--(2.00,6.75) (3.00,7.00)--(3.00,6.50) (4.00,7.00)--(4.00,6.50) ; \draw [fill = lightgray] (0.00,6.50)--(0.00,6.00)--cycle ; \draw [rounded corners = 1pt, fill = lightgray] (0.75,5.75) rectangle (2.25,6.75) ; \node at (1.50,6.25) {$\scriptstyle g$} ; \draw [fill = lightgray] (3.00,6.50)--(3.00,6.00)--cycle ; \draw [fill = lightgray] (4.00,6.50)--(4.00,6.00)--cycle ; \draw (0.00,6.00)--(0.00,5.25) (1.00,5.75)--(1.00,5.25) (2.00,5.75)--(2.00,5.50) (3.00,6.00)--(3.00,5.25) (4.00,6.00)--(4.00,5.25) ; \draw [fill = lightgray] (0.00,5.25)--(0.00,4.75)--cycle ; \draw [fill = lightgray] (1.00,5.25)--(1.00,4.75)--cycle ; \node at (2.00,5.00) {$\scriptstyle c$} ; \draw [fill = lightgray] (3.00,5.25)--(3.00,4.75)--cycle ; \draw [fill = lightgray] (4.00,5.25)--(4.00,4.75)--cycle ; \draw (0.00,4.75)--(0.00,4.25) (1.00,4.75)--(1.00,4.25) (2.00,4.50)--(2.00,4.25) (3.00,4.75)--(3.00,4.25) (4.00,4.75)--(4.00,4.25) ; \draw [fill = lightgray] (0.00,4.25)--(0.00,3.75)--cycle ; \draw [fill = lightgray] (1.00,4.25)--(1.00,3.75)--cycle ; \draw [fill = orange] (2.00,4.25)--(3.00,4.25)--(2.50,3.75)--cycle ; \draw [fill = lightgray] (4.00,4.25)--(4.00,3.75)--cycle ; \draw (0.00,3.75)--(0.00,3.50) (1.00,3.75)--(1.00,3.50) (2.50,3.75)--(2.50,3.50) (4.00,3.75)--(4.00,3.50) ; \draw [fill = lightgray] (0.00,3.50)--(0.00,3.00)--cycle ; \draw [fill = lightgray] (1.00,3.50)--(1.00,3.00)--cycle ; \draw [fill = green] (2.50,3.50)--(3.00,3.00)--(2.00,3.00)--cycle ; \draw [fill = lightgray] (4.00,3.50)--(4.00,3.00)--cycle ; \draw (0.00,3.00)--(0.00,2.50) (1.00,3.00)--(1.00,2.50) (2.00,3.00)--(2.00,2.75) (3.00,3.00)--(3.00,2.75) (4.00,3.00)--(4.00,2.50) ; \draw [fill = lightgray] (0.00,2.50)--(0.00,2.00)--cycle ; \draw [fill = lightgray] (1.00,2.50)--(1.00,2.00)--cycle ; \node at (2.00,2.25) {$\scriptstyle n$} ; \node at (3.00,2.25) {$\scriptstyle y$} ; \draw [fill = lightgray] (4.00,2.50)--(4.00,2.00)--cycle ; \draw (0.00,2.00)--(0.00,1.50) (1.00,2.00)--(1.00,1.50) (2.00,1.75)--(2.00,1.50) (3.00,1.75)--(3.00,1.50) (4.00,2.00)--(4.00,1.50) ; \draw [fill = lightgray] (0.00,1.50)--(0.00,1.00)--cycle ; \draw [fill = lightgray] (1.00,1.50)--(1.00,1.00)--cycle ; \draw [fill = lightgray] (2.00,1.50)--(2.00,1.00)--cycle ; \draw [fill = orange] (3.00,1.50)--(4.00,1.50)--(3.50,1.00)--cycle ; \draw (0.00,1.00)--(0.00,0.50) (1.00,1.00)--(1.00,0.50) (2.00,1.00)--(2.00,0.50) (3.50,1.00)--(3.50,0.75) ; \draw [fill = lightgray] (0.00,0.50)--(0.00,0.00)--cycle ; \draw [fill = lightgray] (1.00,0.50)--(1.00,0.00)--cycle ; \draw [fill = lightgray] (2.00,0.50)--(2.00,0.00)--cycle ; \node at (3.50,0.25) {$\scriptstyle b$} ; \end{scope}
\draw[blue, very thick,dashed]
(-0.18,1.3) rectangle (1.2,4.1);
 \end{tikzpicture}
\,\fll\,
\begin{tikzpicture}[baseline=16.5ex] \begin{scope} [ x = 10pt, y = 10pt, join = round, cap = round, thick, black, solid, -] \draw (0.00,15.50)--(0.00,15.00) (1.00,15.50)--(1.00,15.00) (2.00,15.50)--(2.00,15.25) (3.50,15.50)--(3.50,15.00) ; \draw (0.00,0.00)--(0.00,-0.50) (1.00,0.00)--(1.00,-0.50) (2.00,0.00)--(2.00,-0.50) ; \draw [fill = lightgray] (0.00,15.00)--(0.00,14.50)--cycle ; \draw [fill = lightgray] (1.00,15.00)--(1.00,14.50)--cycle ; \node at (2.00,14.75) {$\scriptstyle m$} ; \draw [fill = green] (3.50,15.00)--(4.00,14.50)--(3.00,14.50)--cycle ; \draw (0.00,14.50)--(0.00,13.75) (1.00,14.50)--(1.00,13.75) (2.00,14.25)--(2.00,13.75) (3.00,14.50)--(3.00,14.00) (4.00,14.50)--(4.00,13.75) ; \draw [fill = lightgray] (0.00,13.75)--(0.00,13.25)--cycle ; \draw [fill = lightgray] (1.00,13.75)--(1.00,13.25)--cycle ; \draw [fill = lightgray] (2.00,13.75)--(2.00,13.25)--cycle ; \node at (3.00,13.50) {$\scriptstyle \raisebox{0pt}{\fontsize{5pt}{5pt}\selectfont$d\!+\!n\!\!-\!\!t$}$} ; \draw [fill = lightgray] (4.00,13.75)--(4.00,13.25)--cycle ; \draw (0.00,13.25)--(0.00,12.75) (1.00,13.25)--(1.00,12.75) (2.00,13.25)--(2.00,12.75) (3.00,13.00)--(3.00,12.75) (4.00,13.25)--(4.00,12.75) ; \draw [fill = lightgray] (0.00,12.75)--(0.00,12.25)--cycle ; \draw [fill = lightgray] (1.00,12.75)--(1.00,12.25)--cycle ; \draw [fill = orange] (2.00,12.75)--(3.00,12.75)--(2.50,12.25)--cycle ; \draw [fill = lightgray] (4.00,12.75)--(4.00,12.25)--cycle ; \draw (0.00,12.25)--(0.00,12.00) (1.00,12.25)--(1.00,12.00) (2.50,12.25)--(2.50,12.00) (4.00,12.25)--(4.00,12.00) ; \draw [fill = lightgray] (0.00,12.00)--(0.00,11.50)--cycle ; \draw [fill = lightgray] (1.00,12.00)--(1.00,11.50)--cycle ; \draw [fill = green] (2.50,12.00)--(3.00,11.50)--(2.00,11.50)--cycle ; \draw [fill = lightgray] (4.00,12.00)--(4.00,11.50)--cycle ; \draw (0.00,11.50)--(0.00,11.00) (1.00,11.50)--(1.00,11.00) (2.00,11.50)--(2.00,11.25) (3.00,11.50)--(3.00,11.00) (4.00,11.50)--(4.00,11.00) ; \draw [fill = lightgray] (0.00,11.00)--(0.00,10.50)--cycle ; \draw [fill = lightgray] (1.00,11.00)--(1.00,10.50)--cycle ; \node at (2.00,10.75) {$\scriptstyle d$} ; \draw [fill = lightgray] (3.00,11.00)--(3.00,10.50)--cycle ; \draw [fill = lightgray] (4.00,11.00)--(4.00,10.50)--cycle ; \draw (0.00,10.50)--(0.00,9.75) (1.00,10.50)--(1.00,10.00) (2.00,10.25)--(2.00,10.00) (3.00,10.50)--(3.00,9.75) (4.00,10.50)--(4.00,9.75) ; \draw [fill = lightgray] (0.00,9.75)--(0.00,9.25)--cycle ; \draw [rounded corners = 1pt, fill = lightgray] (0.75,9.00) rectangle (2.25,10.00) ; \node at (1.50,9.50) {$\scriptstyle f$} ; \draw [fill = lightgray] (3.00,9.75)--(3.00,9.25)--cycle ; \draw [fill = lightgray] (4.00,9.75)--(4.00,9.25)--cycle ; \draw (0.00,9.25)--(0.00,8.75) (1.00,9.00)--(1.00,8.75) (2.00,9.00)--(2.00,8.75) (3.00,9.25)--(3.00,8.75) (4.00,9.25)--(4.00,8.75) ; \draw [fill = orange] (0.00,8.75)--(1.00,8.75)--(0.50,8.25)--cycle ; \draw [fill = lightgray] (2.00,8.75)--(2.00,8.25)--cycle ; \draw [fill = lightgray] (3.00,8.75)--(3.00,8.25)--cycle ; \draw [fill = lightgray] (4.00,8.75)--(4.00,8.25)--cycle ; \draw (0.50,8.25)--(0.50,8.00) (2.00,8.25)--(2.00,8.00) (3.00,8.25)--(3.00,8.00) (4.00,8.25)--(4.00,8.00) ; \draw [fill = green] (0.50,8.00)--(1.00,7.50)--(0.00,7.50)--cycle ; \draw [fill = lightgray] (2.00,8.00)--(2.00,7.50)--cycle ; \draw [fill = lightgray] (3.00,8.00)--(3.00,7.50)--cycle ; \draw [fill = lightgray] (4.00,8.00)--(4.00,7.50)--cycle ; \draw (0.00,7.50)--(0.00,7.00) (1.00,7.50)--(1.00,7.25) (2.00,7.50)--(2.00,7.25) (3.00,7.50)--(3.00,7.00) (4.00,7.50)--(4.00,7.00) ; \draw [fill = lightgray] (0.00,7.00)--(0.00,6.50)--cycle ; \draw [rounded corners = 1pt, fill = lightgray] (0.75,6.25) rectangle (2.25,7.25) ; \node at (1.50,6.75) {$\scriptstyle g$} ; \draw [fill = lightgray] (3.00,7.00)--(3.00,6.50)--cycle ; \draw [fill = lightgray] (4.00,7.00)--(4.00,6.50)--cycle ; \draw (0.00,6.50)--(0.00,5.75) (1.00,6.25)--(1.00,5.75) (2.00,6.25)--(2.00,6.00) (3.00,6.50)--(3.00,5.75) (4.00,6.50)--(4.00,5.75) ; \draw [fill = lightgray] (0.00,5.75)--(0.00,5.25)--cycle ; \draw [fill = lightgray] (1.00,5.75)--(1.00,5.25)--cycle ; \node at (2.00,5.50) {$\scriptstyle c$} ; \draw [fill = lightgray] (3.00,5.75)--(3.00,5.25)--cycle ; \draw [fill = lightgray] (4.00,5.75)--(4.00,5.25)--cycle ; \draw (0.00,5.25)--(0.00,4.75) (1.00,5.25)--(1.00,4.75) (2.00,5.00)--(2.00,4.75) (3.00,5.25)--(3.00,4.75) (4.00,5.25)--(4.00,4.75) ; \draw [fill = lightgray] (0.00,4.75)--(0.00,4.25)--cycle ; \draw [fill = lightgray] (1.00,4.75)--(1.00,4.25)--cycle ; \draw [fill = orange] (2.00,4.75)--(3.00,4.75)--(2.50,4.25)--cycle ; \draw [fill = lightgray] (4.00,4.75)--(4.00,4.25)--cycle ; \draw (0.00,4.25)--(0.00,4.00) (1.00,4.25)--(1.00,4.00) (2.50,4.25)--(2.50,4.00) (4.00,4.25)--(4.00,4.00) ; \draw [fill = lightgray] (0.00,4.00)--(0.00,3.50)--cycle ; \draw [fill = lightgray] (1.00,4.00)--(1.00,3.50)--cycle ; \draw [fill = green] (2.50,4.00)--(3.00,3.50)--(2.00,3.50)--cycle ; \draw [fill = lightgray] (4.00,4.00)--(4.00,3.50)--cycle ; \draw (0.00,3.50)--(0.00,3.00) (1.00,3.50)--(1.00,3.00) (2.00,3.50)--(2.00,3.00) (3.00,3.50)--(3.00,3.25) (4.00,3.50)--(4.00,3.00) ; \draw [fill = lightgray] (0.00,3.00)--(0.00,2.50)--cycle ; \draw [fill = lightgray] (1.00,3.00)--(1.00,2.50)--cycle ; \draw [fill = lightgray] (2.00,3.00)--(2.00,2.50)--cycle ; \node at (3.00,2.75) {$\scriptstyle \raisebox{0pt}{\fontsize{5pt}{5pt}\selectfont$c\!+\!m\!\!-\!\!t$}$} ; \draw [fill = lightgray] (4.00,3.00)--(4.00,2.50)--cycle ; \draw (0.00,2.50)--(0.00,1.75) (1.00,2.50)--(1.00,1.75) (2.00,2.50)--(2.00,2.00) (3.00,2.25)--(3.00,1.75) (4.00,2.50)--(4.00,1.75) ; \draw [fill = lightgray] (0.00,1.75)--(0.00,1.25)--cycle ; \draw [fill = lightgray] (1.00,1.75)--(1.00,1.25)--cycle ; \node at (2.00,1.50) {$\scriptstyle n$} ; \draw [fill = orange] (3.00,1.75)--(4.00,1.75)--(3.50,1.25)--cycle ; \draw (0.00,1.25)--(0.00,0.50) (1.00,1.25)--(1.00,0.50) (2.00,1.00)--(2.00,0.50) (3.50,1.25)--(3.50,0.75) ; \draw [fill = lightgray] (0.00,0.50)--(0.00,0.00)--cycle ; \draw [fill = lightgray] (1.00,0.50)--(1.00,0.00)--cycle ; \draw [fill = lightgray] (2.00,0.50)--(2.00,0.00)--cycle ; \node at (3.50,0.25) {$\scriptstyle b$} ; \end{scope}

\draw[blue, very thick, dashed]
(-0.18,1.4) rectangle (1.2,4.25);

\end{tikzpicture}\,\raisebox{-10.5ex}{,}
\]
where $c+m\leq y$,
\[
\twocell{f}\coloneq
\begin{tikzpicture}[baseline=5ex]
\begin{scope}[
  x = 10pt, y = 10pt,
  join = round, cap = round,
  thick, black, solid, -
]

\draw (0.50,1.50)--(0.50,1.25);

\draw [fill = green]
(0.50,1.25)--(1.00,0.75)--(0.00,0.75)--cycle;

\draw (0.00,0.75)--(0.00,0.50)
      (1.00,0.75)--(1.00,0.50);

\draw[fill=orange]
(0.00,1.75)--(1.00,1.75)--(0.50,1.25)--cycle;

\draw (0.00,1.75)--(0.00,2.00);
\draw (1.00,1.75)--(1.00,2.00);

\begin{scope}[shift={(3.40,3.20)}]

  \draw (0.50,1.50)--(0.50,1.25);

  \draw [fill = green]
  (0.50,1.25)--(1.00,0.75)--(0.00,0.75)--cycle;

  \draw (0.00,0.75)--(0.00,0.50)
        (1.00,0.75)--(1.00,0.50);

  \draw[fill=orange]
  (0.00,1.75)--(1.00,1.75)--(0.50,1.25)--cycle;

  \draw (0.00,1.75)--(0.00,2.00);
  \draw (1.00,1.75)--(1.00,2.00);

\end{scope}

\draw[dashed, line width=1.2pt]
(1.00,2.00) -- (3.40,3.70)
node[midway, above, rotate=35] {$\scriptstyle k-1$};

\end{scope}
\end{tikzpicture}
\quad\text{and}\qquad
\twocell{g}\coloneq
\begin{tikzpicture}[baseline=-1.9ex]
\begin{scope}[
  x = 10pt, y = 10pt,
  join = round, cap = round,
  thick, black, solid, -
]

\draw (0.50,1.50)--(0.50,1.25);

\draw [fill = green]
(0.50,1.25)--(1.00,0.75)--(0.00,0.75)--cycle;

\draw (0.00,0.75)--(0.00,0.50)
      (1.00,0.75)--(1.00,0.50);

\draw[fill=orange]
(0.00,1.75)--(1.00,1.75)--(0.50,1.25)--cycle;

\draw (0.00,1.75)--(0.00,2.00);
\draw (1.00,1.75)--(1.00,2.00);

\begin{scope}[shift={(3.40,-3.20)}]

  \draw (0.50,1.50)--(0.50,1.25);

  \draw [fill = green]
  (0.50,1.25)--(1.00,0.75)--(0.00,0.75)--cycle;

  \draw (0.00,0.75)--(0.00,0.50)
        (1.00,0.75)--(1.00,0.50);

  \draw[fill=orange]
  (0.00,1.75)--(1.00,1.75)--(0.50,1.25)--cycle;

  \draw (0.00,1.75)--(0.00,2.00);
  \draw (1.00,1.75)--(1.00,2.00);

\end{scope}

\draw[dashed, line width=1.2pt]
(1.00,0.55) -- (3.40,-1.10)
node[midway, sloped, above] {$\scriptstyle k-1$};

\end{scope}
\end{tikzpicture}\,\,,
\]
and the $2$-cells in the above boxes cannot be reduced by $\theta_i$ for $i<k$.

\subsubsection{Active context of type III}
We consider the $F^4\NP{\qSPol}$-active context $D_3$ of the form
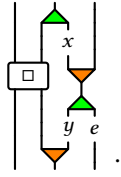
\[
\raisebox{-28.75pt}{\begin{tikzpicture} \begin{scope} [ x = 10pt, y = 10pt, join = round, cap = round, thick, black, solid, -] \draw (0.00,6.00)--(0.00,5.75) (1.50,6.00)--(1.50,5.75) (3.00,6.00)--(3.00,5.75) ; \draw (0.00,0.00)--(0.00,-0.25) (1.50,0.00)--(1.50,-0.25) (3.00,0.00)--(3.00,-0.25) ; \draw [fill = lightgray] (0.00,5.75)--(0.00,5.25)--cycle ; \draw [fill = green] (1.50,5.75)--(2.00,5.25)--(1.00,5.25)--cycle ; \draw [fill = lightgray] (3.00,5.75)--(3.00,5.25)--cycle ; \draw (0.00,5.25)--(0.00,4.75) (1.00,5.25)--(1.00,4.75) (2.00,5.25)--(2.00,5.00) (3.00,5.25)--(3.00,4.75) ; \draw [fill = lightgray] (0.00,4.75)--(0.00,4.25)--cycle ; \draw [fill = lightgray] (1.00,4.75)--(1.00,4.25)--cycle ; \node at (2.00,4.50) {$\scriptstyle x$} ; \draw [fill = lightgray] (3.00,4.75)--(3.00,4.25)--cycle ; \draw (0.00,4.25)--(0.00,3.75) (1.00,4.25)--(1.00,3.75) (2.00,4.00)--(2.00,3.50) (3.00,4.25)--(3.00,3.50) ; \draw [rounded corners = 1pt, fill = white] (-0.25,2.75) rectangle (1.25,3.75) ; \node at (0.50,3.25) {$\scriptstyle \square$} ; \draw [fill = orange] (2.00,3.50)--(3.00,3.50)--(2.50,3.00)--cycle ; \draw (0.00,2.75)--(0.00,2.50) (1.00,2.75)--(1.00,2.50) (2.50,3.00)--(2.50,2.50) ; \draw [fill = lightgray] (0.00,2.50)--(0.00,2.00)--cycle ; \draw [fill = lightgray] (1.00,2.50)--(1.00,2.00)--cycle ; \draw [fill = green] (2.50,2.50)--(3.00,2.00)--(2.00,2.00)--cycle ; \draw (0.00,2.00)--(0.00,1.50) (1.00,2.00)--(1.00,1.50) (2.00,2.00)--(2.00,1.75) (3.00,2.00)--(3.00,1.75) ; \draw [fill = lightgray] (0.00,1.50)--(0.00,1.00)--cycle ; \draw [fill = lightgray] (1.00,1.50)--(1.00,1.00)--cycle ; \node at (2.00,1.25) {$\scriptstyle y$} ; \node at (3.00,1.25) {$\scriptstyle e$} ; \draw (0.00,1.00)--(0.00,0.50) (1.00,1.00)--(1.00,0.50) (2.00,0.75)--(2.00,0.50) (3.00,0.75)--(3.00,0.50) ; \draw [fill = lightgray] (0.00,0.50)--(0.00,0.00)--cycle ; \draw [fill = orange] (1.00,0.50)--(2.00,0.50)--(1.50,0.00)--cycle ; \draw [fill = lightgray] (3.00,0.50)--(3.00,0.00)--cycle ; \end{scope} \end{tikzpicture}}\,\raisebox{-5.5ex}{.}
\]
We then compute 
$
\down{sD_3[\zeta_{a,b,c,d}]}{4}^\preccurlyeq
$
and
$
\down{tD_3[\zeta_{a,b,c,d}]}{4}^\preccurlyeq,
$
which are respectively given by
\[
\sum_t
\Psivec{t}{d+y}{e}{x}{y}
\twocell{(id*0merge*0merge)*1(c*0id*0id*0xjt*0id)
*1(split*0split*0id)*1(merge*0merge*0id)
*1(a*0b*0id*0yjt*0id)*1(id*0split*0split)
*1(emptyx*0bjdjyx*0ex)}
\qquad\text{and}\qquad
\sum_{s,t}
\Psivec{t}{b+d+y-c+s}{e}{x}{y}
\Phivec{s}{c}{b}{d}
\twocell{(id*0merge*0merge)*1(c*0id*0id*0xjt*0id)
*1(split*0split*0id)*1(merge*0merge*0id)
*1(a*0xcjs*0id*0yjt*0id)*1(id*0split*0split)
*1(emptyx*0bjdjyx*0ex)}\,\raisebox{-4.5ex}{,}
\]
where $c\leq b$.
Under the order $\preccurlyeq_{\mathrm{Sch}}$, the corresponding $F^4\NP{\qSPol}$-normal reduction is given by
\[
\down{D_3[\zeta]}{4}^\preccurlyeq:
\twocell{(id*0merge*0merge)*1(c*0id*0id*0x*0id)
*1(split*0split*0id)*1(merge*0merge*0id)
*1(a*0b*0id*0y*0id)*1(id*0split*0split)
*1(emptyx*0bjdjyx*0ex)}
\;\fl\;
-\sum_{t=1}
\Psivec{t}{d+y}{e}{x}{y}
\twocell{(id*0merge*0merge)*1(c*0id*0id*0xjt*0id)
*1(split*0split*0id)*1(merge*0merge*0id)
*1(a*0b*0id*0yjt*0id)*1(id*0split*0split)
*1(emptyx*0bjdjyx*0ex)}
\,+\,
\sum_{s,t}
\Psivec{t}{b+d+y-c+s}{e}{x}{y}
\Phivec{s}{c}{b}{d}
\twocell{(id*0merge*0merge)*1(c*0id*0id*0xjt*0id)
*1(split*0split*0id)*1(merge*0merge*0id)
*1(a*0xcjs*0id*0yjt*0id)*1(id*0split*0split)
*1(emptyx*0bjdjyx*0ex)}\,\raisebox{-5ex}{.}
\]
The target of $\down{D_3[\zeta]}{4}^\preccurlyeq$ is not in $\Nf(\strat{5}\NP{\qSPol})$, since we can further apply $\down{D_3[\zeta]}{4}^\preccurlyeq$ to the first term on the right-hand side above, which yields a new rule in $\strat{5}\NP{\qSPol}$, denoted by $\vartheta_1$: 
\[
\twocell{(id*0merge*0merge)*1(c*0id*0id*0x*0id)
*1(split*0split*0id)*1(merge*0merge*0id)
*1(a*0b*0id*0y*0id)*1(id*0split*0split)
*1(emptyx*0bjdjyx*0ex)}
\;\fl\;
\sum_s
\Phivec{s}{c}{b}{d}
\twocell{(id*0merge*0merge)*1(c*0id*0id*0x*0id)
*1(split*0split*0id)*1(merge*0merge*0id)
*1(a*0xcjs*0id*0y*0id)*1(id*0split*0split)
*1(emptyx*0bjdjyx*0ex)}\,\raisebox{-5ex}{,}
\]
where $c\leq b$ and $d+y>x$.

By repeating the above construction, we can apply the context $D_3$ to the rule $\vartheta_{k-1}$ inductively, thereby obtaining further rules in $\strat{5}\NP{\qSPol}$, which we denote by $\vartheta_{k}$:
\[
\begin{tikzpicture}[baseline=5ex]
\begin{scope}[x=10pt,y=10pt,join=round,cap=round,thick,black,solid,-]
\draw (0.00,5.50)--(0.00,5.25)
      (1.50,5.50)--(1.50,5.25)
      (3.50,5.50)--(3.50,5.25)
      (8.50,5.50)--(8.50,5.25);
\draw (0.00,0.00)--(0.00,-0.25)
      (1.50,0.00)--(1.50,-0.25)
      (3.50,0.00)--(3.50,-0.25)
      (8.50,0.00)--(8.50,-0.25);
\draw[fill=lightgray] (0.00,5.25)--(0.00,4.75)--cycle;
\draw[fill=green]     (1.50,5.25)--(2.00,4.75)--(1.00,4.75)--cycle;
\draw[fill=green]     (3.50,5.25)--(4.00,4.75)--(3.00,4.75)--cycle;
\draw[fill=green]     (8.50,5.25)--(9.00,4.75)--(8.00,4.75)--cycle;
\draw (0.00,4.75)--(0.00,4.50)
      (1.00,4.75)--(1.00,4.25)
      (2.00,4.75)--(2.00,4.25)
      (3.00,4.75)--(3.00,4.25)
      (4.00,4.75)--(4.00,4.25)
      (8.00,4.75)--(8.00,4.25)
      (9.00,4.75)--(9.00,4.25);

\node at (0.00,4.00) {$\scriptstyle c$};
\draw[fill=lightgray] (1.00,4.25)--(1.00,3.75)--cycle;
\draw[fill=lightgray] (2.00,4.25)--(2.00,3.75)--cycle;
\draw[fill=lightgray] (3.00,4.25)--(3.00,3.75)--cycle;
\draw[fill=lightgray] (4.00,4.25)--(4.00,3.75)--cycle;
\draw[fill=lightgray] (8.00,4.25)--(8.00,3.75)--cycle;
\draw[fill=lightgray] (9.00,4.25)--(9.00,3.75)--cycle;
\draw (0.00,3.50)--(0.00,3.25)
      (1.00,3.75)--(1.00,3.25)
      (2.00,3.75)--(2.00,3.25)
      (3.00,3.75)--(3.00,3.25)
      (7.00,3.75)--(7.00,3.25)
      (8.00,3.75)--(8.00,3.25)
      (9.00,3.75)--(9.00,3.25);
\draw[fill=orange] (0.00,3.25)--(1.00,3.25)--(0.50,2.75)--cycle;
\draw[fill=orange] (2.00,3.25)--(3.00,3.25)--(2.50,2.75)--cycle;
\draw[fill=orange] (7.00,3.25)--(8.00,3.25)--(7.50,2.75)--cycle;
\draw[fill=lightgray] (9.00,3.25)--(9.00,2.75)--cycle;
\draw (0.50,2.75)--(0.50,2.50)
      (2.50,2.75)--(2.50,2.50)
      (7.50,2.75)--(7.50,2.50)
      (9.00,2.75)--(9.00,2.50);
\draw[fill=green] (0.50,2.50)--(1.00,2.00)--(0.00,2.00)--cycle;
\draw[fill=green] (2.50,2.50)--(3.00,2.00)--(2.00,2.00)--cycle;
\draw[fill=green] (7.50,2.50)--(8.00,2.00)--(7.00,2.00)--cycle;
\draw[fill=lightgray] (9.00,2.50)--(9.00,2.00)--cycle;
\draw (0.00,2.00)--(0.00,1.75)
      (1.00,2.00)--(1.00,1.75)
      (2.00,2.00)--(2.00,1.50)
      (3.00,2.00)--(3.00,1.50)
      (7.00,2.00)--(7.00,1.50)
      (8.00,2.00)--(8.00,1.50)
      (9.00,2.00)--(9.00,1.50);

\node at (0.00,1.25) {$\scriptstyle a$};
\node at (1.00,1.25) {$\scriptstyle b$};
\draw[fill=lightgray] (2.00,1.50)--(2.00,1.00)--cycle;
\draw[fill=lightgray] (3.00,1.50)--(3.00,1.00)--cycle;
\draw[fill=lightgray] (4.00,1.50)--(4.00,1.00)--cycle;
\draw[fill=lightgray] (8.00,1.50)--(8.00,1.00)--cycle;
\draw[fill=lightgray] (9.00,1.50)--(9.00,1.00)--cycle;
\draw (0.00,0.75)--(0.00,0.50)
      (1.00,0.75)--(1.00,0.50)
      (2.00,1.00)--(2.00,0.50)
      (3.00,1.00)--(3.00,0.50)
      (4.00,1.00)--(4.00,0.50)
      (8.00,1.00)--(8.00,0.50)
      (9.00,1.00)--(9.00,0.50);
\draw[fill=lightgray] (0.00,0.50)--(0.00,0.00)--cycle;
\draw[fill=orange]    (1.00,0.50)--(2.00,0.50)--(1.50,0.00)--cycle;
\draw[fill=orange]    (3.00,0.50)--(4.00,0.50)--(3.50,0.00)--cycle;
\draw[fill=orange]    (8.00,0.50)--(9.00,0.50)--(8.50,0.00)--cycle;

\node at (5.55,3.3) {$\cdots$};
\node at (5.55,1.7) {$\cdots$};
\draw[blue, very thick, dashed] (2.2,-0.5) rectangle (9.3,5.7);

\node at (5.55,2.6) {$\scriptstyle k$};
\end{scope}
\end{tikzpicture}
\fl
\,\sum_s\,
\Phivec{s}{c}{b}{d}
\begin{tikzpicture}[baseline=5ex]
\begin{scope}[x=10pt,y=10pt,join=round,cap=round,thick,black,solid,-]
\draw (0.00,5.50)--(0.00,5.25)
      (1.50,5.50)--(1.50,5.25)
      (3.50,5.50)--(3.50,5.25)
      (8.50,5.50)--(8.50,5.25);
\draw (0.00,0.00)--(0.00,-0.25)
      (1.50,0.00)--(1.50,-0.25)
      (3.50,0.00)--(3.50,-0.25)
      (8.50,0.00)--(8.50,-0.25);
\draw[fill=lightgray] (0.00,5.25)--(0.00,4.75)--cycle;
\draw[fill=green]     (1.50,5.25)--(2.00,4.75)--(1.00,4.75)--cycle;
\draw[fill=green]     (3.50,5.25)--(4.00,4.75)--(3.00,4.75)--cycle;
\draw[fill=green]     (8.50,5.25)--(9.00,4.75)--(8.00,4.75)--cycle;
\draw (0.00,4.75)--(0.00,4.50)
      (1.00,4.75)--(1.00,4.25)
      (2.00,4.75)--(2.00,4.25)
      (3.00,4.75)--(3.00,4.25)
      (4.00,4.75)--(4.00,4.25)
      (8.00,4.75)--(8.00,4.25)
      (9.00,4.75)--(9.00,4.25);

\node at (0.00,4.00) {$\scriptstyle c$};
\draw[fill=lightgray] (1.00,4.25)--(1.00,3.75)--cycle;
\draw[fill=lightgray] (2.00,4.25)--(2.00,3.75)--cycle;
\draw[fill=lightgray] (3.00,4.25)--(3.00,3.75)--cycle;
\draw[fill=lightgray] (4.00,4.25)--(4.00,3.75)--cycle;
\draw[fill=lightgray] (8.00,4.25)--(8.00,3.75)--cycle;
\draw[fill=lightgray] (9.00,4.25)--(9.00,3.75)--cycle;
\draw (0.00,3.50)--(0.00,3.25)
      (1.00,3.75)--(1.00,3.25)
      (2.00,3.75)--(2.00,3.25)
      (3.00,3.75)--(3.00,3.25)
      (7.00,3.75)--(7.00,3.25)
      (8.00,3.75)--(8.00,3.25)
      (9.00,3.75)--(9.00,3.25);
\draw[fill=orange] (0.00,3.25)--(1.00,3.25)--(0.50,2.75)--cycle;
\draw[fill=orange] (2.00,3.25)--(3.00,3.25)--(2.50,2.75)--cycle;
\draw[fill=orange] (7.00,3.25)--(8.00,3.25)--(7.50,2.75)--cycle;
\draw[fill=lightgray] (9.00,3.25)--(9.00,2.75)--cycle;
\draw (0.50,2.75)--(0.50,2.50)
      (2.50,2.75)--(2.50,2.50)
      (7.50,2.75)--(7.50,2.50)
      (9.00,2.75)--(9.00,2.50);
\draw[fill=green] (0.50,2.50)--(1.00,2.00)--(0.00,2.00)--cycle;
\draw[fill=green] (2.50,2.50)--(3.00,2.00)--(2.00,2.00)--cycle;
\draw[fill=green] (7.50,2.50)--(8.00,2.00)--(7.00,2.00)--cycle;
\draw[fill=lightgray] (9.00,2.50)--(9.00,2.00)--cycle;
\draw (0.00,2.00)--(0.00,1.75)
      (1.00,2.00)--(1.00,1.75)
      (2.00,2.00)--(2.00,1.50)
      (3.00,2.00)--(3.00,1.50)
      (7.00,2.00)--(7.00,1.50)
      (8.00,2.00)--(8.00,1.50)
      (9.00,2.00)--(9.00,1.50);

\node at (0.00,1.25) {$\scriptstyle a$};
\node at (1.00,1.25) {$\scriptstyle c-s$};
\draw[fill=lightgray] (2.00,1.50)--(2.00,1.00)--cycle;
\draw[fill=lightgray] (3.00,1.50)--(3.00,1.00)--cycle;
\draw[fill=lightgray] (4.00,1.50)--(4.00,1.00)--cycle;
\draw[fill=lightgray] (8.00,1.50)--(8.00,1.00)--cycle;
\draw[fill=lightgray] (9.00,1.50)--(9.00,1.00)--cycle;
\draw (0.00,0.75)--(0.00,0.50)
      (1.00,0.75)--(1.00,0.50)
      (2.00,1.00)--(2.00,0.50)
      (3.00,1.00)--(3.00,0.50)
      (4.00,1.00)--(4.00,0.50)
      (8.00,1.00)--(8.00,0.50)
      (9.00,1.00)--(9.00,0.50);
\draw[fill=lightgray] (0.00,0.50)--(0.00,0.00)--cycle;
\draw[fill=orange]    (1.00,0.50)--(2.00,0.50)--(1.50,0.00)--cycle;
\draw[fill=orange]    (3.00,0.50)--(4.00,0.50)--(3.50,0.00)--cycle;
\draw[fill=orange]    (8.00,0.50)--(9.00,0.50)--(8.50,0.00)--cycle;

\node at (5.55,3.3) {$\cdots$};
\node at (5.55,1.7) {$\cdots$};
\draw[blue, very thick, dashed] (2.2,-0.5) rectangle (9.3,5.7);

\node at (5.55,2.6) {$\scriptstyle k$};
\end{scope}
\end{tikzpicture}\;\raisebox{-5ex}{,}
\]
where $c\leq b$ and the $2$-cells in the above boxes cannot be reduced by $\vartheta_i$ for $i<k$.

As a result, $\strat{5}\NP{\qSPol}$ consists of the rules $\zeta_{a,b,c,d}$, $\theta_k$, and $\vartheta_k$, together with their closure under $\strat{4}\NP{\qSPol}$-stable contexts.
Note that the rules in $\strat{4}\NP{\qSPol}$ and $\strat{5}\NP{\qSPol}$ change the shapes and the labels of $2$-cells in $\pcat{\qSPol}$, respectively. Hence, we call the rules in $\strat{4}\NP{\qSPol}$ the \emph{shaping rules} and those in $\strat{5}\NP{\qSPol}$ the \emph{labeling rules}.
In Section~\ref{S:PictureCalculus}, we give a pictorial description of the $2$-cells in $\Nf(\strat{5}\NP{\qSPol})$.

\begin{lemma}
\label{L:S4Confluent}
The extension $\strat{5}\NP{\qSPol}$ is confluent.
\end{lemma}
\begin{proof}
Let $(\sigma,\tau)$ be a $\strat{5}\NP{\qSPol}$-critical branching.
Since $\strat{5}\NP{\qSPol}$ consists only of labeling rules, the normal forms
$\down{t\sigma}{5}$ and $\down{t\tau}{5}$ have the same underlying diagrammatic shapes.
It remains to compare their labels and scalar coefficients.
For each $2$-monomial appearing in either normal form, the labels outside
the rewritten region remain unchanged, while those inside are determined
by the original labels and the summation indices.
For instance, a summand indexed by $s$ in the target of
$\zeta_{a,b,c,d}$ carries the labels $a-s$ and $c-s$.
The rules $\theta_k$ and $\vartheta_k$ extend this local relabeling
through the iterated contexts described above.
We therefore collect terms with the same fully labeled diagram and
compare their coefficients, assigning coefficient zero to a diagram
absent from either expression.
Substituting the explicit rule coefficients reduces these comparisons
to scalar identities that can be checked symbolically using
SageMath~\cite{Sagemath2021}.
Hence $\down{t\sigma}{5}=\down{t\tau}{5}$, so every critical branching
is confluent.
\end{proof}

Consequently, each $\strat{i}\NP{\qSPol}$ is left-monomial and confluent for $1\leq i\leq 5$. By Theorem~\ref{T:MainTheorem}, we obtain the main result of this section.
\begin{theorem}
\label{T:MainTheoremqSchur}
The set of normal forms $\Nf(\NP{\qSPol})$ is a hom-basis of the $q$-Schur category.
\end{theorem}

The hom-basis $\Nf(\NP{\qSPol})$ is called the \emph{normal-form basis}, and in Procedure~\ref{SSS:NormalFormHomBasisProcedure}, we give an algorithm for computing it in each hom-space using picture calculus.

\section{A picture calculus for the normal form basis}
\label{S:PictureCalculus}

Given any two 1-cells $(\mathbf a,\mathbf b)$ of the $q$-Schur category $\qSCat$, as pointed out by Brundan \cite{Brundan2025}, the set indexing the hom-basis of $\qSCat(\mathbf a,\mathbf b)$ admits several descriptions, from double coset diagrams to matrices of nonnegative integers of  whose row and column sums are the entries of $\mathbf a$ and of $\mathbf b$. Such a set of matrices will be denoted by Mat$(\mathbf a,\mathbf b)$. For example, take $\mathbf a = (2,5)$ and $\mathbf b = (3,4)$ then there are 3 matrices in Mat$((2,5),(3,4))$:
\[
A = \begin{bmatrix}
0 & 3\\
2 & 2
\end{bmatrix},
\quad
B = \begin{bmatrix}
1 & 2\\
1 & 3
\end{bmatrix},
\quad
C = \begin{bmatrix}
2 & 1\\
0 & 4
\end{bmatrix}.
\]
These sets are also in bijection with the set of double cosets $S_{\mathbf a}\setminus S_N\ /\ S_{\mathbf b}$, where $N$ is the sum of entries in $\mathbf a$ and $\mathbf b$ and $S_{\mathbf a}$ is the subgroup of the symmetric group $S_n$ given by blocks: $S_{\mathbf a} = S_{a_1}\times \cdots\times S_{a_n}$, for $\mathbf a = (a_1,...,a_n)$.

In Subsection~\ref{SS:PictureRewriting}, we introduce diagrams that we call \emph{pictures} and that slightly generalise double coset diagrams. Pictures allow to represent 2-cells in the $q$-Schur category, in a somehow easier way. We perform a rewriting procedure on the set of pictures. 
Then, in Subsection~\ref{SS:ComparsionHomBase}, we compare the normal form basis with three other bases of the $q$-Schur category. 

\subsection{Linear picture rewriting systems}
\label{SS:PictureRewriting}
For the normalization $F^\ast \NP{\qSPol}$ introduced in Section~\ref{S:NormFormHomBasisqSchurCategories}, its set of normal forms satisfies
\[
\Nf(\NP{\qSPol})
= \Nf(F^4\NP{\qSPol}) \cap \Nf(\strat{5}\NP{\qSPol})
= \Nf(F^4\NP{\qSPol}) \cap \bigl(\pcat{\qSPol_2}\setminus [\strat{5}\NP{\qSPol}]\bigr)
= \Nf(F^4\NP{\qSPol})\setminus [\strat{5}\NP{\qSPol}],
\]
where
$
[\strat{5}\NP{\qSPol}] \coloneq \{\, C[s\alpha] \mid
C \in \Cont(\pcat{\qSPol_2}) \text{ and } \alpha \in \strat{5}\NP{\qSPol} \,\}.
$
Therefore, computing the normal form basis of the $q$-Schur category
reduces to two tasks:
\begin{enumerate}
  \item computing all diagrams in $\Nf(F^4\NP{\qSPol})$ with fixed source and target;
  \item excluding diagrams that contain a subdiagram of the form $s\alpha$ for some $\alpha\in \strat{5}\NP{\qSPol}$.
\end{enumerate}
We now describe the normal form basis using pictures, giving a pictorial procedure for the normal forms.

\subsubsection{Linear picture rewriting}
We represent the $2$-cells in $\pcat{\qSPol_2}$ by using the following pictures respectively
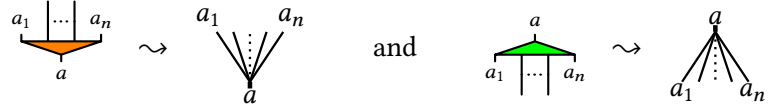
\[
 \twocell{(a1s*0dot*0ans)*1split4*1ax}
\;\leadsto\;
\aup
\qquad\text{and}\qquad
 \twocell{as*1merge4*1(a1x*0dot*0anx)}
\;\leadsto\;
\adown\,.
\]
From these two forms, every $2$-cell of $\Nf(F^3 \NP{\qSPol})$ in \SSS{E:S2-NormalForm} admits a unique picture representation. For instance,
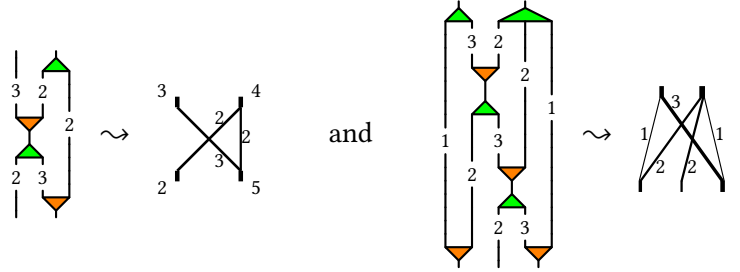
\[
\twocell{(id*0merge)*1(three*0two*0id)*1(split*0two)*1 (merge*0id)*1(two*0three*0id)*1(id*0split)}
\;\leadsto\;
\EXi
\qquad\text{and}\qquad
\twocell{(merge*0merge3)*1(id*0three*0two*0id*0id)*1(id*0split*0two*0id) *1(id*0merge*0id*0one)*1(one*0id*0three*0id*0id)*1(id*0two*0split*0id)*1(id*0id*0merge*0id) *1(id*0id*0two*0three*0id)*1(split*0id*0split)}
\;\leadsto\;
\EXii\,.
\]
The pictures corresponding to $\Nf(F^4 \NP{\qSPol})$ are called \emph{$F_4$-pictures}. They are those in $\Nf(F^3 \NP{\qSPol})$ that do not contain subpictures of the following form:
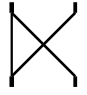
\[
\Fthree\;.
\]
We denote by
$P(\mathbf{a};\mathbf{b})$ the set of all $F^4$-pictures
with source $\mathbf{a}$ and target $\mathbf{b}$, and by $P^\ell(\mathbf{a};\mathbf{b})$ the set obtained by
linearizing the pictures of $P(\mathbf{a};\mathbf{b})$. Given a set $R$ of oriented binary relations on
$P^\ell(\mathbf{a};\mathbf{b})$, we call a pair $(P^\ell(\mathbf{a};\mathbf{b}), R)$ a \emph{linear picture rewriting system}.
\subsubsection{$F^4$-picture procedure}
\label{SSS:NewPictureAlgorithm}
We now present a procedure to compute all $F^4$-pictures with source $\mathbf{a}=(a_1,\ldots,a_n)$ and target $\mathbf{b}=(b^1,\ldots,b^m)$.
We define two families of \emph{oriented segments}
\[
(c^i : b^i \to a_n)_{1 \leq i \leq m}
\qquad\text{and}\qquad
(c_j : a_j \to b^m)_{1 \leq j \leq n},
\]
that decompose into two families oriented subsegments
\[
c^i = (c^{ik})_{1 \leq k \leq n},
\qquad\text{and}\qquad
c_j = (c_{jk})_{1 \leq k \leq m},
\]
where $c^{i1},c^{i2},\ldots,c^{in}$ (resp. $c_{j1},c_{j2},\ldots,c_{jm}$) denote the oriented subsegments resulting by the crossing with the $n$ (resp. $m$) oriented segments from $a_k$ to $b^m$ for $1\leq k \leq n$ (resp. $b^k$ to $a_n$ for $1\leq k \leq m$), as drawn on the following pictures
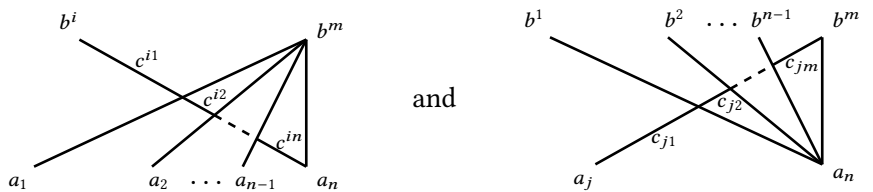
\[
\atob
\qquad\raisebox{3.2em}{$\text{and}$}\qquad
\btoa\,\raisebox{1ex}{.}
\]
For instance, when $n=m=3$, we have
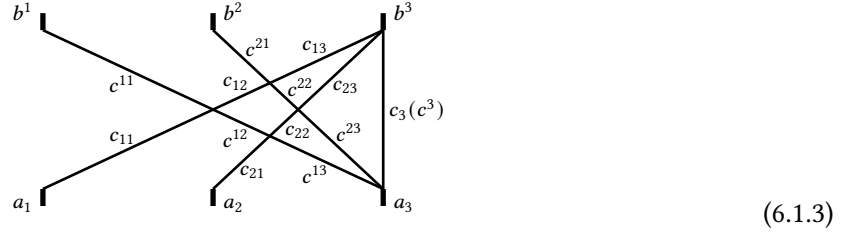
\begin{eqn}{equation}
\label{E:Example-n,m=3}
\nmthree
\end{eqn}
We regard all labels $c^i$ and $c_j$ as natural numbers, and erase the corresponding segment if the label is~$0$.
In order for them to define a pictures, the labels are required to satisfy the following relations:
\[
\left\{
\begin{aligned}
& c^m = c_n, \qquad
c^{i1} = b^i \ (1 \le i \le m-1), \qquad
c_{j1} = a_j \ (1 \le j \le n-1), \\
& c_n + \sum_{i=1}^{n-1} c_{im} = b^m, \qquad
c^m + \sum_{i=1}^{m-1} c^{in} = a_n, \\
& c^{ij} + c_{ji} = c^{i,j+1} + c_{j,i+1},
\quad (1 \le i \le m-1,\ 1 \le j \le n-1).
\end{aligned}
\right.
\]

\begin{lemma}
\label{L:PicturesCorrespondence}
There is a one-to-one correspondence between the set of pictures defined in \SSS{SSS:NewPictureAlgorithm} and $P(\mathbf{a};\mathbf{b})$.
\end{lemma}
\begin{proof}
Set $\mathbf{a}=(a_1,\ldots,a_n)$ and $\mathbf{b}=(b^1,\ldots,b^m)$.
It is straightforward that all pictures defined in \SSS{SSS:NewPictureAlgorithm} are $F^4$-pictures.
Moreover, any additional connecting segment between $b^i$ and $a_j$ with $i<m$ and $j<n$ yields a subpicture of the form 
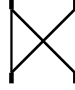
\[
\Fthree
\]
 and hence does not correspond to any picture in $P(\mathbf{a};\mathbf{b})$. 
For example, if we connect the points $a_1$ and $b_2$ in \eqref{E:Example-n,m=3}, the resulting picture contains the subpicture highlighted in red below.
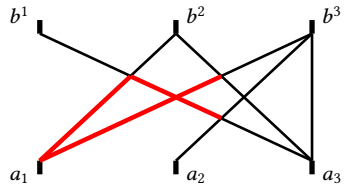
\[
\counterexample\,.
\]
Therefore, the result follows.
\end{proof}

\subsubsection{Collapsing rules}
\label{SSS:CollapsingRules}
We depict the source and target of rules in $\strat{5} \NP{\qSPol}$ as $F^4$-pictures, and define the associated rewriting rules on $P^\ell$:
\[
\Gamma_1:
\sGammaone
\fl
\sum_s
\Phivec{s}{c}{b}{d}
\tGammaone
\qquad
\text{where } c\leq b,
\]
\[
\Gamma^k:
\sGammak
\quad
\fl
\sum_t
\Phivec{t}{c+m}{y}{b-y}
\quad
\tGammak
\qquad
\text{where } c+m\leq y,
\]
\[
\Gamma_k:
\sGammakprime
\quad
\fl
\sum_s
\Phivec{s}{c}{b}{d}
\quad
\tGammakprime
\qquad
\text{where } c\leq b.
\]
where $k$ is the number of iterations.
We define the set
\[
\mathbf{Col}=\{C_1[\Gamma_1],C^k[\Gamma^k],C_k[\Gamma_k]\}
\]
 for all contexts $C_1, C^k, C_k$ such that $C_1[s\Gamma_1], C^k[s\Gamma^k], C_k[s\Gamma_k]$ are $F^4$-pictures.
Since $\Gamma_1$, $\Gamma^k$, and $\Gamma_k$ correspond to labeling rules in $\strat{5} \NP{\qSPol}$, the pictures $C_1[t\Gamma_1]$, $C^k[t\Gamma^k]$, and $C_k[t\Gamma_k]$ are also $F^4$-pictures. Therefore, every rule in $\mathbf{Col}$ is a rewriting rule in $P^\ell$.

We call the rules in $\mathbf{Col}$ the \emph{collapsing rules}. It follows from Lemma~\ref{L:S4Confluent} that the linear picture rewriting system $(P^\ell(\mathbf{a};\mathbf{b}), \mathbf{Col})$ is convergent.

\subsubsection{Normal form basis procedure}
\label{SSS:NormalFormHomBasisProcedure}
The following procedure computes the normal form basis of the $q$-Schur category in terms of pictures.

\begin{algorithm}[H]
\DontPrintSemicolon
\KwIn{Two tuples $\mathbf{a}=(a_1,\ldots,a_n)$ and $\mathbf{b}=(b^1,\ldots,b^m)$ with $a_1+\cdots+a_n=b^1+\cdots+b^m$.}
\KwOut{The normal form basis pictures with source $\mathbf{a}$ and target $\mathbf{b}$.}

$T \leftarrow \varnothing$\;

\ForEach{$p\in P(\mathbf{a};\mathbf{b})$}{
  \eIf{$p\in \Nf(\mathbf{Col})$}{
    $T \leftarrow T \cup \{p\}$\;
  }{
    $T \leftarrow T \cup \operatorname{supp}\bigl(\down{p}{\mathbf{Col}}\bigr)$\;
  }
}

\Return{$T$}\;

\end{algorithm}
When the tuples $\mathbf{a}$ and $\mathbf{b}$ are fixed, the set $P(\mathbf{a};\mathbf{b})$ is finite by the $F^4$-picture procedure in \SSS{SSS:NewPictureAlgorithm}. Hence, Procedure~\eqref{SSS:NormalFormHomBasisProcedure} always terminates. 
By Theorem~\ref{T:MainTheoremqSchur}, we obtain the following result.
\begin{proposition} 
For any tuples $\mathbf{a}$ and $\mathbf{b}$, Procedure~\eqref{SSS:NormalFormHomBasisProcedure} computes a linear basis of the hom-space $\qSCat_2(\mathbf{a},\mathbf{b})$.
\end{proposition}

\subsection{Comparison of hom-bases for the $q$-Schur category}
\label{SS:ComparsionHomBase}

In this subsection, we compute four types of linear bases in hom-space $\qSCat\bigl((2,5);(3,4)\bigr)$ and determine the transition matrices between them. Following $F^4$-picture procedure, we list below the $F_4$-pictures in $P\bigl((2,5);(3,4)\bigr)$ together with their labels
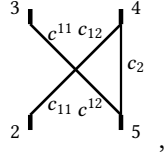
\[
\twofivetothreefour\,,
\]
where $c^{11}=3$, $c_{11}=2$, $c^{11}+c_{12}=c_{11}+c^{12}$, $c_{12}+c_2=4$, and $c^{12}+c_2=5$.
By computation, we obtain the following solutions:
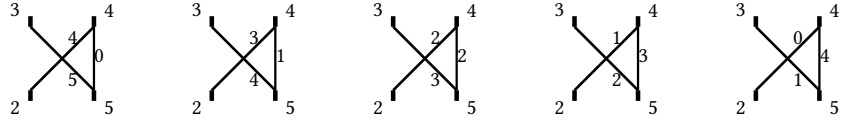
\[
\twofivetothreefouri
\qquad 
\twofivetothreefourii
\qquad
\twofivetothreefouriii
\qquad
\twofivetothreefouriv
\qquad
\twofivetothreefourv\,.
\]
Among them, the following pictures belong to $\Nf(\mathbf{Col})$
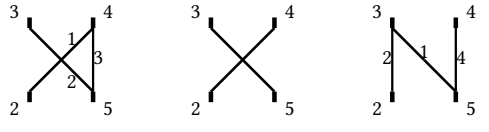
\[
\NFBi
\qquad
\NFBii
\qquad
\NFBiii\,,
\]
which is the normal form basis in $\qSCat((2,5);(3,4))$. 

These pictures correspond to diagrams in $\Nf(F^3 \NP{\qSPol})$. They satisfy the \emph{collapsing relations} induced by the collapsing rules of \SSS{SSS:CollapsingRules}, together with the following two relations:
\[
\sFthree
=
\sum_s\Psivec{s}{a}{b}{c}{d}\;
\tFthree
\qquad\text{and}\qquad
\begin{array}{c}
\sFtwo
\\[0.01em]
\tFtwo
\end{array}
=
\prod_{i=2}^{m}
\left[\begin{array}{c}
a_1+\cdots + a_i \\
a_i
\end{array}\right]_q\,
\ttFtwo
.
\]
These two relations correspond to the rules $\gamma$ and $\gamma^\diamond$ in $\strat{4}\qSPol$ and $\strat{3}\qSPol$, respectively.
Therefore, we can use collapsing relations to compute the transition matrices between the normal form basis and other bases of the $q$-Schur category.

\subsubsection{Standard basis}

In the very first definition of the $q$-Schur category $\qSCat$ given by Brundan {\cite[Theorem 5.1]{Brundan2025}}, the hom-spaces are equipped with a \emph{standard basis}, in the spirit of Iwahori-Hecke algebras. More precisely, given $\mathbf a$ and $\mathbf b$ two 1-cells, $\qSCat(\mathbf a,\mathbf b)$ is the free $\mathbb Z[q,q^{-1}]$-module with basis $(\xi_A)$, where $A$ runs over the set Mat$(\mathbf a,\mathbf b)$.

In the example $(\mathbf a,\mathbf b) = \bigl((2,5);(3,4)\bigr)$, for the standard basis, we get the three pictures
\[
\xi_A =
\raisebox{-0.8cm}{ \SBi} ,
\qquad
\xi_B =\raisebox{-0.8cm}{\SBii} ,
\qquad
\xi_C =\raisebox{-0.8cm}{\SBiii} .
\]
The transition matrix is given by:
\[
\left(
\SBi[baseline=-0.1em]
,
\SBii[baseline=-0.1em]
,
\SBiii[baseline=-0.1em]\right)
=
\left(
\NFBi[baseline=-0.1em]
,
\NFBii[baseline=-0.1em]
,
\NFBiii[baseline=-0.1em]
\right)
\setlength{\arraycolsep}{8pt}
\begin{bmatrix}
\Phi_1u & 1 & 0 \\
\Phi_{-2}u & 0 & 0 \\
\Phi_2u & \left[\begin{array}{c} 4 \\ 3 \end{array}\right]_q \Psi_1v & 1
\end{bmatrix},
\]
where $u = [3\;3\;2]^\top$ and $v=[2\;2\;2\;1]^\top$.

\subsubsection{Canonical basis}

Given two 1-cells $(\mathbf a,\mathbf b)$ in the $q$-Schur category, there is only one element of maximal length $d_A^+$ in the double coset associated to $A$ in $S_{\mathbf a}\setminus S_N/S_{\mathbf b}$ (here the length $\ell$ is in terms of the Coxeter generators of the symmetric group).

The \emph{canonical basis} of the hom-space $\qSCat(\mathbf a,\mathbf b)$ is given as a linear combination in the standard basis as such
\[
\Theta_A = \sum_{B\in \hbox{Mat} (\mathbf a,\mathbf b)} P_{A,B} (q)\xi_B
\] with $P_{A,B}$ is the polynomial defined by 
\[
P_{A,B} = q^{\ell(d_A^+) - \ell(d_B^+)} P_{d_A^+,d_B^+} (q^{-2}),
\] where $P_{d_A^+,d_B^+}$ is the Kazhdan-Lusztig polynomial associated to the two permutations $d_A^+$ and $d_B^+$.

In the case $((2,5),(3,4))$, the computations of the 3 maximal length elements in the 3 double cosets associated to $A$, $B$ and $C$ give 
\[
\begin{array}{rcl}
d_A^+ & = &s_1s_2s_1s_3s_2s_1s_4s_3s_2s_1s_5s_4s_3s_2s_1s_6s_5s_4s_3s_2s_1,\\
d_B^+ & = & s_1s_2s_1s_4s_3s_6s_5s_4s_3s_5s_6s_5s_2s_3s_1s_4s_3s_5,\\
d_C^+ & = & s_1s_2s_1s_4s_5s_4s_6s_3s_4s_5s_4s_3s_6.
\end{array}
\] of respective lengths $\ell(d_A^+) = 21$ (the maximal one), $\ell(d_B^+) = 18$ and $\ell(d_C^+) = 13$. Then we get the Kazhdan-Lusztig polynomials,
\[
P_{A,B} = q^3(1), \quad P_{A,C} = q^8 (1 ), \quad P_{B,C} = q^5(1+q^{-2}),
\] using the computer algebraic system SageMath~\cite{Sagemath2021} for the computation of the last one. Therefore, the canonical basis reads as follows
\[
\begin{array}{rcl}
\Theta_A & = & \xi_A + q^3\xi_B + q^8\xi_C,\\
\Theta_B & = & \xi_B + (q^5 + q^3)\xi_C,\\
\Theta_C & = & \xi_C
\end{array}
\] and the transition matrix with the standard basis is upper triangular. 

\subsubsection{Codeterminant basis}
In \cite{Brundan2025}, Brundan also introduces a third basis, called the \emph{codeterminant basis}. Its definition uses yet another description of the double cosets in $S_{\mathbf a}\setminus S_N/S_{\mathbf b}$ that we do not recall here. But in the case $((2,5),(3,4))$, the computations of pictures associated to the codeterminant basis give 
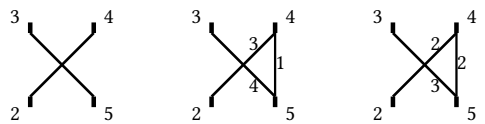
\[
\CoBi
\qquad
\CoBii
\qquad
\CoBiii .
\]

And the transition matrix is given by
\[
\left(
\CoBii[baseline=-0.1em]
,
\CoBi[baseline=-0.1em]
,
\CoBiii[baseline=-0.1em]
\right)
=
\left(
\NFBi[baseline=-0.1em]
,
\NFBii[baseline=-0.1em]
,
\NFBiii[baseline=-0.1em]
\right)
\setlength{\arraycolsep}{8pt}
\begin{bmatrix}
\Phi_1 u & 0 & \Phi_{1}v \\
\Phi_{-2} u & 1 & \Phi_{-2}v \\
\Phi_{2} u & 0 &\Phi_{2}v
\end{bmatrix},
\]
where $u = [3\;4\;1]^\top$ and $v=[3\;3\;2]^\top$.

\medskip

\noindent {\bf Acknowledgements.}
This work was supported by the China Scholarship Council (CSC) under Grant No.~202406140026.

\noindent {\bf Use of computer systems.}
We used the computer algebra system SageMath~\cite{Sagemath2021} to compute \linebreak $q$-binomial coefficients arising in the confluence proofs. Artificial intelligence tools were used for proofreading, literature searches, and assistance with the preparation of formal code and diagrams. They were not used to produce or establish any mathematical results. The authors adhere to the \emph{Leiden Declaration on Artificial Intelligence and Mathematics}. Diagrams were typeset using Catex~\cite{catex}.

\begin{small}
\renewcommand{\refname}{\Large\textsc{References}}
\bibliographystyle{plain}
\bibliography{biblioCURRENT}
\end{small}

\vfill

\begin{footnotesize}

\bigskip
\auteur{Stéphane Gaussent}{stephane.gaussent@univ-st-etienne.fr}
{Université Jean Monnet\\
CNRS, Institut Camille Jordan, UMR5208\\
F-42023 Saint-Etienne, France}

\bigskip
\auteur{Zuan Liu}{zliu@math.univ-lyon1.fr}
{Université Claude Bernard Lyon 1\\
CNRS, Institut Camille Jordan, UMR5208\\
F-69622 Villeurbanne, France}

\bigskip
\auteur{Philippe Malbos}{malbos@math.univ-lyon1.fr}
{Université Claude Bernard Lyon 1\\
CNRS, Institut Camille Jordan, UMR5208\\
F-69622 Villeurbanne, France}
\end{footnotesize}

\vspace{1.5cm}

\begin{small}---\;\;\today\;\;-\;\;\hhmm\;\;---\end{small}

\clearpage

\appendix

\section{Appendices}

\subsection{Normalization of the contexts $C_1$, $C_2$, and $C_3$}
\label{SS:ProofsOfThreeContexts}
We provide a detailed verification of $\down{C_1[\partial\gamma]}{3}=0$ and $\down{C_2[\partial\gamma]}{3}=0$, and show that the rule $\down{C_3[\gamma]}{3}^\preccurlyeq$ is trivial in \SSS{SSS:R3StratifiedCompletionStep}.
\begin{enumerate}
\item We have 
\[
C_1[\gamma_{a,b,c,d}]:
\twocell{merge*1(merge *0 id) *1 (id *0 c *0 id)  *1 (id *0 split)*1(id*0 merge)*1 (id *0 d *0 id) *1(split *0 id)  *1(ax *0 bx)}
\fl
\sum_s\Psivec{s}{a}{b}{c}{d}
\twocell{merge*1(id *0 merge)*1 (id *0 djs *0 id)  *1 (split*0 id)*1 (merge*0 id) *1 (id *0 cjs *0 id)  *1 (id *0 split) *1(ax *0 bx)}\,\raisebox{-5ex}{,}
\]
where $\down{C_1[s\gamma_{a,b,c,d}]}{3}=
\left[\begin{array}{c}
b+d \\
b+d-c
\end{array}\right]_q
\left[\begin{array}{c}
a \\
d
\end{array}\right]_q  \twocell{merge*1(ax*0bx)}$ and
\[
\down{C_1[t\gamma_{a,b,c,d}]}{3}=
\sum_s\Psivec{s}{a}{b}{c}{d}
\left[\begin{array}{c}
a+c-s \\
d-s
\end{array}\right]_q
\left[\begin{array}{c}
b \\
b-c+s
\end{array}\right]_q \twocell{merge*1(ax*0bx)}\,.
\]
A direct computation using the computer algebraic system SageMath~\cite{Sagemath2021} yields
\[
\left[\begin{array}{c}
b+d \\
b+d-c
\end{array}\right]_q
\left[\begin{array}{c}
a \\
d
\end{array}\right]_q =\sum_s\Psivec{s}{a}{b}{c}{d}
\left[\begin{array}{c}
a+c-s \\
d-s
\end{array}\right]_q
\left[\begin{array}{c}
b \\
b-c+s
\end{array}\right]_q.
\]
Consequently, one has $\down{C_1[\partial\gamma_{a,b,c,d}]}{3} = 0$. 
 Alternatively, if they were not equal, then in the $q$-Schur category we would have the $1$-generators $\twocell{merge*1(ax*0bx)} = 0$ for all $a,b > 0$, which is a clear contradiction.
Similarly, one proves that
$\down{C_2[\partial\gamma_{a,b,c,d}]}{3} = 0$.

\item Applying the context
$C_3= \raisebox{-12.50pt}{\begin{tikzpicture} \begin{scope} [ x = 10pt, y = 10pt, join = round, cap = round, thick, black, solid, -] \draw (0.50,2.75)--(0.50,2.50) (2.00,2.75)--(2.00,2.50) ; \draw (0.50,0.00)--(0.50,-0.25) (2.00,0.00)--(2.00,-0.25) ; \draw [fill = green] (0.50,2.50)--(1.00,2.00)--(0.00,2.00)--cycle ; \draw [fill = lightgray] (2.00,2.50)--(2.00,2.00)--cycle ; \draw (0.00,2.00)--(0.00,1.75) (1.00,2.00)--(1.00,1.75) (2.00,2.00)--(2.00,1.75) ; \node at (0.00,1.25) {$\scriptstyle e$} ; \draw [rounded corners = 1pt, fill = white] (0.75,0.75) rectangle (2.25,1.75) ; \node at (1.50,1.25) {$\scriptstyle \square$} ; \draw (0.00,0.75)--(0.00,0.50) (1.00,0.75)--(1.00,0.50) (2.00,0.75)--(2.00,0.50) ; \draw [fill = orange] (0.00,0.50)--(1.00,0.50)--(0.50,0.00)--cycle ; \draw [fill = lightgray] (2.00,0.50)--(2.00,0.00)--cycle ; \end{scope} \end{tikzpicture}}$
to $\gamma_{a,b,c,d}$, we obtain
\[
\twocell{(merge *0 id) *1 (id *0 c *0 id)  *1 (id *0 split)*1(id*0 merge)*1 (id *0 d *0 id) *1(split *0 id)*1(ajex*0bx)}
\xrightarrow{\scriptstyle \down{C_3[\gamma_{a,b,c,d}]\,}{3}^\preccurlyeq}
\sum_s\Psivec{s}{a}{b}{c}{d}
\twocell{(merge *0 merge)*1 (id *0id *0 djs *0 id)  *1 (e *0split*0 id)*1 (id *0merge*0 id) *1 (id *0a *0 cjs *0 id)  *1 (split *0 split)*1(emptyx*0bx) }
\fl
\sum_{s,t}\Psivec{s}{a}{b}{c}{d}\Psivec{t}{a+e}{c-s}{a+c-d}{a}
\twocell{(id*0merge)*1(id*0merge*0id) *1(id*0ajt*0id*0id)*1(split*0id*0bjcjs) *1(merge*0id*0id)*1(id*0ajcjdjt*0id*0id) *1(id*0split*0id)*1(id*0cjs*0id)*1(aje*0split)}
\]
\[
\fl
\sum_{s, t}\Psivec{s}{a}{b}{c}{d}\Psivec{t}{a+e}{c-s}{a+c-d}{a}
\left[\begin{array}{c}
b+d-a-c+t \\
b-c+s
\end{array}\right]_q
\twocell{(id *0 merge)*1 (id *0 ajt *0 id)  *1 (split*0 id)*1 (merge*0 id) *1 (id *0 ajcjdjt *0 id)  *1 (id *0 split) *1(ajex *0 bx)}\,\raisebox{-5ex}{.}
\]
Using the computer algebraic system SageMath~\cite{Sagemath2021}, one can check
\[
\Psivec{s}{a}{b}{c}{d}\Psivec{t}{a+e}{c-s}{a+c-d}{a}
\left[\begin{array}{c}
b+d-a-c+t \\
b-c+s
\end{array}\right]_q
=
\Psivec{r}{a+e}{b}{c}{d}
\]
for all $r=d+t-a$.
Since
\[
\gamma_{a+e,b,c,d}:
\twocell{(merge *0 id) *1 (id *0 c *0 id)  *1 (id *0 split)
*1(id*0 merge)*1 (id *0 d *0 id) *1(split *0 id)*1(ajex*0bx)}
\fl
\sum_r
\Psivec{r}{a+e}{b}{c}{d}
\twocell{(id *0 merge)*1 (id *0 djr *0 id)  *1 (split*0 id)
*1 (merge*0 id) *1 (id *0 cjr *0 id)
*1 (id *0 split) *1(ajex *0 bx)}
\]
belongs to $\strat{4} \qSPol$, it follows that $\down{C_3[\gamma_{a,b,c,d}]}{3}^\preccurlyeq$
is trivial.
\end{enumerate}

\subsection{Confluence of the extension $\strat{4}\NP{\qSPol}$}
\label{ASSS:CriticalBranchingsOfS3}
In this subsection,  we prove the confluence diagrams for two families of critical branchings listed in Lemma~\ref{L:ConfluentOfS3}.
\begin{enumerate}
\item 
The confluent diagram of critical branching
$\left(\gamma
\begin{bmatrix}
b+d\\f\\
b+d-c+f-g\\ b+d-c
\end{bmatrix},
\gamma
\begin{bmatrix}
a\\
b\\
c\\
d
\end{bmatrix}
\right)
$\,:
\[
\xymatrix @R=1em {
\twocell{(emptys*0 gs)*1(merge3*0id)*1(id*0c*0split) *1(id*0id*0merge)*1(id*0split*0id)*1 (id*0merge*0id)*1(id*0 d*0id*0id)*1(split*0id*0id)*1(ax*0bx*0fx)}
  \ar[r]^{\hspace{-0.4cm}}
  \ar[d]_{\hspace{0.5cm}}
&
\sum_s\Psi_s
\twocell{(emptys*0 gs)*1(merge*0merge)*1(id*0split*0id) *1(id*0merge3*0gbdcs)*1(id*0d*0id*0id*0id)*1(split*0id*0split)*1(ax*0bx*0fx)}
  \ar[r]^{\hspace{-0.4cm}}
&
\sum_{s,t}\Psi_t\Psi_s
\twocell{(emptys*0 gs)*1(id*0merge3)*1(id*0djt*0id*0id)*1(split*0id*0id) *1(merge*0id*0gbdcs)*1(id*0split*0id)*1 (id*0merge*0id)*1(id*0id*0split)*1(ax*0bx*0fx)}
  \ar@{=}[d]
\\
\sum_n\Psi_n
\twocell{(emptys*0 gs)*1(merge*0id)*1(id*0split)*1 (id*0merge3)*1(id*0djn*0id*0id)*1(split*0id*0id)*1(merge*0id*0id) *1(id*0cjn*0id*0id)*1(id*0split*0id)*1(ax*0bx*0fx)}
  \ar[r]^{\hspace{-0.4cm}}
&
\sum_{n,m}\Psi_m\Psi_n
\twocell{(emptys*0 gs)*1(id*0merge)*1(id*0djnjm*0id)*1(split*0id)*1 (merge3*0id)*1(id*0cjn*0split)*1(id*0id*0merge) *1(id*0split*0id)*1(ax*0bx*0fx)}
  \ar[r]^{\hspace{-0.4cm}}
&
\sum_{n,m,l}\Psi_l\Psi_m\Psi_n
\twocell{(emptys*0 gs)*1(id*0merge3)*1(id*0djnjm*0id*0id)*1(split*0id*0id) *1(merge*0id*0id)*1(id*0split*0gbdcml)*1 (id*0merge*0id)*1(id*0id*0split)*1(ax*0bx*0fx)}
}
\]
where
\[
\Psi_s\coloneq\Psivec{s}{b+d}{f}{b+d-c+f-g}{b+d-c},
\quad
\Psi_t\coloneq\Psivec{t}{a}{f-g+2b+d-c-s}{f-g+b+d}{d},
\]
\[
\Psi_n\coloneq\Psivec{n}{a}{b}{c}{d},
\quad
\Psi_m\coloneq\Psivec{m}{a+c-n}{b-c+n+f}{b-c+f+d-g}{d-n},
\quad
\Psi_l\coloneq\protect\Psivec{l}{b}{f}{f+b+d-c-g-m}{b-c+n}.
\]
Set $t = n + m$ and $s = m + l$.
With this choice, the last two diagrams above have the same shape.
Using the computer algebraic system SageMath~\cite{Sagemath2021}, one can verify that
$
\Psi_t \Psi_s = \Psi_l \Psi_m \Psi_n, 
$
for all $t = n + m$ and $s = m + l$.
Therefore, this family of critical branchings is confluent.

\item 
The confluent diagram of critical branching
$
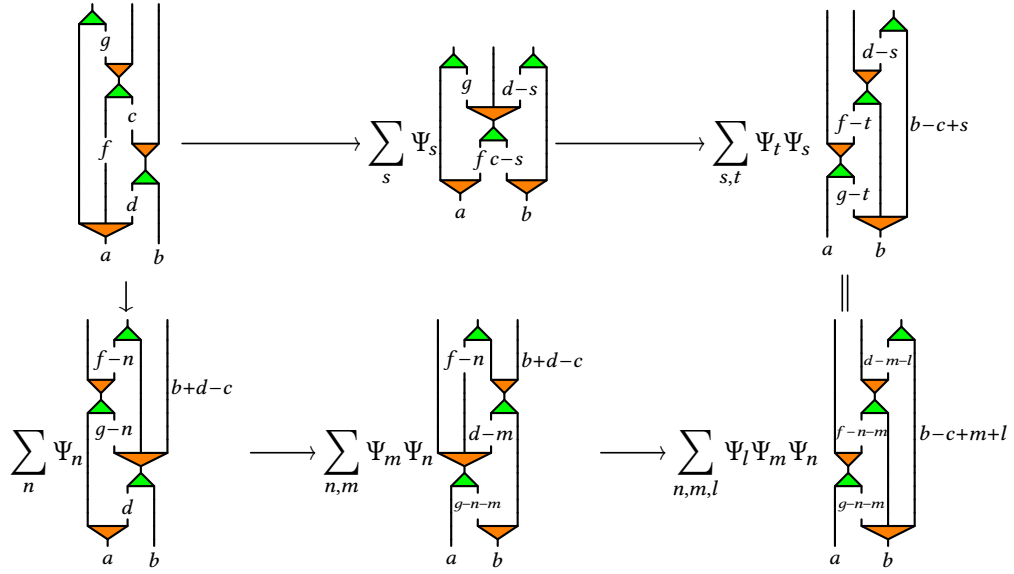
\left(
\gamma
\begin{bmatrix}
f+d \\b\\
c\\d
\end{bmatrix},
\gamma
\begin{bmatrix}
a-d\\c\\
g\\f
\end{bmatrix}\right)
$\,:
\[
\xymatrix @R=1em {
\twocell{(merge*0id*0id)*1(id*0oneg*0id*0id)*1(id*0split*0id) *1(id*0merge*0id)*1(id*0id*0c*0id)*1(id*0onef*0split)*1(id*0id*0merge)*1 (id*0id*0d*0id)*1(split3*0id)*1(ax*0bx)}
  \ar[r]^{\hspace{-0.4cm}}
  \ar[d]_{\hspace{0.5cm}}
&
\sum_s\Psi_s
\twocell{(merge*0id*0merge)*1(id*0oneg*0id*0djs*0id)*1(id*0split3*0id) *1(id*0merge*0id)*1(id*0onef*0cjs*0id)*1(split*0split) *1(ax*0bx)}
  \ar[r]^{\hspace{-0.4cm}}
&
\sum_{s,t}\Psi_t\Psi_s
\twocell{(id*0id*0merge)*1(id*0id*0djs*0id)*1(id*0split*0id)*1 (id*0merge*0id)*1(id*0fjt*0id*0bcs)*1(split*0id*0id)*1 (merge*0id*0id)*1(id*0gjt*0id*0id)*1(id*0split3)*1(ax*0bx)}
  \ar@{=}[d]
\\
\sum_n\Psi_n
\twocell{(id*0merge*0id)*1(id*0fjn*0id*0id)*1(split*0id*0bdc) *1(merge*0id*0id)*1(id*0gjn*0id*0id)*1(id*0split3)*1(id*0merge) *1(id*0d*0id)*1(split*0id)*1(ax*0bx)}
  \ar[r]^{\hspace{-0.4cm}}
&
\sum_{n,m}\Psi_m\Psi_n
\twocell{(id*0merge*0id)*1(id*0fjn*0id*0bdc)*1(id*0id*0split) *1(id*0id*0merge)*1(id*0id*0djm*0id)*1(split3*0id)*1(merge*0id) *1(id*0gjnjm*0id)*1(id*0split)*1(ax*0bx)}
  \ar[r]^{\hspace{-0.4cm}}
&
\sum_{n,m,l}\Psi_l\Psi_m\Psi_n
\twocell{(id*0id*0merge)*1(id*0id*0djmjl*0id)*1(id*0split*0id)*1 (id*0merge*0id)*1(id*0fjnjm*0id*0bcml)*1(split*0id*0id)*1 (merge*0id*0id)*1(id*0gjnjm*0id*0id)*1(id*0split3)*1(ax*0bx)}
}
\]
where
\[
\Psi_s\coloneq
\Psivec{s}{a}{b}{c}{d},
\quad
\Psi_t\coloneq
\Psivec{t}{a}{c-s}{g}{f},
\]
\[
\Psi_n\coloneq
\Psivec{n}{a-d}{c}{g}{f},
\quad
\Psi_m\coloneq
\Psivec{m}{a}{b}{g-n}{d},
\quad
\Psi_l\coloneq
\Psivec{l}{a+g-n-m}{b-g+n+m}{c+n-g}{d-m}.
\]
Set $t = n + m$ and $s = m + l$.
With this choice, the last two diagrams above have the same shape.
Using the computer algebraic system SageMath~\cite{Sagemath2021}, we prove that
$
\Psi_t \Psi_s = \Psi_l \Psi_m \Psi_n, 
$
for all $t = n + m$ and $s = m + l$.
Therefore, this family of critical branchings is confluent.
\end{enumerate}
\end{document}